\documentclass[11pt]{amsart}

\usepackage{parskip}
\allowdisplaybreaks

\usepackage{geometry} 
\usepackage{framed} 

\usepackage{tikz}
\tikzset{every picture/.style={line width=1pt}} 

\usepackage{setspace}
\setdisplayskipstretch{2.5}

\usepackage{amsthm} 
\usepackage{amsmath}
\usepackage{amssymb}
\usepackage{mathrsfs}

\usepackage{hyperref}

\usepackage{amsfonts} 
\newcommand\A{\mathcal{A}}

\newcommand\CC{\mathcal{C}}

\newcommand\BC{\mathbf{C}}
\newcommand\CE{\mathcal{E}}
\newcommand\CF{\mathcal{F}}

\newcommand\G{\mathcal{G}}
\renewcommand\H{\mathcal{H}}

\newcommand\N{\mathbb{N}}

\newcommand\R{\mathbb{R}}
\renewcommand\S{\mathcal{S}}
\newcommand\V{\mathbb{V}}
\newcommand\Z{\mathbb{Z}}

\newcommand\one{\mathbf{1}}

\newcommand\w{\omega}
\newcommand\vphi{\varphi}
\renewcommand\phi{\vphi}
\newcommand\eps{\varepsilon}

\newcommand\Etilt{\mathbf{\hat{E}}}

\usepackage{stmaryrd} 

\newcommand\id{\textnormal{id}}
\newcommand\sing{\textnormal{sing}}
\newcommand\reg{\textnormal{reg}}
\newcommand\spt{\textnormal{spt}}
\newcommand\Lip{\textnormal{Lip}}
\newcommand\dist{\textnormal{dist}}
\newcommand\graph{\textnormal{graph}}
\newcommand\ext{\mathrm{d}}
\newcommand\del{\partial}

\newcommand{\res}{\mathbin{\hspace{0.1em}\vrule height 1.3ex depth 0pt width 0.13ex\vrule height 0.13ex depth 0pt width 1.0ex}} 
\newcommand\level{\textnormal{level}}
\newcommand\grad{\textnormal{grad}}

\newcommand{\weakly}{\rightharpoonup}
\renewcommand{\div}{\textnormal{div}}

\makeatletter
\def\@tocline#1#2#3#4#5#6#7{\relax
  \ifnum #1>\c@tocdepth 
  \else
    \par \addpenalty\@secpenalty\addvspace{#2}%
    \begingroup \hyphenpenalty\@M
    \@ifempty{#4}{%
      \@tempdima\csname r@tocindent\number#1\endcsname\relax
    }{%
      \@tempdima#4\relax
    }%
    \parindent\z@ \leftskip#3\relax \advance\leftskip\@tempdima\relax
    \rightskip\@pnumwidth plus4em \parfillskip-\@pnumwidth
    #5\leavevmode\hskip-\@tempdima
      \ifcase #1
       \or\or \hskip 1em \or \hskip 2em \else \hskip 3em \fi%
      #6\nobreak\relax
    \dotfill\hbox to\@pnumwidth{\@tocpagenum{#7}}\par
    \nobreak
    \endgroup
  \fi}
\makeatother

\newtheoremstyle{newtheoremstyle}
{3pt}
{3pt}
{\itshape}
{\parindent}
{\bfseries}
{.}
{0.5em}
{} 

\newtheoremstyle{newtheoremstyledefn}
{3pt}
{3pt}
{}
{\parindent}
{\bfseries}
{.}
{0.5em}
{}

\theoremstyle{newtheoremstyle}
\newtheorem{theorem}{Theorem}
\newtheorem*{theorem*}{Theorem}
\newtheorem{lemma}[theorem]{Lemma}
\newtheorem{prop}[theorem]{Proposition}
\newtheorem{corollary}[theorem]{Corollary}

\newtheorem{thmx}{Theorem}

\theoremstyle{newtheoremstyledefn}
\newtheorem{defn}[theorem]{Definition}

\numberwithin{equation}{section} 
\numberwithin{theorem}{section}

\newtheorem{remark}[theorem]{Remark}

\usepackage{fancyhdr}
\begin{document}

\title{A Branch Set Stratification for Stationary Varifolds with Epsilon-Regularity}

\author{
	Brian Krummel
	\and
	Paul Minter
	\and
	Neshan Wickramasekera
}

\address{\textnormal{School of Mathematics \& Statistics, University of Melbourne}}
\email{brian.krummel@unimelb.edu.au}
\address{\textnormal{Department of Pure Mathematics and Mathematical Statistics, University of Cambridge}}
\email{pdtwm2@cam.ac.uk}
\address{\textnormal{Department of Pure Mathematics and Mathematical Statistics, University of Cambridge}}
\email{N.Wickramasekera@dpmms.cam.ac.uk}

\begin{abstract}
Suppose $\mathcal{V}$ is a class of stationary integral $n$-varifolds in $B^{n+k}_2(0)\subset\R^{n+k}$ which
is closed under weak limits, homotheties, rotations, and disjoint decomposition, and suppose that $\mathcal {V}$ satisfies an $\eps$-regularity property near planes of (integer) multiplicity $\leq Q\in \{2,3,\dotsc\}$. This last condition, more precisely, requires that there be a constant $\epsilon = \epsilon({\mathcal V}, Q) \in (0, 1)$ such that if $V\in \mathcal{V}$ is, in the unit cylinder ${\mathbb R}^{k} \times B_{1}^{n}(0)$, $\epsilon$-close as varifolds to the plane $\{0\} \times {\mathbb R}^{n}$ taken with multiplicity $\leq Q$ then, in the half-cylinder ${\mathbb R}^{k} \times B_{1/2}^{n}(0)$, $V$ is represented by the graph of a Lipschitz multi-valued function over $B_{1/2}^{n}(0)$ with uniform quantitative estimates of a $C^{1,\alpha}$ nature. For any varifold in such a class $\mathcal{V}$, we prove that the set of branch points with density $\leq Q$ has Hausdorff dimension $\leq n-2$.  
	
By choosing suitable $\mathcal{V}$, a direct consequence of this result and the recently established regularity theorems in \cite{BKMW25, MW24} is that if $V$ is a stationary integral $n$-varifold which is either: (a) represented by the graph of a $2$-valued Lipschitz function; or (b) codimension one, stable, and with no classical singularities of density $<Q$,  then the Hausdorff dimension of the density $Q$ branch set ($Q=2$ in (a)) is at most $n-2$. 
    
The result in (a)  together with the earlier work of Becker-Kahn \cite{BK17} improves on   the recent work of Hirsch--Spolaor \cite{HS24} (which proved the singular-set dimension upper bound $n-1$ for stationary 2-valued Lipschitz graphs) by providing a structural decomposition of the singular set as the union of a piece contained within an $(n-1)$-dimensional $C^{1, \alpha}$ submanifold and a piece that is a closed set of Hausdorff dimension $\leq n-2$. 
    
The varifold class in (b) includes as a special case area minimising hypersurfaces mod $p;$  thus, our  result in this case together with \cite{MW24} subsumes the recent work of De~Lellis--Hirsch--Marchese--Spolaor--Stuvard \cite{DLHMSS22}  (while  also providing the additional conclusion that the mod $p$ minimising hypersurface near every branch point is a Lipschitz multi-valued graph with generalised-$C^{1, \alpha}$ regularity).  
	
Our proof utilises the planar frequency function introduced by the first and third authors in their work on area minimising currents (\cite{KW23a}), and thus does not require the Almgren center manifold for the analysis of branch points except in a single, geometrically canonical case where the center manifold satisfies additional simplifying properties. We prove that the planar frequency is approximately monotone at density $Q$ branch points for $V\in\mathcal{V}$. We then partition the density $Q$ branch set into two pieces, namely those points where the planar frequency value is $\neq 2$ and those where it is exactly $2$. The former we show can be stratified, directly using the planar frequency function, by the spine dimension of tangent maps relative to the tangent plane at branch points. For any point in the latter, we construct a single center manifold which  touches $V$ at that point and at all nearby branch points of the same type, and stratify such branch points by the spine dimension of  the tangent maps arising from a linearisation procedure of $V$ relative to this center manifold.  
    
For situations where  energy convergence in the center-manifold linearisation procedure is not readily available, our analysis uses an elegant alternative argument of Hirsch--Spolaor in \cite{HS24} that treats these tangent maps as multi-valued gradient Young measures. 
\end{abstract} 

\maketitle

\tableofcontents

\section{Introduction}

Understanding the size of the branch set of general stationary integral $n$-varifolds is a question lying at the heart of geometric measure theory. A well-known open question is whether the singular set (or, equivalently, the branch set) can have positive $\H^n$-measure.\footnote{We note that when the mean curvature is non-zero singular sets of positive $\H^n$-measure can occur, even if the mean curvature is arbitrarily close to $0$ in $L^\infty$.}
Even the existence question, namely whether within a given closed Riemannian manifold $(M^{n+k},g)$ one can always find a closed, partially regular, $n$-dimensional minimal surface, is completely open except when $k=1$ (or $n=1$). The Almgren--Pitts min-max theory provides the existence of a stationary integral $n$-varifold, but extremely little is known about its regularity, precisely due to the possibility of a large branch set. Indeed, essentially all that is known in this generality is that the (open) set of regular points is a dense subset of the support of the varifold, which follows from the regularity theory of Allard \cite{All72}.

Let us briefly elaborate on two situations where significant progress has been made. The first is when the codimension is one and the varifold is additionally known to be stable; this is the situation one has in the Almgren--Pitts min-max theory when $k=1$. The work of Schoen--Simon \cite{SS81} gives a regularity and compactness theory for stable minimal hypersurfaces which are \emph{a priori known} to have sufficiently small singular sets (an assumption which, by virtue of the regularity theory for area-minimising hypersurfaces, can be  verified in the context of the Almgren--Pitts min-max theory), showing that the singular set must in fact have Hausdorff dimension $\leq n-7$; see also the recent breakthrough of Bellettini \cite{Bel25}  for extensions. Work of the third author \cite{Wic14a} provides a generalisation of the Schoen--Simon theory. This latter theory has been utilised to give an alternative proof of the existence of minimal hypersurfaces in closed Riemannian manifolds; this  approach uses, in place of direct varifold min-max constructions, a much more elementary PDE theoretic min-max principle applied to the Allen--Cahn equation (\cite{Gua18}). The energy concentration set arising from these PDE min-max solutions is a stationary integral varifold (\cite{HT}) to which the regularity theory of \cite{Wic14a} applies (see \cite{TW12}). Most recently, the work \cite{Wic14a} has been adapted in a similar phase-transition approach to the existence of optimally-regular, constant-mean-curvature and prescribed-mean-curvature hypersurfaces in closed Riemannian manifolds (see \cite{BW18, BW19, BW20}). 
On the regularity side, building on the methods of \cite{Wic14a} and  a key estimate from \cite{Bel25} (and aided also by an intermediate result of Hong--Li--Wang (\cite{HLW24})), a sharp regularity and compactness theory for stable \emph{immersed} hypersurfaces whose singular (i.e.\ non-immersed) points form an ${\mathcal H}^{n-2}$-null set has very recently been established by the second author and Xiao (\cite{MX26}), giving the conclusion that the singular set has Hausdorff dimension at most $n-7$.   

One of the key results in \cite{SS81, Wic14a} is showing that branch points \emph{do not occur} under the given assumptions, and it is through ruling out branch points \emph{entirely} that one is able to deduce further information on the size of the singular set. This is analogous to how in the classical regularity theory of area minimising hypersurfaces, one first rules out branch points using a decomposition of the current into boundaries of Caccioppoli sets (see \cite{Sim83a}),  whose singular points, by a theorem of De~Giorgi (\cite{DG61}, see also \cite{Sim83a}), cannot have tangent cones which are planes. As such, the branch set is empty in these situations.

The other situation where significant progress has been made on the regularity question is area-minimising rectifiable currents of codimension $\geq 2.$ The Federer--Fleming compactness theorem (\cite{FF60}) implies the existence of area (mass)-minimising integral currents (cycles) arising as representatives of non-trivial integral homology classes of  compact  Riemannian manifolds. Similarly, subject to appropriate conditions, this powerful compactness theorem also provides solutions to the (oriented) Plateau problem, producing area-minimisers in their relative homology classes. When the codimension is $\geq 2$, these solutions \emph{can} develop branch points, and it is precisely for this reason that the regularity question in higher codimension is significantly more complicated than in codimension 1. 

Almgren's monumental 1983 work \cite{Alm00} (clarified later by De~Lellis--Spadaro \cite{DLS14, DLS16a, DLS16b}) showed that the Hausdorff dimension of the (interior) singular set of such a minimiser $T$ is $\leq n-2$ where $n$ is the dimension of $T$. This bound is optimal when the codimension is $\geq 2$ (but can be improved to $\leq n-7$ in codimension 1 as discussed above).  With regard to Almgren's method of proof, a major source of additional difficulties is a rather complicated iterative construction of so-called \emph{center manifolds} corresponding to each branch point. This aspect of the classical theory, and the exceeding technical output needed for this procedure, seem unavoidable if the question of the size of the singular set is prioritised and studied in isolation (as done in the classical theory). 

Very recently, a new  framework for this problem has been developed by the first and the third author (see \cite{KW23a, KW23b, KW26c, KW26d, KW26e}). In this approach, the focus is shifted to the local structure of the current. In particular, the question of ${\mathcal H}^{n-2}$-a.e.~uniqueness of tangent cones is tied intimately to the question of the size of the singular set (and to subsequent structural questions). By introducing \emph{planar frequency} -- a new intrinsic frequency function for the current which is shown to be approximately monotone subject to a decay condition -- this approach simultaneously resolves a number of questions \emph{without} employing center manifolds; namely: the $\H^{n-2}$ a.e.\ uniqueness of tangent cones; the dimension of the part of the branch set that is most technically demanding (from the classical view point of using center manifolds); local structural decomposition of this part as a finite union of disjoint locally $(n-2)$-rectifiable pieces; and a similar local structural decomposition of the set of all non-branch point singularities. This is then followed by the dimension bound for the full singular set (and other outcomes), for which a center manifold is invoked only to handle a canonical case where this device becomes necessary and at the same time, importantly, satisfies \emph{key additional geometric properties}. In addition to providing a considerable simplification of Almgren's argument, this limited use of center manifolds is in fact what paves the way to a general local structural study of the current, giving a higher-dimensional generalization of the two-dimensional theory of Chang, White, and Micallef--White (\cite{Cha88, W, MW}). The specific overall new results obtained, in addition to $\H^{n-2}$-a.e.~uniqueness of tangent cones, include: higher order asymptotics at $\H^{n-2}$-a.e.~branch point; a local structural decomposition of the singular set into a finite disjoint union of locally compact, locally $(n-2)$-rectifiable sets; and frequency conditions for a branch point to have a neighborhood in which the support of the current is homeomorphic to an $n$-disk and admits a $C^{1, \mu}$ parameterisation -- a conclusion directly analogous to the Micallef--White asymptotic description of a two-dimensional current near a branch point where the current is irreducible. We mention for completeness that in contemporaneous and independent work, De~Lellis and Skorobogatova (\cite{DelSko1, DelSko2}) have taken an approach based on Almgren's iterative center manifold constructions for all branch points, which together with the work of the second author with De~Lellis and Skorobogatova  (\cite{DLMS23}) provides two of the above conclusions: $\H^{n-2}$-a.e.~uniqueness of tangent cones and countable $(n-2)$-rectifiability of the singular set. These two approaches are conceptually and technically significantly different, except for a certain excess-decay lemma in \cite{DLMS23} and \cite{KW23b} -- based in part, in either case, on \cite{Wic14a} -- which is indispensable for both approaches. 

As previously mentioned, when the ambient manifold has trivial homology of a certain degree one can attempt to find a minimal surface of that dimension using instead the Almgren--Pitts min-max theory, however this object is \emph{not} area minimising and is just a stationary integral varifold (conceivably with a form of Morse index control that could be exploited). Any further partial regularity conclusion has, to date, been unobtainable precisely due to a lack of understanding of the branch set of stationary varifolds.

In the present work, we investigate  the size of the branch set of stationary integral $n$-varifolds $V$ in situations where one has an $\eps$-regularity property near planes of multiplicity at most some fixed $Q\in \{2,3,\dotsc\}$. The $\eps$-regularity  property (see Definition~\ref{defn:eps-reg}) requires that for any integer $Q^{\prime} \in \{2, 3, \ldots, Q\},$ the structure of $V$ locally about a density  $Q^{\prime}$ branch point is given by the graph of a Lipschitz $Q^{\prime}$-valued function, which furthermore is generalised-$C^{1,\alpha}$ (see Definition~\ref{defn:gen-C1}). In turn, this reduces the problem to understanding the size of the density $Q^\prime$ branch set for Lipschitz functions $f:B^n_1(0)\to \A_{Q^\prime}(\R^k)$ whose graphs are stationary for area, and furthermore satisfy an $\eps$-regularity property when $\|Df\|_{L^2(B_1^n(0))}$ is sufficiently small, guaranteeing that $f|_{B^n_{1/2}(0)}$ is generalised-$C^{1,\alpha}$ with uniform, quantitative, estimates.

There are two situations which in recent years have been shown to satisfy such an $\eps$-regularity property for suitable $Q$. These are:
\begin{enumerate}
	\item [(a)] $Q=2$ and $V$ corresponds to the graph of a Lipschitz $2$-valued function. Indeed, in recent work with S.~Becker-Kahn \cite{BKMW25}, the second and third authors show that such $V$ satisfy an $\eps$-regularity theorem near multiplicity $2$ planes. In fact, this conclusion holds under a significantly weaker hypothesis -- a structural condition on the region $\{X \, : \, \Theta_{V}(X) <2\}$ -- in place of the assumption that the entire varifold is a 2-valued Lipschitz graph; see \cite[Definition~0.1]{BKMW25}.
	\item [(b)] $Q$ is arbitrary, $V$ has codimension one and is stable on any orientable portion of the regular part, and $V$ contains no classical singularities of density $<Q$. In the work \cite{MW24}, the second and third authors establish that such $V$ satisfy an $\eps$-regularity theorem near multiplicity $Q$ planes. This situation contains as a special case varifolds associated to currents which are area minimising mod $p$ hypersurfaces.
\end{enumerate}
\begin{remark}
One might hope that in the future it is possible to prove such an $\eps$-regularity property in more general situations, such as when $V$ corresponds to the graph of a Lipschitz $Q$-valued function for arbitrary $Q\in \{2,3,4,\dotsc\}$. This is one of the reasons that the present paper is phrased in this more general framework.
\end{remark}

\subsection{Main results}\label{sec:main-results}

Throughout the paper we fix integers, $n\geq 2$, $k\geq 1$, and write $P_0 := \{0\}^k\times \R^n$. Recall that for a Lipschitz multi-valued function $f:B_1^n(0)\to \A_Q(\R^k)$, $\mathbf{v}(f)$ denotes the varifold associated to $\graph(f)$, i.e.~the varifold push-forward of $B^n_1(0)$ under $f$. We also write $\mathcal{G}$ for the \emph{Almgren metric} on the set $\A_Q(\R^k)$ of unordered $Q$-points in $\R^k$, namely
$$\mathcal{G}\left(\sum_{i=1}^Q\llbracket a_i\rrbracket,\,\sum_{i=1}^Q\llbracket b_i\rrbracket\right):=\left(\inf_{\sigma \in S_Q}\sum^Q_{i=1}|a_i-b_{\sigma(i)}|^2\right)^{1/2}$$
where $S_Q$ denotes the set of bijections $\{1,\dotsc,Q\}\to \{1,\dotsc,Q\}$. As shorthands, for $a = \sum_{i=1}^Q\llbracket a_i\rrbracket\in \A_Q(\R^k)$ and $v\in \R^k$ we write $|a|:= \G(a,Q\llbracket 0\rrbracket)\equiv \left(\sum_{i=1}^Q|a_i|^2\right)^{1/2}$ and $a-v = \sum^Q_{i=1}\llbracket a_i-v\rrbracket$.

To state our main results we introduce the type of varifold class we consider. Let $\mathcal{V}$ denote a set of stationary integral $n$-varifolds $V$ in $B^{n+k}_2(0)\subset\R^{n+k}$. We will assume that $\mathcal{V}$ is closed under \emph{natural geometric operations} and \emph{weak limits}, namely: 
\begin{itemize}
	\item For any $V\in\mathcal{V}$, $x_0\in B^{n+k}_2(0)$, $\rho\in (0, \frac{1}{2}(2-|x_0|))$, and rotation $\Gamma\in SO(n+k)$, we have $(\Gamma\circ\eta_{x_0,\rho})_\#V\res B_2^{n+k}(0)\in \mathcal{V}$.
	\item If $(V_j)_j\subset\mathcal{V}$ has $\limsup_{j\to\infty}\|V_j\|(B^{n+k}_2(0))<\infty$, then there is a subsequence $(j^\prime)$ and $V\in\mathcal{V}$  (possibly $V = 0$) such that $V_{j^\prime}\weakly V$ in $B_{2}^{n+k}(0)$.
    \item If $V\in \mathcal{V}$ is decomposable in the sense that $V = V_1+V_2$ where $V_1$, $V_2$ are integral $n$-varifolds with disjoint supports in $B_{2}^{n+k}(0)$, then $V_1,V_2\in\mathcal{V}$.
\end{itemize}
Here, $\eta_{x_0,\rho}(y):= (y-x_0)/\rho$ is the standard homothety $\R^{n+k}\to \R^{n+k}$ and $\weakly$ denotes convergence as varifolds.

Fix an integer $Q\in \{2,3,\dotsc\}$. The other assumption we make on $\mathcal{V}$ is that varifolds $V\in \mathcal{V}$ satisfy an $\eps$-regularity property near planes of multiplicity at most $Q$:

\begin{defn}\label{defn:eps-reg}
	Let $\mathcal{V}$ and $Q$ be as above. We say that $\mathcal{V}$ obeys the \emph{$\eps$-regularity property up to multiplicity $Q$} if there exists $\eps = \eps(\mathcal{V},Q)\in (0,1)$ such that the following holds. If $V\in \mathcal{V} \setminus \{0\}$ satisfies
	\begin{itemize}
		\item $(\w_n 2^n)^{-1}\|V\|(B^{n+k}_2(0))< Q+1/2$;
		\item $\hat{E}_V<\eps$, where
		$$\hat{E}_V^2:= \int_{\R^k\times B^n_1(0)}\dist^2(x,P_0)\, \ext\|V\|(x),$$
	\end{itemize}
	then, there is $Q^\prime\in \{1,2,\dotsc,Q\}$ and a Lipschitz $Q^\prime$-valued function $u:B^n_{1/2}(0)\to \A_{Q^\prime}(\R^k)$ such that, setting  
    $\widetilde{V}:= V\res B_{15/8}(0)$, conditions (A), (B), and (C) below hold for some constants $C = C({\mathcal V}, Q) \in (0, \infty)$ and $\alpha = \alpha({\mathcal V}, Q) \in (0, 1)$:
    \begin{itemize}
	\item[(A)] $\widetilde{V}\res (\R^k\times B^n_{1/2}(0)) = \mathbf{v}(u)$ and $\sup|u| + \Lip(u) \leq C\hat{E}_V;$
	\item[(B)] $u$ is generalised-$C^{1,\alpha}$ (see Definition \ref{defn:gen-C1}); 
    \item[(C)] Letting 
    $\BC_{x}$ denote the unique tangent cone to $V$ at $x\in \spt\|\widetilde{V}\|\cap (\R^k\times B^n_{1/2}(0))$ (which is equal to the stationary union of $N\in \{2,4,\dotsc,2Q^\prime\}$ half-planes, counted with multiplicity, which may form a union of (at most $Q^\prime$) planes, including possibly a single plane occurring with multiplicity), we have:
	\begin{enumerate}
		\item [(1)] for every $x,y\in \sing(\widetilde{V})\cap \{\Theta_V=Q^\prime\}\cap (\R^k\times B_{1/2}^n(0)),$ 
		$$\dist_\H(\BC_x\cap B^{n+k}_1(0), \BC_y \cap B^{n+k}_1(0)) \leq C\hat{E}_V|x-y|^\alpha;$$
		\item [(2)] for all $x\in \sing(\widetilde{V})\cap \{\Theta_V=Q^\prime\} \cap (\R^k\times B^n_{1/2}(0)),$ 
		$$r^{-n-2}\int_{\R^k\times B^n_r(\pi_{P_0}(x))}\dist^2(y,x+\BC_x)\, \ext\|\widetilde{V}\|(y) \leq C\hat{E}_V^2 r^{2\alpha} \qquad \text{for all }r\in (0,1/4).$$
	\end{enumerate}
    \end{itemize}
\end{defn}

\begin{remark}
    Note that from the Lipschitz bound in Definition \ref{defn:eps-reg}(A), for all $x\in \spt\|\widetilde{V}\|\cap (\R^k\times B^n_{1/2}(0))$ each tangent cone $\BC_x$ as in Definition \ref{defn:eps-reg}(C) obeys
		$$\dist_\H(\BC_x\cap B^{n+k}_1(0), P_0\cap B^{n+k}_1(0)) \leq C\hat{E}_V.$$
\end{remark}

\textbf{Notation:} Given two varifolds $V,W$ in $B^{n+k}_1(0)$ we use the following shorthand notations:
\begin{itemize}
	\item $\dist_\H(V\cap B_1^{n+k}(0), W\cap B^{n+k}_1(0)) \equiv \dist_\H(\spt\|V\|\cap B^{n+k}_1(0), \spt\|W\|\cap B^{n+k}_1(0))$;
	\item For $x\in \R^{n+k}$, $\dist(x,V) \equiv \dist(x,\spt\|V\|)$ and $x+V \equiv (\eta_{-x,1})_\#V$.
\end{itemize}

\begin{remark}\label{remark:lower-Q}
	One would expect an $\eps$-regularity theorem holding near a multiplicity $Q$ plane to imply that any $\eps$-regularity theorem also holds at multiplicity $Q^\prime<Q$ planes, and so the need in the above to work with all multiplicities $\leq Q$ would reduce to working with just multiplicity $Q$. Indeed, this is certainly the case when the class $\mathcal{V}$ is determined by structural or variational assumptions, as if $V\in\mathcal{V}$ was close to a multiplicity $Q^\prime<Q$ plane $P$, one could consider $V+(Q-Q^\prime)(P+z)$ for a small $z\in P^\perp$ (for which $P+z$ and $V$ are disjoint) and this varifold would also lie in $\mathcal{V}$ and now be close a multiplicity $Q$ plane. However, it is not entirely clear how to deduce this abstractly from a single regularity result near multiplicity $Q$ planes. To phrase our results appropriately (and since one expects that in order to verify an $\eps$-regularity theorem near multiplicity $Q$ planes, one needs to have already verified one near planes of multiplicity $<Q$) we work with Definition \ref{defn:eps-reg} as stated.
\end{remark}

\begin{remark}
    When $Q=2$ in Definition \ref{defn:eps-reg}, as soon as one has the conclusion that $\widetilde{V}$ is represented by a \emph{Lipschitz} $2$-valued graph, all the subsequent conclusions regarding generalised-$C^{1,\alpha}$ regularity and the corresponding estimates follow from \cite{BKMW25}. This is the basis for our application to the class of Lipschitz $2$-valued stationary graphs. It is perhaps reasonable to suspect that the $\eps$-regularity theorem in \cite{BKMW25} generalises to all $Q\in \{2,3,4,\dotsc\}$, at which point not only does the work become applicable to Lipschitz $Q$-valued stationary graphs, but Definition \ref{defn:eps-reg} simplifies, only requiring the conclusion of being represented by a Lipschitz multi-valued graph.
\end{remark}

Our main result is then the following. Here, we write $\mathcal{B}_V$ for the \emph{branch set} of a stationary integral varifold $V$, i.e.~the set of singular points where at least one tangent cone is a supported on a plane (which necessarily has integer multiplicity $>1$ by Allard's regularity theorem).

\begin{thmx}\label{thm:main}
	Let $\mathcal{V}$ be a class of varifolds which is closed under natural geometric operations and weak limits (see the discussion preceding Definition~\ref{defn:eps-reg}). Suppose also that $\mathcal{V}$ obeys the $\eps$-regularity property up to multiplicity $Q$  
    (see Definition~\ref{defn:eps-reg}). Then:
	$$\dim_\H \left(\mathcal{B}_V\cap \{\Theta_V=Q\}\right) \leq n-2.$$
	In fact, by induction we have $\dim_\H\left(\mathcal{B}_V\cap \{\Theta_V\leq Q\}\right) \leq n-2$, and so by standard stratification results
	$$\dim_\H(\sing(V)\cap \{\Theta_V < Q+1\}) \leq n-1.$$
\end{thmx}

\begin{remark}\label{remark:stratification}
In fact, our proof of Theorem \ref{thm:main} provides a more precise description of the density $Q$ branch set by providing a certain stratification of it. Writing $\mathcal{B}_V^Q:=\mathcal{B}_V\cap \{\Theta_V=Q\}$, we will show that one can associate a (\emph{planar}) \emph{frequency value} to each $z\in \mathcal{B}_V^Q$, which is a number $\mathcal{N}_V(z)\in[1+\alpha,\infty)$, where $\alpha\in (0,1)$ is as in Definition \ref{defn:eps-reg}. We then have:
$$\mathcal{B}_V^Q \cap \{\mathcal{N}_V<2\} = \mathcal{S}^{(1)}_0\cup \cdots \cup \mathcal{S}^{(1)}_{n-2};$$
$$\mathcal{B}_V^Q \cap \{\mathcal{N}_V\geq 2\} = \mathcal{S}^{(2)}_0 \cup \cdots \cup \mathcal{S}^{(2)}_{n-2};$$
where $\dim_\H(\mathcal{S}^{(i)}_{j}) \leq j$ for each $i\in\{1,2\}$ and $j\in \{0,1,\dotsc,n-2\}$. When $n=2$, we therefore have
$$\mathcal{B}_V^Q = \mathcal{S}^{(1)}_0\cup \mathcal{S}^{(2)}_0$$
and furthermore in this case $\mathcal{S}^{(2)}_0$ is a relatively closed subset of $\mathcal{B}_V^Q$ and both $\mathcal{S}^{(1)}_0$, $\mathcal{S}^{(2)}_0$ are \emph{discrete}. This discreteness follows in an analogous manner to how for stratifications arising from a monotonicity formula, the $0^{\text{th}}$-strata is discrete (as for us a limit point of either $\mathcal{S}^{(1)}_0$ or $\mathcal{S}^{(2)}_0$ would create a non-collapsed $2$-dimensional blow-up which has a 1-dimensional touching set, which is impossible). Thus, when $n=2$ the only possible branch points which are not isolated are those of planar frequency $\geq 2$ which are limit points of branch points of planar frequency $<2$. In fact, for any $\eps>0$, we can analogously stratify $\mathcal{B}_V^Q\cap \{\mathcal{N}_V\geq \frac{5}{3}+\eps\} = \tilde{\S}_0^{(2)}\cup \cdots\cup \tilde{\S}_{n-2}^{(2)}$ for arbitrary $n$, meaning that when $n=2$ the only situation in which we can have (distinct) density $Q$ branch points $(Z_j)_j\subset \mathcal{B}^Q_{V}$, $Z\in\mathcal{B}_V^Q$ with $Z_j \to Z$ is when $\mathcal{N}_V(Z)\geq 2$ and\footnote{See Remark \ref{remark:lower-exponent-in-cm} and Remark \ref{remark:lower-frequency-in-cm} for the explanation for why we get this improved upper bound on the frequency value on the $Z_j$.} $\limsup_{j\to\infty}\mathcal{N}_V(Z_j)\leq \frac{5}{3}$, and so there must be a genuine jump in the frequency value (in fact, one must even have $\mathcal{N}_V(Z)\in\{2,3,4,\dotsc\}$, else one can get a contradiction by performing a suitable blow-up off the tangent plane at $Z$).  Conjecturally this should not happen when $n=2$.
\end{remark}

Theorem \ref{thm:main} is directly applicable in several situations. Indeed, consider the following two classes:
\begin{itemize}
	\item For $L\in (0,\infty)$, write $\mathcal{V}_L$ for the set of stationary integral $n$-varifolds in $B^{n+k}_2(0)$ which are represented by the graph of a Lipschitz $2$-valued function (or a single-valued function) over some $n$-dimensional subspace of $\R^{n+k}$ with Lipschitz constant $\leq L$.
	\item For $Q\in\{2,3,\dotsc\}$, write $\mathcal{S}_Q$ for the set of stationary integral $n$-varifolds in $B^{n+1}_2(0)$ which are stable on their regular part and have no classical singularities of density $<Q$.
\end{itemize}
These classes have been studied in \cite{BKMW25, MW24}, respectively, with $\eps$-regularity theorems established for both. As such, $\mathcal{V}_L$ and $\mathcal{S}_Q$ satisfy the assumptions of Theorem \ref{thm:main} (with $Q=2$ in the former case). We therefore have the following.

\begin{corollary}\label{cor:Lip-1}
	Let $V$ be a stationary integral varifold in $\R^k\times B^n_1(0)\subset\R^{n+k}$ such that there is a Lipschitz function $u:B^n_1(0)\to \A_2(\R^k)$ for which $V = \mathbf{v}(u)$. Then
	$$\dim_\H(\mathcal{B}_V) \leq n-2.$$
	Consequently, $\dim_\H(\sing(V))\leq n-1$. Moreover, when $n=2$ we have
	$$\mathcal{B}_V = \S_0^{(1)}\cup \S^{(2)}_0$$
	where $\S^{(1)}_0$ and $\S^{(2)}_0$ are discrete sets as in Remark \ref{remark:stratification} and furthermore $\S^{(2)}_0$ is a relatively closed subset of $\mathcal{B}_V$.
\end{corollary}

\begin{corollary}\label{cor:stable-1}
	Let $Q\in \{2,3,\dotsc\}$ and $V\in \mathcal{S}_Q$. Then
	$$\dim_\H(\mathcal{B}_V\cap \{\Theta_V=Q\}) \leq n-2.$$
	Consequently, $\dim_\H(\sing(V)\cap \{\Theta_V<Q+1\}) \leq n-1$. Moreover, when $n=2$ we have
	$$\mathcal{B}_V\cap \{\Theta_V=Q\} = \S_0^{(1)}\cup \S^{(2)}_0$$
	where $\S^{(1)}_0$ and $\S^{(2)}_0$ are discrete sets as in Remark \ref{remark:stratification}, and furthermore $\S^{(2)}_0$ is a relatively closed subset of $\mathcal{B}_V\cap \{\Theta_V<Q+1\}$.
\end{corollary}

\begin{remark}
	In Corollary \ref{cor:Lip-1}, the claim on the \emph{full} singular set follows directly from the branch set dimension bound and general stratification theorems for singularities of stationary integral varifolds. This part of Corollary \ref{cor:Lip-1} was recently established by Hirsch--Spolaor \cite{HS24}. In the special case of Corollary \ref{cor:Lip-1} where one knows that $u$ is $C^{1,\alpha}$ rather than just Lipschitz, Corollary \ref{cor:Lip-1} was already known from the work of Simon and the third author \cite{SW16}. The proof in \cite{SW16} uses purely PDE-theoretic methods 
    (in particular without the need for an Almgren-type center manifold construction, as, morally, the average of the two values is used as the center manifold) and in fact subsequent work by the first and third authors \cite{KW21} gives (again without employing the Almgren center manifolds) that the branch set is countably $(n-2)$-rectifiable, with unique tangent maps at $\H^{n-2}$-a.e.~branch point. Notice that the results in \cite{SW16} are not otherwise applicable in Corollary \ref{cor:Lip-1}, since the $\eps$-regularity theorem in \cite{BKMW25} only provides \emph{generalised-}$C^{1,\alpha}$ regularity locally about branch points, not $C^{1,\alpha}$ regularity (although in codimension one it is applicable, as then generalised-$C^{1,\alpha}$ regularity combined with stationarity implies $C^{1,\alpha}$).
\end{remark}

\begin{remark}
	For $p\geq 2$ even, the class $\mathcal{S}_{p/2}$ contains all varifolds associated to currents which are area minimising hypersurfaces modulo $p$, and thus Corollary \ref{cor:stable-1} includes as a special case the main result established in \cite{DLHMSS22} (this is because all branch points have density $p/2$ in this case). Since branch points in area minimising hypersurfaces mod $p$ necessarily\footnote{This follows from the analysis in \cite{DLHMSS22}. It should be noted that this is a significant difference to the stable hypersurface setting analysed in \cite{MW24}, as in the latter case it is possible for branch points to have planar frequency $<2$. The reason for this difference is precisely due to the fact that an area minimising hypersurface mod $p$ must satisfy a special orientation condition at classical singularities, namely that the unit conormals at the boundary must be orientated in the same manner. This is a consequence of the minimising mod $p$ assumption. Nonetheless, if one were to assume such an orientation structure holds locally about a given branch point in a stable minimal hypersurface in $\S_{p/2}$, one could similarly show that the planar frequency of the branch point must be $\geq 2$. We refer the reader to \cite{MS26} for further discussion.} have planar frequency $\geq 2$, it also shows that a two-dimensional surface in a smooth $3$-manifold which is locally area minimising mod $p$ has a discrete branch set (cf.~Remark \ref{remark:stratification}). We therefore have (this has also independently been recently established in \cite{SSS25}):
\end{remark}

\begin{corollary}
	A $2$-dimensional area minimising surface mod $p$ in a closed smooth $3$-dimensional Riemannian manifold has isolated branch points.
\end{corollary}

In the two-valued Lipschitz case as in Corollary~\ref{cor:Lip-1}, the work of Hirsch--Spolaor \cite{HS24} establishes the Hausdorff dimension bound $n-1$ (on the \emph{entire} singular set) by showing that the classical Almgren program in its full extent -- namely, the construction of sequences of center manifolds and intervals of flattening for arbitrary branch points -- can be adapted to this setting. In this argument, as in the classical theory, all branch points are treated equally, since the argument proceeds without the knowledge of the uniqueness of tangent cones (and suitable quantitative estimates).
Consequently, the argument does not seem to readily provide a mechanism for dealing with possible ``frequency $1$'' branch points,  meaning that it only yields the conclusion that the branch set has Hausdorff dimension $\leq n-1$. (This is not as issue in the classical setting of are-minimisers, and the optimal $n-2$ conclusion follows, as the linearisation process produces Dirichlet energy minimisers.)

By contrast, in the present work,  we use the regularity theory of \cite{BKMW25} to establish the optimal result, namely that the branch set has Hausdorff dimension $\leq n-2$. Though arguably the regularity theory in \cite{BKMW25} is a ``heavy tool,'' it leads to the optimal results and also to considerable simplifications is other ways. It  allows us to introduce, to the present not-necessarily-minimising setting, the \emph{planar frequency function} developed by the first and third authors in their work on area minimising currents \cite{KW23a}. The planar frequency function allows one to treat branch points with planar frequency $<2$ \emph{directly} (and even those with planar frequency $>2$), namely by performing a blow-up procedure relative to the (unique) tangent plane; no center manifolds are needed. For the branch points of planar frequency $\geq 2$, we are able to use a significantly simplified version of Almgren's program, building just \emph{one} center manifold, avoiding the need for both changing center manifolds and intervals of flattening -- which are technically demanding foundational ingredients in the classical theory. In situations where energy convergence is not readily available in the blow-up procedure relative to this center manifold (cf.~Remark \ref{remark:energy-convergence}), we use an elegant alternative argument from \cite{HS24} which treats blow-ups as multi-valued gradient Young measures.

\begin{remark}\label{remark:area-min}
	In situations where one can show energy convergence when blowing-up relative to the center manifold (as is the case when the varifold is that associated to an area minimising current, or a current which is area minimising mod $p$, using a competitor argument), the language of multi-valued gradient Young measures is not needed, and one can work with multi-valued $W^{1,2}$ functions instead (indeed, the former then reduces to the latter).
\end{remark}

Combining Corollary \ref{cor:Lip-1} with the regularity theories in \cite{BK17, BKMW25} as well as simple tangent cone analysis, we immediately deduce the following structural result for stationary $2$-valued Lipschitz graphs.

\begin{thmx}\label{thm:main-2}
	There exists $\alpha = \alpha(n,k)\in (0,1)$ such that the following is true. Suppose that $u:B^n_1(0)\to \A_2(\R^k)$ is a Lipschitz $2$-valued function with the property that the varifold $V = \mathbf{v}(u)$ is stationary. Write $U\subset B^n_1(0)$ for the (open) subset of $B^n_1(0)$ consisting of the points locally about which $u$ is generalised-$C^{1,\alpha}$. Then we have
		$$\dim_\H(B^n_1(0)\setminus U)\leq n-3.$$
		In fact, we have the disjoint union
		$$\spt\|V\| = \mathscr{R}\cup \mathscr{B}\cup\mathscr{C}\cup \mathscr{K}$$
		and that $U = B^n_1(0)\setminus\pi(\mathscr{K})$, where $\pi:\R^{k}\times\R^{n}\to \R^n$ is the orthogonal projection map and:
		\begin{enumerate}
			\item [(i)]
			\begin{itemize}
				\item $\mathscr{R}$ is the set of regular points of $V$;
				\item $\mathscr{B}$ is the set of singular points of $V$ where at least one tangent cone is a plane of multiplicity $2$ (this is then the unique tangent cone and locally about the point $V$ has a generalised-$C^{1,\alpha}$ structure by \cite[Theorem A]{BKMW25});
				\item $\mathscr{C}$ is the set of singular points of $V$ where at least one tangent cone consists of either a sum of $2$ distinct planes or $4$ distinct half-planes with a common boundary (this is then the unique tangent cone and locally about the point $V$ has a generalised-$C^{1,\alpha}$ structure by \cite{BK17});
				\item $\mathscr{K}$ is a closed subset which consists of the singular points of $V$ where no tangent cone can be written as the sum of $4$ half-planes (including the possibility of two planes which may or may not coincide).
			\end{itemize}
			\item [(ii)]
			\begin{itemize}
				\item $\dim_\H(\mathscr{B})\leq n-2$;
				\item $\dim_\H(\mathscr{C})\leq n-1$ (in fact, $\mathscr{C}$ is locally contained within an $(n-1)$-dimensional $C^{1,\alpha}$ submanifold by \cite{BK17});
				\item $\dim_\H(\mathscr{K})\leq n-3$.
			\end{itemize}
			In particular, as $\sing(V) = \mathscr{B}\cup\mathscr{C}\cup\mathscr{K}$, we have $\dim_\H(\sing(V))\leq n-1$.
		\end{enumerate}
\end{thmx}

\begin{proof}
	All the claims except the Hausdorff dimension on $\mathscr{K}$ follow from Theorem \ref{thm:main} and the regularity theories established in \cite{BK17, BKMW25}. To establish the Hausdorff dimension on $\mathscr{K}$, we will analyse the possible tangent cones to show that $\mathscr{K}$ must be a subset of the usual $(n-3)$-stratum, and thus the claimed dimension bound follows.
	
	Indeed, suppose $x\in \sing(V)$ and let $\BC$ be a tangent cone at $x$. If $\dim(S(\BC))\in \{n-1,n\}$, then one may deduce that $\BC$ is the unique tangent cone and $V$ has a $GC^{1,\alpha}$ structure locally about $x$ by applying the regularity theorems in \cite{BK17, BKMW25}. Thus, $x\not\in \mathscr{K}$. Now suppose $\dim(S(\BC)) = n-2$. We claim that $\BC$ must be a sum of two planes: once we have shown this one may apply \cite{BK17} to see that $V$ is smoothly immersed locally about $x$, and thus $x\not\in\mathscr{K}$, which completes the proof.
	
	Since $V$ is a two-valued Lipschitz graph it follows that $\BC$ is also represented by the graph of a Lipschitz $2$-valued function $f:\R^n\to \A_2(\R^k)$. By quotienting out by the spine of $\BC$, we may in fact assume that $\BC$ is a $2$-dimensional cone with $S(\BC) = \{0\}$, so that $f:\R^2\to \A_2(\R^k)$. By the homogeneity of $\BC$ and Theorem \ref{thm:main}, we know that the branch set of $\BC$ is contained within $\{0\}$, and thus the only possible singularities of $\BC$ away from $\{0\}$ must be classical singularities. Each classical singularity determines a ray of classical singularities in the domain of $f$: let $N\in \{0,1,2,\dotsc\}$ be the number of such rays (we allow for $N=0$ where there are no classical singularities). Note that by a simple compactness argument on the link $\BC\cap S^{k+1}$ there must be only finitely many such rays, so $N<\infty$. If $N\geq 1$, denote the rays (through the origin) of the classical singularities by $\ell_1,\dotsc,\ell_N\subset\R^2$. Then $\R^2\setminus \cup_{i=1}^N\ell_i = \cup_{i=1}^NW_i$, where in  (suitable chosen) polar coordinates $(r,\theta)$ in $\R^2$ each $W_i$ takes the form $W_i = \{r>0, \theta_i<\theta<\theta_{i+1}\}$ for some $0 = \theta_1 < \cdots <\theta_{N+1} = 2\pi$. On the regions $W_i\times\R^k$, $\BC$ is a 2-dimensional minimal cone which is the cone over two (smooth) geodesic arcs in $\S^{k+1}$, and thus $\BC$ must be the sum of two planes here, which intersect only along $\theta = \theta_i, \theta_{i+1}$. But two distinct planes can only intersect along a subspace, and thus we need $\theta_i + \pi = \theta_{i+1}$. But since $0\leq \theta_i\leq 2\pi$ for all $i$, this forces $N = 2$ in this situation (with $\theta_1=\pi$), and therefore in general we must have $N \in \{0,2\}$.
    But then in the case $N = 0$, we have that $\BC$ coincides with the sum of two planes, and so we are done. The only remaining case is when $N=2$, in which case $\BC$ is the sum of $4$ half-planes. But then this would force $\Theta_{\BC}(0) = 2$, and since the classical singularities also have density $2$, we would need that the classical singularities belong to the spine, and so $\dim(S(\BC)) = 1$, again giving the contradiction. Thus we need $\BC$ to be the sum of two planes only intersecting along a subspace of dimension $n-2$, completing the proof.
\end{proof}

\begin{remark}
It might be possible to improve the bound on the Hausdorff dimension of $\mathscr{K}$ in Theorem \ref{thm:main-2} (and hence of $B^n_1(0)\setminus U$) to $\leq n-4$, which would be the best possible dimension bound in arbitrary codimension due to the examples constructed by Lawson--Osserman \cite{LO77}. In codimension one, in fact one has $\mathscr{K}=\emptyset$ (and hence has $U = B^n_1(0)$), see \cite[Theorem 5.2]{Hie20}.
\end{remark}

Combining Corollary \ref{cor:stable-1} with the regularity theories in \cite{MW24, Min24}, we immediately deduce the following structural result for $V\in \mathcal{S}_Q$:

\begin{thmx}\label{thm:main-4}
	Suppose $V\in\S_Q$. Then,
	$$\spt\|V\|\cap \{\Theta_V<Q+1\} = \mathcal{R}^{\leq Q}\cup\mathcal{B}^Q\cup\mathcal{C}^Q\cup\mathcal{\tilde{C}}^{Q+1/2}\cup \mathcal{K}\, ,$$
	where
	\begin{itemize}
		\item $\mathcal{R}^{\leq Q}$ is the set of regular points of $V$ with density in $\{1,2,\dotsc,Q\}$, and thus about which $\spt\|V\|$ is a locally smooth submanifold;
		\item $\mathcal{B}^Q\equiv\mathcal{B}_V\cap\{\Theta_V=Q\}$ is the set of density $Q$ flat singular points of $V$; by \cite{MW24}, locally about each point in $\mathcal{B}^Q$ we have that $V$ is a $Q$-valued generalised-$C^{1,\alpha}$ graph, and by Corollary \ref{cor:stable-1} we know $\mathcal{B}^Q$ has Hausdorff dimension $\leq n-2$;
		\item $\mathcal{C}^Q$ is the set of density $Q$ classical singularities of $V$, and locally about each such point $V$ is a smooth classical singularity (see \cite{Kru14});
		\item $\mathcal{\tilde{C}}^{Q+1/2}$ is the set of points where one tangent cone to $V$ is a classical cone of density $Q+1/2$; by \cite{Min24}, it follows that such a tangent cone is unique, and $V$ is locally about each such point a (possibly branched) generalised-$C^{1,\alpha}$ perturbation of its tangent cone;
		\item $\mathcal{K}$ is a relatively closed subset of $\{\Theta_V<Q+1\}$ which has dimension $\leq n-2$; in fact, $\mathcal{K}$ equals the $(n-2)$-strata of $V$ in $\{\Theta_V<Q+1\}$.
	\end{itemize}
	In particular, $\dim_\H(\sing(V)\cap \{\Theta_V<Q+1\})\leq n-1$. Here, $\alpha = \alpha(n,Q)\in (0,1)$.
\end{thmx}

\subsection{Overview of the proof}

We begin with some technical preliminaries which follow from the $\eps$-regularity property in Definition \ref{defn:eps-reg}. In particular, we establish suitable doubling conditions for the energy of the Lipschitz multi-valued graphs representing $V\in\mathcal{V}$ about a branch point.

The first key step in the proof is to use these preliminaries to prove approximate monotonicity of the planar frequency function, as introduced in \cite{KW23a}. Indeed, we will show that at \emph{every} density $Q$ branch point (which by hypothesis has a unique tangent cone and, again by hypothesis, about which the structure of the varifold is given by the graph of a generalised-$C^{1,\alpha}$ function as in Definition \ref{defn:eps-reg}), the planar frequency function is approximately monotone at \emph{all} scales. Thus, we may associate a planar frequency value at \emph{every} such point, which we then show takes values in $[1+\alpha,\infty)$. In particular, the planar frequency value always has a fixed gap above $1$, thus ruling out ``frequency $1$'' branch points; we stress that this is stronger than in the corresponding situation for area minimisers (cf.~\cite{KW23a}).

We then partition the set of (density $Q$) branch points into two (disjoint) pieces:
\begin{itemize}
	\item $\mathcal{B}^{\neq 2}$, those where the planar frequency is valued in $[1+\alpha,\infty)\setminus\{2\}$;
	\item $\mathcal{B}^2$, those where the planar frequency is equal to $2$.
\end{itemize}
We then show (as observed in \cite{KW26c}) that $\mathcal{B}^{\neq 2}$ has Hausdorff dimension $\leq n-2$ using \emph{only} the planar frequency function. Indeed, we will show that the coarse blow-ups arising from a sequence of vertical rescalings of a given $V\in\mathcal{V}$ about a branch point are always homogeneous, with degree of homogeneity equal to the planar frequency value at the chosen point. At this stage, we can then argue the dimension bound for $\mathcal{B}^{\neq 2}$ using either of the two usual ``dimension reduction'' arguments. We opt for that of Almgren \cite{Alm00} as opposed to Federer \cite{Fed70}, namely we stratify $\mathcal{B}^{\neq 2}$ by the spine dimension of the coarse blow-ups and show that nearby points in $\mathcal{B}^{\neq 2}$ of ``good'' planar frequency value must accumulate around the spine of the coarse blow-up. Due to the planar frequency value lying in $(1,\infty)$, the spine of these coarse blow-ups must be a subspace of dimension $\leq n-2$; we stress that having a frequency gap above $1$ is what ensures that this spine dimension is \emph{not} $n-1$. The desired dimension bound for this subset of the branch set then follows by standard measure-theoretic arguments.

Let us elaborate on this last point that nearby points in $\mathcal{B}^{\neq 2}$ of ``good'' planar frequency value must accumulate around the spine of the coarse blow-up, as it will illustrate why branch points of planar frequency \emph{exactly} $2$ must be treated differently. The planar frequency function at a given branch point measures the frequency of $V$ relative to the tangent plane to $V$ at the point (which is unique, by Definition \ref{defn:eps-reg}). However, the tangent plane may differ from point-to-point, and thus branch points with different tangent planes have different ``frequency functions''. The spine of the coarse blow-up is however determined by the points of maximum frequency for a \emph{specific} choice of plane, and so it is not a priori clear that different branch points should accumulate along the spine. However, we will show (similarly to \cite{KW26c}), for $\mathcal{B}^{\neq 2}$, that \emph{when we perform a coarse blow-up, all these potentially different tangent planes at the relevant points in $\mathcal{B}^{\neq 2}$ coincide in the coarse blow-up}. As such, in the coarse blow-up the frequency functions at the limit points of these branch points are measuring the frequency of the coarse blow-up with respect to the \emph{same} plane, and this is what forces the accumulation of these points around the spine of the coarse blow-up. The reason this argument fails for points where the planar frequency is exactly $2$ is simply because one cannot in general guarantee that the (possibly different) tangent planes at different points of planar frequency $2$ coincide in the coarse blow-up limit\footnote{The reader may find it useful to think of this difference from the viewpoint of a Taylor expansion: measuring the frequency with respect to the tangent plane is morally subtracting off the linear part of the Taylor expansion, leading to the first possibly non-zero term being the quadratic term. Planar frequency $2$ means this part does not vanish, which is not a particularly stringent condition. Planar frequency $>2$ forces this quadratic term to vanish, which is more restrictive for, say, a harmonic function, and this is what allows us to show that the different tangent planes coincide in this case.}.

To show that the set $\mathcal{B}^2$ has Hausdorff dimension $\leq n-2$, we construct a center manifold analogously to Almgren \cite{Alm00} (see also \cite{DLS16a}). This is needed, for instance, when the coarse blow-ups collapse to a single (harmonic, homogeneous of degree $2$) function with multiplicity $Q$. In fact, this argument will also show that the set of branch points with planar frequency $\geq 2$ has Hausdorff dimension $\leq n-2$, although at this stage of the argument we only need to show this for the part of the branch set where the planar frequency is exactly $2$.

The reader who is familiar with Almgren's center manifold construction will then note the benefit of this approach: if one only wishes to study branch points where the decay to the (unique) tangent plane is at least quadratic, the center manifold is \emph{significantly} simpler to use, as one may avoid the need for a sequence of changing center manifolds and intervals of flattening. Consequently, one only needs to build a \emph{single} center manifold, which furthermore will touch the varifold at all nearby branch points of the same density where the decay to the tangent plane is at least quadratic. (This fact follows directly from the construction and was likely known to Almgren in the 1980's, and is noted in the work of S.~Chang \cite{Cha88}.) As such, we only need this simpler construction. This is one of the significant advantages of using the planar frequency function, as already observed in \cite{KW23a}, namely by avoiding the need to use a center manifold at branch points with \emph{sub}quadratic decay to the tangent plane (or with non-unique tangent cone); this is the situation requiring an infinite sequence of center manifolds in the classical theory. Consequently, we will give an essentially self-contained construction of the center manifold, including the proof of the almost monotonicity of the frequency function relative to the center manifold, with certain refinements. A significant portion of this will follow the relevant parts of the works \cite{DLS16a, DLS16b}, which in turn are a technically streamlined version of Almgren's original work.

\begin{remark}
One might wonder, given the Lipschitz multi-valued graph structure about our branch points, whether it is possible to use a simpler approach than Almgren's center manifold to establish the dimension bound for $\mathcal{B}^2$. For instance, one might wonder if it is possible to perform a blow-up relative to the average of the multi-valued graph instead of the center manifold. The issue with this at the moment appears to be a lack of regularity for the average, since if one wishes to show monotonicity of a frequency quantity relative to some submanifold, that submanifold needs to be $C^{3}$ due to error terms involving derivatives of the mean curvature of the submanifold appearing in the first variation computations (which arise from derivatives of the exponential map). Notice that the center manifold is $C^{3,\beta}$, whilst we do not even know if the average-part of a stationary Lipschitz multi-valued graph is $C^2$ in general\footnote{It should be noted in the case of $2$-valued $C^{1,\alpha}$ stationary graphs the work \cite{SW16} provides a way around this problem, utilising PDE-theoretic techniques to show that one can directly work with the average-free part of the function.}. Simply mollifying the average does not help either, as one (ideally) wants the object to touch $V$ at the desired branch points. If however one was in a situation where it was known that the average did have this improved regularity, then it seems reasonable that one could use the average as a replacement for the center manifold (for instance, if the average-part was identically zero, this does work and indeed it reduces to the planar frequency).
\end{remark}

\begin{remark}
One could alternatively attempt a method which falls between the planar frequency and the center manifold, namely by trying to use a modification of the planar frequency via measuring the frequency of $V$ at a given branch point relative to some choice of, say, minimal surface touching $V$ at the branch point, and allowing the surface to vary from point-to-point. For planar frequency, this choice of (minimal) surface is the tangent plane. However, it is unclear how one might choose such surfaces in a manner which ensures both: (a) there is an approximately monotone notion of frequency; and (b) the possibly different choices of surface at nearby points collapse to a fixed one in the coarse blow-up limit (as was the case with planar frequency at points of planar frequency $\neq 2$); this latter condition is needed to know that in the blow-up the frequency functions are the same. We stress that, for branch points of planar frequency $\geq 2$, \emph{both} of these points are satisfied in the present work by taking the surface at different points to be a fixed one, namely the \emph{single} center manifold.
\end{remark}

To aid the reader, we briefly mention that using a center manifold to control the dimension of the branch points with planar frequency $\geq 2$ is morally an argument using quantitative unique continuation. Indeed, one can control the size of the touching set between two smooth minimal surfaces using a frequency function. The geometric viewpoint of this is that one measures the frequency of one surface relative to the other, and stratifies the touching set by looking at the tangent maps, using elliptic estimates to linearise the problem (which in the setting of two minimal graphs gives rise to homogeneous harmonic functions). When controlling the size of the branch set of a given varifold, morally we wish to control the set of singular points where the surface is touching itself (with ``genuine'' branch points being those where the ``sheets'' of the varifold are indistinguishable, and so one cannot reduce the problem to the touching set of two minimal graphs). The idea then is that, if one can find \emph{another} (appropriate\footnote{The surface cannot be chosen completely arbitrarily due to the need for an approximately monotone frequency function.}) surface which touches the varifold along some subset of the branch set, then one can control the size of that subset by controlling this part of the touching set of two geometric objects as before. \emph{This is what the center manifold precisely achieves}, for the subset of the branch points where the planar frequency is $\geq 2$. With this geometric viewpoint, the argument is therefore a natural extension of typical quantitative unique continuation arguments. 

We also stress that, with this viewpoint, the main property one needs regarding the linearisation procedure is that, when $V$ is close to the manifold one is measuring frequency relative to (namely the center manifold, or the tangent plane for planar frequency), it is close \emph{strongly} in $W^{1,2}$ to the linearised problem (after appropriate rescaling). This enables one to pass certain first variation identities from $V$ to the linearised problem (for example, to establish monotonicity of the frequency function at the linearised level), which is all that is needed to conclude the argument. Other than this, the precise form of the linearised problem is essentially irrelevant.

\begin{remark}\label{remark:energy-convergence}
    In more general situations, as in the present paper, this last point regarding \emph{strong} convergence in $W^{1,2}$ to the linearised problem (either relative to a plane or the center manifold) is a non-trivial fact. In Appendix \ref{app:energy-convergence} we will see that under Definition \ref{defn:eps-reg} (as thus in the settings of Corollary \ref{cor:Lip-1} and Corollary \ref{cor:stable-1}) we \emph{do have} strong $W^{1,2}_{\text{loc}}(B^n_1(0))$ convergence when linearising relative to a plane (this uses the energy non-concentration estimate in \cite[$(\mathfrak{B}6)$]{BKMW25}). This will be of great importance in our arguments, namely in establishing approximate monotonicity of the planar frequency function. 
    
    For the linearisation procedure relative to the center manifold, in certain situations strong convergence in 
    $W^{1, 2}_{\rm loc}$ can be directly verified (see Remark \ref{remark:area-min}); for instance, in Almgren's work on area minimisers this is verified using a competitor argument.) In the present work, we use an argument of Hirsch--Spolaor \cite{HS24} which allows us to sidestep this issue by constructing blow-ups as multi-valued gradient Young measures. These measures then form our blow-up class relative to the center manifold, and we can similarly show that one may stratify branch points of planar frequency $\geq 2$ by the spine dimension of these measures. This allows us to conclude the dimension bound again by standard arguments.
\end{remark}

\begin{remark} As a final comment, we note that our approach is rather different to that used by the third author and Simon in \cite{SW16} when studying $2$-valued $C^{1,\alpha}$ stationary graphs. As mentioned, in \cite{SW16} it was possible to avoid the use of Almgren center manifolds completely, using a frequency function for the average-free part of the graph directly. Morally, if one can `subtract off' the average and maintain all other analytic properties, such as an approximately monotone frequency function and energy convergence to the linearised problem, then the above geometric discussion is trivial, as one can simply look at the touching set between the average-free part and a single plane (morally, the problem is then reduced to planar frequency). The reason why this simplification was possible in \cite{SW16} was that they were able to show that in fact the $2$-valued function was $C^{1,1/2}$-regular, meaning that the coefficients of the corresponding (multi-valued) PDE for the average-free part of the function are Lipschitz. This allows one to directly show that a suitable variant of Almgren's frequency function, now for the average-free part, is approximately monotone using ideas from the work of Garofalo--Lin \cite{GL86, GL87}. For us, we do not expect in general to have $C^{1,1/2}$ regularity of our multi-valued graphs and so we seemingly cannot argue like this. Our present argument in particular avoids looking at multi-valued PDEs.
\end{remark}

\textbf{Organisation of paper:} In Section \ref{sec:prelim} we give some preliminaries and results concerning the $\eps$-regularity property. In Section \ref{sec:pff} we introduce the planar frequency function, show its approximate monotonicity at all density $Q$ branch points, and then use it to establish the dimension bound on the set of branch points where the planar frequency is $\neq 2$. In Section \ref{sec:cm}, we then construct a single center manifold and use it to show that the set of branch points where the planar frequency is exactly $2$ (in fact, is $\geq 2$) also has Hausdorff dimension $\leq n-2$, thus completing the proof of Theorem \ref{thm:main}. In Appendix \ref{app:energy-convergence} we discuss the problem of energy convergence for coarse blow-ups, and in Appendix \ref{app:ym} we develop the theory of multi-valued gradient Young measures introduced by Hirsch--Spolaor \cite{HS24}.

\textbf{Acknowledgements:} This research was conducted during the period P.M.~served as a Clay Research Fellow. The authors would like to thank Sidney Stanbury for reading and providing valuable feedback on a preliminary version of this paper.

\section{Preliminaries}\label{sec:prelim}

Throughout we fix $Q\in \{2,3,\dotsc\}$ and a class $\mathcal{V}$ of stationary integral $n$-varifolds in $B^{n+k}_2(0)$ which satisfies:
\begin{itemize}
	\item $\mathcal{V}$ is closed under natural geometric operations and weak limits;
	\item $\mathcal{V}$ obeys the $\eps$-regularity property up to multiplicity $Q$ (Definition \ref{defn:eps-reg}).
\end{itemize}
In the present section, we will deduce consequences of the $\eps$-regularity property holding up to multiplicity $Q$. We refer the reader to \cite{BKMW25, MW24} for further explanations of terminology where needed. For simplicity, when constants depend on both $\mathcal{V}$ and $Q$, we simply write that they depend on $\mathcal{V}$.

For $V\in\mathcal{V}$ and $P$ an $n$-dimensional subspace of $\R^{n+k}$, we define the (\emph{one-sided}) \emph{$L^2$-height excess of $V$ relative to $P$} by
$$\hat{E}^2_{V,P}:= \int_{\pi_P^{-1}(P\cap B^{n+k}_1(0))}\dist^2(x,P)\, \ext\|V\|(x)$$
where $\pi_P:\R^{n+k}\to P$ is the orthogonal projection onto $P$. We abbreviate $\hat{E}_{V,P_0}$ by $\hat{E}_V$, where we recall $P_0 := \{0\}^k\times\R^n$. We also define the $L^2$\emph{-tilt excess of $V$ relative to $P$} by:
$$\Etilt_{V,P}^2 := \frac{1}{2}\int_{\pi_P^{-1}(P\cap B^{n+k}_1(0))}\|\pi_x-\pi_P\|^2\, \ext\|V\|(x)$$
where here $\pi_x:=\pi_{T_xV}$ whenever $T_xV$ exists (which is $\|V\|$-a.e.), and for a linear map $A:\R^{n+k}\to \R^{n+k}$, $\|A\|^2:=\sum_{i,j}A_{ij}^2$ denotes the Hilbert--Schmidt norm of $A$. We abbreviate $\Etilt_{V,P_0}$ by $\Etilt_V$.

Now consider a $Q$-valued Lipschitz function $u:\Omega\to \A_Q(\R^k)$, where $\Omega\subset\R^n$ is open. We write $u = \sum^Q_{\alpha=1}\llbracket u^\alpha\rrbracket$, and for each $\kappa\in \{1,\dotsc,k\}$ we write $u^\kappa\equiv u\cdot e_\kappa = \sum^Q_{\alpha=1}\llbracket u^{\kappa,\alpha}\rrbracket$, where we assume that $u^{\kappa,1}\leq\cdots\leq u^{\kappa,Q}$; here $\{e_1,\dotsc,e_k\}$ denotes the standard basis for $\R^k$. We stress that this does \emph{not} mean that $u^1 = (u^{1,1},u^{2,1},\dotsc,u^{k,1})$, for instance, as the orderings for the different coordinate functions may be different.

We now define what it means for such $u$ to be \emph{generalised-$C^{1,\alpha}$}. We will define this inductively on $Q$: when $Q=1$, generalised-$C^{1,\alpha}$ simply means $C^{1,\alpha}$ in the usual, single-valued, sense. We write
$$\Omega^Q_u := \{x\in \Omega: u(x) = Q\llbracket u_a(x)\rrbracket\}$$
for the $Q$-coincidence set of $u$; here $u_a:= \frac{1}{Q}\sum_{\alpha=1}^Qu^\alpha$ is the average-part of $u$. Notice that $\Omega^Q_u\subset\Omega$ is a relatively closed subset, and that for each $x\in \Omega\setminus \Omega^Q_u$, there is a radius $\rho = \rho(x,u)>0$ such that $u = u^x_1 + u^x_2$, where $u^x_i:B_\rho(x)\to \A_{Q_i}(\R^k)$ are Lipschitz $Q_i$-valued functions, where $Q_1,Q_2\in \{1,\dotsc,Q-1\}$ and $Q_1+Q_2 = Q$.

\begin{defn}\label{defn:gen-C1}
	Fix $\alpha\in (0,1)$. We say that a Lipschitz function $u:\Omega\to \A_Q(\R^k)$ is \emph{generalised-$C^{1,\alpha}$} in $\Omega$, and belongs to $GC^{1,\alpha}(\Omega)$, if the following holds:
	\begin{enumerate}
		\item [(1)] For $x\in \Omega\setminus \Omega^Q_u$, one may choose $\rho = \rho(x, u)$ and functions $u^x_1,u^x_2$ as above such that $u_1^x$ and $u_2^x$ are generalised-$C^{1,\alpha}$ in $B_{\rho}(x)$.
		\item [(2)] To each $x\in \Omega^Q_u$, we can assign a function $A_x:\R^n\to \A_Q(\R^k)$ with $A_x(0) = \llbracket 0\rrbracket$ satisfying
		$$\lim_{h\to 0}|h|^{-1}\G(u(x+h),u_a(x) + A_x(h)) = 0$$
		such that we may write $A_x = \sum^Q_{\alpha=1}\llbracket A^\alpha_x\rrbracket$, where for each $\alpha\in \{1,\dotsc,Q\}$, $A^\alpha_x:\R^n\to \R^k$ is a function of the following form: there is an $(n-1)$-dimensional subspace $S^\alpha\subset \R^n$ such that if $H_1^\alpha,H^\alpha_2$ are the (two, distinct) closed half-planes of $\R^n$ with boundary $S^\alpha$, then there are linear functions $L^\alpha_i:H^\alpha_i\to \R^k$ with $L^\alpha_1|_{S^\alpha} = L^\alpha_2|_{S^\alpha}$ such that $A^\alpha_x(y) = L^\alpha_i(y)$ whenever $y\in H^\alpha_i$.
		
		In particular, the varifold $\BC_x$ associated to the graph of $A_x$ is a cone taking the form
		$$\BC_x = \sum^{N_x}_{i=1}Q_x^i|\widetilde{H}^i_x|$$
		where $N_x\in\{1,2,\dotsc,Q\}$, $\widetilde{H}_i^x$ are distinct half-planes, and $Q_x^i\in \{1,\dotsc,Q\}$ are such that $\sum^{N_x}_{i=1}Q_x^i = 2Q$. Moreover, we assume that if $Q^i_x = Q$ for some $i$, then $\BC_x = Q|P|$ for some plane $P$.
		\item [(3)] The map sending $x\mapsto A_x$ is $\alpha$-Hölder continuous on the set $\Omega_u^Q$, where we define the metric on the image by $d(A_x,A_y) := \sup_{z\in B_1^n(0)}\G(A_x(z),A_y(z))$.
	\end{enumerate}
\end{defn}
Notice that this is a very crude Hölder-type assumption: we are just requiring some uniform Hölder continuity of the tangent cones on the $Q$-coincidence set, and separately local $\alpha$-Hölder regularity of the tangent cones away from the $Q$-coincidence set. This is therefore less restrictive than the definitions given in the special cases seen in \cite{BKMW25, MW24}; however it suffices for our purposes.

Now we recall the (\emph{multiplicity $Q$}) \emph{coarse blow-up class} associated to $\mathcal{V}$, denoted $\mathfrak{B}^Q_{\mathcal{V}}$, which consists of all \emph{coarse blow-ups} $v\in W^{1,2}(B^n_1(0);\A_Q(\R^k))$ generated by $\mathcal{V}$. These are functions generated by performing a vertical (inhomogeneous) blow-up procedure of the sequence of multi-valued Lipschitz approximations corresponding to a sequence $(V_j)_j\subset\mathcal{V}$ which converge, as varifolds in $\R^k\times B^n_1(0)$, to the plane $P_0$ with multiplicity $Q$. The reader can consult \cite[Part 2]{BKMW25} for the general definition of a coarse blow-up, with a ``coarse blow-up generated by $\mathcal{V}$'' being one in which the sequence of varifolds all belong to $\mathcal{V}$.

Our first aim is to establish a modification of the $\eps$-regularity property which guarantees closeness in energy to a coarse blow-up when instead normalising the derivative of the function in Definition \ref{defn:eps-reg} by the \emph{tilt excess}. This readily follows from the property given in Definition \ref{defn:eps-reg}; we give the argument for completeness. Here and throughout we will write $\mathfrak{B}_{\mathcal{V}}$ for the set of $Q$-valued functions given by sums of the form
$$v = \sum^N_{i=1}\sum^{Q_i}_{\alpha=1}\llbracket c_iv^\alpha_i\rrbracket$$
where $N\in \{1,\dotsc,Q\}$ and $Q_1,\dotsc,Q_N \in \{1,\dotsc,Q\}$ are positive integers such that $\sum^N_{i=1}Q_i=Q$, $c_i\in [0,\infty)$ are constants (possibly zero) and $v_i = \sum^{Q_i}_{\alpha=1}\llbracket v^\alpha_i(x)\rrbracket$ belong to $\mathfrak{B}_{\mathcal{V}}^{Q_i}$. As a shorthand, we will write such a sum as $v = \sum^N_{i=1}c_iv_i$, although the reader should note the (significant) abuse of notation here, as when working with Dirac masses certainly $2\llbracket P \rrbracket \neq \llbracket 2P\rrbracket$, however as for us the total number of values will always be fixed this notation shouldn't create any confusion.

\begin{theorem}\label{thm:eps-reg}
	There exists $\eps = \eps(\mathcal{V})\in (0,1)$ such that the following holds. Suppose $V\in \mathcal{V}$ satisfies
	\begin{itemize}
		\item $(\w_n 2^n)^{-1}\|V\|(B^{n+k}_2(0))<Q+1/2$;
		\item $Q-1/2\leq \w_n^{-1}\|V\|(\R^k\times B^n_1(0))\leq Q+1/2$;
		\item $\min\{\hat{E}_V,\Etilt_V\}<\eps$.
	\end{itemize}
	Then, there is a Lipschitz $Q$-valued function $u:B^n_{1/2}(0)\to \A_Q(\R^k)$ such that if $\widetilde{V} = V\res B_{15/8}(0)$, then
	$$\widetilde{V}\res (\R^k\times B^n_{1/2}(0)) = \mathbf{v}(u) \qquad \text{and} \qquad \Lip(u) \leq C\Etilt_V.$$
	Furthermore, for any $\eta>0$, there exists $\eps_0 = \eps_0(\mathcal{V},\eta)\in (0,1)$ such that if additionally $\eps<\eps_0$, then there is $v\in \mathfrak{B}_{\mathcal{V}}$ such that
	$$\int_{B^n_{1/2}(0)}\left||\Etilt_V^{-1}Du|^2 - |Dv|^2\right| < \eta.$$
	Here, $C = C(\mathcal{V})\in (0,\infty)$ and $\alpha = \alpha(\mathcal{V})\in (0,1)$.
\end{theorem}

\begin{proof}
	Suppose that there is a sequence $(\mathcal{V}_j)_j\subset\mathcal{V}$ with $V_j$ satisfying the assumptions of the theorem with
	\begin{equation}\label{E:eps-reg-1}
		\min\{\hat{E}_{V_j},\Etilt_{V_j}\}\to 0.
	\end{equation}
	To prove the theorem, it suffices to show that all the conclusions hold for infinitely many $j$.
	
	We claim that it suffices to assume instead of \eqref{E:eps-reg-1} that
	\begin{equation}\label{E:eps-reg-2}
		\hat{E}_{V_j} + \Etilt_{V_j}\to 0.
	\end{equation}
	Indeed, if from \eqref{E:eps-reg-1} we have $\hat{E}_{V_j}\to 0$, then one can use the reverse Poincaré inequality for stationary integral varifolds to deduce that the tilt excess of $V_j$ on any cylinder $\R^k\times B^n_\sigma(0)$ must $\to 0$, and therefore after a small rescaling we can assume \eqref{E:eps-reg-2}. If instead \eqref{E:eps-reg-1} gives that $\Etilt_{V_j}\to 0$, then since the $V_j$ have bounded supports (being varifolds in $B^{n+k}_2(0)$) we have that the $V_j$ converge as varifolds to at most $Q$ planes parallel to $P_0$. If the limit is supported on more than one plane, then since varifold convergence implies local convergence in Hausdorff distance, for any $\sigma>0$ and all sufficiently large $j$ we must have that $\spt\|V_j\|\cap (B^k_{31/16}(0)\times B^n_{1-\sigma}(0))$ has at least two connected components; the result then follows by induction, applying the $\eps$-regularity property to each component (each of which is close to a plane of multiplicity $<Q$)\footnote{Notice that, if a component $V_{j,i}$ of $V_j$ converges to a plane $P_i$ parallel to $P_0$, after translating we may assume $P_i = \{0\}^k\times\R^n$, and we construct coarse blow-ups by rescaling by $\hat{E}_{V_{j,i},P_i}$, with the correct scaling for the gradient being determined by a limit of $\hat{E}_{V_{j,i},P_i}/\Etilt_{V_{j,i}}\equiv \hat{E}_{V_{j,i},P_i}/\Etilt_{V_{j,i},P_i}\to c_i\in [0,\infty)$. In particular, by induction we would have $\Etilt_{V_{j,i}}^{-1}Du_{j,i}\to Dv_i$ in $L^2(B^n_{1/2}(0))$ for some $v_i$ in the blow-up class. The limit we are then interested in is $v = \sum^N_{i=1}\tilde{c}_iv_i$, where $\tilde{c}_i = \lim_{j\to\infty}\Etilt_{V_{j,i}}/\Etilt_{V_j}\in [0,1]$. This is why we must allow the possibility for distinct, and possibly zero, coefficients weighting each individual coarse blow-up in the definition of $\mathfrak{B}_{\mathcal{V}}$.}; the $Q=1$ case follows from Allard's regularity theorem and elliptic estimates for the minimal surface system. Hence, the only remaining case is when the limit is a single plane (with multiplicity $Q$), and hence the $L^2$ height excess on any $\R^k\times B^n_\sigma(0)$, $\sigma\in (0,1)$, must $\to 0$ as $j\to\infty$. Again, by rescaling $V_j$ by a small amount and a vertical translation, we can then reduce to \eqref{E:eps-reg-2}.
	
	We therefore have $V_j\weakly Q|P_0|$ as varifolds in $\R^k\times B^n_1(0)$. The $\eps$-regularity property from Definition \ref{defn:eps-reg} then establishes the existence of the function $u_j:B^n_{3/4}(0)\to \A_Q(\R^k)$ with 
	$$V_j\res (\R^k\times B^n_{3/4}(0)) = \mathbf{v}(u_j) \qquad \text{and} \qquad \Lip(u_j) \leq C\hat{E}_{V_j}.$$
	We now wish to control $Du$ by $\Etilt_V$ rather than $\hat{E}_V$. For this, notice that the Poincaré inequality for multi-valued functions (see \cite[Proposition 2.12]{DLS11}) gives
	\begin{equation}\label{E:eps-reg-3}
	\int_{B^n_{3/4}(0)}\G(u_j,\xi_j)^2 \leq C\int_{B^n_{3/4}(0)}|Du_j|^2
	\end{equation}
	for some $\xi_j\in \mathcal{A}_Q(\R^k)$. Notice that we can without loss of generality assume that $(\xi_j)_a = 0$, since $|\xi_j|^2 \leq C\int_{B_{3/4}^n(0)}|u_j|^2 + C\int_{B^n_{3/4}(0)}|Du_j|^2$ and the right-hand side of this inequality $\to 0$ as $j\to\infty$, and so translating $V_j$ by $(\xi_j)_a$ the resulting varifolds still converge to $P_0$. Write $\xi_j = \sum^Q_{\alpha=1}\llbracket \xi_j^\alpha\rrbracket$ where $\xi_j^\alpha\in \R^k$, and define $\boldsymbol{\pi}_j:= \cup^Q_{\alpha=1}\pi^\alpha_j$ where $\pi^\alpha_j := \{\xi^\alpha_j\}\times\R^n$. Notice that \eqref{E:eps-reg-3} gives
	\begin{align}
		\int_{\R^k\times B^n_{3/4}(0)}\dist^2(x,\boldsymbol{\pi}_j)\, \ext\|V_j\| + \int_{\R^k\times B^n_{3/4}(0)}\dist^2(x,&\,V_j)\, \ext\|\boldsymbol{\pi}_j\|\nonumber\\
		& \leq C\int_{\R^k\times B^n_{3/4}(0)}\|\pi_x-\pi_{P_0}\|^2\, \ext\|V_j\|.\label{E:eps-reg-4}
	\end{align}
	We now claim that for $\eta = \eta(\mathcal{V})\in (0,1/8)$ sufficiently small, if
	\begin{equation}\label{E:eps-reg-5}
		\int_{\R^k\times B^n_{3/4}}\dist^2(x,\boldsymbol{\pi}_j)\, \ext\|V_j\| + \int_{\R^k\times B^n_{3/4}}\dist^2(x,V_j)\, \ext\|\boldsymbol{\pi}_j\|\leq \eta^2\int_{\R^k\times B^n_{3/4}}\dist^2(x,P_0)\, \ext\|V_j\|,
	\end{equation}
	then $\spt\|V_j\|\cap (\R^k\times B_{5/8}^n(0))$ has two connected components; at this point, we can argue inductively as before to conclude the theorem. Indeed, this can be seen by a contradiction argument via performing a coarse blow-up and using the uniform convergence to the coarse blow-up (which would be a union of at least $2$ distinct planes, using that $\xi_j$ is average-free); notice that \eqref{E:eps-reg-5} implies $\xi_j\neq 0$, else $\boldsymbol{\pi}_j = P_0$.\footnote{One can alternatively argue directly that there must be two components under \eqref{E:eps-reg-5}, using the Lipschitz regularity of $u_j$ and simple measure estimates to show that one has $\spt\|V_j\|\cap (\R^k\times B^n_{5/8}(0))\cap \{\dist(\cdot,\boldsymbol{\pi}_j)\geq \frac{1}{4}\min_{\xi^\alpha,\xi^\beta\ \text{distinct}}|\xi^\alpha-\xi^\beta|\} = \emptyset$ for all $j$ sufficiently large and for $\eta$ sufficiently small. The main point here is that if this latter set were not empty, then by the Lipschitz regularity of $u_j$ there would be a sufficiently large ball where $V_j$ remains $\geq \frac{1}{8}\min_{\xi^\alpha,\xi^\beta\ \text{distinct}}|\xi^\alpha-\xi^\beta|$ from each plane in $\boldsymbol{\pi}_j$, which gives $\min_{\xi^\alpha,\xi^\beta\ \text{distinct}}|\xi^\alpha-\xi^\beta|\leq C\eta\hat{E}_{V_j}$, which when combined with \eqref{E:eps-reg-5} gives a contradiction via the triangle inequality.} Thus, we have established the theorem under \eqref{E:eps-reg-5}. 
	
	We are therefore left with the case when \eqref{E:eps-reg-5} fails. In which case, combining with \eqref{E:eps-reg-4} we have
	$$\int_{\R^k\times B^n_{3/4}(0)}\dist^2(x,P_0)\, \ext\|V_j\| \leq C\eta^{-2}\int_{\R^k\times B^n_{3/4}(0)}\|\pi_x-\pi_{P_0}\|^2.$$
	Hence, if we now apply our $\eps$-regularity property to $(\eta_{0,3/4})_\#V_j$ in place of $V_j$ and use a simple covering argument we get
	\begin{align*}
	\Lip(u_j|_{B^n_{1/2}(0)}) & \leq C\left(\int_{\R^k\times B^n_{3/4}(0)}\dist^2(x,P_0)\, \ext\|V\|\right)^{1/2}\\
	& \leq C\eta^{-1}\left(\int_{\R^k\times B^n_{3/4}(0)}\|\pi_x-\pi_{P_0}\|^2\, \ext\|V_j\|\right)^{1/2} \leq C\eta^{-1}\Etilt_V
	\end{align*}
	The second claim follows similarly, now using Theorem \ref{thm:energy-convergence} in order to deducing the energy convergence.
\end{proof}

\begin{remark}\label{remark:tilt-height-comparison}
	Under the assumptions of Theorem \ref{thm:eps-reg}, if one additionally assumes that $\Theta_V(0)\geq Q$ then one can directly verify the Poincaré-style inequality
	$$\int_{\R^k\times B^n_{\sigma/2}(0)}\dist^2(x,P_0)\, \ext\|V\| \leq C\int_{\R^k\times B^n_\sigma(0)}\|\pi_x-\pi_{P_0}\|^2\, \ext\|V\|$$
	for any given $\sigma\in (0,1)$, providing $\eps$ is small depending on $\sigma$ as well. This follows from the first variation formula and Almgren's (weak) Lipschitz approximation with error terms controlled by the tilt excess (arguing analogously to that seen in \cite[(7.4)]{Wic14a}). Thus in this situation, after rescaling, one can assume that the height excess and tilt excess on different cylinders are comparable.
\end{remark}

We now recall some of the main properties concerning coarse blow-ups $v\in \mathfrak{B}_{\mathcal{V}}$. These are direct consequences of the strong convergence in $W^{1,2}$ to the coarse blow-up guaranteed in Definition \ref{defn:eps-reg} (cf.~Appendix \ref{app:energy-convergence}) and the stationarity of the varifolds involved.

\begin{theorem}\label{thm:blow-up}
	If $v\in\mathfrak{B}_{\mathcal{V}}$ has $v\not\equiv 0$ in $B^n_1(0)$, then for any $y\in B^n_1(0)$ and $\rho\in (0,1-|y|)$ we have that $v\not\equiv 0$ on any ball $B^n_\rho(y)\subset B_1^n(0)$, and that the frequency function
	$$N_{v,y}(\rho):=\frac{\rho^{2-n}\int_{B_\rho(y)}|Dv|^2}{\rho^{1-n}\int_{\del B_\rho(y)}|v|^2}$$
	is well-defined and is a monotonically non-decreasing function of $\rho$. In particular, the frequency $N_v(y):= \lim_{\rho\downarrow 0}N_{v,y}(\rho)\in [0,\infty)$ exists for each $y\in B^n_1(0)$ and is an upper semi-continuous function of $y$.
	
	Furthermore, all of these conclusions also hold for $v_f$, the average-free part of $v$, in place of $v$.
\end{theorem}

From this we deduce some elementary consequences for $v\in \mathfrak{B}_{\mathcal{V}}$ and the $Q$-valued functions provided by Definition \ref{defn:eps-reg}. These will be used when proving approximate monotonicity of the planar frequency function in Section \ref{sec:pff}. Analogous estimates will also be needed when proving the monotonicity of the frequency function relative to a center manifold.

\begin{lemma}\label{lemma:doubling-1}
	Fix $c\in (0,\infty)$. Then there exists $\beta = \beta(\mathcal{V},c)\in (0,1)$ such that the following is true. If $v\in \mathfrak{B}_{\mathcal{V}}$ obeys
	$$\int_{B_{\rho/2}(z)}|Dv|^2 \geq c^2\int_{B_\rho(z)}|Dv|^2$$
	for some $B_\rho(z)\subset B_{3/4}(0)$, then for any $\zeta\in B_{\rho/5}(z)$ we have
	$$\int_{B_{\rho/20}(\zeta)}|Dv|^2\geq\beta^2\int_{B_\rho(z)}|Dv|^2.$$
\end{lemma}
\begin{proof}
	We argue by contradiction. Suppose instead that for some fixed $c\in (0,\infty)$ we have functions $v_j\in\mathfrak{B}_{\mathcal{V}}$, balls $B_{\rho_j}(z_j)\subset B_{3/4}(0)$, and points $\zeta_j\in B_{\rho_j/5}(z_j)$ such that
	\begin{equation}\label{E:doubling-1-1}
		\int_{B_{\rho_j/2}(z_j)}|Dv_j|^2\geq c^2\int_{B_{\rho_j}(z_j)}|Dv_j|^2 \qquad \text{yet} \qquad \int_{B_{\rho_j/20}(\zeta_j)}|Dv_j|^2 < \frac{1}{j^2}\int_{B_{\rho_j}(z_j)}|Dv_j|^2
	\end{equation}
	for all $j$. In particular, we must have $\int_{B_{\rho_j}(z_j)}|Dv_j|^2>0$ for all $j$. First, notice that by arguing\footnote{Recalling the definition of $\mathcal{B}_{\mathcal{V}}$, this means arguing as in \cite[$(\mathfrak{B}3\text{I})$]{BKMW25} for each $v_{j,i}$ in the decomposition $v_j = \sum_i c_{j,i}v_{j,i}$.} as in \cite[$(\mathfrak{B}3\text{I})$]{BKMW25} and using Theorem \ref{thm:eps-reg} with Appendix \ref{app:energy-convergence} (which gives the strong convergence in $L^2$ of the derivative along the blow-up sequence when normalising by the tilt excess) we have that there is $\tilde{v}_j\in\mathfrak{B}_{\mathcal{V}}$ for which $D\tilde{v}_j(x) = \left(\rho_j^{2-n}\int_{B_{\rho_j}(z_j)}|Dv_j|^2\right)^{-1/2}\rho_jDv_j(z_j+\rho_j x)$. Combining this with \eqref{E:doubling-1-1} we have
	\begin{equation}\label{E:doubling-1-2}
		\int_{B_{1/2}(0)}|D\tilde{v}_j|^2 \geq c^2 \qquad \text{yet} \qquad \int_{B_{1/20}(\tilde{\zeta}_j)}|D\tilde{v}_j|^2 < \frac{1}{j^2}
	\end{equation}
	where $\tilde{\zeta}_j := \rho_j^{-1}(\zeta_j-z_j)\in B_{1/5}(0)$. Now applying the compactness property for $\mathfrak{B}_{\mathcal{V}}$ (cf.~\cite[$(\mathfrak{B}3\text{IV})$]{BKMW25}) we may pass to a subsequence for which $D\tilde{v}_j\to D\tilde{v}$ for some $\tilde{v}\in\mathfrak{B}_{\mathcal{V}}$, where the convergence (again, due to Theorem \ref{thm:eps-reg} and Appendix \ref{app:energy-convergence}) is now \emph{strongly} in $L^{2}_{\text{loc}}(B_1(0))$. In particular, taking $j\to\infty$ in \eqref{E:doubling-1-2} gives
	$$\int_{B_{1/2}(0)}|D\tilde{v}|^2 \geq c^2 \qquad \text{and} \qquad \int_{B_{1/40}(\tilde{\zeta})}|D\tilde{v}|^2 = 0$$
	where we have also passed to a subsequence so that $\tilde{\zeta}_j\to \tilde{\zeta}\in \overline{B}_{1/5}(0)$; in particular, $|D\tilde{v}| = 0$ on $B_{1/40}(\tilde{\zeta})$. But as $\tilde{v}\in\mathfrak{B}_{\mathcal{V}}$ has a monotone frequency function by Theorem \ref{thm:blow-up}, this would imply that $|D\tilde{v}| = 0$ on $B_{1}(0)$, contradicting $\int_{B_{1/2}(0)}|D\tilde{v}|^2\geq c^2$. This contradiction completes the proof.
\end{proof}

\textbf{Remark:} The restriction on the ball $B_\rho(z)$ in Lemma \ref{lemma:doubling-1} is simply due to the fact that we do not know whether we have energy convergence on the whole of $B_1(0)$ along the blow-up sequence, but only locally on $B_1(0)$.

We now ``de-linearise'' Lemma \ref{lemma:doubling-1}, giving a version for the varifold itself.

\begin{lemma}\label{lemma:doubling-2}
	Fix $c\in (0,\infty)$, and let $\eps_* = \eps_*(\mathcal{V})$ be the constant from Theorem \ref{thm:eps-reg}. Then, there exists $\eps = \eps(\mathcal{V},c)\in (0,\eps_*)$ and $\beta = \beta(\mathcal{V},c)\in (0,1)$ such that the following holds. Suppose $V\in\mathcal{V}$ satisfies
	\begin{itemize}
		\item $Q-1/2 \leq \w_n^{-1}\|V\|(\R^k\times B_1^n(0))<Q+1/2$;
		\item $\Etilt_V<\eps$.
	\end{itemize}
	Then, if $u:B_{1/2}^n(0)\to \A_Q(\R^k)$ is the function given by Theorem \ref{thm:eps-reg}, and if for some $B_{8\sigma}(\xi)\subset B_{7/16}(0)$ we have
	\begin{equation}\label{E:doubling-2-0}
	\int_{B_{\sigma}(\xi)}|Du|^2 \geq c^2\int_{B_{8\sigma}(\xi)}|Du|^2,
	\end{equation}
	then for any $\zeta\in B_{2\sigma/5}(\xi)$ we have
	$$\int_{B_{\sigma/10}(\zeta)}|Du|^2\geq \beta^2\int_{B_{8\sigma}(\xi)}|Du|^2.$$
\end{lemma}

\begin{proof}
	Of course, we may assume that $|Du|\not\equiv 0$ on $B_{8\sigma}(\xi)$, else the result is trivial. Then, similarly to as in the proof of Lemma \ref{lemma:doubling-1}, we may use \cite[$(\mathfrak{B}3\text{I})$]{BKMW25} and Theorem \ref{thm:eps-reg} to translate and rescale to assume that $8\sigma=1$ and furthermore that $\Etilt_V^{-1}Du$ is close to $Dv$ in $L^2(B_{1/2}(0))$ for some $v\in\mathfrak{B}_{\mathcal{V}}$ (the important point here is that the normalising factor $\Etilt_V$ is defined on the cylinder with the same radius as the ball in the right-hand side of \eqref{E:doubling-2-0}, and so can be assumed to satisfy $(1-\delta)\|Du\|^2_{L^2(B_1(0))}\leq \Etilt_V^2\leq (1+\delta)\|Du\|^2_{L^2(B_1(0))}$ for a suitable choice of $\delta>0$, say $\delta = 1/2$).
	
	Thus, we may reduce to the case where we are assuming
	$$\int_{B_{1/8}(0)}|Du|^2 \geq c^2\int_{B_1(0)}|Du|^2$$
	and, for a fixed $\eta\in (0,1)$ (to be determined depending on $\mathcal{V}$ and $c$), for $\eps = \eps(\mathcal{V},c)>0$ sufficiently small we have
	$$\int_{B_{1/4}(0)}\left|\Etilt_V^{-1}|Du| - |Dv|\right|^2 \leq \eta$$
	for some $v\in\mathfrak{B}_{\mathcal{V}}$. Thus, the triangle inequality gives
	$$\int_{B_{1/4}(0)}|Dv|^2 \leq 2\int_{B_{1/4}(0)}\Etilt_V^{-2}|Du|^2 + 2\eta \leq 8,$$
	where the last inequality here is provided $\eps$ is sufficiently small. We also have
	$$\int_{B_{1/8}(0)}|Dv|^2 \geq \frac{1}{2}\int_{B_{1/8}(0)}\Etilt_V^{-2}|Du|^2 - \eta \geq \frac{c^2}{2}\Etilt_V^{-2}\int_{B_1(0)}|Du|^2 - \eta \geq \frac{c^2}{4}$$
	provided $\eta = \eta(\mathcal{V},c)\in (0,1)$ is sufficiently small. Combining the above two inequalities we get
	$$\int_{B_{1/8}(0)}|Dv|^2 \geq \frac{c^2}{32}\int_{B_{1/4}(0)}|Dv|^2.$$
	We may therefore apply Lemma \ref{lemma:doubling-1} to deduce the existence of $\beta = \beta(\mathcal{V},c)>0$ such that for any $\zeta\in B_{1/20}(0)$,
	$$\int_{B_{1/80}(\zeta)}|Dv|^2 \geq \beta^2\int_{B_{1/4}(0)}|Dv|^2.$$
	But then we have
	\begin{align*}
		\int_{B_{1/80}(\zeta)}|Du|^2 & \geq \frac{1}{2}\int_{B_{1/80}(\zeta)}\Etilt_V^2|Dv|^2 - \eta\Etilt_V^2\\
		& \geq \frac{\beta^2}{2}\int_{B_{1/4}(0)}\Etilt_V^2|Dv|^2 - \eta\Etilt_V^2\\
		& \geq \frac{\beta^2}{4}\int_{B_{1/4}(0)}|Du|^2 - \eta\left(1+\frac{\beta^2}{2}\right)\Etilt_V^2\\
		& \geq \frac{\beta^2 c^2}{4}\int_{B_1(0)}|Du|^2 - \eta\left(1+\frac{\beta^2}{2}\right)\Etilt_V^2\\
		& \geq \frac{\beta^2 c^2}{8}\int_{B_1(0)}|Du|^2
	\end{align*}
	provided $\eta = \eta(\mathcal{V},c)>0$ is sufficiently small. Applying the inverse homothety to that from the start of the proof completes the argument.
\end{proof}

\section{Planar Frequency}\label{sec:pff}

In this section we introduce the \emph{planar frequency function}, originally used by the first and third authors in their work on analysing the branch set of an area minimising current \cite{KW23a}. In this section we show its approximate monotonicity at density $Q$ branch points of varifolds $V\in\mathcal{V}$, and then subsequently use it to show that the set of (density $Q$) branch points where the planar frequency value is \emph{not} equal to $2$ has Hausdorff dimension $\leq n-2$.

Throughout, we fix $V\in\mathcal{V}$ and suppose $Z\in \spt\|V\|$ is such that $V$ has one tangent cone at $Z$ being a plane $P$ with multiplicity $Q$, and furthermore there is no $\delta>0$ for which $V\res B^{n+k}_\delta(Z) = Q|P+Z|\res B^{n+k}_\delta(Z)$ (e.g.~$Z\in \sing(V)$). Such assumptions hold, for instance, about a density $Q$ branch point of $V$. By Definition \ref{defn:eps-reg}, we know that this tangent cone is unique and, by translating and rotating, we may without loss of generality assume that $Z=0$ and $P = P_0 \equiv \{0\}^k\times\R^n$. Moreover, we know that there is some scale $\rho>0$ for which the conclusions of Definition \ref{defn:eps-reg} apply to $V$ in $B_\rho(0)$; by considering instead $(\eta_{0,\rho})_\#V$ in place of $V$, we may without loss of generality assume that $\rho=1$. In particular, we may assume that there is a generalised-$C^{1,\alpha}$ function $u:B^n_{1/2}(0)\to \A_Q(\R^k)$ with $u(0) = Q\llbracket 0\rrbracket$, $Du(0) = Q\llbracket 0\rrbracket$, which further obeys the estimates and conclusions in Definition \ref{defn:eps-reg} and Theorem \ref{thm:eps-reg} (recall also Remark \ref{remark:tilt-height-comparison}, which gives comparability of the height and tilt excess under the present assumptions; we also replace $V$ with $\widetilde{V}$ as in Theorem \ref{thm:eps-reg}). Thus, from the decay rate of the tilt excess as well as the Lipschitz bound, we have the pointwise bound
\begin{equation}\label{E:Lip-bound}
	|Du(x)| \leq C\Etilt_Vr^\alpha \qquad \text{for a.e.~$x\in B_{1/2}^n(0)$,}
\end{equation}
where $r = |x|$. For notational simplicity, we will write $L = C\Etilt_V$ for this Lipschitz constant. We also remark that one can assume that (cf.~\eqref{E:formula-4}), for $x$ where $Du(x)$ is defined, if $S$ is the tangent space to $V$ at the point $(u^\alpha(x),x)$ then we have
\begin{equation}\label{E:tilt-comparison}
\frac{1}{3}\|\pi_S-\pi_{P_0}\|^2 \leq |Du^\alpha(x)|^2 \leq \frac{2}{3}\|\pi_S-\pi_{P_0}\|^2.
\end{equation}
In particular, we can arrange so that $\frac{1}{2}\Etilt_V \leq \|Du\|_{L^2(B_1(0))}\leq 2\Etilt_V$.

\subsection{Approximate monotonicity of the planar frequency function}

Let $\phi:[0,\infty)\to \R$ be the Lipschitz function defined by
$$\phi(s):=\begin{cases}
	1 & \text{if }s\in [0,1/2);\\
	2-2s & \text{if }s\in [1/2,1);\\
	0 & \text{if }s\in [1,\infty).
\end{cases}$$
Technically, in our arguments we will need to take a smooth approximation of $\phi$ and then take a limit (cf.~\cite[Proof of Lemma 4.5]{KW23a}). For ease of presentation, we shall not comment further on this detail.

For $\rho\in (0,1/2)$ we define:
$$H(\rho):= -\rho^{1-n}\int\dist^2(X,P_0)|\nabla^S r|^2\frac{1}{r}\phi^\prime(r/\rho)\, \ext V(X,S),$$
$$D(\rho):= \frac{1}{2}\rho^{2-n}\int\|\pi_S-\pi_{P_0}\|^2\phi(r/\rho)\, \ext V(X,S).$$
Here, $r = |x|$ and $X = (y,x)\in \R^k\times \R^n\cong \R^{n+k}$, so that $\dist(X,P_0) \equiv |y|$. These are the same intrinsic expressions as defined in \cite{KW23a}. We then define the \emph{planar frequency function} by
$$\mathcal{N}(\rho):=\frac{D(\rho)}{H(\rho)}.$$
Note that this is well-defined. Indeed, if it were not then necessarily $H(\rho)=0$, i.e.~$|u|\equiv 0$ on $B_\rho(0)\setminus\overline{B}_{\rho/2}(0)$. But now take $\zeta\in C^1_c(B^n_1(0);\R)$ with $\zeta\equiv 1$ on $B_{3\rho/4}(0)$ and $\zeta\equiv 0$ outside $B_\rho(0)$. Consider its vertical extension $\tilde{\zeta}(y,x):= \zeta(x)$ for $(y,x)\in \R^k\times B^n_1(0)$, and take any $\bar{\zeta}\in C^1_c(\R^k\times B^n_1(0))$ which agrees with $\tilde{\zeta}$ on a neighbourhood of $\spt\|V\|$. Then, for any $\kappa\in \{1,\dotsc,k\}$, by taking the function $x^\kappa\bar{\zeta}e_\kappa$ in the first variation formula for $V$, where $e_\kappa$ is the unit vector in the $x^\kappa$-direction, we get
$$\int\bar{\zeta}|\nabla^Vx^\kappa|^2\, \ext\|V\| = -\int x^\kappa\nabla^Vx^\kappa\cdot\nabla^V\bar{\zeta}.$$
By choice of $\zeta$, the integrand on the right-hand side is supported on $\R^k\times (B_\rho(0)\setminus \overline{B}_{3\rho/4}(0))\cap \spt\|V\|$, which is where we know $x^\kappa\equiv 0$ (as $|u|\equiv 0$ here). Thus, we have for each $\kappa\in \{1,\dotsc,k\},$
$$\int\bar{\zeta}|\nabla^Vx^\kappa|^2\, \ext\|V\| = 0$$
which in turn implies that $|Du|\equiv 0$ on $B_{3\rho/4}(0)$, i.e.~$u$ is constant in $B_{3\rho/4}(0)$. But as $|u|\equiv 0$ on $B_{3\rho/4}(0)\setminus \overline{B}_{\rho/2}(0)$, this gives that $|u|\equiv 0$ on $B_\rho(0)$. But then this means that $V\res (\R^k\times B_\rho^n(0)) \equiv Q|B^n_\rho(0)|$, giving a contradiction. Thus, $\mathcal{N}(\rho)$ is well-defined for each $\rho\in (0,1/2)$.

\textbf{Notation:} We write $H_{V,P,Z}(\rho)$, $D_{V,P,Z}(\rho)$, and $\mathcal{N}_{V,P,Z}(\rho)$ for the above (intrinsic) quantities centred at a point $Z$ and relative to a plane $P$ passing through $Z$. These are exactly the above expressions when applied to the varifold $(\Gamma\circ\tau_Z)_\#V$, where $\tau_Z(X):= X-Z$ and $\Gamma$ is any rotation of $\R^{n+k}$ sending $P-Z$ to $P_0$.

\begin{remark}
When the codimension is $1$, the above expressions for $H(\rho)$ and $D(\rho)$ take a much simpler looking form in terms of the function $u$. This is due to the normal space to the graph of $u$ being $1$-dimensional. Indeed, by direct computation (similar to those done below) one may check that when $k=1$:
$$H(\rho) = -\rho^{1-n}\int|u|^2r^{-1}\phi^\prime(r/\rho)\cdot\frac{1+|D_\theta u|^2}{\sqrt{1+|Du|^2}},$$
$$D(\rho) = \rho^{2-n}\int|Du|^2\phi(r/\rho)\cdot\frac{1}{\sqrt{1+|Du|^2}}.$$
Here, $|D_\theta u|^2 = |Du|^2 - |D_ru|^2$ is a shorthand notation, where $D_r u := \frac{x}{|x|}\cdot Du$. Here, we are also implicitly summing over the $Q$-values of $u$ meaning, for example, that technically we have
$$D(\rho) = \sum^Q_{\alpha=1}\rho^{2-n}\int|Du^\alpha|^2\phi(r/\rho)\cdot\frac{1}{\sqrt{1+|Du^\alpha|^2}}.$$
This makes very explicit the comparison between the intrinsic expressions for $H(\rho)$ and $D(\rho)$ to those defined for a Dirichlet-minimising function. Here, we have additional factors coming from the non-linear nature of the minimal surface equation. These additional factors (and their counterparts in higher codimension) effectively allow us to deal with error terms which appear in the first variation formula when establishing the approximate monotonicity of the planar frequency function. Notice that each of these additional factors is comparable to $1$, since
$$\frac{1+|D_\theta u|^2}{\sqrt{1+|Du|^2}}\in \left[\frac{1}{\sqrt{1+L^2}},\sqrt{1+L^2}\right] \qquad \text{and} \qquad \frac{1}{\sqrt{1+|Du|^2}} \in \left[\frac{1}{\sqrt{1+L^2}}, 1\right].$$
In particular, as $L = C\Etilt_V$, we see that as we have strong convergence in $W^{1,2}_{\text{loc}}(B_1)$ along coarse blow-up sequences, these expressions will converge to the usual ones at the blow-up level (this will be important later). The same is also true in codimension $>1$, as in general:
$$H(\rho) = -\rho^{1-n}\int\sum^Q_{\alpha=1}\sum^n_{i,j=1}\sqrt{G(Du^\alpha)}G^{ij}(Du^\alpha)\frac{x_ix_j}{r^2}|u^\alpha|^2r^{-1}\phi^\prime(r/\rho),$$
$$D(\rho) = \rho^{2-n}\int\sum^Q_{\alpha=1}\sum^n_{i,j=1}\sqrt{G(Du^\alpha)}G^{ij}(Du^\alpha)D_iu^\alpha\cdot D_ju^\alpha \phi(r/\rho),$$
where here $G_{ij}(Du^\alpha) = \delta_{ij}+D_iu^\alpha\cdot D_ju^\alpha$ and, representing $(G_{ij}(Du^\alpha))_{ij}$ as an $n\times n$ matrix, $(G^{ij}(Du^\alpha))_{ij}$ denotes the inverse matrix of $(G_{ij}(Du^\alpha))_{ij}$ and $G(Du^\alpha)$ denotes the determinant of $(G_{ij}(Du^\alpha))_{ij}$. One can also use the first variation formula for $V$ to similarly write down general inner and outer variational formulas for $u$, certain special cases of which we will use below.
\end{remark}

We now show that the planar frequency function is approximately monotone. We begin by computing the derivatives of $H(\rho)$ and $D(\rho)$: these computations mirror those in \cite{KW23a}, except they are in fact simpler due to $V$ being graphical on the region of interest, along with the estimate provided in \eqref{E:Lip-bound} (this allows us to simplify the arguments needed to control certain error terms which arise).

We first compute the form of some key expressions at a point $(u^\alpha(x),x)$ in the graph of $u$ with tangent space $S$ (which is a plane a.e.). Recall that $X = (y,x)\in \R^k\times\R^n$ and $r=|x|$. Firstly, one may find an orthonormal basis $\{\tau_1,\dotsc,\tau_k\}$ for $S^\perp$ which takes the form
$$\tau_i = \frac{(e_i,-Du^{\alpha,i})}{\sqrt{1+|Du^{\alpha,i}|^2}} + \sum^{i-1}_{j=1}O(|Du|^2)e_j + \sum^{n+k}_{j=k+1}O(|Du|^3)e_j$$
(each coefficient in the two sums may depend on the summand, but are of the order written). For notational simplicity, we suppress the index $\alpha$ and just write $u(x)$. From this it follows that:
\begin{align}
	\pi_{P_0^\perp}(\nabla^S r) & = D_r u + O(|Du|^3);\label{E:formula-1}\\
	|\nabla^S r|^2 & = 1 - |D_r u|^2 + O(|Du|^4);\label{E:formula-2}\\
	|\nabla^{S^\perp}r|^2 & = |D_ru|^2 + O(|Du|^4);\label{E:formula-3}\\
	\frac{1}{2}\|\pi_S-\pi_{P_0}\|^2 & = |Du|^2 + O(|Du|^4).\label{E:formula-4}
\end{align}

\textbf{Step 1: Derivative of $H$.} The derivative of $H$ is (recall $|y| = \dist(X,P_0)$):
\begin{align*}
H^\prime(\rho) = (n-1)\rho^{-n}\int |y|^2|\nabla^S r|^2\frac{1}{r}\phi^\prime(r/\rho)\, \ext V(X,S)  + \rho^{-1-n}\int |y|^2|\nabla^S r|^2\phi^{\prime\prime}(r/\rho)\, \ext V(X,S).
\end{align*}
Now take $\zeta = |y|^2\phi^\prime(r/\rho)\nabla^{\R^{n+k}} r$ in the first variation formula for $V$. We get:
\begin{align*}
	\int 2(y,0)\cdot\pi_{P_0^\perp}(\nabla^S r)&\phi^\prime(r/\rho) + \frac{1}{\rho}|y|^2|\nabla^S r|^2\phi^{\prime\prime}(r/\rho)\\
	& + |y|^2r^{-1}\phi^\prime(r/\rho)\left[n-\frac{1}{2}\|\pi_S-\pi_{P_0}\|^2 - |\nabla^S r|^2\right]\, \ext V(X,S) = 0.
\end{align*}
Multiplying this by $\rho^{-n}$ and rearranging, using the above expression for $H^\prime(\rho)$, we see
\begin{align*}
H^\prime(\rho) = -2\rho^{-n}\int(y,0)\cdot\pi_{P_0^\perp}&(\nabla^S r)\phi^\prime(r/\rho)\, \ext V(X,S)\\
& \underbrace{-\rho^{-n}\int|y|^2 r^{-1}\phi^\prime(r/\rho)\left[n|\nabla^{S^\perp}r|^2 - \frac{1}{2}\|\pi_S-\pi_{P_0}\|^2\right]\, \ext V(X,S)}_{=:\,\text{error}_1}.
\end{align*}
So, we have
\begin{equation}\label{E:H-prime}
	H^\prime(\rho) = -2\rho^{-n}\int(y,0)\cdot\pi_{P_0^\perp}(\nabla^S r)\phi^\prime(r/\rho)\, \ext V(X,S) + \text{error}_1
\end{equation}
where $\text{error}_1$ is as above. Using \eqref{E:formula-3} and \eqref{E:formula-4} along with \eqref{E:Lip-bound}, we see that we can control this error term by
\begin{equation}\label{E:H-prime-error}
	|\text{error}_1| \leq 2\cdot Q(n+1)\cdot L^2\rho^{2\alpha}\cdot \rho^{-1}\cdot H(\rho) = 2Q(n+1)L^2\rho^{2\alpha-1}H(\rho),
\end{equation}
where the factor of $2$ at the front is coming from the $O(|Du|^4)$ terms in \eqref{E:formula-3}, \eqref{E:formula-4} and the error from $1$ in \eqref{E:formula-2}.

\textbf{Step 2: Alternative expression for $D$.} Now let us turn to $D(\rho)$. If one takes $\zeta = \phi(r/\rho)(y,0)$ in the first variation formula for $V$, we get
$$\int\frac{1}{2}\|\pi_S-\pi_{P_0}\|^2\phi(r/\rho) + \frac{1}{\rho}\phi^\prime(r/\rho)(y,0)\cdot \pi_{P^\perp_0}(\nabla^S r)\, \ext V(X,S) = 0.$$
Multiplying by $\rho^{2-n}$ we see this immediately gives
\begin{equation}\label{E:D-alt}
	D(\rho) = -\rho^{1-n}\int(y,0)\cdot\pi_{P_0^\perp}(\nabla^S r)\phi^\prime(r/\rho)\, \ext V(X,S).
\end{equation}
\textbf{Step 3: Derivative of $D$.} Next, note that we have (from the original expression for $D(\rho)$):
\begin{equation}\label{E:pff-D-prime-1}
D^\prime(\rho) = \frac{2-n}{2}\rho^{1-n}\int\|\pi_S-\pi_{P_0}\|^2\phi(r/\rho)\, \ext V(X,S) - \frac{1}{2}\rho^{-n}\int\|\pi_S-\pi_{P_0}\|^2 r\phi^\prime(r/\rho)\, \ext V(X,S).
\end{equation}
Now if we take $\zeta = \phi(r/\rho)(0,x)$ in the first variation formula for $V$, we get
\begin{equation}\label{E:pff-D-prime-2}
\int n\phi(r/\rho) - \frac{1}{2}\|\pi_S-\pi_{P_0}\|^2\phi(r/\rho) + |\nabla^S r|^2 \frac{r}{\rho}\phi^\prime(r/\rho)\, \ext V(X,S) = 0.
\end{equation}
Now by the divergence theorem (or using the above in the special case where $V$ is the multiplicity one plane $P_0$):
$$0 = \int\div(x,\phi(r/\rho))\, \ext\H^n(x) = \int n\phi(r/\rho) + \frac{r}{\rho}\phi^\prime(r/\rho)\, \ext\H^n(x).$$
We know (from the fact $V$ is represented by a multi-valued graph) that $V$ is a $Q$-cover of $P_0$ in this region, and so this and the area formula gives
\begin{equation}\label{E:pff-D-prime-3}
0 = \int \left(n\phi(r/\rho) + \frac{r}{\rho}\phi^\prime(r/\rho)\right) J\pi_{P_0}\, \ext V(X,S)
\end{equation}
where $J\pi_{P_0}$ is the Jacobian of $\pi_{P_0}:\spt\|V\|\to P_0$. If we now consider $\eqref{E:pff-D-prime-1} + 2\rho^{1-n}\cdot\eqref{E:pff-D-prime-2} - 2\rho^{1-n}\cdot\eqref{E:pff-D-prime-3}$, we get, using $|\nabla^S r|^2 = 1-|\nabla^{S^\perp}r|^2$,
\begin{equation}
	D^\prime(\rho) = -2\rho^{-n}\int|\nabla^{S^\perp}r|^2 r\phi^\prime(r/\rho)\, \ext V(X,S) + \text{error}_2
\end{equation}
where
$$\text{error}_2 := 2\rho^{1-n}\int\left(1-\frac{1}{4}\|\pi_S-\pi_{P_0}\|^2 - J\pi_{P_0}\right)\left(n\phi(r/\rho) + \frac{r}{\rho}\phi^\prime(r/\rho)\right)\, \ext V(X,S).$$
We then make one small further modification to the expression for $D^\prime(\rho)$. Notice that (see \cite[(3.61)]{KW23a})
$$\left||\nabla^{S^\perp}r|^2 - \frac{|\pi_{P_0^\perp}(\nabla^S r)|^2}{|\nabla^S r|^2}\right| = O(\|\pi_S-\pi_{P_0}\|^4)$$
(note that $\nabla^S r\neq 0$ here, since $\nabla^S r = 0$ would mean $S$ is orthogonal to $P_0$ which cannot happen due to our estimates in Definition \ref{defn:eps-reg}), and so we actually have
\begin{equation}\label{E:pff-D-prime-4}
	D^\prime(\rho) = -2\rho^{-n}\int\frac{|\pi_{P_0^\perp}(\nabla^S r)|^2}{|\nabla^S r|^2}\cdot r\phi^\prime(r/\rho)\, \ext V(X,S) + \text{error}_2 + \text{error}_3
\end{equation}
where (recall $-\phi^\prime\geq 0$)
$$|\text{error}_3| \leq C\rho^{1-n}\int\|\pi_S-\pi_{P_0}\|^4\cdot\left(-\frac{r}{\rho}\phi^\prime(r/\rho)\right)\, \ext V(X,S).$$
\textbf{Step 4: Control $\text{error}_2$ and $\text{error}_3$.} In order to show the approximate monotonicity of the planar frequency (which is done in Step 5), we will need to control $|\text{error}_2|$ and $|\text{error}_3|$ in terms of $D(\rho)$. Notice first that from \eqref{E:formula-4}, we know that $D(\rho)$ is comparable to
$$\widetilde{D}(\rho) = \rho^{2-n}\int |Du|^2\phi(r/\rho)\, \ext x$$
and that (see \cite[(3.62)]{KW23a})
$$1 - \frac{1}{4}\|\pi_S-\pi_{P_0}\|^2 - J\pi_{P_0} = O(\|\pi_S-\pi_{P_0}\|^4).$$
Thus, in order to control $\text{error}_2$ and $\text{error}_3$, we can instead control
$$\text{error}_{2,1} := \rho^{1-n}\int|Du|^4\phi(r/\rho) \qquad \text{and} \qquad \text{error}_{2,2} := \rho^{1-n}\int_{B_\rho\setminus B_{\rho/2}}|Du|^4$$
in terms of $\widetilde{D}(\rho)$ (here, we have used that $\phi^{\prime}(r/\rho) = -2\one_{\R^k\times (B_\rho\setminus\overline{B_{\rho/2}})}$, and so $r/\rho\in [1/2,1]$ in this part of the integrand). These are fourth order error terms in $Du$. Notice that using the bound $|Du(x)|\leq L r^\alpha$ a.e.~from \eqref{E:Lip-bound}, we get
$$\text{error}_{2,1} \leq L^2\rho^{2\alpha-1}\cdot \rho^{2-n}\int|Du|^2\phi(r/\rho) = L^2\rho^{2\alpha-1}\widetilde{D}(\rho).$$
Thus, we just need to bound $\text{error}_{2,2}$. The reader will then note the difficulty: the integral in $\text{error}_{2,2}$ occurs over $B_\rho\setminus B_{\rho/2}$, whilst the ``main'' part of $\widetilde{D}(\rho)$ (i.e.~coming from that where $\phi(r/\rho)$ is $1$) is $\int_{B_{\rho/2}}|Du|^2$. One therefore wishes to control the energy of $u$ on $B_{\rho}\setminus \overline{B}_{\rho/2}$ by that on $B_{\rho/2}$: this essentially amounts to wanting $Du$ to satisfy a doubling condition. However, even though we know that the coarse blow-ups satisfy a doubling condition, we cannot pass this to the blow-up sequence at every scale $\rho$. Our argument therefore needs to be more refined: for this we argue as in \cite{KW23a}, utilising Lemma \ref{lemma:doubling-2}.

Thus, in order to control $\text{error}_{2,2}$, it suffices to bound
$$\text{error}_2^*:= \rho^{1-n}\int_{B_\rho}|Du|^4$$
in terms of $\widetilde{D}(\rho)$. For this, fix $\rho\in (0,1/4)$. For each $\xi\in B_\rho(0)$ set
$$\overline{\sigma}_\xi := \sup\left\{\sigma\in (0,\rho/2]:\int_{B_\sigma(\xi)}|Du|^2 \geq \w_n L^2 (64\sigma)^{n+2\alpha}\right\}$$
where we take the convention that $\sup\emptyset = 0$. Here, $\w_n = \H^n(B^n_1(0))$. We know from the estimate \eqref{E:Lip-bound} that for all $\xi\in B_\rho(0)$ and $\sigma\in [\rho/32,\rho/2]$,
$$\int_{B_\sigma(\xi)}|Du|^2 \leq \int_{B_{3\rho/2}(0)}|Du|^2 \leq L^2(3\rho/2)^{2\alpha}\cdot \w_n(3\rho/2)^n \leq L^2\w_n (48\sigma)^{n+2\alpha}$$
and thus we see that $\overline{\sigma}_\xi \leq \rho/32$. Now set
$$\S:= \{\xi\in B_\rho(0):\overline{\sigma}_\xi >0\}.$$
Notice that for $\xi\in B_\rho(0)\setminus \S$ we have $|Du(\xi)| = 0$.
By the Vitali covering lemma, there is a countable set of points $\{\xi_i\}_i\subset \S$ such that $\{B_{2\sigma_i/5}(\xi_i)\}_i$ is a pairwise disjoint collection of balls such that $\{B_{2\sigma_i}(\xi_i)\}_i$ covers $\S$, where $\sigma_i := \overline{\sigma}_{\xi_i}$. By definition of $\overline{\sigma}_\xi$ we have
$$\int_{B_{4\sigma_i}(\xi_i)}|Du|^2 \leq \w_nL^2(256\sigma_i)^{n+2\alpha}$$
and so, as $L = C\Etilt_V$ is small, this corresponds to saying that the tilt-excess of $V$ in the cylinder $\R^k\times B_{4\sigma_i}(\xi_i)$ is small. Thus, applying Theorem \ref{thm:eps-reg} in this cylinder, we see that on $B_{2\sigma_i}(\xi_i)$ we have
$$|Du| \leq \sigma_i^{-1}\cdot C\left((4\sigma_i)^{2-n}\int_{B_{4\sigma_i}(\xi_i)}|Du|^2\right)^{1/2} \leq CL\sigma_i^\alpha \qquad \text{a.e.,}$$
where here $C = C(\mathcal{V})\in (0,\infty)$, and hence
\begin{equation}\label{E:D-0}
	\int_{B_{2\sigma_i}(\xi_i)}|Du|^4 \leq CL^4\sigma_i^{n+4\alpha}.
\end{equation}
Since $\xi_i\in B_\rho(0)$ and $\sigma_i\leq \rho/32$, we may find a ball $B_{\sigma_i/5}(\zeta_i)\subset B_{2\sigma_i/5}(\xi_i)\cap B_\rho(0)$. By definition of $\sigma_i$ we have
\begin{equation}\label{E:D-1}
	\int_{B_{\sigma_i}(\xi_i)}|Du|^2\geq \w_n L^2(64\sigma_i)^{n+2\alpha}
\end{equation}
and from the original bound on $|Du|$ within $B_{8\sigma_i}(\xi_i)\subset B_{1/2}(0)$ from the definition of $\overline{\sigma}_{\xi_i}$ we have
$$\int_{B_{8\sigma_i}(\xi_i)}|Du|^2 \leq \w_n L^2(64\cdot 8\sigma_i)^{n+2\alpha}$$
and so, using the fact that $\alpha<1$, we get
$$\int_{B_{\sigma_i}(\xi_i)}|Du|^2 \geq 8^{-n-2}\int_{B_{8\sigma_i}(\xi_i)}|Du|^2.$$
Thus, we are in a situation to use Lemma \ref{lemma:doubling-2}: we get that there exists a constant $\beta = \beta(\mathcal{V})\in (0,1)$ such that
\begin{equation}\label{E:D-2}
	\int_{B_{\sigma_i/10}(\zeta_i)}|Du|^2\geq \beta^2\int_{B_{8\sigma_i}(\xi_i)}|Du|^2.
\end{equation}
Using simple inclusions and \eqref{E:D-1}, \eqref{E:D-2}, we clearly have
\begin{equation}\label{E:D-3}
	\int_{B_{2\sigma_i/5}(\xi_i)\cap B_\rho(0)}|Du|^2 \geq \int_{B_{\sigma_i/5}(\zeta_i)}|Du|^2 \geq \beta^2\int_{B_{8\sigma_i}(\xi_i)}|Du|^2 \geq \beta^2\w_n L^2 (64\sigma_i)^{n+2\alpha}.
\end{equation}
Now since $B_{\sigma_i/5}(\zeta_i)\subset B_\rho(0)$, we have that $r\leq \rho - \sigma_i/10$ on $B_{\sigma_i/10}(\zeta_i)$, and thus $\phi(r/\rho)\geq \sigma_i/(5\rho)$ on $B_{\sigma_i/10}(\zeta_i)$. Thus, from \eqref{E:D-3} and \eqref{E:D-2},
\begin{align*}
	\beta^2\w_n L^2(64\sigma_i)^{n+1+2\alpha}\leq 64\sigma_i \int_{B_{2\sigma_i/5}(\xi_i)}|Du|^2 & \leq 64\beta^{-2}\sigma_i\int_{B_{\sigma_i/10}(\zeta_i)}|Du|^2\\
	& \leq 320\beta^{-2}\rho\int_{B_\rho(0)}|Du|^2\phi(r/\rho)\\
	& \leq 320\beta^{-2}\cdot \rho^{n-1}\widetilde{D}(\rho),
\end{align*}
which gives the bound
\begin{equation}\label{E:D-4}
	\sigma_i \leq CL^{-\frac{2}{n+1+2\alpha}}\widetilde{D}(\rho)^{\frac{1}{n+1+2\alpha}}.
\end{equation}
where $C = C(\mathcal{V})$.

Recalling now that we had a cover of $\S$, we therefore have from \eqref{E:D-0}, \eqref{E:D-3}, and \eqref{E:D-4} (as $|Du| = 0$ on $B_\rho(0)\setminus \S$)
\begin{align*}
	\int_{B_\rho(0)}|Du|^4 = \int_{\S} |Du|^4 & \leq \sum_i \int_{B_{2\sigma_i}(\xi_i)}|Du|^4\\
	& \leq C\sum_i L^4\sigma_i^{n+4\alpha}\\
	& \leq C\sum_i L^2\sigma_i^{2\alpha}\int_{B_{2\sigma_i/5}(\xi_i)\cap B_\rho(0)}|Du|^2\\
	& \leq CL^{2-\frac{4\alpha}{n+1+2\alpha}}\widetilde{D}(\rho)^{\frac{2\alpha}{n+1+2\alpha}}\sum_i\int_{B_{2\sigma_i/5}(\xi_i)\cap B_\rho(0)}|Du|^2\\
	& \leq CL^{2-2\gamma}\widetilde{D}(\rho)^\gamma \int_{B_\rho(0)}|Du|^2,
\end{align*}
where here $C = C(\mathcal{V})\in (0,\infty)$ and $\gamma = \frac{2\alpha}{n+1+2\alpha}$.

Notice that by using \eqref{E:formula-4}, we can then instead change this into a bound depending on terms closer to $D(\rho)$ instead of $\widetilde{D}(\rho)$, and so we have
$$\int_{B_\rho(0)}|Du|^4 \leq CL^{2-2\gamma}D(\rho)^\gamma\int_{\R^k\times B_\rho(0)}\|\pi_S-\pi_{P_0}\|^2\, \ext V(X,S).$$
But now noting that $\phi(r/\rho) - \frac{r}{\rho}\phi^\prime(r/\rho) = 1$ on $\R^k\times B_{\rho/2}(0)$ and $=2$ on $\R^k\times (B_\rho(0)\setminus\overline{B}_{\rho/2}(0))$, we have
\begin{align*}
\int_{\R^k\times B_\rho(0)}\|\pi_S-\pi_{P_0}\|^2\, \ext V(X,S) & \leq \int\|\pi_S-\pi_{P_0}\|^2\left(\phi(r/\rho)-\frac{r}{\rho}\phi^\prime(r/\rho)\right)\\
& = 2(n-1)\rho^{n-2}D(\rho) + 2\rho^{n-1}D^\prime(\rho).
\end{align*}
Thus, we see that
$$\int_{B_\rho(0)}|Du|^4 \leq CL^{2-2\gamma}D(\rho)^\gamma\left((n-1)\rho^{n-2}D(\rho) + \rho^{n-1}D^\prime(\rho)\right)$$
and hence
$$|\text{error}_2^*| \leq CL^{2-2\gamma}\rho^{-1}D(\rho)^{\gamma}\left((n-1)D(\rho) + \rho D^\prime(\rho)\right)$$
which is our desired bound. To summarise, we have now shown
\begin{equation}\label{E:D-prime-error}
|\text{error}_2| + |\text{error}_3| \leq L^2\rho^{2\alpha-1}D(\rho) + CL^{2-2\gamma}\rho^{-1}D(\rho)^{\gamma}\left((n-1)D(\rho) + \rho D^\prime(\rho)\right).
\end{equation}
\textbf{Step 5: Establish the approximate monotonicity of the planar frequency.} Now we are in a position to prove the almost monotonicity of the planar frequency function. We compute, using \eqref{E:pff-D-prime-4}, \eqref{E:D-alt}, \eqref{E:H-prime}:
\begin{align*}
	\frac{\ext}{\ext\rho}\log\mathcal{N}(\rho) & = \frac{D^\prime(\rho) H(\rho) - H^\prime(\rho) D(\rho)}{D(\rho)H(\rho)}\\
	& = \frac{2\rho^{1-2n}}{D(\rho)H(\rho)}\left[\left(\int\frac{|\pi_{P_0^\perp}(\nabla^S r)|^2}{|\nabla^S r|^2}r(-\phi^\prime(r/\rho))\, \ext V(X,S)\right)\left(\int|y|^2|\nabla^Sr|^2\frac{1}{r}(-\phi^\prime(r/\rho))\right)\right.\\
	& \hspace{2em} - \left.\left(\int(y,0)\cdot\pi_{P_0^\perp}(\nabla^S r)(-\phi^\prime(r/\rho))\right)^2\right] + \frac{\text{error}_2}{D(\rho)} + \frac{\text{error}_3}{D(\rho)} - \frac{\text{error}_1}{H(\rho)}.
\end{align*}
By Cauchy--Schwarz (noting that $-\phi^\prime\geq 0$) the non-error part of this expression is $\geq 0$, and thus we have
$$\frac{\ext}{\ext\rho}\log\mathcal{N}(\rho) \geq -\frac{|\text{error}_1|}{H(\rho)} - \frac{|\text{error}_2|}{D(\rho)} - \frac{|\text{error}_3|}{D(\rho)}.$$
Now to control these error terms, we use \eqref{E:H-prime-error} and \eqref{E:D-prime-error}. Indeed, \eqref{E:H-prime-error} gives
$$\frac{|\text{error}_1|}{H(\rho)} \leq 2Q(n+1)L^2\rho^{2\alpha-1}$$
and \eqref{E:D-prime-error} gives
$$\frac{|\text{error}_2|+|\text{error}_3|}{D(\rho)} \leq L^2\rho^{2\alpha-1} + CL^{2-2\gamma}\rho^{-1}D(\rho)^{\gamma-1}\left((n-1)D(\rho) + \rho D^\prime(\rho)\right).$$
Recalling from \eqref{E:Lip-bound} that $|Du(x)|\leq L r^\alpha$ for a.e.~$x$, we therefore have $D(\rho) \leq CL^2\rho^{2+2\alpha}$, and so we get
\begin{align*}
\frac{|\text{error}_2|+|\text{error}_3|}{D(\rho)} & \leq L^2\rho^{2\alpha-1} + CL^2\rho^{2\gamma + 2\gamma\alpha -1} + CL^{2-2\gamma}D(\rho)^{\gamma-1}D^\prime(\rho)\\
& \leq CL^2\rho^{2\gamma(1+\alpha)-1} + CL^{2-2\gamma}D(\rho)^{\gamma-1}D^\prime(\rho)
\end{align*}
since $2\alpha-1\geq 2\gamma+2\gamma\alpha-1$. Thus we see
\begin{align*}
\frac{\ext}{\ext\rho}\log\mathcal{N}(\rho) & \geq - CL^2\rho^{2\alpha-1} - CL^2\rho^{2\gamma(1+\alpha)-1} - CL^{2-2\gamma}D(\rho)^{\gamma-1}D^\prime(\rho)\\
& \geq - CL^2\rho^{2\gamma(1+\alpha)-1} - CL^{2-2\gamma}D(\rho)^{\gamma-1}D^\prime(\rho).
\end{align*}
Integrating this inequality over an interval $[\sigma,\rho]$ we get
$$\log\mathcal{N}(\rho) - \log\mathcal{N}(\sigma) \geq -CL^2\left(\rho^{2\gamma(1+\alpha)}-\sigma^{2\gamma(1+\alpha)}\right) - CL^{2-2\gamma}\left(D(\rho)^\gamma-D(\sigma)^\gamma\right)$$
i.e.
$$\mathcal{N}(\sigma)\leq \mathcal{N}(\rho)\cdot e^{CL^2(\rho^{2\gamma(1+\alpha)}-\sigma^{2\gamma(1+\alpha)})+L^{2-2\gamma}(D(\rho)^\gamma-D(\sigma)^\gamma)}.$$
Bounding again $D(\rho) \leq CL^2\rho^{2+2\alpha}$, we therefore get for any $0<\sigma\leq\rho<1/4$,
$$\mathcal{N}(\sigma)\leq \mathcal{N}(\rho)\cdot e^{CL^2\rho^{2\gamma(1+\alpha)}}$$
where $C = C(\mathcal{V})$ and $\gamma = \gamma(\mathcal{V}) \equiv \frac{2\alpha}{n+1+2\alpha}$. This is the desired almost monotonicity of the planar frequency.

To summarise, we have shown the following:

\begin{theorem}[Approximate Monotonicity of Planar Frequency]\label{thm:monotonicity}
	There exists $\eps = \eps(\mathcal{V})\in (0,1)$ such that the following holds. Suppose $V\in\mathcal{V}$ is such that $0\in \sing(V)$ is a density $Q$ singular point of $V$ with (unique) tangent cone $Q|P_0|$ and $\Etilt_V<\eps$. Then, for all $0<\sigma\leq\rho<1/4$,
	$$e^{C\Etilt_V^2\sigma^{2\gamma(1+\alpha)}+C\Etilt_V^{2-2\gamma}D(\sigma)^\gamma}\mathcal{N}(\sigma) \leq e^{C\Etilt_V^2\rho^{2\gamma(1+\alpha)} + C\Etilt_V^{2-2\gamma}D(\rho)^\gamma}\mathcal{N}(\rho),$$
	and hence
	$$\mathcal{N}(\sigma)\leq \mathcal{N}(\rho)\cdot e^{C\Etilt_V^2\rho^{2\gamma(1+\alpha)}}.$$
	Here, $\alpha = \alpha(\mathcal{V})\in (0,1)$ is the constant from Definition \ref{defn:eps-reg}, $\gamma = \frac{2\alpha}{n+1+2\alpha}$, and $C = C(\mathcal{V})\in (0,\infty)$.
\end{theorem}

\begin{remark}
In fact, now that one has a monotone frequency function, one can show a simpler version of the doubling property one needed to control $\text{error}_2^*$. Consequently, one can then prove
$$\frac{|\text{error}_2|+|\text{error}_3|}{D(\rho)} \leq CL^2\rho^{2\alpha-1}$$
and consequently that
$$\frac{\ext}{\ext \rho}\log\mathcal{N}(\rho) \geq -CL^2\rho^{2\alpha-1}.$$
This then gives the simpler-looking monotonicity property: for all $0<\sigma\leq\rho<1/4$,
$$e^{C\Etilt_V^2\sigma^{2\alpha}}\mathcal{N}(\sigma)\leq e^{C\Etilt_V^2\rho^{2\alpha}}\mathcal{N}(\rho).$$
This then simplifies the form of some properties later as well. We will not use this however, and stick with the version in Theorem \ref{thm:monotonicity}. We stress that the constant $C$ in this remark would also depend on the frequency value at (say) scale $1$, whilst the corresponding constant $C$ in Theorem \ref{thm:monotonicity} does not.
\end{remark}

In the setting of Theorem \ref{thm:monotonicity}, we see that the limit
$$\mathcal{N}_V(0) := \lim_{\rho\downarrow 0}\mathcal{N}(\rho)$$
always exists; we call $\mathcal{N}_V(0)$ the \emph{planar frequency of $V$ at $0$}. Of course, at any singular point of such a varifold $V$ where the (unique) tangent cone is a plane of multiplicity $Q$, then up to rotation, translation, and scaling, the hypotheses of Theorem \ref{thm:monotonicity} are satisfied, and thus at any such singular point $Z$ we can define the frequency $\mathcal{N}_V(Z)$ as the limit at $0$ of the planar frequency centred at the point $Z$ and rotated so that the tangent plane is $Q|P_0|$.

\textbf{Note:} In the area-minimising case as in \cite{KW23a}, it is not known a priori whether there is a decay rate above $1$ at branch points, and the planar frequency function (among other things) is used to establish that at ${\mathcal H}^{n-2}$ a.e.\ branch point the decay rate must exceed $1 + \alpha$ for some fixed constant $\alpha$ depending only on $n$, $m$ and a mass bound. See \cite{KW23a}, \cite{KW23b}. In contrast to this, here the generalised-$C^{1,\alpha}$ regularity of $u$ guarantees (as just shown) that the planar frequency exists and is approximately monotone at \emph{every} such singular point, and the planar frequency takes values $\geq 1 + \alpha$. 

Next, we deduce some elementary consequences of such a monotonicity result. The first is control on the scale-invariant $L^2$-height in terms of the planar frequency function, which tells us that the planar frequency captures the decay rate of the varifold to its tangent plane. This also provides a uniform lower bound on the planar frequency value $\mathcal{N}_V(0)$, which will be important for our blow-up analysis later.

\begin{corollary}\label{cor:decay-estimates}
	There exists $\eps = \eps(\mathcal{V})\in (0,1)$ such that if $V$ is as in Theorem \ref{thm:monotonicity} with this choice of $\eps$, then:
	\begin{enumerate}
		\item [(i)] For all $0<\sigma\leq\rho<1/4$ we have
		$$e^{-C\Etilt_V^2\rho^{2\alpha}}\left(\frac{\sigma}{\rho}\right)^{2\mathcal{N}(\rho)e^{C\Etilt_V^2\rho^{2\gamma(1+\alpha)}}}H(\rho) \leq H(\sigma) \leq e^{C(\mathcal{N}_V(0)+1)\Etilt_V^2\rho^{2\gamma(1+\alpha)}}\left(\frac{\sigma}{\rho}\right)^{2\mathcal{N}_V(0)}H(\rho);$$
		\item [(ii)] For all $0<\sigma\leq\rho\leq 1/4$ we have
		\begin{align*}
			&\frac{1}{16}e^{-C\Etilt_V^2\rho^{2\alpha}}\left(\frac{\sigma}{\rho}\right)^{2\mathcal{N}(\rho)e^{C\Etilt_V^2\rho^{2\gamma(1+\alpha)}}}\rho^{-n}\int_{B_\rho(0)}|u|^2\\
			& \hspace{13em} \leq \sigma^{-n}\int_{B_\sigma(0)}|u|^2\\
			& \hspace{13em} \leq 16e^{C(\mathcal{N}_V(0) + 2)\Etilt_V^2\rho^{2\gamma(1+\alpha)}}\left(\frac{\sigma}{\rho}\right)^{2\mathcal{N}_V(0)}\rho^{-n}\int_{B_\rho(0)}|u|^2;
		\end{align*}
		\item [(iii)] $\mathcal{N}_V(0)\geq 1+\alpha$.
	\end{enumerate}
	Here, $\alpha = \alpha(\mathcal{V})\in (0,1)$ is as in Theorem \ref{thm:eps-reg}, $\gamma = \frac{2\alpha}{n+1+2\alpha}$, and $C = C(\mathcal{V})\in (0,\infty)$.
\end{corollary}

\textbf{Remark:} Thus, if $V\in\mathcal{V}$ and $Z$ is a singular point at which the (unique) tangent cone is a plane of multiplicity $Q$, then $\mathcal{N}_V(Z)\geq 1+\alpha$, where $\alpha = \alpha(\mathcal{V})\in (0,1)$ is as in Theorem \ref{thm:eps-reg}.

\begin{proof}
	Fix $0<\sigma\leq\rho<1/4$. From \eqref{E:H-prime} and \eqref{E:D-alt}, we have for any $\tau\in [\sigma,\rho]$,
	$$H^\prime(\tau)-2\tau^{-1}D(\tau) = \text{error}_1,$$
	which by dividing by $H(\tau)$ and using \eqref{E:H-prime-error} gives, for $\eps = \eps(\mathcal{V})$ sufficiently small,
	$$\left|\frac{H^\prime(\tau)}{H(\tau)} - \frac{2\mathcal{N}(\tau)}{\tau}\right| \leq C\Etilt_V^2\tau^{2\alpha-1}$$
	where $C = C(\mathcal{V})$. Using the approximate monotonicity of $\mathcal{N}$ from Theorem \ref{thm:monotonicity} and the simple fact that $e^{-x}\geq 1-x$ for all $x\geq 0$, we get
	$$\frac{2\mathcal{N}_V(0)}{\tau} - C(\mathcal{N}_V(0)+1)\Etilt_V^2\tau^{2\gamma(1+\alpha)-1} \leq \frac{H^\prime(\tau)}{H(\tau)} \leq \frac{2\mathcal{N}(\rho)e^{C\Etilt_V^2\rho^{2\gamma(1+\alpha)}}}{\tau} + C\Etilt_V^2\tau^{2\alpha-1}.$$
	Integrating this over $\tau\in [\sigma,\rho]$ gives exactly (i).
	
	To see (ii), note first that from the definition of $H(\tau)$ and \eqref{E:formula-2}, for $\tau\in (0,1/4)$ we have
	$$H(\tau) \leq 2\tau^{1-n}\int_{B_\tau(0)\setminus B_{\tau/2}(0)}|u|^2 r^{-1}(1+C\Etilt_V^2\tau^{2\alpha}) \leq 4\tau^{-n}(1+C\Etilt_V^2\tau^{2\alpha})\int_{B_\tau(0)\setminus B_{\tau/2}(0)}|u|^2,$$
	and so for $\eps$ sufficiently small,
	$$\frac{1}{8}H(\tau) \leq \tau^{-n}\int_{B_\tau(0)\setminus B_{\tau/2}(0)}|u|^2.$$
	Moreover, again from the definition of $H(\tau)$ and \eqref{E:formula-2},
	$$H(\tau) \geq 2\tau^{1-n}\int_{B_\tau(0)\setminus B_{\tau/2}(0)}|u|^2 r^{-1}(1-C\Etilt_V^2\tau^{2\alpha}) \geq 2\tau^{-n}(1-C\Etilt_V^2\tau^{2\alpha})\int_{B_\tau(0)\setminus B_{\tau/2}(0)}|u|^2.$$
	So combining the above we have
	$$\frac{1}{8}H(\tau)\leq \tau^{-n}\int_{B_\tau(0)\setminus B_{\tau/2}(0)}|u|^2 \leq 2(1+C\Etilt_V^2\tau^{2\alpha})H(\tau).$$
	Apply this now with $\tau = \sigma_k := 2^{-k}\sigma$, $k\in \{0,1,2,\dotsc\}$; noting that $e^x\geq 1+x$ for $x\geq 0$, this gives
	$$\frac{1}{8}H(\sigma_k) \leq \sigma_k^{-n}\int_{B_{\sigma_k}(0)\setminus B_{\sigma_k/2}(0)}|u|^2 \leq 2e^{C\Etilt_V^2\sigma_k^{2\alpha}}H(\sigma_k).$$
	Then, if we set $\rho_k := 2^{-k}\rho$, we may apply the result in (i) of the present lemma with $\sigma = \sigma_k$, $\rho = \rho_k$ to get (noting $\sigma_k/\rho_k = \sigma/\rho$ for all $k$, and also using Theorem \ref{thm:monotonicity})
	\begin{align*}
		&\frac{1}{16}e^{-C\Etilt_V^2\rho^{2\alpha}}\left(\frac{\sigma}{\rho}\right)^{n+2\mathcal{N}(\rho)e^{C\Etilt_V^2\rho^{2\gamma(1+\alpha)}}}\int_{B_{\rho_k}(0)\setminus B_{\rho_k/2}(0)}|u|^2\\
		&\hspace{15em} \leq \int_{B_{\sigma_k}(0)\setminus B_{\sigma_k/2}(0)}|u|^2\\
		&\hspace{15em} \leq 16 e^{C(\mathcal{N}_V(0)+2)\Etilt_V^2\rho^{2\gamma(1+\alpha)}}\left(\frac{\sigma}{\rho}\right)^{n+2\mathcal{N}_V(0)}\int_{B_{\rho_k}(0)\setminus B_{\rho_k/2}(0)}|u|^2.
	\end{align*}
	Summing this over $k\in \{0,1,2,\dotsc\}$ then gives (ii).
	
	Finally, to see (iii) we use (ii), which gives for any $0<\sigma\leq\rho<1/4$ (using as well the estimates from Theorem \ref{thm:eps-reg})
	$$\frac{1}{16}e^{-C\Etilt_V^2\rho^{2\alpha}}\left(\frac{\sigma}{\rho}\right)^{n+2\mathcal{N}(\rho)e^{C\Etilt_V^2\rho^{2\gamma(1+\alpha)}}}\int_{B_\rho(0)}|u|^2 \leq \int_{B_\sigma(0)}|u|^2 \leq C\Etilt_V^2\sigma^{n+2+2\alpha}.$$
	For fixed $\rho$, if we now compare the powers of $\sigma$ here, we see that one needs
	$$e^{C\Etilt_V^2\rho^{2\gamma(1+\alpha)}}\mathcal{N}(\rho) \geq 1+\alpha$$
	(else, dividing both sides by $\sigma^{n+2+2\alpha}$, we would get a contradiction if we took $\sigma\downarrow 0$ for fixed $\rho>0$). Since this inequality must hold for every $\rho>0$, if we now take $\rho\downarrow 0$ we see that $\mathcal{N}_V(0)\geq 1+\alpha$.
\end{proof}

We also have upper semi-continuity of the planar frequency among singular points where the (unique) tangent cone is a plane with multiplicity $Q$.

\begin{corollary}[Upper Semi-Continuity of Planar Frequency]\label{cor:usc-1}
    Suppose $(V_j)_j\subset\mathcal{V}$ and $V\in \mathcal{V}$ are such that $V_j\weakly V$. Suppose also that $Z\in \sing(V)$ is such that $V$ has a (unique) tangent cone $P$ at $Z$ which is a multiplicity $Q$ tangent plane, and $Z_j\in \spt\|V_j\|$ are such that $V_j$ has a (unique) tangent cone $P_j$ at $Z_j$ which is a multiplicity $Q$ tangent plane. Then:
	$$\mathcal{N}_V(Z)\geq \limsup_{j\to\infty}\mathcal{N}_{V_j}(Z_j).$$
    In fact, for any $\rho_j\to \rho$ with $\rho>0$ sufficiently small we have
    $$\lim_{j\to\infty}\mathcal{N}_{V_j,P_j,Z_j}(\rho_j) = \mathcal{N}_{V,P,Z}(\rho).$$
\end{corollary}

\begin{proof}
	We may translate to assume without loss of generality that $Z=0$, and rescale to assume (by Definition \ref{defn:eps-reg}) that $V$ is represented by a Lipschitz $Q$-valued graph in $B_1^{n+k}(0)$. We may also rotate to assume that the (unique) tangent cone to $V$ at $Z$ is $Q|P_0|$. Assuming $j$ is sufficiently large, we may assume that for all $j$ we have $\Etilt_{V_j}<\eps$, for $\eps$ as in Theorem \ref{thm:eps-reg}. We may also translate and rotate the $V_j$ to assume that $Z_j=0$ and that the unique tangent cone to $V_j$ at $Z_j$ is $P_j := P_0$ for all $j$ (this is possible as such operations do not change the value of the frequency function). Thus, $V_j$, $V$ (respectively) are all expressible in $\R^k\times B^n_{1/2}(0)$ as generalised-$C^{1,\alpha}$ functions $u_j, u$ (respectively) $B^n_{1/2}(0)\to \A_Q(\R^k)$, obeying the estimates in Theorem \ref{thm:eps-reg}; note that these estimates have constants which are independent of $j$. In particular, we in fact have local uniform convergence of the $u_j$ to $u$, as well as locally strong convergence in $W^{1,2}$. This readily gives
	$$H_{V_j,P_j,Z_j}(\rho) \to H_{V,P,Z}(\rho) \qquad \text{and} \qquad D_{V_j,P_j,Z_j}(\rho)\to D_{V,P,Z}(\rho)$$
	for all $\rho\in (0,1/4)$. Thus, as $H_{V,P,Z}(\rho)>0$ for all $\rho\in (0,1/4)$ (if not, then as shown before $V$ must agree with $2|P_0|$ in some neighbourhood of $Z$ and so as mentioned the conclusion is automatically true), we have
	$$\mathcal{N}_{V_j,P_j,Z_j}(\rho)\to \mathcal{N}_{V,P,Z}(\rho) \qquad \text{for any }\rho\in (0,1/4).$$
	But then, using the approximate monotonicity of the frequency given by Theorem \ref{thm:monotonicity}, we have for any $\rho\in (0,1/4)$,
	$$\limsup_{j\to\infty}\mathcal{N}_{V_j}(Z_j) \leq\limsup_{j\to\infty}e^{C\Etilt_{V_j}^2\rho^{2\gamma(1+\alpha)}}\mathcal{N}_{V_j,P_j,Z_j}(\rho) = e^{C\Etilt_V^2\rho^{2\gamma(1+\alpha)}}\mathcal{N}_{V,P,Z}(\rho)$$
	(we can ensure that $\Etilt_{V_j}\to \Etilt_V$ using the strong convergence in $W^{1,2}$, up to rescaling the domain). Here, $C = C(\mathcal{V})\in (0,\infty)$, $\alpha = \alpha(\mathcal{V})\in (0,1)$, and $\gamma = \frac{2\alpha}{n+1+2\alpha}$ are independent of $j$. Taking $\rho\downarrow 0$ we then get the result.
\end{proof}

Finally, we establish a form of upper semi-continuity between the planar frequency at points along a blow-up sequence with the Almgren frequency function of the limiting coarse blow-up.

\begin{corollary}[Upper Semi-Continuity along Blow-Up Sequences]\label{cor:usc-2}
	Let $v\in\mathfrak{B}_{\mathcal{V}}$ be the coarse blow-up of a sequence of varifolds $(V_j)_j\subset\mathcal{V}$. Suppose that $Z_j = (\zeta_j,\xi_j)\in \sing(V_j)\cap (\R^k\times B^n_{1/2}(0))$ are such that the (unique) tangent cone to $V_j$ at $Z_j$ is a plane $P_j$ with multiplicity $Q$. Suppose that $\xi_j\to \xi\in \overline{B}^n_{1/2}(0)$, and set $\ell_{v_a,\xi}(x):= v_a(\xi) + (x-\xi)\cdot Dv_a(\xi)$. Finally, assume that $v\not\equiv Q\llbracket \ell_{v_a,\xi}\rrbracket$ in $B^n_{1/2}(0)$. Then:
	$$N_{v-\ell_{v_a,\xi}}(\xi)\geq \limsup_{j\to\infty}\mathcal{N}_{V_j}(Z_j).$$
    In fact, for any $\rho_j\to \rho$ with $\rho>0$ sufficiently small we have
    $$\lim_{j\to\infty}\mathcal{N}_{V_j,P_j,Z_j}(\rho_j) = N_{v-\ell_{v_a,\xi},\xi}(\rho).$$
\end{corollary}

\begin{proof}
	By Theorem \ref{thm:eps-reg} we know that for all $j$ sufficiently large we have $V_j\res (\R^k\times B^n_{3/4}(0))$ is represented by the graph of a generalised-$C^{1,\alpha}$ function $u_j : B^n_{3/4}(0)\to \A_Q(\R^k)$, with an estimate. By translating and rescaling, we may assume that $\xi = 0$. Note that the plane $P_j$ is in fact given by the tangent plane to the average $(u_j)_a$, i.e.~it is the graph of $\ell_j(x):= (u_j)_a(\xi_j) + (x-\xi_j)\cdot D(u_j)_a(\xi_j)$. 
	
	We now claim that $\ell_j/\hat{E}_{V_j}$ converges locally in $C^1$ to $\ell_{v_a,\xi}$; to see this, it suffices to show that $(u_j)_a(\xi_j)/\hat{E}_{V_j}\to v_a(\xi)$ and $D(u_j)_a(\xi_j)/\hat{E}_{V_j}\to Dv_a(\xi)$. The former follows from the proof of the Hardt--Simon inequality for coarse blow-ups (cf.~\cite[$(\mathfrak{B}5)$]{BKMW25}). For the latter, it suffices to show\footnote{Using the standard analysis fact that a necessary condition for a sequence to converge is if every subsequence always has a further subsequence which converges to the given limit.} that one can always find a subsequence for which it is true. For this, notice that from the estimates in Definition \ref{defn:eps-reg} we always have $|D(u_j)_a(\xi_j)|\leq C\hat{E}_{V_j}$, and so up to passing to a subsequence we know $D(u_j)_a(\xi_j)/\hat{E}_{V_j}\to \Lambda$ for some $\Lambda$. But then since for any given $\rho>0$, again from the estimates in Definition \ref{defn:eps-reg}, we have
	$$\rho^{-n-2}\int_{\R^k\times B_\rho^n(\xi_j)}\dist^2(x,P_j)\, \ext\|V_j\|(x) \leq C\rho^{2\alpha}\hat{E}^2_{V_j},$$
	if we divide this by $\hat{E}_{V_j}^2$ and take $j\to\infty$, we get
	$$\rho^{-n-2}\int_{B^n_\rho(\xi)}|v-v_a(\xi) - \Lambda\cdot (x-\xi)|^2 \leq C\rho^{2\alpha}$$
	and thus
	$$\rho^{-n-2}\int_{B^n_\rho(\xi)}|v_a-v_a(\xi)-\Lambda\cdot(x-\xi)|^2 \leq C\rho^{2\alpha}.$$
	Since this is true for all $\rho>0$ and $v_a$ is a smooth (harmonic) function (cf.~\cite[$(\mathfrak{B}2)$]{BKMW25}), we see that we must have $\Lambda = Dv_a(\xi)$. Hence, we have now shown that $D(u_j)_a(\xi_j)/\hat{E}_{V_j}\to Dv_a(\xi)$.
	
	We now know that $\ell_j/\hat{E}_{V_j}$ converges locally in $C^1$ to $\ell_{v_a,\xi}$. In particular, if we translate $V_j$ by $Z_j$ and then perform an appropriate (small) rotation which maps the plane given by the graph of $\ell_j - (u_j)_a(\xi_j)$ (which is a subspace) to $P_0 \equiv \{0\}^k\times\R^n$, then we see that, arguing as in \cite[$(\mathfrak{B}3\text{III})$]{BKMW25}, that the blow-up of the resulting sequence is, up to a translation and rescaling, $v-\ell_{v_a,\xi}$. Thus, it suffices to prove the case where $Z_j\equiv 0$, $P_j\equiv P_0$, $(u_j)_a(0) = 0$, $D(u_j)_a(0) = 0$, and thus $\ell_{v_a,\xi}\equiv \ell_{v_a,0} = 0$.
	
	However in this case, due to the locally strong convergence in $W^{1,2}$ along the blow-up sequence provided by Definition \ref{defn:eps-reg}, for each $\rho\in (0,1/4)$ we have
	$$\frac{1}{\hat{E}_{V_j}^2}H_{V_j,P_j,Z_j}(\rho) \to -\rho^{1-n}\int|v|^2r^{-1}\phi^\prime(r/\rho) \qquad \text{and} \qquad \frac{1}{\hat{E}_{V_j}^2}D_{V_j,P_j,Z_j}(\rho) \to \rho^{2-n}\int|Dv|^2\phi(r/\rho),$$
	i.e.~for every $\rho\in (0,1/4)$ we have
	$$\mathcal{N}_{V_j,P_j,Z_j}(\rho)\to \frac{\rho^{2-n}\int|Dv|^2\phi(r/\rho)}{-\rho^{1-n}\int|v|^2r^{-1}\phi^\prime(r/\rho)} =: N_{v,0}(\rho)$$
	where $N_{v,0}(\rho)$ is the Almgren frequency function of $v$ (with the usual modification incorporating the cut-off function $\phi$). Notice that here we may assume that $N_{v,0}(\rho)>0$ for each such $\rho$ (and thus $H_{V_j,P_j,Z_j}(\rho)>0$ for $j$ sufficiently large), else $|v|\equiv 0$ on $B_\rho(0)$ which contradicts that $v\not\equiv Q\llbracket 0\rrbracket$ on $B^n_{1/2}(0)$. As the above is true for every $\rho>0$, using the approximate monotonicity of the planar frequency function from Theorem \ref{thm:monotonicity}, we see that for every $\rho\in (0,1/4)$,
	$$\limsup_{k\to\infty}\mathcal{N}_{V_j}(Z_j)\leq\limsup_{j\to\infty} e^{C\Etilt_{V_j}^2\rho^{2\gamma(1+\alpha)}}\mathcal{N}_{V_j,P_j,Z_j}(\rho) = N_{v,0}(\rho).$$
	Taking $\rho\downarrow 0$ then completes the proof.
\end{proof}

\textbf{Remark:} The expression for the Almgren frequency function of a coarse blow-up $v\in\mathfrak{B}_{\mathcal{V}}$ appearing in Corollary \ref{cor:usc-2} is slightly different than the one used in Theorem \ref{thm:blow-up}, since in Corollary \ref{cor:usc-2} it includes the cut-off function in the definition. If one takes a more general $\phi$ and lets $\phi\to \one_{[0,1]}$, one can recover the version of the frequency function used in Theorem \ref{thm:blow-up}. The difference is not substantial, however, and in Theorem \ref{thm:blow-up} one could alternatively used the modified version above. We stress however that the difference \emph{is} more substantial for the planar frequency function, as indeed our proof of the estimate controlling $\text{error}_2^*$ depended on the gradient of the cut-off $\phi$, and the constant involved goes to $\infty$ as $\phi\to \one_{[0,1]}$.

\subsection{Dimension bound for branch points with planar frequency $\neq 2$}\label{sec:freq-not-2}
Here we argue similarly to the corresponding part in \cite{KW26c}. Fix $V\in\mathcal{V}$ and write $\mathcal{B}^Q_V$ for the set of singular points $Z$ of $V$ at which $V$ has a (unique) tangent cone which is a plane with multiplicity $Q$. We know from Theorem \ref{thm:monotonicity} that each $Z\in\mathcal{B}^Q_V$ has a well-defined planar frequency value which, from Corollary \ref{cor:decay-estimates}(iii), takes values in $[1+\alpha,\infty)$ for some fixed $\alpha = \alpha(\mathcal{V})\in (0,1)$. We then partition $\mathcal{B}^Q_V$ into three subsets based on the planar frequency value as follows:
$$\mathcal{B}^Q_V = \mathcal{B}_V^{<2}\cup \mathcal{B}_V^2 \cup\mathcal{B}_V^{>2}$$
where for $p\in [1,\infty)$, $\mathcal{B}_V^{<p} = \{Z\in \mathcal{B}^Q_V:\mathcal{N}_V(Z)<p\}$, $\mathcal{B}_{V}^p = \{Z\in \mathcal{B}^Q_V:\mathcal{N}_V(Z)=p\}$, and $\mathcal{B}_V^{>p} = \{Z\in\mathcal{B}^Q_V:\mathcal{N}_V(Z)>p\}$ (we have dropped the $Q$ dependence in the notation for these latter sets for simplicity). 

In this section, we use the properties of the planar frequency previously established to prove a dimension bound on the size of $\mathcal{B}_V^{<2}$ and $\mathcal{B}_V^{>2}$. In Section \ref{sec:cm} we will show the same dimension bound for $\mathcal{B}_V^2$ (and in fact, the argument in Section \ref{sec:cm} also establishes the dimension bound for $\mathcal{B}_V^{>2}$).

Our main result here is therefore:

\begin{prop}\label{prop:dimension-not-equal-2}
	If $V\in\mathcal{V}$, then $\dim_\H(\mathcal{B}_V^{<2}\cup \mathcal{B}_V^{>2}) \leq n-2$.
\end{prop}

\textbf{Remark:} Our argument will utilise Almgren's dimension reduction via stratification \cite{Alm00}. One could also use Federer's dimension reduction, arguing by contradiction assuming that $\H^s(B_1(0)\cap (\mathcal{B}_V^{<2}\cup \mathcal{B}_V^{>2}))>0$ for some $s>n-2$ and then performing a coarse blow-up about a good density point for this set, using the approximate monotonicity of the planar frequency function to ultimately reach a contradiction by producing a coarse blow-up which has a branch set of Hausdorff dimension $>n-2$. The advantage of arguing via Almgren's stratification is that it will detect a subspace which nearby points of ``good frequency'' must accumulate to, providing the natural starting point towards proving that the branch set is countably $(n-2)$-rectifiable (although we do not pursue this question in the present work).

\begin{proof}
	First, fix $Z\in\mathcal{B}^Q_V$. Up to rotation and translation, without loss of generality we may assume that $Z=0$ and that the (unique) tangent cone to $V$ at $Z$ is $Q|P_0|$. Take any sequence $\sigma_j\downarrow 0$, and produce a coarse blow-up $v\in \mathfrak{B}_{\mathcal{V}}$ of the sequence $V_j:= (\eta_{0,\sigma_j})_\#V$. Corollary \ref{cor:usc-2}, with $P_j\equiv P_0$ and $Z_j\equiv 0$ for all $j$, gives that for each $\rho>0$,
	$$\mathcal{N}_{V_j}(\rho) \equiv \mathcal{N}_{V_j,P_0,0}(\rho)\to N_{v,0}(\rho)$$
	(note that all of these quantities are well-defined by the estimates in Corollary \ref{cor:decay-estimates}). However, as the $V_j$ are simply rescalings of $V$ about a fixed point, we know that $\mathcal{N}_{V_j}(\rho) = \mathcal{N}_{V_j,P_0,Z}(\rho\rho_j)\to \mathcal{N}_V(Z)$ for every $\rho>0$, and thus we see that
	$$N_{v,0}(\rho) = \mathcal{N}_V(Z)$$
	for every $\rho>0$. In particular, this tells us that $N_v(0) = \mathcal{N}_V(Z)$. As $N_{v,0}(\rho)$ is constant in $\rho$, we have that $v$ is a homogeneous function\footnote{The reader should note that the estimates in Definition \ref{defn:eps-reg} and Arzelà--Ascoli readily give that the blow-up is $\beta$-Hölder continuous for any $\beta\in (0,1)$. In fact, by Definition \ref{defn:eps-reg} and Corollary \ref{cor:decay-estimates}, for each $\rho\in [1,\infty)$ we have $V_j\res C_\rho(0) = \mathbf{v}(u_j)$ where $\|u_j\|_{C^{0,1}(B_\rho(0))}\leq C_\rho\hat{E}_{V_j}$, provided $j$ is sufficiently large, where $C_\rho\in (0,\infty)$ is a constant depending on $\rho$ but not on $j$. Thus, Arzelà--Ascoli gives that $v_j:=u_j/\hat{E}_{V_j}\to v$ uniformly on compact subsets of $\R^n$. Since $v_j\to v$ locally uniformly, we also know that $\Lip(v|_{B_\rho(0)})\leq\liminf_{j\to\infty}\Lip(v_j|_{B_\rho(0)})\leq C_\rho$. Thus, we may assume that $v$ is defined on all of $\R^n$.} with degree of homogeneity exactly $\mathcal{N}_V(Z)$ (here, as explained in the footnote, we can ensure that $v$ is naturally defined on all of $\R^n$). In particular, $v\not\equiv 0$, since $\|v_j\|_{L^2(B_1(0)}\equiv 1$ and as $v_j\to v$ uniformly on compact subsets of $\R^n$ we therefore have $\|v\|_{L^2(B_1(0))}=1$. Moreover, by standard properties of frequency functions for homogeneous blow-ups, we know that $N_v(z)\leq N_v(0)$ for all $z\in \R^n$, and that if $S(v)$ denotes the spine of $v$ (namely, the set of points under which $v$ is translation invariant, which is a subspace due to the homogeneity of $v$), then
	$$S(v) = \{z\in \R^n: N_v(z) = N_v(0)\}.$$
	From Corollary \ref{cor:decay-estimates}(iii), we know that $N_v(0) = \mathcal{N}_V(Z)\geq 1+\alpha$. Thus, $\dim S(v)\not\in \{n-1,n\}$. Indeed, if $\dim S(v) = n$, then $v$ is constant, yet we know $v(0) = 0$, which means that $v\equiv 0$, which contradicts $v$ not being identically zero. Also, if $\dim S(v) = n-1$, then $\graph(v)$ would be a union of half-planes whose boundaries meet along an $(n-1)$-dimensional subspace which would mean $N_v(0) = 1$, contradicting $N_v(0)\geq 1+\alpha$.
	
	Thus, $\dim S(v) \in \{0,1,2,\dotsc,n-2\}$. We may therefore stratify $\mathcal{B}^Q_V$ as follows. For each $Z\in \mathcal{B}^Q_V$, write $\mathfrak{B}_Z$ for the set of all coarse blow-ups constructed in the above manner by sequences of rescalings of $V$ centred at $Z$. Then, necessarily each $v\in \mathfrak{B}_Z$ is a homogeneous function in $\mathfrak{B}_{\mathcal{V}}$ which has degree of homogeneity $\geq 1+\alpha$ and spine dimension at most $n-2$. For each $j\in \{0,1,\dotsc,n-2\}$, set
	$$\mathcal{S}_j := \{Z\in\mathcal{B}^Q_V:\dim S(v)\leq j\text{ for all }v\in \mathfrak{B}_Z\}.$$
	Clearly by definition $\mathcal{S}_0\subset\mathcal{S}_1\subset\cdots\subset \mathcal{S}_{n-2} = \mathfrak{B}^Q_V$. So far everything we have discussed has been for any $Z\in\mathcal{B}_V^Q$. We now claim that for each $j\in \{0,1,2,\dotsc,n-2\}$,
	\begin{equation}\label{E:not-2-1}
		\dim_\H(\mathcal{S}_j\cap (\mathcal{B}_V^{<2}\cup \mathcal{B}_V^{>2})) \leq j.
	\end{equation}
	Evidently \eqref{E:not-2-1} proves the proposition by taking $j=n-2$. Proving \eqref{E:not-2-1} will follow by standard arguments related to stratifications provided we can show that, for each $Z\in\mathcal{B}_V^{<2}\cup \mathcal{B}_V^{>2}$, nearby points with planar frequency $\geq \mathcal{N}_V(Z) - \eps$ must concentrate around the spine of a blow-up in $\mathfrak{B}_Z$ (for sufficiently small $\eps$). Indeed, without loss of generality we may assume $Z=0$ with the unique tangent plane at $Z$ being $Q|P_0|$, take any sequence $\sigma_j\downarrow 0$ and again set $V_j:= (\eta_{0,\sigma_j})_\#V$. Consider any point $z\in B^n_1(0)$ for which there exist $Z_j = (\zeta_j,\xi_j)$ singular points of $V_j$ at which the (unique) tangent cone is a multiplicity $Q$ plane such that $\xi_j\to z$, and $\mathcal{N}_{V_j}(Z_j)\geq\mathcal{N}_V(Z) -\delta_j$ for some sequence $\delta_j\downarrow 0$. Then, from Corollary \ref{cor:usc-2}, we know that
	$$N_{v-\ell_{v_a,z}}(z) \geq \limsup_{j\to\infty}\mathcal{N}_{V_j}(Z_j) \geq\mathcal{N}_V(Z) = N_v(0).$$
	Thus, to complete the proof we need to show that this implies $z\in S(v)$. To do this, we need to show that $N_v(z) = N_v(0)$, which from the above is equivalent to showing that $\ell_{v_a,z}\equiv 0$; this is where we will need to use that $\mathcal{N}_V(Z)\neq 2$. We split this into two cases (recall here $N_v(0) = \mathcal{N}_V(Z) \in [1+\alpha,\infty)\setminus \{2\}$).
	
	First, suppose that $N_v(0)\in [1+\alpha,\infty)\setminus\Z$, i.e.~$N_v(0)$ is \emph{not} an integer. Then if $v_a\not\equiv 0$, $v_a$ would be a harmonic function (cf.~\cite[$(\mathfrak{B}2)$]{BKMW25}) which is homogeneous of some non-integer degree: this is impossible, and thus we must have $v_a\equiv 0$. This of course then implies that $\ell_{v_a,z}\equiv 0$ for \emph{every} $z\in B^n_1(0)$, and thus $N_{v-\ell_{v_a,z}}(z) = N_{v}(z)\geq N_v(0)$, implying that $N_{v}(z) = N_v(0)$ and so $z\in S(v)$. This shows that $Z_k$ must be concentrating on $S(v)$, which has dimension $\leq k$.
	
	Lastly, we just need to show this claim when $N_v(0)\in \{3,4,\dotsc\}$. Here, we cannot argue that $v_a\equiv 0$. Instead, we will argue that $\ell_{v_a,z}\equiv 0$ for the points $z$ as above. Indeed, notice that here we have, by assumption, that $N_{v-\ell_{v_a,z}}(z)\geq N_v(0) \geq 3$. This tells us that, as $v_a$ is a smooth harmonic function, that $v_a-\ell_{v_a,z}$ vanishes to degree at least $2$. Hence, $0=D_x^2(v_a-\ell_{v_a,z})(z) = D^2_x v_a(z)$, as $\ell_{v_a,z}$ is linear and so $D^2_x\ell_{v_a,z}(x)\equiv 0$. We now argue that this with the homogeneity of $v_a$ implies $Dv_a(z) = 0$. Indeed, consider for any $i\in \{1,\dotsc,n\}$ the function $w(t):= D_iv_a(tz)$ for $t>0$. Since $D^2v_a(z) = 0$, we evidently have $w^\prime(1) = 0$. However, since $v_a$ is homogeneous of degree $\mathcal{N}_v(0)$, we know that $D_iv_a$ is homogeneous of degree $N_v(0)-1$, and hence $w(t) = t^{N_v(0)-1}w(1) = t^{N_v(0)-1}D_iv_a(z)$. But then $w^\prime(t) = (N_v(0)-1)t^{N_v(0)-2}D_iv_a(z)$, and hence $0 = w^\prime(1) = (N_v(0)-1)D_iv_a(z)$. As $N_v(0) - 1\geq 2$, evidently we must have $D_iv_a(z) = 0$, and hence as this is true for each $i\in \{1,\dotsc,n\}$ we see that $Dv_a(z) = 0$. But we can then of course repeat this argument (with $w(t) = v_a(tz)$), as $N_v(0)>2$, to similarly get that $v_a(z) = 0$. In particular, at such $z$ we have $\ell_{v_a,z}\equiv 0$, and so again $N_v(z) = N_{v-\ell_{v_a,z}}(z) = N_v(0)$, and thus $z\in S(v)$.
	
	This therefore shows that the set $\mathcal{S}_k\cap (\mathcal{B}_V^{<2}\cup \mathcal{B}_V^{>2})$ has the desired subspace approximation property in order to show that its Hausdorff dimension is at most $k$, completing the proof (cf. \cite[Section 3.4, Lemma 3]{Sim96}).
\end{proof}

\textbf{Note:} One sees that the above proof does not show that the set $\mathcal{B}_V^2$ of points of planar frequency \emph{exactly} $2$ has Hausdorff dimension $\leq n-2$. Indeed, one cannot show that $\ell_{v,z}\equiv 0$ at the points $z$ in question, as it is possible for $v_a-\ell_{v_a,z}$ to vanish to degree $2$ (i.e.~be of quadratic order) at \emph{every} point, e.g.~if $n=2$ and $v_a(x,y) = x^2-y^2$. Indeed, if $v_a$ is non-zero and homogeneous of degree $2$, then $v_a-\ell_{v_a,z}$ \emph{must} vanish to degree $2$ at every point $z$ (by Taylor's theorem and the homogeneity).

\begin{remark}
Whilst the above proof does not control the size of $\mathcal{B}_V^2$, if for each $n$-dimensional subspace $P\subset\R^{n+k}$ we set
$$\mathcal{B}^{2,P}_V:= \{Z\in\mathcal{B}^2_V:T_Z V = P\}$$
then the proof \emph{does} show that $\dim_\H(\mathcal{B}^{2,P}_V) \leq n-2$, i.e.~points where the planar frequency is $2$ \emph{and} the tangent cone is a given plane has Hausdorff dimension at most $n-2$ (this follows because one can arrange that the linear function $\ell_{v,z}$ arising at the end of the proof of Proposition \ref{prop:dimension-not-equal-2} has vanishing derivative, from which knowing the planar frequency is $2$ is enough to conclude in the analogous manner). In particular, if we knew that the stationary varifold $V$ was represented by a $Q$-valued Lipschitz graph \emph{with zero average}, then the tangent plane at any branch point is $P_0$ (say), and so we can conclude that the \emph{whole} branch set has Hausdorff dimension $\leq n-2$ at this stage with no further argument needed. Thus, the only situation which we do not know how to address as of now is when locally about a frequency $2$ branch point, there is a whole \emph{continuum} of frequency $2$ branch points all with different tangent planes. Note that also for this, the problematic case (if one was arguing by contradiction, blowing-up around a point of positive $\H^s$-upper density for some $s>0$) is when the coarse blow-up coincides with a multiplicity $Q$ single-valued harmonic function which is homogeneous of degree $2$. To deal with this, we need a more complicated argument involving a center manifold (see Section \ref{sec:cm}).
\end{remark}

\begin{remark}
The above proof also has the following interesting consequence. Suppose $V\in\mathcal{V}$ (for instance, we could take $\mathcal{V} = \mathcal{V}_L$ as in Corollary \ref{cor:Lip-1}). Suppose $Z\in \mathcal{B}^Q_V$ is a singular point of $V$ where the (unique) tangent cone is $Q|P_0|$, and the planar frequency of $V$ at $Z$ is $\neq 2$. Then, for all $\delta>0$, there is a radius $r>0$ (depending on $V$ and $\delta$) such that if $\tilde{Z}\in\mathcal{B}_V$ has planar frequency $\mathcal{N}_V(\tilde{Z})\geq \mathcal{N}_V(Z)-\delta$ and $|Z-\tilde{Z}|<r$, then not only must $\tilde{Z}$ be within $\delta$ of a subspace of dimension $\leq n-2$, but also
\begin{equation}\label{E:consequence-lojasiewicz}
\dist(\tilde{Z},P_0) \leq \delta\hat{E}_{V,P_0}(B_{2r}(Z)).
\end{equation}
Indeed, this follows from combining the proofs of Corollary \ref{cor:usc-1} and Proposition \ref{prop:dimension-not-equal-2}, as indeed in the proof of Proposition \ref{prop:dimension-not-equal-2} we showed that the linear part of the average of the blow-up vanishes at the limit of such points. Furthermore, these also give
\begin{equation}\label{E:consequence-lojasiewicz-2}
\dist_\H(T_{\tilde{Z}}V\cap B_1(0), P_0\cap B_1(0)) \leq \delta \Etilt_{V,P_0}(B_{2r}(Z)).
\end{equation}
The inequality \eqref{E:consequence-lojasiewicz} should be compared to the corollary of the Łojasiewicz inequality which says that for a (single-valued) minimal graph $u$, if $Du(x) = 0$ then there is a neighbourhood of $x$ such that any point $y$ in this neighbourhood which obeys $Du(y) = 0$ necessarily also obeys $u(y) = u(x)$. Because points of non-integer planar frequency do not occur in the setting of single-valued minimal graphs, both \eqref{E:consequence-lojasiewicz} and \eqref{E:consequence-lojasiewicz-2} are most directly comparable when the point has planar frequency an integer $\geq 3$. Indeed, for single-valued minimal graphs $u$, the Łojasiewicz inequality gives that at any point with $D^2u(x) = 0$, there is a neighbourhood of $x$ for which any point $y$ in this neighbourhood with $D^2u(y)=0$ necessarily has $Du(x) = Du(y)$, and hence also $u(x) = u(y)$. Both \eqref{E:consequence-lojasiewicz} and \eqref{E:consequence-lojasiewicz-2} are weakened versions of these facts for $V\in\mathcal{V}$, which due to the branching behaviour are not in general known to satisfy a version of the Łojasiewicz inequality.
\end{remark}

\section{Dimension Bound on Planar Frequency $\geq 2$ Branch Points via Center Manifold}\label{sec:cm}

To study the set of points where the planar frequency is $2$ (or even $\geq 2$) we utilise a \emph{center manifold}. We wish to take some time to summarise the idea, in light of planar frequency. The basic philosophy is that we wish to measure the frequency of $V$ relative to a more general manifold $\mathcal{M}$ instead of a plane; indeed, if $\mathcal{M}$ is a plane, our discussion will reduce to the planar frequency. The manifold $\mathcal{M}$ needs to satisfy some conditions for this to be possible:
\begin{enumerate}
	\item [(A)] \emph{$\mathcal{M}$ needs to be $C^3$.} This is necessary to mimic the first variation arguments we used to show the approximate monotonicity of the planar frequency function, namely variations normal to $\mathcal{M}$ and variations within $\mathcal{M}$. The latter involves differentiating a vector field whose flow lines match the geodesic arcs emitting from a point, and for this $\mathcal{M}$ needs to be $C^3$.
	\item [(B)] \emph{$\mathcal{M}$ cannot be arbitrary, but needs to have some ``variational structure'' itself.} This is perhaps expected if one wishes there to be a frequency function with a monotonicity property in the first place, as the monotonicity arises from control on suitable variational identities. This requirement manifests itself analytically when one needs to establish suitable doubling conditions of $V$ relative to $\mathcal{M}$ (cf.~Lemma \ref{lemma:doubling-2}). This doubling property is needed to control certain error terms arising from our variational identities, as we saw in Section \ref{sec:pff} for the planar frequency.
	\item [(C)] \emph{$\mathcal{M}$ should} (ideally\footnote{One could imagine a more complicated argument where the manifold $\mathcal{M}$ is allowed to change with the radius, such that these manifolds converge in some strong way as the radius goes to $0$ to a manifold which is tangent. One could then measure the frequency with respect to these changing manifolds. Such a problem does arise in Almgren's original work when trying to use a center manifold for \emph{all} branch points, which in our language corresponds to including also branch points of planar frequency $<2$. We however do not need to do this, as the planar frequency allows us to deal with such points using blow-ups relative to tangent planes as we saw in Section \ref{sec:pff}. In some sense, the need for changing the center manifold at a given point is replaced by a much simpler situation where one changes the ``center manifold'' at nearby points instead (i.e.~we take the ``center manifold'' at points of planar frequency $<2$ to be the tangent plane at the point).}) \emph{touch $V$ at all points of interest, namely any point of planar frequency $\geq 2$.} This guarantees that the frequency value is non-trivial at points of interest. Morally, one is then able to control the size of the set of points with planar frequency $\geq 2$ by instead controlling the touching set of two surfaces (or at least the relevant part of the touching set), in an analogous spirit to classical quantitative unique continuation problems.
	\item [(D)] A more subtle condition arises for general $\mathcal{M}$: there are certain error terms which arise in the first variation arguments which are \emph{linear} in the \emph{average} of the normal map representing $V$ over $\mathcal{M}$ (assuming such a map exists); this is in addition to other errors which are, in theory, higher order. One needs to control such linear error terms in some manner. If $\mathcal{M}$ were a minimal surface itself, these linear errors would in fact vanish as they also involve the mean curvature of $\mathcal{M}$. Likewise, if the normal map representing $V$ over $\mathcal{M}$ had zero average, then these terms would also vanish. It appears hard however to guarantee either of these situations, and thus for general $\mathcal{M}$ when this term does not vanish, we need to control it suitably in order to establish approximate monotonicity of the associated frequency function relative to $\mathcal{M}$. In order to show the appropriate control, namely that this error is higher order (say, approximately cubic), $\mathcal{M}$ \emph{needs to somehow mirror} (the average of) $V$; for instance, at a point where $V$ is close to a plane with some rate, $\mathcal{M}$ should be also. Whilst this ties in nicely to (C) above, it is crucial that \emph{everywhere} $\mathcal{M}$ is a reflection of the decay rate of the varifold to its average, rather than just at a single point. This is important for certain estimates, which without one cannot control these linear error terms.
\end{enumerate}

Typically in the literature the center manifold is described as being an ``approximation to the average of the sheets'' of the varifold, with its use being to construct a blow-up which is non-collapsed and thus has singularities; this is done by showing that the blow-up relative to the center manifold has zero average, yet the full blow-up is non-zero (the latter being a result of having a bounded frequency function with a monotonicity-type property). We stress that whilst this is loosely the truth, the center manifold needs to be \emph{much} more than just a manifold-approximation to the average of the sheets in order for there to be an approximately monotone frequency function with respect to it. Indeed, what is necessary is for \emph{the center manifold to be close to the average whilst also mirroring the decay of the varifold towards its tangent plane(s)}, as described in (D) above. This is \emph{much} stronger than just being an approximation to the average, and this is consistent with the center manifold needing to be an object which inherently has a variational structure built into it in order for there to be an approximately monotone frequency function. Thus, ``an approximation to the average'' is perhaps more accurately described as ``an approximation to the average \emph{which reflects the decay behaviour between the varifold and its average}''. Again, the reason for this is \emph{not} to get a blow-up with zero average, but instead to show the approximate monotonicity of the frequency relative to the center manifold at points of interest (indeed, in our approach, \emph{it is not the fact that the average of the blow-up is zero which allows us to ultimately conclude our argument}\footnote{The average being zero is useful from the perspective of improved \emph{regularity}, however. Indeed, knowing that you have a non-trivial blow-up with zero average means that the blow-up cannot collapse into a single harmonic function with multiplicity $Q$ in the limit, and so it must `separate'. In principle, this separation then allows one to understand further regularity structures for the varifold about branch points in an inductive fashion, using the reduction of the multiplicity on the sheets of the blow-up. This is a key mechanism in the regularities theories first seen in \cite{Cha88, Wic14a}, for instance.}, but only that the average of the graphical representation of the varifold over the center manifold is suitably small in order to control the error terms needed to show the relevant frequency function is approximately monotone). The size of the relevant singular set is then controlled by controlling an appropriate subset of the touching set of $V$ and the center manifold, in the same spirit as classical quantitative unique continuation arguments. The average is a suitable candidate for building the center manifold for another reason, namely because it is close to harmonic, and the error term in this closeness is \emph{cubic} in $|Du|$ (which is the correct order we wanted to make these linear errors involving the average). It is ultimately this fact that the error between the average and a harmonic function is cubic which enables us to establish the regularity point from (A) above for our center manifold.

\textbf{Remark:} In Almgren's construction \cite{Alm00} (and also in De Lellis--Spadaro \cite{DLS16a}) the center manifold construction is much more complicated than what we need here, as we only need the center manifold to pass through points of planar frequency $\geq 2$, which we will be able to guarantee. If one tries to construct a center manifold which can be used for \emph{every} branch point, it will not pass through those branch points where the decay rate is $<2$. (It also won't necessarily pass through density $Q$ classical singularities, which one may think of as points with ``frequency $1$'').

\textbf{Remark:} It is perhaps instructive to consider how one establishes a dimension bound on flat singular points which are \emph{not} branch points, and why that case is significantly more simple. Indeed, there the varifold is locally given by the sum of minimal graphs, and so one can reduce to bounding the Hausdorff dimension of the touching set of two minimal graphs, say $f$ and $g$, defined over the (common) tangent plane. One can then look at the difference $f-g$, which solves an elliptic PDE, for which one can use methods from \cite{GL86} to establish directly a form of monotonicity of a  frequency function (this is also the philosophy taken by Simon and the third author in \cite{SW16}). Using elliptic estimates, the result then follows simply by blowing up. In the language of center manifolds, one could view this argument in two ways. One could instead look at the two-valued function $\llbracket f\rrbracket + \llbracket g\rrbracket$ and compare it to the average $\frac{1}{2}(f+g)$ of the sheets, which would morally then be looking at the two-valued symmetric function $\pm\frac{1}{2}\llbracket f-g\rrbracket$; notice that this is then the same as looking at the single-valued difference $f-g$. Alternatively, one could use the graph of $g$ as the ``center manifold'', and measure the frequency of $f$ relative to $g$; morally this would be looking at $f-g$. However, even in this very simple situation the center manifold construction will \emph{not} agree with either of these ``natural'' choices, and as a result one has to deal with the extra linear error term mentioned in (D). The center manifold is then built essentially to deal with this error term when there is seemingly no other ``natural'' candidate for a manifold to measure frequency with respect to. It would be interesting to investigate whether in some specialised situations, such as area minimisers coming from calibrated submanifolds, one can find a more natural candidate, such as another calibrated submanifold.

\begin{remark}\label{remark:cm-pulse}
Loosely speaking, the center manifold can be described as follows: for each $x\in B^n_1(0)$, let $r(x)$ denote the largest radius $\sigma>0$ such that both the height and tilt excess of $V$ relative to the tangent plane of $u_{\text{a}}$ at $x$ fail to decay at a rate of $2$ (or even $2-\delta$ for some small $\delta>0$ fixed) in the cylinder $\R^k\times B_\sigma^n(x)$; notice that $r(x) = 0$ if $x$ is (the projection of) a flat singular point of planar frequency $\geq 2$. Then, the center manifold is morally the graph of
$$\wp(x):= (u_{\text{a}}\ast\varrho_{r(x)})(x)$$
where $\varrho$ is a suitable radial bump function and $\varrho_{r(x)}(y):= r(x)^{-n}\varrho(y/r(x))$; here, if $r(x) =0$ we set $\wp(x) := u_{\text{a}}(x)$. Notice that such a function obeys the last three points above: provided $r(x)$ is small enough, $\wp(x)$ is close to the average, which is almost harmonic, which would allow us to establish (B). Point (C) is clear by definition. Finally, certainly this definition reflects the decay rate of $V$ to its average, so (D) holds. The regularity conclusion required in (A) is, however, unclear. In our actual construction of a center manifold, we take a suitable discretisation of this definition, which will enable us to show the regularity whilst also maintaining the other three conditions. We also note that this heuristic definition also shows why one should expect Almgren's observation that when $u$ is actually a single-valued function whose graph is minimal, then the center manifold coincides with the graph of $u$ and so provides an improved regularity estimate from $C^{1,\alpha}$ to $C^{3,\beta}$ (as in the single-valued case, the planar frequency is $\geq 2$ \emph{everywhere}, and so $r(x)\equiv 0$ for all $x$).
\end{remark}

We begin by setting up the notation and assumptions we will need throughout this section.

\textbf{Definition:} Let us write
$$\Etilt^2_{V,P}(B_r(x)) := \frac{1}{2}r^{-n}\int_{B^{n+k}_r(x)}\|\pi_X-\pi_P\|^2\ \ext\|V\|(X)$$
for the tilt-excess of $V$ in the ball $B_r(x)$ with respect to the plane $P$ passing through the origin, and 
$$\Etilt_{V}(B_r(x)):= \inf_{P} \Etilt_{V,P}(B_r(x))$$
for the optimal tilt-excess of $V$ on the ball $B^{n+k}_r(x)$ over planes $P$ through $0$.

\textbf{Definition:} We write the \emph{height of $V$ on a set $A$ relative to a plane $P$} by
$$\mathbf{h}_{V,P}(A):= \sup_{x,y\in \spt\|V\|\cap A}|\pi_{P^\perp}(x) - \pi_{P^\perp}(y)|\, .$$
For a ball $\mathbf{B}$, we then set
$$\mathbf{h}_{V}(\mathbf{B}) := \inf_{P\in \mathcal{P}_{\mathbf{B}}}\mathbf{h}_{V,P}(\mathbf{B})$$
where here $\mathcal{P}_{\mathbf{B}}$ is the set of all hyperplanes $P$ attaining the infimum for $\Etilt_V(\mathbf{B})$. Throughout, we will write $P_{\mathbf{B}}$ for a choice of hyperplane in $\mathcal{P}_{\mathbf{B}}$ which achieves $\mathbf{h}_V(\mathbf{B})$; in particular, this choice also achieves the infimum in $\Etilt_V(\mathbf{B})$.

For a plane $P$ and radius $r>0$ we define the cylinder over $P$ of radius $r$ and center $x$ by $C_r(x,P):= \pi_P^{-1}(P\cap B_r^{n+k}(\pi_P(x)))$. For cylinders we can similarly define a tilt-excess, except now $\Etilt_V(C_r(x,P))$ denotes $\Etilt_{V,P}(C_r(x,P))\equiv \Etilt_{(\eta_{x,r})_\#V,P}$ and $\mathbf{h}_{V}(C_r(x,P)) \equiv \mathbf{h}_{V,P}(C_r(x,P))$, i.e.~the plane in the tilt-excess integrand is always taken to be the one defining the cylinder.

Recall that $P_0 := \{0\}^k\times\R^n$. Throughout this section we will assume the following:

\textbf{Assumptions:} Let $\eps_2\in (0,1]$ be a constant to be specified, and let $V$ be such that $(\eta_{0,3\sqrt{n}})_\# V\in\mathcal{V}$. We assume that
\begin{itemize}
	\item $0$ is a branch point of $V$ with $\Theta_V(0) = Q$ and $\mathcal{N}_V(0)\geq 2$;
	\item $\left((6\sqrt{n})^n\w_n\right)^{-1}\|V\|(B_{6\sqrt{n}}^{n+k}(0))\leq Q+\eps_2$;
	\item $\Etilt_V\equiv \Etilt_{V}(B^{n+k}_{6\sqrt{n}}(0)) = \Etilt_{V,P_0}(B_{6\sqrt{n}}^{n+k}(0))\leq \eps_2$.
\end{itemize}
Note that, by working on a slightly smaller ball and rescaling, Remark \ref{remark:tilt-height-comparison} regarding the Poincaré inequality means that we can also assume without loss of generality that the tilt-excess of $V$ relative to $P_0$ controls the height excess of $V$ relative to $P_0$. Moreover, from Theorem \ref{thm:eps-reg}, if $\eps_2$ is sufficiently small we may assume that $V$ is represented by the graph of a generalised-$C^{1,\alpha}$ function $u:P_0\cap B^{n+k}_{6\sqrt{n}}(0)\to \A_Q(\R^k)$. The constant $\alpha = \alpha(\mathcal{V})\in (0,1)$ from Theorem \ref{thm:eps-reg} will be fixed throughout.

\subsection{Construction of a center manifold}

We break this down into several steps.

\textbf{Step 1: Initial Set-Up and Discretisation.} For $j\in \Z_{\geq 1}$, write $\mathcal{C}^j$ for the collection of (closed) cubes $L$ of $P_0$ of the form
$$L = [a_1,a_1+2\ell]\times\cdots\times[a_n,a_n+2\ell]\subset P_0$$
where $2\ell = 2^{1-j}=:2\ell(L)$ is the side-length of the cube $L$, $a_i\in 2^{1-j}\Z$ with $a_i\in [-4,4-2\ell]$. The \emph{center} of $L$ is $x_l:= (a_1+\ell,\dotsc,a_n+\ell)\in\R^n$. We also write $\mathcal{C}:= \bigcup_{j=1}^\infty\mathcal{C}^j$. If $H,L\in\mathcal{C}$ are such that $H\subset L$, we call $L$ an \emph{ancestor} of $H$ and $H$ a \emph{descendant} of $L$. If in addition $\ell(L) = 2\ell(H)$, we call $H$ a \emph{child} of $L$ and $L$ the \emph{parent} of $H$. Furthermore, if $L\in\mathcal{C}^j$ we say $j$ is the \emph{level} of $L$, and write $\text{level}(L):=j$.

We now build a Whitney--style decomposition of $[-4,4]^n$ which reflects the decay behaviour of $V$ in each region. Cubes will be chosen depending on the failure of decay of either the tilt-excess or the height (in $L^\infty$) at suitable scales. We will have numerous choices in our construction, determined by choices of constants. These constants are:
\begin{enumerate}
	\item [(i)] $M_0\geq 4$, determining the scale of the balls determined by our Whitney-style decomposition, and thus how much overlap different balls have (thus $M_0$ is a parameter determining the level of `discretisation');
	\item [(ii)] $N_0\geq 1$, determining how small the cubes we are looking at are by controlling the maximal cube size (morally, this gives an upper bound on the convolution scale by $2^{-N_0}$ in Remark \ref{remark:cm-pulse}, thus forcing a certain degree of closeness to the average);
	\item [(iii)] $C_e\geq 1$ determining the failure of tilt-excess decay;
	\item [(iv)] $C_h\geq 1$ determining the failure of $L^\infty$ height decay.
\end{enumerate}	
We will always be choosing $M_0$ first depending on $\mathcal{V}$ (in particular, we take this to include the dependence on $n,k,Q$, for simplicity), followed by $N_0$, then followed by $C_e$, then $C_h$, and finally $\eps_2$. In this ordering of constant choice, later constants are therefore allowed to depend on earlier ones, e.g.~$C_e$ is allowed to depend on $M_0$ and $N_0$, as in the end these will only depend on $\mathcal{V}$. As an initial requirement, we impose that $M_0$, $N_0$ obey
\begin{equation}\label{E:M_0-N_0-condition}
M_0\sqrt{n}2^{7-N_0}\leq 1.
\end{equation}
For each $L\in \mathcal{C}$, set $p_L\equiv (y_L,x_L):= (u_1(x_L),x_L)$ for some choice of $u_1$ (e.g.~order the values of $u$ by the natural ordering of the function $e_1\cdot u :B^n_1(0)\to \A_Q(\R)$ and take the smallest one in this ordering). Set $r_L:= M_0\sqrt{n}\ell(L)$ and $\mathbf{B}_L:= B^{n+k}_{64r_L}(p_L)$. Write $P_L:= P_{\mathbf{B}_L}$ for the choice of plane which achieves the infimums in both $\Etilt_V(\mathbf{B}_L)$ and $\mathbf{h}_V(\mathbf{B}_L)$, as defined above. We then define families of cubes $\mathcal{R}$, $\S\subset\mathcal{C}$, with $\mathcal{S} = \mathcal{S}_e\cup\mathcal{S}_h\cup\S_n$, $\mathcal{R} = \cup_j\mathcal{R}^j$, $\mathcal{S}_\star = \cup_j \S^j_\star$ for $\star\in\{e,h,n\}$\footnote{This is a slight abuse of notation, as $n$ does not denote a dimension here but rather indicates `neighbouring' cubes.}, and $\S^j = \S^j_e\cup\S^j_h\cup\S^j_n$, as follows. For $j<N_0$ we set $\mathcal{R}^j = \S^j = \emptyset$. For $j\geq N_0$, we proceed inductively. If no ancestor of $L\in \mathcal{C}^j$ is in $\S$, then
\begin{enumerate}
	\item [(EX)] $L\in\S^j_e$ if $\Etilt_V(\mathbf{B}_L)> C_e\Etilt_V\ell(L)^{1-\tfrac{1}{12}}$;
	\item [(HT)] $L\in\S^j_h$ if $L\not\in \S^j_e$ and $\mathbf{h}_V(\mathbf{B}_L)>C_h\Etilt_V\ell(L)$;
	\item [(NN)] $L\in\S_n^j$ if $L\not\in \S^j_e\cup\S^j_h$ but $L$ intersects an element of $\S^{j-1}$.
\end{enumerate}
If none of (EX), (HT), or (NN) occur, then $L\in \mathcal{R}^j$. (Notice that we never consider $L\in \CC^j$ which has an ancestor in $\S$.) Finally, after the inductive construction is complete, set $\Gamma:= [-4,4]^n\setminus\bigcup_{L\in\S}L$.

\begin{remark}\label{remark:lower-exponent-in-cm}
	One might wonder why we chose a power of $1-\frac{1}{12}$ in condition (EX), as opposed to the (seemingly) more natural power of $1$. It is actually very important that this power is $<1$, with the reason being because one cannot distinguish between regular points and singular points when the power is $1$ (naturally, smooth functions come into their tangent planes at a rate of at least $2$ by Taylor's theorem). As such we do not want to see failure of condition (EX) because the sheets locally coincide, and the rate of decay to the tangent plane is exactly $2$; morally at such points the center manifold should just be the sheet itself, but this would not necessarily happen if (EX) failed at such points\footnote{Compare this to the single-valued case. One would want to construct a center manifold which coincides with the graph itself (indeed, it was already noted by Almgren that his construction achieves this). In this situation, (HT) always fails. However, if the power in (EX) was $1$, there could be points where (EX) holds -- even though the graph is smooth! To get around this, we decrease the power slightly, and this issue disappears.}. This technicality manifests itself when we try to prove the doubling condition for cubes in (EX) (see Lemma \ref{lemma:cm-17}) analogously to what we did in Lemma \ref{lemma:doubling-2}, except now we will want to control the whole energy by that of the average-free part. But one cannot do this if the sheets all coincide, as the average-free part will vanish! And the issue is exactly because the average could have decay rate exactly $2$. The choice of power $\frac{11}{12}$ is, however, essentially arbitrary: any power in $(2/3,1)$ would suffice, and we choose $11/12$ simply because it makes some of the constants nicer; we will make it clear in our statements where this choice appears. The upper bound of $1$ here is exactly for the above reason, whilst the lower bound of $2/3$ is the needed to prove that the center manifold is $C^{3,\beta}$ for some $\beta>0$. As a final comment, one might think that if one takes an exponent of $1$ in (EX), one could rule out coincidence with a single-valued function in an alternative manner by taking the constant $C_e$ large enough (as a homogeneous degree $2$ harmonic function with energy $1$ at unit scale cannot have proportionally large energy at a smaller scale). However, the problem is more a result of the relative change in the tilt-excess from one scale to the next, which is what one needs for the doubling condition in Lemma \ref{lemma:cm-17} (which is unchanged by changing $C_e$), rather than the size of the constant.
\end{remark}

Morally, at each stage condition (NN) forces all neighbouring cubes to those included at the previous stage in $\S$ to ``spawn'' and be included in $\S$; they might not lie in $\S_n^j$ (as they might lie in $\S_e^j$ or $\S_h^j$), but they will lie in $\S^j$. It is this inductive spawning, combined with possibly smaller cubes appearing from (EX) and (HT), which gives rise to $\S$. The complement of this is then $\Gamma$. Note also that for each $j$, $\mathcal{R}^j\cup\bigcup_{i\leq j}\S^i = [-4,4]^n$.

Our first observation is that the pair $(\Gamma,\S)$ forms a (\emph{weak}) \emph{Whitney decomposition} of $[-4,4]^n$, in the sense that:
\begin{enumerate}
	\item [(W1)] $\Gamma$ is closed, $\Gamma\cup\bigcup_{L\in\S}L = [-4,4]^n$, and $\Gamma\cap L = \emptyset$ for all $L\in\S$;
	\item [(W2)] if $L_1,L_2\in\S$ are distinct, then $\text{int}(L_1)\cap \text{int}(L_2)=\emptyset$, where $\text{int}(L)$ denotes the interior of the (closed) cube $L$;
	\item [(W3)] if $L_1,L_2\in \S$ have non-empty intersection, then $|\text{level}(L_1)-\text{level}(L_2)|\leq 1$; in particular, $\ell(L_1)/\ell(L_2)\in \{\tfrac{1}{2},1,2\}$.
\end{enumerate}
We leave the reader to verify these simple properties. It can also be seen that (W1) -- (W3) imply that for every $L\in\S$ we have $\dist(\Gamma,L)\geq 2\ell(L)$ (we stress that here for two sets $A,B\subset\R^{n}$ we define $\dist(A,B):=\inf_{a\in A,\,b\in B}|a-b|$, which is different from the Hausdorff distance $\dist_\H(A,B)$). However, we do not necessarily have any bound of the form $\dist(\Gamma,L)\leq C\ell(L)$ (such a bound would mean that $(\Gamma,\mathcal{S})$ is a Whitney decomposition of $[-4,4]^n$ in the more traditional sense, hence why we called the above a `weak' Whitney decomposition).

\textbf{Step 2: Initial Comparisons between Decay Regions.} The decomposition $(\Gamma,\S)$ is the decomposition we want which reflects the decay behaviour of $V$. Let us now prove some basic properties regarding it.

\begin{lemma}\label{lemma:cm-1}
	There exists $\eps = \eps(\mathcal{V})\in (0,1)$ and $C^* = C^*(\mathcal{V},M_0,N_0)\in (0,\infty)$ such that if $\eps_2<\eps$, $C_e\geq C^*$, and $C_h\geq C^*C_e$, then $\S^j = \emptyset$ for all $j\leq N_0+6$.
\end{lemma}

\begin{proof}
	The proof is simply playing around with the constants and inequalities in the definitions.

	Fix $N_0\leq j\leq N_0+6$, and take $L\in \mathcal{C}^j$. By \eqref{E:M_0-N_0-condition}, we know that $r_L \leq 2^{-7}$, and so $64r_L\leq \tfrac{1}{2}$. Moreover, we know that $|p_L|\leq 4\sqrt{n}+C\Etilt_V$ from Theorem \ref{thm:eps-reg} (combined with the Poincaré inequality to control the height excess by the tilt excess), where $C = C(\mathcal{V})$. In particular, this shows that for $\eps = \eps(\mathcal{V})$ sufficiently small, we have $\mathbf{B}_L\subset B^{n+k}_{5\sqrt{n}}(0)$. Thus,
	\begin{equation}\label{E:cm-1-1}
	\Etilt_V(\mathbf{B}_L)\leq \Etilt_{V,P_0}(\mathbf{B}_L)\leq \frac{(6\sqrt{n})^{\tfrac{n}{2}}}{\left(64M_0\sqrt{n}2^{-N_0-6}\right)^{\tfrac{n}{2}}}\Etilt_{V,P_0}(B^{n+k}_{6\sqrt{n}}(0)).
	\end{equation}
	Thus, as $\ell(L)\geq 2^{-N_0-6}$, for suitable $C^* = C^*(n,k,M_0,N_0)$, if $C_e\geq C^*$ this implies
	$$\Etilt_V(\mathbf{B}_L)\leq C_e\Etilt_V\ell(L)$$
	i.e.~(EX) does not hold.
	
	Now, as $p_L\in\spt\|V\|$, we know (e.g.~from the Lipschitz graph structure or the monotonicity formula) that $\|V\|(\mathbf{B}_L)\geq c_0r_L^n$, for some $c_0 = c_0(n)$. Thus, using again \eqref{E:cm-1-1},
	$$\|\pi_{P_L}-\pi_{P_0}\|\leq C_0\left(\Etilt_{V,P_0}(\mathbf{B}_L) + \Etilt_{V,P_L}(\mathbf{B}_L)\right) \leq 2C_0 C_e\Etilt_V\ell(L)$$
	where $C_0 = C_0(n,k,Q,M_0,N_0)$. So, we have
	\begin{align*}
	\mathbf{h}_V(\mathbf{B}_L)\leq \widetilde{C}_0\ell(L)\cdot\|\pi_{P_L}-\pi_{P_0}\| + \mathbf{h}_{V,P_0}(\mathbf{B}_L) & \leq 2\widetilde{C}_0C_0C_e\Etilt_V\ell(L)^2 + \widetilde{C}_0\ell(L)\cdot\mathbf{h}_V(\R^k\times B^n_{5\sqrt{n}}(0))\\
	& \leq \bar{C}_0(C_e+1)\Etilt_V\cdot\ell(L)
	\end{align*}
	where $\widetilde{C}_0 = \widetilde{C}_0(n,k,Q,M_0,N_0)$ and $\bar{C}_0 = \bar{C}_0(n,k,Q,M_0,N_0)$. Thus, if $C^* = C^*(n,k,Q,M_0,N_0)$ is chosen sufficiently large and $C_h\geq C^*C_e$, then
	$$\mathbf{h}_V(\mathbf{B}_L)\leq C_h\Etilt_V\ell(L).$$
	Thus, (HT) also fails for $L$. As this applies for every such $L$, this shows that $\mathcal{S}^j=\emptyset$ for every $j\leq N_0+6$.
\end{proof}

\textbf{Note:} The same proof shows that, for any fixed $K_0$, we can ensure $\S^j=\emptyset$ for all $j\leq N_0+K_0$, provided we allow $C^*$ to depend on $K_0$ also.

The next lemma gives comparisons between the planes $P_L$ realising $\Etilt_V(\mathbf{B}_L)$ for different cubes $L\in \mathcal{R}\cup\mathcal{S}$. What this lemma is effectively doing is controlling the planes across neighbouring cubes, which will allow us to ``glue'' the planes together when forming the center manifold to establish its existence and regularity. It also says that $V$ remains close to the relevant plane in the relevant region.

\begin{lemma}\label{lemma:cm-3}
	Let $C^* = C^*(\mathcal{V},M_0,N_0)\in (0,\infty)$ be as in Lemma \ref{lemma:cm-1}. Then, there exists $\eps = \eps(\mathcal{V},M_0,N_0,C_e,C_h)\in (0,1)$ such that if $\eps_2<\eps$, $C_e\geq C^*$, and $C_h\geq C^*C_e$, then:
	\begin{equation}\label{lemma:cm-2}
		\mathbf{B}_H\subset\mathbf{B}_L\subset B^{n+k}_{5\sqrt{n}}(0) \qquad \text{for any }H,L\in\mathcal{R}\cup\mathcal{S} \text{ with }H\subset L.
	\end{equation}
	Moreover, suppose $H,L\in\mathcal{R}\cup\mathcal{S}$ are such that either \textnormal{(a)} $H\subset L$, or \textnormal{(b)} $H\cap L\neq\emptyset$ and $\textnormal{level}(L)\leq\textnormal{level}(H)\leq\textnormal{level}(L)+1$. Then:
	\begin{enumerate}
		\item [(i)] $\|\pi_{P_L}-\pi_{P_H}\|\leq C_1\Etilt_V\ell(L)^{1-\frac{1}{12}}$;
		\item [(ii)] $\|\pi_{H}-\pi_{P_0}\|\leq C_1\Etilt_V$;
		\item [(iii)] $\mathbf{h}_V(C_{32r_H}(p_H,P_0)) \leq C_2\Etilt_V\ell(H)$ and $\spt\|V\|\cap C_{32r_H}(p_H,P_0)\subset\mathbf{B}_H$;
		\item [(iv)] $\mathbf{h}_V(C_{32r_L}(p_L,P_H)) \leq C_2\Etilt_V \ell(L)$ and $\spt\|V\|\cap C_{32r_L}(p_L,P_H)\subset\mathbf{B}_L$.
	\end{enumerate}
	Here, $C_1 = C_1(\mathcal{V},M_0,N_0,C_e)$ and $C_2 = C_2(\mathcal{V},M_0,N_0,C_e,C_h)$.
\end{lemma}

\begin{proof}
	The proof is a bit long but is effectively bookkeeping. The idea is that, since our cube refinement only stops when the decay condition stops, we still have control on each cube as they become finer and finer.
	
	Notice first that the second inclusion in \eqref{lemma:cm-2} we have already shown in the proof of Lemma \ref{lemma:cm-1}; note that the same proof there works for any $L\in \mathcal{R}\cup\mathcal{S}$. Notice also that the second claims in (iii) and (iv), respectively, follow from the first claims in (iii) and (iv), respectively, by choosing $C_2\eps_2<1$, i.e.~$\eps_2 = \eps_2(C_2)$ is sufficiently small. Thus, to prove \eqref{lemma:cm-2} we only need to prove the first inclusion in \eqref{lemma:cm-2}, and to prove (iii) and (iv) we just need to prove the first inequality of each.
	
	Consider first situation (a), i.e.~when $H\subset L$ (this includes \eqref{lemma:cm-2}). We start by proving claims \eqref{lemma:cm-2}, (i), (ii), and (iii) by induction on the level of $H$ (we will return to (iv) afterwards). The base case is when $\text{level}(H) = N_0$; this forces $H=L$ and so $\text{level}(L) = N_0$ also. In particular, as $H = L$, the first inclusion in \eqref{lemma:cm-2} is trivial, as is claim (i). To see (ii), as $\ell(H) = 2^{-N_0}$, the argument from the proof of Lemma \ref{lemma:cm-1} will give that $\|\pi_{P_H}-\pi_{P_0}\|\leq \bar{C}\Etilt_V$ for some $\bar{C} = \bar{C}(\mathcal{V},M_0,N_0,C_e)$, which proves (ii). For the first claim in (iii), notice that
	$$C_{32r_H}(p_H,P_0)\subset C_{5\sqrt{n}}(0,P_0)$$
	(indeed, $|x_H| + 32r_H \leq 5\sqrt{n}$ holds because $|x_H| \leq 4\sqrt{n}$ and $r_H \leq 2^{-7}$). Thus, as by assumption $\mathbf{h}_V(C_{5\sqrt{n}}(0,P_0))\leq C(\mathcal{V})\Etilt_V$, we have
	$$\mathbf{h}_V(C_{32r_H}(p_H,P_0)) \leq \mathbf{h}_{V,P_0}(C_{5\sqrt{n}}(0,P_0)) \leq C(\mathcal{V})\Etilt_V = 2^{N_0}C(\mathcal{V})\cdot \Etilt_V\ell(H)$$
	as $\ell(H) = 2^{-N_0}$. This proves (iii) when $\text{level}(H) = N_0$.
	
	Thus, under the assumption of (a), we have verified \eqref{lemma:cm-2} and (i) -- (iii) when $\text{level}(H) = N_0$. Now we move to the inductive step.
	
	First consider \eqref{lemma:cm-2}. Notice again that when $H=L$, the first inclusion in \eqref{lemma:cm-2} is trivial, and so we may assume $H\subsetneq L$. We can therefore find $H \equiv H^i\subset H^{i-1}\subset\cdots\subset H^j\equiv L$ such that $i>j$ and $H^k\in \mathcal{R}^k$ for all $k<i$ (we make no claim on whether $H^i\equiv H$ is in $\mathcal{R}^i$ or $\mathcal{S}^i$). It therefore suffices to show $\mathbf{B}_{H^i} \subset \mathbf{B}_{H^{i-1}}$, as then we would be done by induction. By our inductive assumption, we know that (iii) holds for $H^{i-1}$, and thus $|y_{H^i}-y_{H^{i-1}}| \leq C_2\Etilt_V\ell(H^{i-1})$. Since also $|x_{H^i}-x_{H^{i-1}}|\leq \sqrt{n}\ell(H^{i-1})$, we see that
	$$|p_{H^i}-p_{H^{i-1}}| \leq (\sqrt{n} + C_2\Etilt_V)\ell(H^{i-1}).$$
	In particular, as $\ell(H^{i-1}) = 2\ell(H^i)$, if we choose $\eps_2$ small enough we get $|p_{H^i}-p_{H^{i-1}}| \leq 3\sqrt{n}\ell(H^i)$. Now, to show $\mathbf{B}_{H^i}\subset\mathbf{B}_{H^{i-1}}$, we need
	$$64r_{H^i} + |p_{H^i}-p_{H^{i-1}}| \leq 64r_{H^{i-1}}.$$
	As $\ell(H^{i-1}) = 2\ell(H^i)$, it therefore suffices to have
	$$64M_0\sqrt{n}\ell(H^i) + 3\sqrt{n}\ell(H^i) \leq 128M_0\sqrt{n}\ell(H^i)$$
	i.e.~$64M_0 + 3 \leq 128M_0$, which is clearly true as $M_0\geq 1$. This therefore proves \eqref{lemma:cm-2} by induction.
	
	Now we prove (i); again the case $H=L$ is trivial, and so we can assume $H\subsetneq L$. Again, we may find $H\equiv H^i\subset H^{i-1}\subset\cdots \subset H^{j}\equiv L$ such that $i>j$ and $H^k\in\mathcal{R}^k$ for all $k<i$ (we make no claim on whether $H^i\equiv H$ is in $\mathcal{R}^i$ or $\mathcal{S}^i$). Then, for $a\in \{j+1,\dotsc,i\}$, we know from the monotonicity formula that $\|V\|(\mathbf{B}_{H^a})\geq c_0(64r_{H^a})^n$ for some $c_0 = c_0(n)$. Thus,
	$$\|\pi_{P_{H^a}}-\pi_{P_{H^{a-1}}}\|^2 \leq \frac{4(64r_{H^a})^n}{\|V\|(\mathbf{B}_{H^a})}\left(\Etilt_{V,P_{H^a}}^2(\mathbf{B}_{H^a}) + \Etilt_{V,P_{H^{a-1}}}^2(\mathbf{B}_{H^a})\right).$$
	Now as $\Etilt_{V,P_{H^a}}(\mathbf{B}_{H^a}) = \Etilt_V(\mathbf{B}_{H^a}) \leq \Etilt_{V,P_{H^{a-1}}}(\mathbf{B}_{H^a})$ by definition of $P_{H^a}$, and since $\mathbf{B}_{H^a}\subset\mathbf{B}_{H^{a-1}}$ by \eqref{lemma:cm-2} (which we know is true by induction, assuming $\eps_2$ is sufficiently small) we see that
	$$\|\pi_{P_{H^a}}-\pi_{P_{H^{a-1}}}\|^2 \leq \frac{8\cdot 2^n}{c_0}\Etilt_{V,P_{H^{a-1}}}^2(\mathbf{B}_{H^{a-1}}) = \frac{2^{n+3}}{c_0}\Etilt_V^2(\mathbf{B}_{H^{a-1}}).$$
	But as $H^{a-1}\in \mathcal{R}^{a-1}$, we know that $\Etilt_V(\mathbf{B}_{H^{a-1}}) \leq C_e\Etilt_V\ell(H^{a-1})^{1-\frac{1}{12}}$. So, we get
	$$\|\pi_{P_{H^a}}-\pi_{P_{H^{a-1}}}\| \leq \left(\frac{2^{n+3}}{c_0}\right)^{1/2}C_e\Etilt_V\ell(H^{a-1})^{1-\frac{1}{12}}.$$
	Now using the triangle inequality, we see that
	$$\|\pi_{P_H}-\pi_{P_L}\| \leq \sum^i_{a=j+1}\|\pi_{P_{H^a}}-\pi_{P_{H^{a-1}}}\| \leq \left(\frac{2^{n+3}}{c_0}\right)^{1/2}C_e\Etilt_V\sum^i_{a=j+1}\ell(H^{a-1})^{1-\frac{1}{12}}.$$
	But noting that this sum is bounded by $\sum^\infty_{a=j}(2^{11/12})^{-a} \leq C(2^{-j})^{11/12} \equiv C\ell(L)^{11/12}$, we see that
	$$\|\pi_{P_H}-\pi_{P_L}\|\leq \bar{C}_1\Etilt_V\ell(L)^{1-\frac{1}{12}}$$
	where $\bar{C}_1 = \bar{C}_1(n,C_e)$. This proves (i) in the case of (a).
	
	 To see (ii), this follows readily from (i). Indeed, take any $L\in \mathcal{S}^{N_0}$ with $H\subset L$ (note such $L$ exists by Lemma \ref{lemma:cm-1}) and apply (i) to get, as $\ell(L) = 2^{-N_0}$,
	 $$\|\pi_{P_H} - \pi_{P_L}\| \leq \bar{C}_1\Etilt_V\cdot (2^{1-\frac{1}{12}})^{-N_0}.$$
	 But as (ii) also holds for $L$ (this was the base case of the induction), we know $\|\pi_{P_L}-\pi_{P_0}\|\leq \bar{C}\Etilt_V$ for some $\bar{C} = \bar{C}(\mathcal{V},M_0,N_0,C_e)$. The triangle inequality then gives
	 $$\|\pi_{P_H}-\pi_{P_0}\|\leq C_1\Etilt_V$$
	 for some $C_1 = C_1(\mathcal{V},M_0,N_0,C_e)$, which shows (ii).
	 
	 Next we prove (iii); as already mentioned, we just need to prove the first claim of (iii). Suppose $\text{level}(H) = i+1$, and choose $H^i\supset H\equiv H^{i+1}$ with $H^i\in \mathcal{R}^i$. By induction, we know that (iii) holds for $H^i$, and so we know in particular that $\spt\|V\|\cap C_{32r_{H^i}}(p_{H^i},P_0)\subset \mathbf{B}_{H^i}$.
	 
	 But we know by construction that $|x_{H^{i+1}}-x_{H^i}| \leq 2\sqrt{n}\ell(H^{i+1}) \leq 2r_{H^{i+1}}$ and $r_{H^{i+1}} = \frac{1}{2}r_{H^i}$, meaning that $|x_{H^{i+1}}-x_{H^i}| + 32r_{H^{i+1}} \leq 17r_{H^i}$. This gives
	 $$C_{32r_{H^{i+1}}}(p_{H^{i+1}},P_0)\subset C_{17r_{H^i}}(p_{H^i},P_0)\subset C_{32r_{H^i}}(p_{H^i},P_0).$$
	 Hence we have
	 \begin{align*}
	 	\mathbf{h}_V(C_{32r_{H^{i+1}}}(p_{H^{i+1}},P_0)) & \leq \mathbf{h}_V(C_{32r_{H^i}}(p_{H^i},P_0))\\
		& \leq \mathbf{h}_{V,P_0}(\mathbf{B}_{H^i})\\
		& \leq \mathbf{h}_V(\mathbf{B}_{H^i}) + C(n,k)r_{H^i}\|\pi_{P_{H^i}}-\pi_{P_0}\|\\
		& \leq C_h\Etilt_V\ell(H^i) + C(\mathcal{V},M_0,N_0,C_e)\Etilt_V\ell(H^i)\\
		& \equiv \frac{1}{2}C_2\Etilt_V\ell(H^i) \equiv C_2\Etilt_V\ell(H^{i+1}),
	 \end{align*}
	 where the first line follows from the inclusion of cylinders, the second line follows from the induction hypothesis, the third is the triangle inequality, and the fourth follows from the fact that $H^i\in\mathcal{R}^i$ and so (HT) fails, as well as using (ii) (which is true inductively as shown). As $\ell(H^i) = 2\ell(H^{i+1})$, this completes the proof of (iii).
	 
	 Thus, now we have completed the proof of \eqref{lemma:cm-2}, as well as that of (i), (ii), and (iii) in case (a). To complete case (a), we just need to verify (iv).
	 
	To prove (iv) under assumption (a), we will work by induction on $\text{level}(L)\in \{N_0,N_0+1,\dotsc\}$. Consider the base case, so that $L\in\mathcal{R}^{N_0}$. We claim that
	$$C_{32r_L}(p_L,P_H)\cap B^{n+k}_{6\sqrt{n}}(0)\subset C_{5\sqrt{n}}(0,P_0).$$
	To see this, recall that $|p_L| \leq 4\sqrt{n} + C(\mathcal{V})\Etilt_V$ and $r_L\leq 2^{-7}$, and so a simple computation shows that this inclusion is valid provided $\|\pi_{P_H}-\pi_{P_0}\|$ is sufficiently small, which is guaranteed by (ii) if $\eps_2 = \eps_2(\mathcal{V},M_0,N_0,C_e)$ is sufficiently small; this proves the claimed inclusion. But then this gives
	$$\spt\|V\|\cap C_{32r_L}(p_L,P_H)\subset C_{5\sqrt{n}}(0,P_0),$$
	and hence
	$$\mathbf{h}_V(C_{32r_L}(p_L,P_H))\leq \mathbf{h}_V(C_{5\sqrt{n}}(0,P_0)) + C(n)r_L\|\pi_{P_H}-\pi_{P_0}\| \leq \bar{C}\Etilt_V \ell(L)$$
	where $\bar{C} = \bar{C}(\mathcal{V},M_0,N_0,C_e)$; here we have used (ii), our a priori height bound (from Theorem \ref{thm:eps-reg} and the relevant Poincaré inequality to control the height excess by the tilt excess as already explained there), and that $\ell(L) = 2^{-N_0}$. This proves the first claim of (iv) when $\text{level}(L) = N_0$ (from which the second claim of (iv) follows by taking $\eps_2 = \eps_2(\bar{C})$ sufficiently small, as already explained).
	
	Now suppose by induction that we have established (iv) in case (a) when $\text{level}(L)\in \{N_0,N_0+1,\dotsc,i\}$, and we wish to prove it for $\text{level}(L) = i+1$. We may therefore choose $L^i\in\mathcal{R}^i$ with $L^{i+1}\equiv L\subset L^i$. We claim:
	$$C_{32r_{L^{i+1}}}(p_{L^{i+1}},P_H)\subset C_{32r_{L^i}}(p_{L^i},P_H).$$
	This follows in essentially the same way as in the base case. Indeed, from (iii) we know $|y_{L^{i+1}}-y_{L^i}| \leq C_2\Etilt_V\ell(L^i)$, and so as $|x_{L^{i+1}}-x_{L^i}| \leq \sqrt{n}\ell(L^i)$, provided $\eps_2$ is small enough, we get $|p_{L^{i+1}}-p_{L^i}| \leq 3\sqrt{n}\ell(L^{i+1})$. To prove the inclusion, we need $|p_{L^{i+1}}-p_{L^i}| + 32r_{L^{i+1}} \leq 32r_{L^i}$, which is true as $r_{L^i} = 2r_{L^{i+1}}$ and $M_0\geq 1$. But then by our induction assumption, we know that
	$$\spt\|V\|\cap C_{32r_{L^{i+1}}}(p_{L^{i+1}},P_H)\subset \spt\|V\|\cap C_{32r_{L^i}}(p_{L^i},P_H)\subset \mathbf{B}_{L^i}$$
	and therefore
	\begin{align*}
		\mathbf{h}_V(C_{32r_{L^{i+1}}}(p_{L^{i+1}},P_H)) & \leq \mathbf{h}_V(\mathbf{B}_{L^i}) + C(n,k)r_{L^i}\|\pi_{P_{L^i}}-\pi_{P_H}\|\\
		& \leq C_h\Etilt_V\ell(L^i) + C(\mathcal{V},M_0,N_0,C_e)\Etilt_V\cdot\ell(L_i)^{1+1-\frac{1}{12}}\\
		& \leq C_2\Etilt_V\ell(L^{i+1}),
	\end{align*}
	where in the second inequality we have used that $L^i\in\mathcal{R}^i$ and so (HT) fails, as well as the fact that $H\subset L^i$ and so (i) holds, as well as $\ell(L^i) = 2\ell(L^{i+1})$; here, $C_2 = C_2(\mathcal{V},M_0,N_0,C_e,C_h)$. This establishes the first claim in (iv) by induction (from which the second follows in the same manner as before). This therefore completes the proof of (iv), and thus completes the proof of the lemma under the assumption that (a) holds.
	
	Now we turn our attention to the case where (b) holds. We only need to prove (i) and (iv), as (ii) and (iii) follow from the case of (a).

    For (i), choose a cube $\widetilde{L}\in \mathcal{R}$ such that if $L\in\mathcal{R}^{N_0}$ then $\widetilde{L}=L$ and otherwise $\widetilde{L}$ is the parent of $L$. Then, as we have argued before, we have $|p_{\widetilde{L}}-p_L|\leq 2\sqrt{n}\ell(L)$. Similarly, using that $L$ and $H$ are adjacent, we have $|p_H-p_L|\leq 3\sqrt{n}\ell(L)$. So if $\widetilde{L}$ is the parent of $L$ then $|p_H-p_{\widetilde{L}}|\leq 6\sqrt{n}\ell(L) = 3\sqrt{n}\ell(\widetilde{L})$. Notice that the same bound $|p_H-p_{\widetilde{L}}|\leq 3\sqrt{n}\ell(\widetilde{L})$ is still true if $\widetilde{L} = L$. Thus, we see that $\mathbf{B}_H\cup \mathbf{B}_L\subset \mathbf{B}_{\widetilde{L}}$. Hence, arguing similarly to as we have before,
	\begin{equation}\label{E:lemma-cm-3-1}
		\|\pi_{P_H}-\pi_{P_{\widetilde{L}}}\| \leq C(n)\left(\Etilt_V(\mathbf{B}_H) + \Etilt_V(\mathbf{B}_{\widetilde{L}})\right) \leq C(n)\Etilt_V(\mathbf{B}_{\widetilde{L}})
	\end{equation}
	and similarly
	$$\|\pi_{P_L}-\pi_{P_{\widetilde{L}}}\| \leq C(n)\left(\Etilt_V(\mathbf{B}_L) + \Etilt_V(\mathbf{B}_{\widetilde{L}})\right) \leq C(n)\Etilt_{V}(\mathbf{B}_{\widetilde{L}}).$$
	Using the triangle inequality we therefore have
	$$\|\pi_{P_H}-\pi_{P_L}\|\leq C(n)\Etilt_V(\mathbf{B}_{\widetilde{L}}).$$
	Thus, if $\widetilde{L}\neq L$, then $\widetilde{L}\in \mathcal{R}$ and so (EX) fails for $\widetilde{L}$, meaning that $\Etilt_V(\mathbf{B}_{\widetilde{L}})\leq C_e\Etilt_V\ell(\widetilde{L})^{1-\frac{1}{12}} \leq 2C_e\Etilt_V\ell(L)^{1-\frac{1}{12}}$, and so the result follows. Otherwise, $L = \widetilde{L}$, and then $L\in\mathcal{R}^{N_0}$; but then one can control $\Etilt_V(B_{\widetilde{L}})$ by $\Etilt_V(B^{n+k}_{5\sqrt{n}}(0))$ up to constants depending on $N_0$ (using \eqref{lemma:cm-2}), and so the result follows. Thus, (i) is established.

    Finally, we come to (iv) in the case of (b). In fact, we claim that the conclusion holds not only for $L$ but also for any ancestor $\widetilde{L}$ of $L$. We prove this by induction on the level of $L$, and for a given level we work inductively on the level of the ancestor. Indeed, one sees that the argument under the assumption of (a) goes through in this case, except at the end when one must apply (i) to bound $\|\pi_{P_{\widetilde{L}}}-\pi_{P_H}\|$, where $\widetilde{L}\in\mathcal{R}$ is an ancestor of $L$. However, one does have the desired bound in this case, namely that as in (i), because one can verify \eqref{E:lemma-cm-3-1} analogously (by induction on the level) for all ancestors $\widetilde{L}$ of $L$ in the present case. Hence the argument passes through in an analogous manner, giving the result.
\end{proof}

We may now prove from relatively simple parts of the above lemmas that both the tilt-excess and height remain controlled on \emph{all} the balls $\mathbf{B}_L$, for $L\in \mathcal{R}\cup\mathcal{S}$. This is perhaps expected since as soon as either decay condition fails we stop refining the cube.

\begin{lemma}\label{lemma:cm-4}
	Let $C^*$ be as in Lemma \ref{lemma:cm-1} and $C_e\geq C^*$, $C_h\geq C^*C_e$. Then, for sufficiently small $\eps_2 = \eps_2(\mathcal{V},M_0,N_0,C_e,C_h)\in (0,1)$ we have for any $L\in\mathcal{R}\cup\mathcal{S}$,
	$$\Etilt_V(\mathbf{B}_L)\leq C\Etilt_V\ell(L)^{1-\frac{1}{12}}\ \ \ \text{and}\ \ \ \mathbf{h}_V(\mathbf{B}_L)\leq C\Etilt_V\ell(L).$$
	Here, $C = C(\mathcal{V},M_0,N_0,C_e,C_h)\in (0,\infty)$.
\end{lemma}

\begin{proof}
	The claim for $L\in \mathcal{R}$ follows from the construction, namely because (EX) and (HT) fail. So we only need to prove the claim for $L\in\mathcal{S}$. From Lemma \ref{lemma:cm-1} we know $\text{level}(L)\geq N_0+7$, and so it has a parent $\tilde{L}\in \mathcal{R}$. But by \eqref{lemma:cm-2} from Lemma \ref{lemma:cm-3} we know $\mathbf{B}_L\subset\mathbf{B}_{\tilde{L}}$, and so
	$$\Etilt_V(\mathbf{B}_L)\leq 2^{n/2}\Etilt_V(\mathbf{B}_{\tilde{L}}) \leq 2^{n/2}\cdot C_e\Etilt_V\ell(\tilde{L})^{11/12} = \left(2^{\frac{n}{2}+11/12}C_e\right)\Etilt_V\ell(L)^{11/12}$$
	which shows the first claim. For the second, we have from Lemma \ref{lemma:cm-3}(a)(i),
\begin{align*}
	\mathbf{h}_V(\mathbf{B}_L) & \leq \mathbf{h}_V(\mathbf{B}_{\tilde{L}}) + C(n,k)r_L\|\pi_{P_L}-\pi_{P_{\tilde{L}}}\|\\
	& \leq C_h\Etilt_V\ell(\tilde{L}) + C(n,k)r_L\cdot C_1\Etilt_V\ell(\tilde{L})^{11/12}\\
	& \leq C(\mathcal{V},M_0,N_0,C_e,C_h)\cdot\Etilt_V\ell(L)
\end{align*}
using $\ell(\tilde{L}) = 2\ell(L)$; this proves the claim.
\end{proof}

\textbf{Step 3: Graphical Representation in Decay Regions and Regularity Estimates.} Since we have control over the tilt-excess and the height over cubes in our Whitney-style decomposition, for suitable choices of parameter we will be able to apply our $\eps$-regularity theorem (Theorem \ref{thm:eps-reg}) to deduce a certain graphical representation holds over the cube regions, with estimates reflecting the decay of $V$ on the given cube to a plane. We will then use these graphs to construct the center manifold. In fact, we also prove more to allow us to compare the graphs over neighbouring cubes.

\begin{lemma}\label{lemma:cm-5}
	Let $C^*$ be as in Lemma \ref{lemma:cm-1} and $C_e\geq C^*$, $C_h\geq C^*C_e$. Then, for sufficiently small $\eps_2 = \eps_2(\mathcal{V},M_0,N_0,C_e,C_h)\in (0,1)$, the following holds. Let $H,L\in\mathcal{R}\cup\mathcal{S}$ be such that either: 
	\begin{enumerate}
		\item [(a)] $H\subset L$; or
		\item [(b)] $H\cap L\neq\emptyset$ and $\level(L) \leq \level(H) \leq \level(L)+1$.
	\end{enumerate}
	Then, $V\res C_{32r_L}(p_L,P_H)$ satisfies the assumptions (suitably rescaled) of Theorem \ref{thm:eps-reg} in the cylinder $C_{32r_L}(p_L,P_H)$.
\end{lemma}

\begin{proof}
	From Lemma \ref{lemma:cm-3}(iv) we know that in either case
	$$\spt\|V\|\cap C_{32r_L}(p_L,P_H)\subset\mathbf{B}_L$$
	(moreover from \eqref{lemma:cm-2} we know that $\mathbf{B}_L\subset B^{n+k}_{5\sqrt{n}}(0)$). Moreover, as by assumption $V$ is represented by a multi-valued graph over $B^n_{6\sqrt{n}}(0)$ in $\R^k\times B^n_{6\sqrt{n}}(0)$, with Lipschitz constant $\leq C(\mathcal{V})\Etilt_V$ (by Theorem \ref{thm:eps-reg}), Lemma \ref{lemma:cm-3}(ii) tells us that $V\res C_{32r_L}(p_L,P_H)$ is still a $Q$-valued Lipschitz graph over $P_H$ with Lipschitz constant controlled by $C_1\Etilt_V$, where $C_1 = C_1(\mathcal{V},M_0,N_0,C_e)$. In particular, the mass assumption in Theorem \ref{thm:eps-reg} holds for $V\res C_{32r_L}(p_L,P_H)$. But then from Lemma \ref{lemma:cm-3}(iv) we know
	\begin{align*}
		\Etilt_V^2(C_{32r_L}(p_L,P_H)) & \leq 2^n\Etilt_{V,P_H}^2(\mathbf{B}_L)\\
		& \leq 2^{n+1}\Etilt_V^2(\mathbf{B}_L) + 2^{n+1}\cdot (Q+1)\w_n\|\pi_{P_H}-\pi_{P_L}\|^2\\
		& \leq 2^{n+1}\cdot C^2\Etilt_V^2 \ell(L)^{2(1-\frac{1}{12})} + 2^{n+1}\cdot (Q+1)\w_n\cdot C_1^2\Etilt_V^2 \ell(L)^{2(1-\frac{1}{12})}
	\end{align*}
	where in the last inequality we have used Lemma \ref{lemma:cm-4} and Lemma \ref{lemma:cm-3}(i). To summarise, we have
	$$\Etilt_V(C_{32r_L}(p_L,P_H))\leq \tilde{C}\Etilt_V\ell(L)^{1-\frac{1}{12}}$$
	where $\tilde{C} = \tilde{C}(\mathcal{V},M_0,N_0,C_e,C_h)\in (0,\infty)$. Thus, if $\eps_2 = \eps_2(\mathcal{V},M_0,N_0,C_e,C_h)$ is such that $\tilde{C}\eps_2<\eps$, where $\eps = \eps(\mathcal{V})$ is as in Theorem \ref{thm:eps-reg}, we see that Theorem \ref{thm:eps-reg} is applicable here.
\end{proof}

From Lemma \ref{lemma:cm-5} we therefore see that one can control $V$ in the relevant regions by the size of the cube of interest. Thus, for cubes $H,L$ as in Lemma \ref{lemma:cm-5}, we may apply Theorem \ref{thm:eps-reg} to $V\res C_{32r_L}(p_L,P_H)$ to find a function $u_{HL}: B_{16r_L}(p_L,P_H)\to \A_Q(P_H^\perp)$ which is $GC^{1,\alpha}$ and obeys:
\begin{enumerate}
	\item [(I)] $V\res C_{16r_L}(p_L,P_H) = \mathbf{v}(u_{HL})$;
	\item [(II)] $\|u_{HL}\|_{C^0} \leq \widetilde{C}\Etilt_V\ell(L)$ and $\|Du_{HL}\|_{GC^{0,\alpha}}\leq \widetilde{C}\Etilt_V\ell(L)^{1-\frac{1}{12}}$
\end{enumerate}
where $\widetilde{C} = \widetilde{C}(\mathcal{V},M_0,N_0,C_e,C_h)$. Indeed, the bounds in (II) come from the bounds on the tilt-excess and height on the relevant balls provided by Lemma \ref{lemma:cm-4}, using also that $p_L\in \spt\|V\|$ as well as Lemma \ref{lemma:cm-3}(iv) so that we know $\spt\|V\|\cap C_{32r_L}(p_L,P_H)\subset\mathbf{B}_L$ (conclusion (I) above also uses this).

These functions will form the building blocks of our center manifold construction as well as for the comparison between $V$ and the center manifold. We call $u_{HL}$ the $(H,L)$\emph{-representation of $V$}.

We now begin building the center manifold from these building blocks. First, fix $\varrho\in C^\infty_c(B^n_1(0))$ a non-negative radial function with $\int_{\R^n}\varrho = 1$. For $\lambda>0$, we write $\varrho_\lambda(x):= \lambda^{-n}\varrho(x/\lambda)$.

For each $H,L\in \mathcal{R}\cup\mathcal{S}$ as in Lemma \ref{lemma:cm-5}, given the corresponding $(H,L)$-representation $u_{HL}$ of $V$, write
$$h_{HL}:= (u_{HL})_{\text{a}}\ast\varrho_{\ell(L)}$$
for the smoothed average, i.e.~we convolute by an amount proportional to the side-length of the cube the function was constructed relative to (cf.~Remark \ref{remark:cm-pulse}). We call $h_{HL}$ the \emph{tilted} $(H,L)$\emph{-convoluted average} (``tilted'' because it is defined over a subset of $P_H$, not of $P_0$).

Now, as the Lipschitz constant of $h_{HL}$ is bounded by $\tilde{C}\Etilt_V\ell(L)$ (from Lemma \ref{lemma:cm-5}), and moreover since $\|\pi_{P_H}-\pi_{P_0}\|\leq C_1\Etilt_V$ (from Lemma \ref{lemma:cm-3}(ii)), we may find a (smooth) function $g_{HL}: B^n_{8r_L}(x_L,P_0)\to \R^k$ such that
$$\graph(g_{HL}) = \graph(h_{HL})\cap (\R^k\times B^n_{8r_L}(p_L,P_H)).$$
We call $g_{HL}$ the $(H,L)$\emph{-convoluted average}.

\textbf{Remark:} In the special case $H=L$, we will call $u_L\equiv u_{LL}$ the $L$\emph{-representation of $V$}, $h_L\equiv h_{LL}$ the \emph{tilted $L$-convoluted average}, and $g_L\equiv g_{LL}$ the $L$\emph{-convoluted average}.

It is the functions $g_L$ which we will glue together to form the center manifold. In order to deduce (1) the relevant regularity of the gluing; and (2) that the center manifold is close to the average of $V$ in a very strong way, we will prove some estimates regarding the $g_{HL}$. The key fact for all this is that a multi-valued minimal graph representing a stationary integral varifold close to a plane has an average which is close to being harmonic, where the error term is \emph{cubic} (it is this cubic nature of the error which will give us the correct estimates for both (1) and (2) here). More precisely:

\begin{lemma}\label{lemma:cm-6}
	Let $C^*$ be as in Lemma \ref{lemma:cm-1} and $C_e\geq C^*$, $C_h\geq C^*C_e$. Assume also that $\eps_2 = \eps_2(\mathcal{V},M_0,N_0,C_e,C_h)$ is small enough so that the conclusions of Lemmas \ref{lemma:cm-3}, \ref{lemma:cm-4}, \ref{lemma:cm-5} hold. Let $H,L\in \mathcal{R}\cup\S$ be as in Lemma \ref{lemma:cm-5}. Then,
	$$\left|\int D(u_{HL})_{\textnormal{a}}\cdot D\zeta\right| \leq C\Etilt_V^3 r_L^{3(1-\frac{1}{12})}\|D\zeta\|_{L^1}$$
	for all $\zeta\in C^\infty_c(B_{16r_L}(p_L,P_H);P_H^\perp)$; here $C = C(\mathcal{V},M_0,N_0,C_e,C_h)$.
\end{lemma}
\begin{proof}
	Fix $\zeta\in C^\infty_c(B_{16r_L}(p_L,P_H);P_H^\perp)$. Extend $\zeta$ vertically to a function $\tilde{\zeta}$ via $\tilde{\zeta}(\tilde{y},\tilde{x}):= \tilde{\zeta}(\tilde{x})$, where $\tilde{y}\in P_H^\perp$ and $\tilde{x}\in p_L+P_H$. Then, take $\bar{\zeta}$ to be any function $C^1_c(\R^{n+k};P_H^\perp)$ which agrees with $\tilde{\zeta}$ in a neighbourhood of $\spt\|V\|$. Then, if we take the vector field $\chi = \bar{\zeta}$ in the first variation formula for $V$, we get (using the area formula)
	$$\left|\sum_j \int D u^j_{HL}\cdot D\zeta\right| \leq C \sum_j\int |Du^j_{HL}\cdot D\zeta| |Du_{HL}|^2.$$
	Estimating the integral on the right-hand side, we have
	\begin{align*}
		\left|\int D(u_{HL})_{\text{a}}\cdot D\zeta\right| & \leq C\sup_{B_{16r_L}(p_L,P_H)}|Du_{HL}|^3\int|D\zeta|\\
		& \leq C\Etilt_V^3\ell(L)^{11/4}\|D\zeta\|_{L^1}\\
		& \equiv C\Etilt_V^3r_L^{11/4}\|D\zeta\|_{L^1}
	\end{align*}
	where here $C = C(\mathcal{V},M_0,N_0,C_e,C_h)$; in the second inequality here we used the pointwise bound on $|Du_{HL}|$ provided by the application of Theorem \ref{thm:eps-reg} enabled by Lemma \ref{lemma:cm-5}.
\end{proof}

The simple estimate in Lemma \ref{lemma:cm-6} will be crucial for numerous reasons. First, we use it to show that the functions $h_{HL}$ stay close to the original average $(u_{HL})_{\text{a}}$ in a very strong sense, with a better estimate than one might a priori expect. This will be crucial for controlling the ``bad'' error terms later.

\begin{lemma}\label{lemma:cm-7}
	Assume the assumptions and set-up of Lemma \ref{lemma:cm-6}. Then,
	$$\|h_{HL}-(u_{HL})_{\textnormal{a}}\|_{L^1(B_{15r_L}(p_L,P_H))}\leq C\Etilt_V^3r_L^{n+3(1-\frac{1}{12})+1},$$
	where $C = C(\mathcal{V},M_0,N_0,C_e,C_h)$.
\end{lemma}

\begin{proof}
	This is a simple fact regarding convolutions using the estimate in Lemma \ref{lemma:cm-6}. For simplicity, let us write $h\equiv h_{HL}$, $u\equiv u_{HL}$, and $\ell = \ell(L)$. Then
	\begin{align*}
		h(x)-u_{\text{a}}(x) & = \int\varrho_{\ell}(y)(u_{\text{a}}(x-y)-u_{\text{a}}(x))\ \ext y\\
		& = \int\varrho_{\ell}(y)\left(\int^1_0Du_{\text{a}}(x-\sigma y)\cdot (-y)\ \ext\sigma\right)\ext y\\
		& = \int\int^1_0\varrho_{\ell}(y/\sigma)Du_{\text{a}}(x-y)\cdot\frac{-y}{\sigma^{n+1}}\ \ext\sigma\ext y\\
		& = \int Du_{\text{a}}(x-y)\cdot\underbrace{\left((-y)\int^1_0\varrho_{\ell}(y/\sigma)\sigma^{-n-1}\ \ext\sigma\right)}_{=:\,\Gamma(y)}\ext y.
	\end{align*}
	But note that if we set
	$$\zeta(y) := \int^\infty_{|y|}\tau\int^1_0\varrho_{\ell}\left(\frac{y\tau}{|y|\sigma}\right)\sigma^{-n-1}\ \ext\sigma\ext\tau$$
	then, because $\varrho_\ell$ is radial, we have $D\zeta = \Gamma$. As $\varrho$ has compact support we know that $\zeta$ has compact support, and moreover that $\spt(\zeta)\subset B_{\ell}(0)$. Note however that, whilst $\Gamma$ is smooth on $\R^n\setminus\{0\}$, it is unbounded on any neighbourhood of $0$ (as $n>1$), and so $\zeta$ is not $C^1$. We will instead show that $\zeta\in W^{1,1}$; this is enough for its applicability in Lemma \ref{lemma:cm-6} by an approximation argument. So, notice that
	$$\zeta(y) = |y|^2\int^\infty_1 t\int^1_0\varrho_{\ell}\left(\frac{yt}{\sigma}\right)\sigma^{-n-1}\ \ext\sigma\ext t$$
	and hence
	\begin{align*}
		\|\zeta\|_{L^1}&\leq\int\int^\infty_1\int^1_0 t|y|^2\varrho\left(\frac{yt}{\ell\sigma}\right)\ell^{-n}\sigma^{-n-1}\ \ext\sigma\ext t\ext y\\
		& = \ell^2 \left(\int|u|^2\varrho(u)\ext u\right)\left(\int^1_0\sigma\ext\sigma\right)\left(\int^\infty_1 t^{-n-1}\ext t\right)\\
		& \equiv C\ell^2<\infty
	\end{align*}
	and
	$$\|D\zeta\|_{L^1} = \|\Gamma\|_{L^1} \leq \int\int^1_0|y|\varrho\left(\frac{y}{\ell\sigma}\right)\ell^{-n}\sigma^{-n-1}\ext\sigma\ext y = \ell\left(\int |y|\varrho(y)\ext y\right)\left(\int^1_0\ext\sigma\right) \equiv C\ell<\infty$$
	which shows $\zeta\in W^{1,1}$. Thus, we have
	$$h(x)-u_{\text{a}}(x) = \int Du_{\text{a}}(x-y)\cdot D\zeta(y)\ \ext y = \int Du_{\text{a}}(y)\cdot D\zeta(x-y)\ \ext y$$
	and thus, using Lemma \ref{lemma:cm-6},
	\begin{align*}
		\|h-u_{\text{a}}\|_{L^1(B_{15r_L}(p_L,P_H)} & = \int_{B_{15r_L}(p_L,P_H)}\left|\int Du_{\text{a}}(y)\cdot D\zeta(x-y)\ \ext y\right|\ext x\\
		& \leq C\Etilt_V^3 r_L^{11/4}\int_{B_{15r_L}(p_L,P_H)}\int|D\zeta(x-y)|\ \ext y\ext x\\
		& \leq C\Etilt_V^3r_L^{11/4}\cdot C(n)r_L^n\cdot\|D\zeta\|_{L^1}\\
		& \equiv C\Etilt^3_V r_L^{n+15/4}
	\end{align*}
	as desired.
\end{proof}

The $L^1$ estimate in Lemma \ref{lemma:cm-7} will be used in two crucial ways. One, as already mentioned, is to control the ``bad'' error terms in the frequency computations later. Another way is to prove $C^{3,\beta}$-estimates for the functions $h_{HL}$, $g_{HL}$, for different $H,L$ when the domains overlap. This is done via certain interior Schauder estimates in the following lemma. We begin by comparing functions defined on subsets of the same plane.

\begin{lemma}\label{lemma:cm-8}
	Assume the assumptions and set-up of Lemma \ref{lemma:cm-6}. Then, for each $j\in \{0,1,2,3\}$,
	$$\|h_{HL}-h_H\|_{C^j(B_{6r_H}(p_H,P_H))} + \|g_{HL}-g_H\|_{C^j(B_{4r_H}(p_H,P_0))}\leq C\Etilt_V^3\ell(L)^{3(1-\frac{1}{12})+1-j},$$
	where $C = C(\mathcal{V},M_0,N_0,C_e,C_h)$. Moreover, for any $\beta\in (0,3/4)$
	$$\|h_{HL}-h_H\|_{C^{3,\beta}(B_{6r_H}(p_H,P_H))} + \|g_{HL}-g_H\|_{C^{3,\beta}(B_{4r_H}(p_H,P_0))} \leq C_\beta \Etilt_V^3\ell(L)^{3(1-\frac{1}{12})-2-\beta}$$
	where $C_\beta = C_\beta(\mathcal{V},M_0,N_0,C_e,C_h,\beta)$.
\end{lemma}

\textbf{Remark:} Notice that $h_{HL}$ is defined on $B_{16r_L}(p_L,P_H)$ and $h_H$ is defined on $B_{16r_H}(p_H,P_H)$, and so since $|p_L-p_H|\leq 2\sqrt{n}\ell(L)$ both these domains contain $B_{6r_H}(p_H,P_H)$.

\textbf{Note:} The range of $\beta$ is $(0,3/4) \equiv (0,3(1-\frac{1}{12})-2)$; this indicates the dependence of $\beta$ on the choice of exponent in (EX), and why it needs to be $>2/3$.

\begin{proof}
	We have $H\in \mathcal{R}^j\cup \S^j$ and $L$ is either (a) an ancestor of $H$, or (b) $H\cap L\neq\emptyset$ with\footnote{In fact, going forward we will only need the case of (b) where $\text{level}(L) = \text{level}(H)$.} $\text{level}(L)-\text{level}(H)\in \{0,1\}$. Consider also another cube, $J$, where in the case of (a) $J$ is the parent of $L$ and in the case of (b) $J$ is the parent of $H$.
	
	We start by comparing $h_{HL}$ and $h_{HJ}$; the general case will follow by iteration and the triangle inequality. Notice that, where defined, the only difference between $h_{HL}$ and $h_{HJ}$ is the scaling of the convolution factor (this is simply because the average-part of $u_{HL}$ and $u_{HJ}$ coincide, since both functions coincide with $V$ on their domains of definition). So, in fact, where both are defined,
	$$h_{HL} - h_{HJ} = (u_{HL})_{\text{a}}\ast (\varrho_{\ell(L)}-\varrho_{\ell(J)}).$$
	Fix any $\zeta\in C^\infty_c(B_{9r_L}(p_L,P_H))$. Notice that
	$$\spt(\zeta\ast\varrho_{\ell(L)}),\, \spt(\zeta\ast\varrho_{\ell(J)})\subset B_{9r_L + r_J}(p_L,P_H) \subset B_{11r_L}(p_L,P_H)\subset B_{16r_J}(p_J,P_H)$$
	which guarantees the applicability of Lemma \ref{lemma:cm-6} with both $\zeta\ast\varrho_{\ell(L)}$ and $\zeta\ast\varrho_{\ell(J)}$ in place of $\zeta$. Again, on $\spt(\zeta)$ we have $(u_{HL})_{\text{a}} = (u_{HJ})_{\text{a}}$. So,
	$$\int\zeta\cdot\Delta(h_{HL}-h_{HJ}) = -\int D(h_{HL}-h_{HJ})\cdot D\zeta = -\int D(u_{HL})_{\text{a}}\cdot D(\zeta\ast(\varrho_{\ell(L)}-\varrho_{\ell(J)}))$$
	and so applying Lemma \ref{lemma:cm-6} we see (noting $\ell(J)/\ell(L)\in \{1,2\}$)
	$$\left|\int\zeta\cdot\Delta(h_{HL}-h_{HJ})\right|\leq C\Etilt_V^3 r_L^{11/4}\|D(\zeta\ast(\varrho_{\ell(L)}-\varrho_{\ell(J)}))\|_{L^1}.$$
	So, bounding
	\begin{align*}
	\|D(\zeta\ast(\varrho_{\ell(L)}-\varrho_{\ell(J)}))\|_{L^1} = \|\zeta\ast(D\varrho_{\ell(L)}-D\varrho_{\ell(J)})\|_{L^1} & \leq \|\zeta\||_{L^1}\|D\varrho_{\ell(L)}-D\varrho_{\ell(J)}\|_{L^1}\\
	& \leq C\ell(L)^{-1}\|\zeta\|_{L^1}.
	\end{align*}
	Combining the above inequalities, we get
	$$\left|\int\zeta\cdot\Delta(h_{HL}-h_{HJ})\right|\leq C\Etilt_V^3 r_L^{7/4}\|\zeta\|_{L^1}.$$
	As this is true for all $\zeta\in C^\infty_c(B_{9r_L}(p_L,P_H))$, this immediately gives
	$$\|\Delta(h_{HL}-h_{HJ})\|_{C^0(B_{9r_L}(p_L,P_H))}\leq C\Etilt_V^3r_L^{7/4}.$$
	Notice also that an identical argument gives for any $k\geq 0$ (simply integrating by parts to move the extra derivatives on the $\varrho_{\ell(L)}$ and $\varrho_{\ell(J)}$ terms, which ultimately gives an extra factor of $\ell(L)^{-k}$ in the bound)
	$$\left|\int\zeta\cdot \Delta(D^k(h_{HL}-h_{HJ}))\right|\leq C\Etilt_V^3r_L^{7/4-k}\|\zeta\|_{L^1}$$
	and thus
	$$\|\Delta(D^k(h_{HL}-h_{HJ}))\|_{C^0(B_{9r_L}(p_L,P_H))}\leq C\Etilt_V^3 r_L^{7/4-k}.$$
	Also, from Lemma \ref{lemma:cm-7} and the fact that $(u_{HL})_{\text{a}} = (u_{HJ})_{\text{a}}$ on $B_{9r_L}(p_L,P_H)$, we immediately get from the triangle inequality
	$$\|h_{HL}-h_{HJ}\|_{L^1(B_{9r_L}(p_L,P_H))}\leq C\Etilt_V^3r_L^{n+15/4}.$$
	To summarise, setting for ease of notation $\xi:= h_{HL}-h_{HJ}$, we currently have estimates
	$$\|\xi\|_{L^1(B_{9r_L}(p_L,P_H))}\leq C\Etilt_V^3 r_L^{n+15/4}$$
	and for each $k\geq 0$,
	$$\|\Delta(D^k\xi)\|_{C^0(B_{9r_L}(p_L,P_H))}\leq C\Etilt_V^3 r_L^{7/4-k}.$$
	In particular, we have for any distinct $x,y\in B_{9r_L}(p_L,P_H)$ and $\beta\in (0,3/4)$,
	\begin{align*}
		\frac{|\Delta(D^k\xi)(x)-\Delta(D^k\xi)(y)|}{|x-y|^\beta} & = \frac{|\Delta(D^k\xi)(x) - \Delta(D^k\xi)(y)|}{|x-y|}\cdot |x-y|^{1-\beta}\\
		& \leq 18^{1-\beta}\|\Delta(D^{k+1}\xi)\|_{C^0(B_{9r_L}(p_L,P_H))}\cdot r_L^{1-\beta}\\
		& \leq 18^{1-\beta}\cdot C\Etilt_V^3 r^{7/4-k-\beta}.
	\end{align*}
	Hence,
	$$[\Delta(D^k\xi)]_{C^{0,\beta}(B_{9r_L}(p_L,P_H))}\leq C\Etilt_V^3r_L^{7/4-k-\beta}.$$
	Using these bounds, the interior $C^{2,\beta}$ Schauder estimates give (applied with the operator $\Delta$ and the $L^1$ norm on the right-hand side of the Schauder estimate)
	$$\|\xi\|_{C^k(B_{17r_L/2}(p_L,P_L))}\leq C\Etilt_V^3 r_L^{15/4-k} \qquad \text{for }k\in \{0,1,2\}.$$
	In turn, using this supremum bound for $D\xi$ as well as the previous supremum bound on $\Delta(D\xi)$, applying the interior Schauder estimates again but now to $D\xi$, we get
	$$\|\xi\|_{C^k(B_{8r_L}(p_L,P_H))}\leq C\Etilt_V^3r_L^{15/4-k} \qquad \text{for }k\in\{0,1,2,3\}$$
	and
	$$\|\xi\|_{C^{3,\beta}(B_{8r_L}(p_L,P_H))}\leq C\Etilt_V^3r_L^{3/4-\beta}.$$
	This proves the desired estimates, except that for $h_{HL} - h_{HJ}$.
	
	We now deduce the general form of the first estimate in both claims of the lemma by iterating the above estimates. We split into the two different cases. Suppose first that $H\cap L\neq\emptyset$ and $\level(L)-\level(H)\in\{0,1\}$. Then, let $J$ be the parent of $H$. The above argument allows us to bound $h_{HL}-h_{HJ}$. But it also allows us to bound $h_H - h_{HJ}$ (if we apply it with $L=H$); the triangle inequality then gives the desired bounds on $h_{HL}-h_H$, since $B_{6r_H}(p_H,P_H)\subset B_{8r_L}(p_L,P_H)$ (again, as $|p_H-p_L|\leq 2\sqrt{n}\ell(L)$).
	
	In the second case, we have that $L$ is an ancestor of $H$. So take a chain $H\equiv L^j\subset L^{j-1}\subset\cdots\subset L^i\equiv L$, where $L^k\in \mathcal{R}^k\cup\S^k$ for all $k$. By the above (with $L^k$ in place of $L$ and $L^{k-1}$ in place of $J$) we know how to bound $h_{HL^k}-h_{HL^{k-1}}$ on the ball $B_{8r_{L^k}}(p_{L^k},P_H)$. But we know $B_{6r_H}(p_H,P_H)\subset B_{8r_{L^k}}(p_{L^k},P_H)$ for all such $k$, and so we may bound
	\begin{align*}
	\|h_H-h_{HL}\|_{C^{3,\beta}(B_{6r_H}(p_H,P_H))} & \leq \sum^j_{k=i+1}\|h_{HL^k}-h_{HL^{k-1}}\|_{C^{3,\beta}(B_{6r_H}(p_H,P_H))}\\
	& \leq C\Etilt_V^3\sum^j_{k=i+1}(2^{3/4-\beta})^{-k}\\
	& \leq C\Etilt_V^3\ell(L)^{3/4-\beta}
	\end{align*}
	which proves the first bound in the second claimed estimate of the lemma. The first bound in first claimed estimate of the lemma follows in an analogous fashion. This proves all the desired estimates on $h_{HL}-h_H$. The corresponding estimates for $g_{HL}-g_H$ then follow by simple facts regarding changing coordinates (see \cite[Lemma B.1]{DLS16a} for a precise statement).
\end{proof}

Whilst Lemma \ref{lemma:cm-8} controlled convoluted averages which are defined on the same plane, we also need to compare the convoluted averages which are defined on planes corresponding to adjacent cubes. This is needed to understand the regularity of the center manifold when we move from cube-to-cube. This is achieved in the following lemma.

\begin{lemma}\label{lemma:cm-9}
	Assume the assumptions and set-up of Lemma \ref{lemma:cm-6}, except those on $H,L$. Fix cubes $H,L\in \mathcal{R}\cup\S$ and suppose that $H$ has an ancestor $J$ (which we allow to be $H$ itself) such that $J\cap L\neq\emptyset$ and $\level(J) = \level(L)$. Then, there is a map $\hat{h}_L:B_{8r_J}(p_J,P_H)\to P_H^\perp$ such that:
	\begin{enumerate}
		\item [(a)] $\graph(\hat{h}_L) = \graph(h_L)\cap C_{8r_J}(p_J,P_H)$;
		\item [(b)] $\|h_{HJ}-\hat{h}_L\|_{L^1(B_{8r_J}(p_J,P_H))}\leq C\Etilt_V^3 \ell(J)^{n+3(1-\frac{1}{12})+1}$.
	\end{enumerate}
	Here $C = C(\mathcal{V},M_0,N_0,C_e,C_h)$.
\end{lemma}

\begin{proof}
	Note first from Lemma \ref{lemma:cm-3} that
	$$\|\pi_{P_H}-\pi_{P_L}\|\leq \|\pi_{P_H}-\pi_{P_J}\| + \|\pi_{P_J}-\pi_{P_L}\|\leq 2C_1\Etilt_V\ell(J)^{11/12}.$$
	Note also that, by Lemma \ref{lemma:cm-7},
	$$\|h_{HJ}-(u_{HJ})_{\text{a}}\|_{L^1(B_{15r_J}(p_J,P_H))}\leq C\Etilt_V^3 r_J^{n+15/4}$$
	and (as $r_L = r_J$)
	$$\|h_L-(u_L)_{\text{a}}\|_{L^1(B_{15r_J}(p_L,P_L))}\leq C\Etilt_V^3r_J^{n+15/4}.$$
	We may find functions $\hat{h}_L,\hat{u}:B_{8r_J}(p_J,P_H)\to P_H^\perp$ such that $\graph(\hat{h}_L) = \graph(h_L)\cap C_{8r_J}(p_J,P_H)$ and $\graph(\hat{u})=\graph((u_L)_{\text{a}})\cap C_{8r_J}(p_J,P_H)$, and moreover from the above $L^1$ bound (using \cite[Lemma B.1]{DLS16a}),
	$$\|\hat{h}_L-\hat{u}\|_{L^1(B_{8r_J}(p_J,P_H))}\leq C\Etilt_V^3 r_J^{n+15/4}.$$
	Thus by the triangle inequality, writing $B = B_{8r_J}(p_J,P_H)$ for notational simplicity,
	\begin{align*}
		\|h_{HJ}-\hat{h}_L\|_{L^1(B)} & \leq \|h_{HJ}-(u_{HJ})_{\text{a}}\|_{L^1(B)} + \|(u_{HJ})_{\text{a}}-\hat{u}\|_{L^1(B)} + \|\hat{u}-\hat{h}_L\|_{L^1(B)}\\
		& \leq C\Etilt_V^3r_J^{n+15/4} + \|(u_{HJ})_{\text{a}}-\hat{u}\|_{L^1(B)}
	\end{align*}
	and so it suffices to estimate $\|(u_{HJ})_{\text{a}}-\hat{u}\|_{L^1(B)}$. This follows again from simple facts regarding changing coordinates (indeed, one can write down an expression for $\hat{u}$ in terms of $(u_{L})_{\text{a}}$, the inverse of the map $x\mapsto x+(u_L)_{\text{a}}(x)$, and the projection maps to the planes in question). We refer the reader to \cite[Lemma 5.6]{DLS16a} for the details.
\end{proof}

\textbf{Step 4: Gluing and completion of the construction.} We are now in a position to prove the main estimates on the convoluted averages $g_H$ in order to glue them together to form the center manifold. First, we introduce some further notation which we will need for this. For each $j$ set
$$\mathcal{P}^j:=\mathcal{R}^j\cup\bigcup^j_{i=N_0}\mathcal{S}^i.$$
We also fix a function $\vartheta\in C^\infty_c(\left[-\tfrac{17}{16},\tfrac{17}{16}\right]^n;[0,1])$ with $\vartheta\equiv 1$ on $[-1,1]^n$ and $|D\vartheta|\leq C(n)$. For a cube $L\in \mathcal{R}\cup \S$ we set $\vartheta_L(y):=\vartheta\left(\frac{y-x_L}{\ell(L)}\right)$. We then define the \emph{glued interpolation} at step $j$, $\wp_j:[-4,4]^n\to \R^k$, by
$$\wp_j:= \frac{\sum_{L\in\mathcal{P}^j}\vartheta_Lg_L}{\sum_{L\in\mathcal{P}^j}\vartheta_L}.$$
We will show that the $\wp_j$ converge (in $C^3$) to a function $\wp$, the graph of which will be our center manifold. First, we need some final estimates relating the $g_H$ to one another and to $\wp_j$.

\begin{lemma}\label{lemma:cm-10}
	Assume the assumptions and set-up of Lemma \ref{lemma:cm-6}, except those on $H,L$. Then for any $H,L\in\mathcal{P}^j$ we have:
	\begin{enumerate}
		\item [(i)] For any $\beta\in (0,3/4)$, $\|g_H\|_{C^{3,\beta}(B_{4r_H}(x_H,P_0))}\leq C_\beta\Etilt_V$;
		\item [(ii)] If $H\cap L\neq\emptyset$, then $\|g_H-g_L\|_{C^j(B_{r_H}(x_H))}\leq C_\beta\Etilt_V\ell(H)^{3(1-\frac{1}{12})+1-j-\beta}$ for any $j\in \{0,1,2,3\}$ and $\beta\in (0,3/4)$;
		\item [(iii)] For any $\beta\in (0,3/4)$, $|D^3g_H(x_H)-D^3g_L(x_L)|\leq C_\beta\Etilt_V|x_H-x_L|^{3(1-\frac{1}{12})-2-\beta}$;
		\item [(iv)] $\|g_H-y_H\|_{C^0(B_{4r_H}(x_H,P_0))}\leq C\Etilt_V\ell(H)$ and, if $P_x^{g_H}$ is the tangent plane to $\graph(g_H)$ at $x$, then $\|\pi_{P_H}-\pi_{P_x^{g_H}}\|\leq C\Etilt_V\ell(H)^{1-\frac{1}{12}}$ for all $x\in H$;
		\item [(v)] If $\tilde{L}$ is the cube concentric to $L\in\S^j$ with $\ell(\tilde{L}) = \frac{9}{8}\ell(L)$, then for all $i\geq j$,
		$$\|\wp_i-g_L\|_{L^1(\tilde{L})}\leq C\Etilt_V^3 \ell(L)^{n+3(1-\frac{1}{12})+1}.$$
	\end{enumerate}
	Here, $C$ is a constant depending on $\mathcal{V},M_0,N_0,C_e,C_h$, and $C_\beta$ is a constant allowed to also depend on $\beta$.
\end{lemma}

\begin{proof}
	Claim (i) follows immediately from Lemma \ref{lemma:cm-8}: simply take $L^{N_0}\in \mathcal{R}^{N_0}$ to be an ancestor of $H$ in Lemma \ref{lemma:cm-8} and use the triangle inequality, since the estimates on $\|g_{HL^{N_0}}\|_{C^{3,\beta}}$ follow immediately from those on $h_{HL^{N_0}}$ (which in turn follow from Lemma \ref{lemma:cm-8} using estimates for $u_{HL^{N_0}}$ which follow from the discussion after Lemma \ref{lemma:cm-5}).
	
	Regarding (ii), fix $H,L$ such that $H\cap L \neq \emptyset$. From our construction of $\mathcal{R}\cup\mathcal{S}$, this means that we either have $\ell(H) = \ell(L)$ or (without loss of generality by swapping $H$ and $L$) $\ell(H) = \frac{1}{2}\ell(L)$. We claim that in either case we have
	\begin{equation}\label{E:cm-lemma-10-1}
		\|h_H-\hat{h}_L\|_{L^1(B_{2r_H}(p_H,P_H))}\leq C\Etilt_V^3\ell(H)^{n+15/4}.
	\end{equation}
	Indeed, in the case $\ell(H) = \ell(L)$, we can just apply Lemma \ref{lemma:cm-9} (take $J=H$) to get this (in fact, we get control on the $L^1$ norm on the larger set $B_{8r_H}(p_H,P_H)$). In the case $\ell(H) = \frac{1}{2}\ell(L)$, let $J$ be the parent of $H$; obviously $J\cap L\neq\emptyset$. So apply Lemma \ref{lemma:cm-9} to get $\|h_{HJ}-\hat{h}_L\|_{L^1(B_{8r_J}(p_J,P_H))}\leq C\Etilt_V^3\ell(J)^{n+15/4}$. But from Lemma \ref{lemma:cm-8} we know $\|h_H - h_{HJ}\|_{L^1(B_{6r_H}(p_H,P_H))}\leq C\Etilt_V^3\ell(J)^{n+15/4}$, and so from the triangle inequality we get $\|h_H-\hat{h}_L\|_{L^1(B_{2r_H}(p_H,P_H))}\leq C\Etilt_V^3\ell(H)^{n+15/4}$, giving \eqref{E:cm-lemma-10-1}. Thus, \eqref{E:cm-lemma-10-1} is established. But this estimate then passes to $g_H-g_L$ on a smaller ball by changing coordinates (see \cite[Lemma B.1]{DLS16a}), giving
	$$\|g_H-g_L\|_{L^1(B_{3r_H/2}(p_H,P_0))}\leq C\Etilt_V^3\ell(H)^{n+15/4}.$$
	Finally, since we know $\|g_H-g_L\|_{C^{3,\beta}(B_{3r_H/2}(p_H,P_0))}\leq C_\beta\Etilt_V$ from (i), we can apply the interpolation inequalities for Hölder spaces (or interior Schauder estimates) to get (ii).

    For (iii), notice that if we applied the argument for (ii) instead assuming that $L$ is the parent of $H$, we would deduce that for any $0<\gamma<\beta<3/4$ we have
        $$[D^3g_H-D^3g_L]_{\gamma,B_{r_H}(x_H)}\leq C_{\beta,\gamma}\Etilt_V\ell(H)^{\beta-\gamma}.$$
        Hence (taking e.g.~$\gamma = \beta/2$)
        \begin{align*}
            |D^3g_H(x_H)-D^3g_L(x_L)| & \leq |D^3g_H(x_H)-D^3g_H(x_L)| + |D^3g_H(x_L)-D^3g_L(x_L)|\\
            & \leq \ell(H)^\gamma [D^3g_H-D^3g_L]_{\gamma,B_{r_H}(x_H)} + |D^3g_H(x_L)-D^3g_L(x_L)|\\
            & \leq \ell(H)^\gamma\cdot C_{\beta,\beta/2}\Etilt_V\ell(H)^{\beta-\gamma} + C_{3/4-\beta}\Etilt_V\ell(H)^{\beta}\\
            & \equiv C_{\beta}\Etilt_V\ell(L)^{\beta}.
        \end{align*}
	If we iterate this and use the triangle inequality, we also get that for any ancestor $L$ of $H$, for any $\beta\in (0,3/4)$,
	$$|D^3g_H(x_H)-D^3g_L(x_L)|\leq C_\beta\Etilt_V\ell(L)^{\beta}.$$
	But now given any $H,L\in\mathcal{P}^j$, let $\tilde{H}$ and $\tilde{L}$ be the first ancestors of $H$ and $L$ with the same level and non-empty intersection. The above bound combined with (ii) then gives
	\begin{align*}
		|D^3g_H(x_H)-D^3g_L(x_L)| & \leq |D^3g_H(x_H)-D^3g_{\tilde{H}}(x_{\tilde{H}})|\\
		&\hspace{3em} + |D^3g_{\tilde{H}}(x_{\tilde{H}})-D^3g_{\tilde{L}}(x_{\tilde{L}})| + |D^3g_{\tilde{L}}(x_{\tilde{L}}) - D^3g_L(x_L)|\\
		& \leq C_\beta\Etilt_V\ell(\tilde{L})^{\beta}.
	\end{align*}
	But one can check from the cube construction that $|x_L-x_H|\geq c(n)\ell(\tilde{L})$, which therefore completes the proof of (iii) (one changes $\beta$ to $3/4-\beta$ to get the exact form of (iii)).
	
	To see (iv), the first bound follows simply from the height bound in the application of Theorem \ref{thm:eps-reg} which gave $h_H$ (see Lemma \ref{lemma:cm-5} and the subsequent discussion) and then changing coordinates to $g_H$ (see \cite[Lemma B.1]{DLS16a}). Moreover, also from the application of Theorem \ref{thm:eps-reg} in Lemma \ref{lemma:cm-5}, we know that
	$$\|Dh_H\|_{C^{0}(B_{6r_H}(p_H,P_H))}\leq C\Etilt_V\ell(H)^{11/12}$$
	and thus the second inequality in (iv) follows immediately, with $\pi_{P_x^{h_H}}$ in place of $\pi_{P_x^{g_H}}$. But then the graph of $g_H$ over $B_{4r_H}(x_H,P_0)$ is a subset of the graph of $h_H$ over $B_{5r_H}(p_H,P_H)$, and so the tangent planes coincide on this region; thus, the result with $\pi_{P_x^{g_H}}$ follows.
	
	Finally we prove (v). Fix $L\in\mathcal{S}^j$. For $i\geq j$, let $\mathcal{P}^i(L)$ denote the set of all cubes in $\mathcal{P}^i$ which intersect $L$. If $\tilde{L}$ is as in (v), then noting that if $H\in\mathcal{P}^i$ intersects $\tilde{L}$ it must be in $\mathcal{P}^i(L)$, we have by definition of $\wp_i$
	$$\|\wp_i-g_L\|_{L^1(\tilde{L})}\leq C(n)\sum_{H\in\mathcal{P}^i(L)}\|g_H-g_L\|_{L^1(B_{r_L}(p_L,P_0))}\leq C\Etilt_V^3\ell(L)^{n+15/4}$$
	where we have used the fact that the number of cubes in $\mathcal{P}^i(L)$ is bounded by some constant depending only on $n$.
\end{proof}

Finally, we can construct our center manifold from the $\wp_j$.

\begin{theorem}[Existence of Center Manifold]\label{thm:cm}
	Let $C^*$ be as in Lemma \ref{lemma:cm-1} and $C_e\geq C^*$, $C_h\geq C^*C_e$. Then, if $\eps_2 = \eps_2(\mathcal{V},M_0,N_0,C_e,C_h)\in (0,1)$ is sufficiently small, we have:
	\begin{enumerate}
		\item [(i)] For any $\beta\in (0,3/4)$, $\|\wp_j\|_{C^{3,\beta}([-4,4]^n)}\leq C_\beta\Etilt_V$;
		\item [(ii)] if $L\in\mathcal{S}^i$ and $H$ is a cube concentric to $L$ with $\ell(H) = \frac{9}{8}\ell(L)$, then $\wp_j = \wp_k$ on $H$ for any $j,k\geq i+2$;
		\item [(iii)] $\wp_j$ converges in $C^{3,\beta}$, for any $\beta\in (0,3/4)$, to a function $\wp:[-4,4]^n\to\R$ which is $C^{3,\gamma}$ for any $\gamma\in (0,3/4)$.
	\end{enumerate}
	Here, $C_\beta = C_\beta(\mathcal{V},M_0,N_0,C_e,C_h,\beta)\in (0,\infty)$.
\end{theorem}

The submanifold of $\R^{n+k}$ given by
$$\mathcal{M}:= \graph(\wp)$$
is called a \emph{center manifold of $V$ relative to $P_0$}. It is $C^{3,\gamma}$ for any $\gamma\in (0,3/4)$. If we write $\Phi:[-4,4]^n\to \R^{n+k}$ for the graphing function $\Phi(y) := (\wp(y),y)$, to each $L\in\S$ we associate a corresponding \emph{Whitney region} $\mathcal{L}$ on $\mathcal{M}$ via:
$$\mathcal{L}:= \Phi\left(\tilde{L}\cap \left[-\tfrac{7}{2},\tfrac{7}{2}\right]^n\right)$$
where $\tilde{L}$ is the cube concentric to $L$ with $\ell(\tilde{L}) = \frac{17}{16}\ell(L)$.

\begin{proof}[Proof of Theorem \ref{thm:cm}]
	For $H\in\mathcal{P}^j$, write $\chi_H := \vartheta_H/\left(\sum_{L\in\mathcal{P}^j}\vartheta_L\right)$. Clearly we have
	\begin{itemize}
		\item $\sum_{H\in\mathcal{P}^j}\chi_H = 1$ on $[-4,4]^n$;
		\item $\|\chi_H\|_{C^i}\leq C(n)\ell(H)^{-i}$ for each $i\geq 0$.
	\end{itemize}
	For such $H$, set $\mathcal{P}^j(H):=\{L\in\mathcal{P}^j:L\cap H\neq\emptyset\}\setminus\{H\}$ for the set of adjacent cubes to $H$ in $\mathcal{P}^j$. By construction of $\mathcal{P}^j$, we know $|\level(L)-\level(H)|\leq 1$ for all $L\in\mathcal{P}^j(H)$, and moreover $|\mathcal{P}^j(H)|\leq C(n)$. Since for each $y\in H$ we have
	$$\wp_j(y) = \sum_{L\in\mathcal{P}^j(H)\cup\{H\}}\chi_L(y)g_L(y)$$
	the bound $\|\wp_j\|_{C^0}\leq C\Etilt_V$ immediately follows from the above facts and Lemma \ref{lemma:cm-10}(i). Moreover, for such $y$ we also may write
	$$\wp_j(y) = g_H(y) + \sum_{L\in\mathcal{P}^j(H)}(g_L-g_H)\chi_L(y).$$
	Therefore, using Lemma \ref{lemma:cm-10}(i)--(ii), for $k\in \{1,2,3\}$,
	\begin{align*}
		\|D^k\wp_j\|_{C^0(H)} & \leq \|D^kg_H\|_{C^0(H)} + C(n)\sum^k_{i=0}\sum_{L\in\mathcal{P}^j(H)}\|g_L-g_H\|_{C^i(H)}\ell(L)^{-(k-i)}\\
		& \leq C\Etilt_V\left(1+\ell(H)^{15/4-k}\right)
	\end{align*}
	(here, we have used that $H\subset B_{r_L}(x_L)$ for $L\in\mathcal{P}^j(H)$, which is true as $M_0\geq 4$). We also obtain by direct computation, for any $\beta\in (0,3/4)$,
	\begin{align*}
		[D^3\wp_j&]_{C^{0,\beta}(H)} \leq [D^3g_H]_{C^{0,\beta}(H)}\\
		& \hspace{2em} + C(n)\sum^3_{i=0}\sum_{L\in\mathcal{P}^j(H)}\ell(H)^{-(3-i)}\left(\ell(H)^{-\beta}\|D^i(g_L-g_H)\|_{C^0(H)} + [D^i(g_L-g_H)]_{C^{0,\beta}(H)}\right)\\
		& \leq C_\beta\Etilt_V\ell(H)^{3/4-\beta}
	\end{align*}
	again using Lemma \ref{lemma:cm-10}(i)--(ii). This completes the proof of the bound on $\|\wp_j\|_{C^{3,\beta}(H)}$, which in particular shows that $\|\wp_j\|_{C^3}\leq C\Etilt_V$. We now need to extend the bound on the H\"older semi-norm of $D^3\wp_j$ across cubes. For this, fix $x,y\in [-4,4]^n$, and let $H,L\in\mathcal{P}^j$ be such that $x\in H$ and $y\in L$. If $H\cap L\neq\emptyset$, then clearly by taking a point in the intersection $H\cap L$ and using the triangle inequality, we get
	$$|D^3\wp_j(x)-D^3\wp_j(y)|\leq \left([D^3\wp_j]_{C^{0,\beta}(H)} + [D^3\wp_j]_{C^{0,\beta}(L)}\right)|x-y|^\beta \leq C_\beta\Etilt_V|x-y|^{\beta}.$$
	So we may now assume that $H\cap L=\emptyset$; without loss of generality we can assume $\ell(H)\leq\ell(L)$. Observe that
	$$\max\{|x-x_H|,|y-x_L|\}\leq \sqrt{n}\ell(L) \leq 2\sqrt{n}|x-y|.$$
	Moreover, by construction we know $\wp_j$ is identically equal to $g_H$ in a neighbourhood of its center $x_H$; in particular, $D^3\wp_j(x_H) = D^3g_H(x_H)$. Thus, we get
	\begin{align*}
		|D^3\wp_j(x) - &D^3\wp_j(y)|\\
		& \leq |D^3\wp_j(x)-D^3\wp_j(x_H)| + |D^3g_H(x_H)-D^3g_L(x_L)| + |D^3\wp_j(x_L)-D^3\wp_j(y)|\\
		& \leq C_\beta\Etilt_V(|x-x_H|^{\beta} + |x_H-x_L|^\beta + |y-x_L|^\beta)\\
		& \leq C\Etilt_V|x-y|^\beta
	\end{align*}
	where we have used the first bound for intersecting cubes and Lemma \ref{lemma:cm-10}(iii) in the second inequality. Thus, this establishes (i).
	
	Regarding (ii), let $L\in\mathcal{S}^i$. Observe that, by construction, $\mathcal{P}^j(L) = \mathcal{P}^{i+2}(L)$ for all $j\geq i+2$. Thus, if $H$ is the cube concentric with $L$ with $\ell(H) = \frac{9}{8}\ell(L)$, we have $\spt(\vartheta_M)\cap H=\emptyset$ for all $M\not\in\mathcal{P}^j(L)$ for $j\geq i+2$. So,
	$$\left.\wp_j\right|_H = \frac{\sum_{M\in\mathcal{P}^j(L)}\vartheta_Mg_M}{\sum_{M\in\mathcal{P}^j(L)}\vartheta_M} = \frac{\sum_{M\in\mathcal{P}^{i+2}(L)}\vartheta_Mg_M}{\sum_{M\in\mathcal{P}^{i+2}(L)}\vartheta_M} = \left.\wp_{i+2}\right|_{H}$$
	which proves (ii).
	
	Finally, to prove (iii) we will show that $\|\wp_j-\wp_{j+1}\|_{C^0([-4,4]^n)}\leq C2^{-j}$ for some choice of $C = C(\mathcal{V},M_0,N_0,C_e,C_h)$; this shows that $\wp_j$ converges uniformly to some function $\wp$ which is continuous. The bounds in (i), combined with Arzelà--Ascoli, show that in fact this convergence is in $C^{3,\beta}$ for all $\beta\in (0,3/4)$, and the limiting function $\wp$ is in $C^{3,\gamma}$ for all $\gamma\in (0,3/4)$, hence proving (iii).
	
	So fix $x\in [-4,4]^n$. We then have $x\in L\cap H$, where $L\in \mathcal{P}^j$ and $H\in\mathcal{P}^{j+1}$, and moreover such that either $H=L$ or $H$ is a child of $L$. Now, if $\level(L)\leq j-2$, we know from (ii) that $\wp_j(x) = \wp_{j+1}(x)$ and so there is nothing to prove. So we may assume $\level(L) \in\{j-1,j\}$. But then from (i) and Lemma \ref{lemma:cm-10}(iv) we have
	\begin{align*}
		|\wp_j(x)-\wp_{j+1}(x)| & \leq |\wp_j(x)-\wp_j(x_L)| + |g_L(x_L) - g_H(x_H)| + |\wp_{j+1}(x_H)-\wp_{j+1}(x)|\\
		& \leq C(\|D\wp_j\|_{C^0} + \|D\wp_{j+1}\|_{C^0})\cdot 2^{-j} + \|g_H-y_H\|_{C^0} + \|g_L-y_L\|_{C^0} + |y_H-y_L|\\
		& \leq C\Etilt_V\cdot 2^{-j} + C\Etilt_V\cdot 2^{-j}\\
		& \equiv C\Etilt_V\cdot 2^{-j}
	\end{align*}
	where we have bounded $|y_H-y_L|$ using Lemma \ref{lemma:cm-3}. Hence, taking the supremum over all $x$ we get $\|\wp_j-\wp_{j+1}\|_{C^0}\leq C\Etilt_V 2^{-j}$, as desired. This completes the proof.
\end{proof}

This completes the construction of the center manifold. We stress that we have built into the construction of the center manifold the decay structure of $V$ towards its tangent planes (which should be compared with the method used to control $\text{error}_2$ in the proof of Theorem \ref{thm:monotonicity}). This will be crucial for later estimates regarding frequency. Our next aim is to write $V$ as a graph over the center manifold, and use this decay structure of $V$ to deduce several estimates for this graph. We will always assume the assumptions of Theorem \ref{thm:cm} from now on.

\subsection{Construction of normal graph and its decay properties}

Let $\mathcal{M}$ be the center manifold given by Theorem \ref{thm:cm}. Since $\mathcal{M}$ is $C^2$, we have the standard differential-geometric fact that there exists a tubular neighbourhood $\mathbf{U}$ of $\mathcal{M}$ on which there exists an orthogonal projection map. More precisely, there exists an open set $\mathbf{U}\supset \mathcal{M}$ of $\R^{n+k}$ such that for each $x\in\mathbf{U}$ there is a unique point $y\in\mathcal{M}$ with $|x-y|<1$ and $x-y\perp T_y\mathcal{M}$ (of course $\mathcal{M}$ currently is only a subset of the cylinder $\R^k\times[-4,4]^n$ and we are only interested in the orthogonal projection restricted to this cylinder). We write $\mathbf{p}(x):=y$ for this uniquely determined point, and thus $\mathbf{p}:\mathbf{U}\to \mathcal{M}$ is the orthogonal projection map. Moreover, since $\mathcal{M}$ is in fact $C^{3,1/2}$ with its $C^{3,1/2}$ norm (i.e.~that of the graphing function $\wp$) controlled by $C\Etilt_V$, provided $\eps_2$ is sufficiently small we may assume that $\mathbf{p}$ is a $C^{2,1/2}$ map, and moreover that $\spt\|V\|\subset \mathbf{U}$.

Before constructing the graph over $\mathcal{M}$ representing $V$, we first give a lemma regarding the touching set of $V$ and $\mathcal{M}$. We show that $\mathcal{M}$ touches $V$ at every point of $\Phi(\Gamma)$, i.e.~the image of the set $\Gamma$ in our Whitney decomposition used to form $\mathcal{M}$, and moreover every flat singular point of $V$ with planar frequency $\geq 2$ lies in $\Phi(\Gamma)$. In particular, $\mathcal{M}$ touches $V$ at \emph{every} flat singular point of planar frequency $\geq 2$. Recall that as $V$ is represented by a $GC^{1,\alpha}$ function $u$ over $P_0$, we have $\mathcal{B}_u^{\geq 2} = \pi_{P_0}(\mathcal{B}_V^{\geq 2})$ (here $\mathcal{B}^{\geq 2}_u$ means the obvious thing, namely branch points of $u$ where the decay rate to the tangent plane is $\geq 2$).

\begin{lemma}\label{lemma:cm-12}
	Assume the assumptions of Theorem \ref{thm:cm}. Then, if $p\in \Phi(\Gamma)$, then $\spt\|V\|\cap \mathbf{p}^{-1}(p) = {p}$, $\Theta_V(p) = Q$, and $T_p\mathcal{M}$ is the unique tangent plane to $\spt\|V\|$ at $p$. Moreover,
	$$\mathcal{B}^{\geq 2}_u\cap [-4,4]^n\subset\Gamma$$
	or equivalently, $\mathcal{B}_V^{\geq 2}\cap (\R^k\times [-4,4]^n)\subset\Phi(\Gamma)$. Furthermore, we even have the more quantitative conclusion that if $L\in\S$ then $\ell(L)<\frac{1}{64\sqrt{n}}\dist(\mathcal{B}_u^{\geq 2},L)$.
\end{lemma}

\begin{proof}
	To prove the first claims, fix $p\in\Phi(\Gamma)$ and let $\tilde{p}\in \Gamma$ be such that $\Phi(\tilde{p}) = p$. By construction, there is an infinite chain $L_{N_0}\supset L_{N_0+1}\supset\cdots$ of cubes with $L_j\in \mathcal{R}^j$ for all $j$ such that $\{\tilde{p}\} = \bigcap_j L_j$. From Lemma \ref{lemma:cm-3}(i) we know that $(P_{L_j})_j$ converge to a plane $P$ with $\|\pi_{P_{L_j}}-\pi_P\|\leq C\Etilt_V (2^{11/12})^{-j}$. We also know by construction that $|p-p_{L_j}|\leq 2\sqrt{n}2^{-j}$. Combined with Lemma \ref{lemma:cm-3}(iv), we then have
    \begin{equation}\label{E:cm-12-height-decay}
    \mathbf{h}_V(C_{16r_{L_j}}(p,P)) \leq C\Etilt_V 2^{-j}.
    \end{equation}
    In particular, by Theorem \ref{thm:eps-reg} we know $V\res C_{2^{3-N_0}}(p,P)$ is the graph of a $Q$-valued $GC^{1,\alpha}$ function over $P$. Hence by \eqref{E:cm-12-height-decay}, $\spt\|V\|\cap \mathbf{p}^{-1}(p) = \{p\}$ and $\Theta_V(p) = Q$. Now, from the failure of (EX), we know that $(\eta_{p_{L_j},2^{-j}})_\# V$ converges to $Q|P|$; hence $(\eta_{p,2^{-j}})_\#V = (\eta_{(p-p_{L_j})/2^{-j},1}\circ\eta_{p_{L_j},2^{-j}})_\#V \weakly Q|P|$ so that $P$ is the approximate tangent space to $\spt\|V\|$ at $p$. By $\|\pi_{L_j}-\pi_P\|\leq C\Etilt_V (2^{11/12})^{-j}$ and Lemma \ref{lemma:cm-10}(iv) (cf.~Theorem \ref{thm:cm}), we have $T_p\mathcal{M}=P$. This establishes the first claims.
	
	Now let us turn to the second claims regarding flat singular points with planar frequency $\geq 2$. We know in general from Corollary \ref{cor:decay-estimates} and Theorem \ref{thm:monotonicity} that at any $x_0\in\mathcal{B}_u^{\geq 2}$, if $P_{q_0}$ is the corresponding tangent plane to $V$ at $q_0 = (y_0,x_0)\in\spt\|V\|$ (note that there is exactly one such $y_0$ as this point has density $Q$), then we have for $r\in (0,\sqrt{n})$,
	$$\Etilt_V(B_r(q_0),P_{q_0})\leq C(\mathcal{V})\Etilt_Vr,$$
	$$\mathbf{h}_V(B_r(q_0),P_{q_0})\leq C(\mathcal{V})\Etilt_V r^2.$$
	We also know from the Lipschitz regularity and the bound provided by Theorem \ref{thm:eps-reg} that for any $(y,x)\in \spt\|V\|$ we have $|y-y_0|\leq C(\mathcal{V})\Etilt_V|x-x_0|$ and also $\|\pi_{P_q}-\pi_{P_0}\|\leq C(\mathcal{V})\Etilt_V$ (this claim involving the Lipschitz regularity is using that $(y_0,x_0)$ is a $Q$-coincidence point of $V$).
	
	Thus, now suppose that $L\in\mathcal{C}^j$, $j\geq N_0$, is such that there is $q_0 = (y_0,x_0)$ with $q_0\in\mathcal{B}_V^{\geq 2}$ and $|x_0-x_L|\leq 128\sqrt{n}\ell(L)$. Then, by the above inequalities we know that $|y_L-y_0|\leq C(\mathcal{V})\Etilt_V|x_L-x_0| \leq C(\mathcal{V})\Etilt_V\ell(L)$, and thus $\mathbf{B}_L\subset B^{n+k}_{C(\mathcal{V},M_0)\ell(L)}(q_0)$. Thus,
	$$\Etilt_V(\mathbf{B}_L)\leq C(\mathcal{V},M_0)^n\Etilt_V(B^{n+k}_{C(\mathcal{V},M_0)\ell(L)}(q_0)) \leq \tilde{C}(\mathcal{V},M_0)\Etilt_V\ell(L).$$
	We get from this also that (arguing similarly to that in Lemma \ref{lemma:cm-3})
	$$\|\pi_{P_L}-\pi_{P_{q_0}}\|\leq C\left(\Etilt_V(\mathbf{B}_L) + \Etilt_V(B^{n+k}_{C(\mathcal{V},M_0)\ell(L)}(q_0),P_{q_0})\right) \leq C\Etilt_V\ell(L)$$
	from which it follows that
	$$\mathbf{h}_V(\mathbf{B}_L)\leq \mathbf{h}_V(B_{C(\mathcal{V},M_0)\ell(L)}(q_0),P_{q_0}) + C\ell(L)\|\pi_{P_L}-\pi_{P_{q_0}}\| \leq C\Etilt_V^2\ell(L)^2.$$
	We stress here that the constants above only depend on $\mathcal{V},M_0$. In particular, this shows that whenever $L\in \mathcal{C}^j$, $j\geq N_0$, is such that there exists $q_0 = (y_0,x_0)$ a point of planar frequency $\geq 2$ in $V$ with $|x_0-x_L|\leq 128\sqrt{n}\ell(L)$, then $L\not\in \mathcal{S}^j_e\cup\S^j_h$ (provided $C_e,C_h$ are sufficiently large depending only on $\mathcal{V},M_0$).
	
	We now use this to prove the lemma. Suppose for contradiction that there is $q_0 = (y_0,x_0)$ with $x_0\in \mathcal{B}_u^{\geq 2}\cap [-4,4]^n$ yet $x_0\not\in\Gamma$. Then, by definition of $\Gamma$, there must be $j\geq N_0$ and $L\in \S^j$ with $|x_L-x_0|\leq \sqrt{n}\ell(L)$. By the above, we must have $L\in \mathcal{S}^j_n$. But then there must be a cube $\tilde{L}\in \mathcal{S}^k_e\cup\mathcal{S}^k_h$, where $k\in \{N_0,N_0+1,\dotsc,j-1\}$, and cubes $L\equiv L^j, L^{j-1},\dotsc,L^k\equiv\tilde{L}$, with $L^t\in \mathcal{S}^t_n$ for each $j\leq t\leq k+1$, and $L^t\cap L^{t-1}\neq\emptyset$. But then
	$$|x_{\tilde{L}}-x_0|\leq 2\sqrt{n}\sum^j_{i=k}2^{-i} \leq 4\sqrt{n}\ell(\tilde{L})$$
	and so $\tilde{L}$ satisfies the above assumptions, which gives the desired contradiction.
	
	Finally, to prove the last claim of the lemma, suppose for contradiction that there is $L\in \mathcal{S}^j$ with $\ell(L)>\frac{1}{64\sqrt{n}}\dist(x_0,L)$, for some such $q_0 = (y_0,x_0)\in\mathcal{B}_V^{\geq 2}$. Then
	$$|x_L-x_0|\leq \dist(x_0,L)+\sqrt{n}\ell(L)\leq 65\sqrt{n}\ell(L).$$
	But then we can argue just as above to get a contradiction.
\end{proof}

\begin{remark}\label{remark:lower-frequency-in-cm}
Due to our choice of exponent in (EX) the above proof shows that points of planar frequency $\geq 1+(1-\frac{1}{12})$ also lie in $\Phi(\Gamma)$. By changing the exponent in (EX) we can guarantee points of planar frequency $\in [5/3+\eta,\infty)$ belong to $\Gamma$ for any choice of $\eta\in (0,1/3]$.
\end{remark}

Now we begin giving the properties of our normal graph. Clearly the above lemma tells us that on $\Phi(\Gamma)$, the normal graph will vanish, as $\mathcal{M}$ touches $V$ on this set. We now give some properties about the rest.

\begin{lemma}\label{lemma:cm-13}
	Assume the assumptions of Theorem \ref{thm:cm}. Then we have:
	\begin{enumerate}
		\item [(i)] $\sum_{x\in\mathbf{p}^{-1}(y)}\Theta_V(x) = Q$ for all $y\in\mathcal{M}$;
		\item [(ii)] if $L\in\mathcal{S}$, then for each $x\in L$ we have
		$$\spt\|V\|\cap \mathbf{p}^{-1}(\Phi(x))\subset B^{n+k}_{C\Etilt_V\ell(L)}(\Phi(x)),$$
		where $C = C(\mathcal{V},M_0,N_0,C_e,C_h)$.
	\end{enumerate}
\end{lemma}

\begin{proof}
	Claim (i) is a simple consequence of the fact that $V$ is a $Q$-valued Lipschitz graph with Lipschitz constant controlled by $C\Etilt_V$, as well as the fact $\|\wp\|_{C^3}\leq C\Etilt_V$.
	
	To see (ii), fix $L\in\mathcal{S}$ and $x\in L$. Let $J$ be the cube concentric with $L$ with $\ell(J) = \frac{17}{16}\ell(L)$. By definition of $\wp$, Theorem \ref{thm:cm}(ii), and Lemma \ref{lemma:cm-10}(ii), we have
	$$\|\wp-g_L\|_{C^0(J)}\leq \sum_{H\in\S:H\cap L\neq\emptyset}\|g_L-g_H\|_{C^0(J)}\leq C\Etilt_V\ell(L)^{15/4}.$$
	But we know that the graph of $g_L$ coincides with that of $h_L$ and moreover that $\|u_L\|_{C^0}\leq \tilde{C}\Etilt_V\ell(L)$. So, if $p_L = (z_L,w_L)\in P_L\times P_L^\perp$, since $p_L \in \spt\|V\|$ we immediately get $\|(u_L)_{\text{a}}-w_L\|_{C^0} \leq C\Etilt_V\ell(L)$. Standard estimates with convolutions then give $\|h_L-w_L\|_{C^0}\leq C\Etilt_V\ell(L)$. These in turn imply that for any point of $\spt\|V\|$ in $C_{15r_L}(p_L,P_L)$, we have that the distance to the graph of $h_L$ is at most $C\Etilt_V\ell(L)$, which proves the claim because $\spt\|V\|\cap \mathbf{p}^{-1}(\Phi(x))\subset C_{15r_L}(p_L,P_L)$ simply because $V$ is represented exactly by $u_L$ here.
\end{proof}

Notice that Lemma \ref{lemma:cm-13}(i) already tells us that $V$ must be a $Q$-valued graph over $\mathcal{M}$, which moreover must be $GC^{1,\alpha}$ due to Theorem \ref{thm:eps-reg} and the $C^3$ regularity of $\mathcal{M}$. Moreover, Lemma \ref{lemma:cm-13}(ii) gives us an a priori bound on the $C^0$-norm of this function. The key property, however, of our construction is that on each cube we can control the average of the normal graph, and moreover the average of the normal graph is \emph{higher order} (namely, it is higher order than quadratic) when compared to the overall graph. This is made precise in the following theorem and is one of the crucial aspects of building the center manifold the way we did (as this will be needed to control certain error terms later on; the other crucial aspects are the doubling-style conditions coming from Lemma \ref{lemma:cm-15} and Lemma \ref{lemma:cm-17}).

\begin{theorem}\label{thm:cm-14}
	Assume the assumptions of Theorem \ref{thm:cm}. Then if $\eps_2 = \eps_2(\mathcal{V},M_0,N_0,C_e,C_h)\in (0,1)$ is sufficiently small, there is a $GC^{1,\alpha}$ function $F:\mathcal{M}\to \A_Q(\mathbf{U})$ taking the form $F(x) = \sum_{i=1}^Q \llbracket x+N_i(x)\rrbracket$, where $N \equiv \sum^Q_{i=1}\llbracket N_i\rrbracket:\mathcal{M}\to \A_Q(\R^{n+k})$ is a $GC^{1,\alpha}$ function with $N_i(x)\perp T_x\mathcal{M}$ for all $x\in\mathcal{M}$, such that $V\res \mathbf{U}$ is the integral varifold associated to the image of $F$. Moreover, on every Whitney region $\mathcal{L}$ (associated to a cube $L\in\S$) we have:
	\begin{enumerate}
		\item [(i)] $\|N|_{\mathcal{L}}\|_{C^0}\leq C\Etilt_V\ell(L)$;
		\item [(ii)] $\|DN|_{\mathcal{L}}\|_{GC^{0,\alpha}}\leq C\Etilt_V\ell(L)^{1-\frac{1}{12}}$;
		\item [(iii)] $\int_{\mathcal{L}}|N_{\textnormal{a}}|\leq C\Etilt_V^3\ell(L)^{n+2(1-\frac{1}{12})+1}$.
	\end{enumerate}
    Here, $C = C(\mathcal{V},M_0,N_0,C_e,C_h)$, and moreover $\alpha = \alpha(\mathcal{V})\in (0,1)$ is as in Theorem \ref{thm:eps-reg}.
\end{theorem}

\textbf{Definition:} We call the map $N:\mathcal{M}\to \A_Q(\R^{n+k})$ given by Theorem \ref{thm:cm-14} the \emph{normal map}; it represents $V$ as a $Q$-valued graph over $\mathcal{M}$.

\begin{proof}
	We already know from Lemma \ref{lemma:cm-13} and Theorem \ref{thm:eps-reg} that a $GC^{1,\alpha}$ map $F$ of the given form exists whose associated varifold coincides with $V\res\mathbf{U}$. We only need to justify the other claims. By Lemma \ref{lemma:cm-13}(ii) the $C^0$ bound on $N$ follows, proving (i). For (ii), this follows also from the corresponding bound on $u_L$, and the fact that by Lemma \ref{lemma:cm-10}(iv) and the $C^2$ estimate on $\wp$ given by Theorem \ref{thm:cm}, $\mathcal{M}\cap C_{r_L}(p_L,P_L)$ is the graph of a function $\wp_L:B_{r_L}(p_L,P_L)\to P_L^\perp$ obeying $|D\wp_L|\leq C\Etilt_V\ell(L)^{11/12}$ at the point $(\wp(x_L),x_L)$ and so $\|D\wp_L\|_{C^0}\leq C\Etilt_V\ell(L)^{11/12}$, as $\|D^2\wp_L\|_{C^0}\leq C\Etilt_V$ (we refer the reader to \cite[Theorem 5.1]{DLS13} for the precise details on reparameterising multi-valued graphs). This shows (ii).
	
	For (iii), let $p\in\mathcal{L}$. We then have by reparameterising (see again \cite[Theorem 5.1]{DLS13}),
	\begin{align}
	|N_{\text{a}}(p)| & \leq C|(u_L)_{\text{a}}(\pi_{P_L}(p))-\pi^\perp_{P_L}(p)| + C\text{Lip}(N|_{\mathcal{L}})\|\pi_{T_p\mathcal{M}}-\pi_{P_L}\|\cdot|N(p)|\nonumber\\
	& \leq C|(u_L)_{\text{a}}(\pi_{P_L}(p))-\pi^\perp_{P_L}(p)| + C\Etilt_V^3\ell(L)^{17/6}\label{E:cm-14-1}
	\end{align}
	using the bounds from (i) and (ii). Now, if we let $\wp_L:P_L\to P_L^\perp$ be such that $\graph(\wp_L) = \mathcal{M}$, we see that
	\begin{align*}
		\int_{\mathcal{L}}|(u_L)_{\text{a}}(\pi_{P_L}(p))-\pi^\perp_{P_L}(p)| & \leq C\int_{\pi_{P_L}(\mathcal{L})}|(u_L)_{\text{a}}-\wp_L|\\
		& \leq C\|(u_L)_{\text{a}}-h_L\|_{L^1(\pi_{P_L}(\mathcal{L}))} + C\|g_L-\wp\|_{L^1(\tilde{L})}\\
		& \leq C\Etilt_V^3\ell(L)^{n+15/4}
	\end{align*}
	where in this last line we have used Lemma \ref{lemma:cm-7} and Lemma \ref{lemma:cm-10}(v) (here, $\tilde{L}$ is the cube concentric with $L$ and $\ell(\tilde{L}) = \frac{9}{8}\ell(L)$). So, integrating \eqref{E:cm-14-1} gives the result.
\end{proof}

This concludes the ``new'' estimates we will need to use the center manifold in the first place, namely the smallness of the average in a very controlled way. Next, we need to establish more ``standard'' estimates, which mirror those seen in Section \ref{sec:pff} when dealing with the error terms from the planar frequency function. In Section \ref{sec:pff}, we found regions where certain decay conditions stopped and used an appropriate Vitali cover and a doubling condition to control the error term $\text{error}^*_2$ therein. Here, the regions where the decay stops \emph{have already been built into our center manifold construction}: these are the cubes in $\S$. Since the center manifold was not constructed ``randomly'', but is built from $u_{\text{a}}$, which itself is close to a harmonic function, we can therefore establish the equivalent estimates and doubling conditions for the normal map $N$. The argument is somewhat more technical, as $\mathcal{M}$ is not flat and moreover there are three types of cubes to consider, but the reader should note the similarity in this part of the argument to that used for planar frequency.

Let us start with the simplest case, namely the height cubes. The main goal is to control in every height cube the size of the average-free part of $N$ from below by the length of the respective cube; this is a stronger form of a doubling-type condition as it is an $L^\infty$ statement rather than an $L^2$ statement. This is achieved in (ii) of the following lemma:

\begin{lemma}[Height Cube Control]\label{lemma:cm-15}
	There is a constant $\bar{C}^* = \bar{C}^*(M_0)\in (0,\infty)$ such that the following holds. Suppose that $C_h\geq \bar{C}^*C_e$. Then, if $\eps_2 = \eps_2(\mathcal{V},M_0,N_0,C_e,C_h)\in (0,1)$ is sufficiently small, then the following hold for every $L\in \S_h$:
	\begin{enumerate}
		\item [(i)] $\Theta_V(p)\leq Q-1$ for every $p\in B_{16r_L}^{n+k}(p_L)$;
		\item [(ii)] $|N_f(x)|\geq \frac{1}{4}C_h\Etilt_V\ell(L)$ for every $x\in \Phi(B_{2\sqrt{n}\ell(L)}^n(x_L,P_0))$;					\item [(iii)] $L\cap H=\emptyset$ for every $H\in\mathcal{S}_n$ with $\ell(H)\leq\frac{1}{2}\ell(L)$.
	\end{enumerate}
\end{lemma}

\begin{proof}
	(i) is a simple consequence of the $GC^{1,\alpha}$ regularity of $V$: if we had a point of density $Q$ in $B_{16r_L}^n(p_L)$, then we would get that\footnote{Indeed, one can consider the father $\widetilde{L}$ of $L$ to see this, as $V\res C_{48r_{\widetilde{L}}}(p_{\widetilde{L}})$ is the graph of a $GC^{1,\alpha}$ function with Lipschitz constant $\leq C(\mathcal{V})\Etilt_V\ell(L)$ and that $\mathbf{B}_L\subset C_{48r_{\widetilde{L}}}(p_{\widetilde{L}})$.} the separation of $V$ in $\mathbf{B}_L$ is $\leq C(\mathcal{V})\Etilt_V\ell(L)$, and so (HT) would not fail for $L$, i.e.~$L\not\in \mathcal{S}_h$. Thus, there are no points of density $Q$ in $B^n_{16r_L}(p_L)$ and thus as $\Theta_V\in \{1,2,\dotsc,Q\}$ (again from the graph structure) we see it must be $\leq Q-1$ everywhere.
	
	Claim (ii) also follows from the $GC^{1,\alpha}$ regularity of $V$. Since by definition of $\mathcal{S}_h$ there are two points in $\mathbf{B}_L$ which have separation $\geq C_h\Etilt_V\ell(L)$, and since $V$ is a multi-valued graph with Lipschitz constant $\leq C(\mathcal{V})C_e\Etilt_V \ell(L)^{11/12}$, on a region of radius proportional to $\ell(L)$ each value of $V$ can only change by an amount at most $C(\mathcal{V})C_h\Etilt_V\ell(L)^{1+11/12}$ (as $C_e\leq C_h$). Thus, if $N_0 = N_0(\mathcal{V})$ is sufficiently large, this is $<\frac{1}{16}C_h\Etilt_V\ell(L)$, which combined with the other estimates on $\mathcal{M}$ proves (ii). (Alternatively, one can ensure this using $C_h\geq C^*C_e$ and choosing $C^*$ sufficiently large.)
	
	Finally, (iii) follows in essentially the same way as (ii), by showing that if (HT) holds on $L$, then on every $H\in\mathcal{C}^j$ with $j\geq\level(L)+1$ and $H\cap L \neq\emptyset$, then (HT) still holds on $H$. Thus, it can never be the case that $H\in\mathcal{S}_n$. Indeed, suppose there is $H\in\mathcal{S}$ with $\level(H)\geq\level(L)+1$ and $H\cap L\neq\emptyset$. From our construction of $\S$ it follows that $\level(H) = \level(L)+1$ and so $\ell(H) = \frac{1}{2}\ell(L)$. Moreover, from Lemma \ref{lemma:cm-3} we know $\|\pi_{P_H}-\pi_{P_L}\|\leq C_1\Etilt_V\ell(H)^{11/12}$. Arguing as in (ii), since (HT) holds on $L$, we see that for $\eps_2$ sufficiently small,
$$|\pi_{P_H^\perp}(x-y)|\geq\frac{3}{4}C_h\Etilt_V\ell(L) = \frac{3}{2}C_h\Etilt_V\ell(H)$$
for \emph{some} $x,y\in\spt\|V\|\cap C_{32r_L}(p_L,P_L)$. But then this gives that $\mathbf{B}_H$ must contain two points $x,y$ with $|\pi_{P_H^\perp}(x-y)|\geq\frac{3}{2}C_h\Etilt_V\ell(H)$, showing that $H$ satisfies (HT) and so $H\not\in \S_n$. This completes the proof.
\end{proof}

The next case we analyse is that of the neighbouring cubes, i.e.~$\mathcal{S}_n$. Notice that Lemma \ref{lemma:cm-15}(iii) tells us that if $H\in\mathcal{S}_n$ then $H$ must in fact spawn from a cube in $\mathcal{S}_e$, as it cannot be adjacent to any larger cube in $\mathcal{S}_h$. From this, we can deduce the following fact, morally saying that any cube in $\mathcal{S}_n$ is ``covered'' by a cube in $\mathcal{S}_e$. Indeed, Lemma \ref{lemma:cm-15}(iii) gives that if $H\in\mathcal{S}_n$, then we can find a chain $L_0,L_1,\dotsc,L_j\equiv H$ such that $L_0\in\mathcal{S}_e$, $L_i\in \S_n$ for all $i\geq 1$, and $L_i\cap L_{i-1}\neq\emptyset$ with $\level(L_i) = \level(L_{i-1})+1$ for all $i\geq 1$. In particular, for such a chain we have
$$H\subset B_{3\sqrt{n}\ell(L_0)}(x_{L_0},P_0).$$
Now fix an ordering on the (countable) set $\S_e$ such that the level is increasing along the ordering. For $H\in\S_n$, we then choose $L_0\in \S_e$ \emph{minimal} in this ordering for which there is a chain as just described starting at $L_0$ and ending at $H$. We call $L_0$ the \emph{spawner} of $H$. We may therefore partition $\S_n$ into disjoint sets $\S_n(J)$, for $J\in \S_e$, where $\S_n(J)$ is the set of all $H\in \S_n$ whose spawner is $J$. We call $\S_n(J)$ the \emph{spawners of $J$}.

The following lemma regarding the spawner of $H\in\S_n$ is a simple consequence of the analysis so far. Here, for $q\in \mathcal{M}$ we write $\mathcal{B}_r(q)$ for the geodesic ball in $\mathcal{M}$ of radius $r$ centred at $q$. The reason we need such a lemma is because currently we do not know any decay property for the cubes in $\S_n$, and so we will use the spawner to control them, and so we need to know the relation between $H$ and its spawner.

\begin{lemma}[Spawners of Neighbouring Cubes $\S_n$]\label{lemma:cm-16}
	If $\eps_2 = \eps_2(\mathcal{V},M_0,N_0,C_e,C_h)\in (0,1)$ is sufficiently small, then the following holds. Fix $r>0$ and suppose $H\in \S_n$ intersects $B:=\pi_{P_0}(\mathcal{B}_r(0))$. Then, if $J\in\S_e$ is the spawner of $H$, we have
	$$\ell(J)\leq\frac{3r}{64\sqrt{n}} \qquad \text{and}\qquad \textnormal{dist}(J,B),\, \textnormal{dist}(H,J)\leq 3\sqrt{n}\ell(J).$$
\end{lemma}
\textbf{Notation:} Here, recall $\dist(A,B):= \inf\{|a-b|:a\in A, b\in B\}$.

\begin{proof}
	We know by definition of the spawner that $\dist(H,J) \leq 2\sqrt{n}\ell(J)$, and thus
	$$\dist(J,B)\leq \dist(H,J) + 2\sqrt{n}\ell(H) \leq 3\sqrt{n}\ell(J)$$
	which proves the second claim. If $r\geq 1/4$, the first claim is trivial given that $\ell(J)\leq 2^{-N_0-6}$ for all cubes in $\S$ (cf.~Lemma \ref{lemma:cm-1}) and moreover that $N_0$ is assumed to obey $\sqrt{n}2^{7-N_0}\leq 1$. Assume therefore that $r\leq 1/4$. Since $\|D\wp\|_{C^{2,1/2}}\leq C\eps_2$ from Theorem \ref{thm:cm}, if $\eps_2$ is sufficiently small we know that $B_{r/2}^n(0)\subset B\subset B_r(0)$ for any $r$. But as $H$ intersects $B\subset B_r(0)$, this tells us (with the fact $\dist(H,J)\leq 2\sqrt{n}\ell(J)$) that $J$ intersects $B_{r+3\sqrt{n}\ell(J)}(0)$, and so $\dist(0,J)\leq r+3\sqrt{n}\ell(J)$. But then from Lemma \ref{lemma:cm-12}, we have (as $0\in \mathcal{B}^{\geq 2}_V$)
	$$\ell(J) \leq \frac{1}{64\sqrt{n}}\dist(0,L) \leq \frac{r+3\sqrt{n}\ell(J)}{64\sqrt{n}}$$
	which if we rearrange for $\ell(J)$ gives $\ell(J)\leq \frac{r}{61\sqrt{n}}\leq \frac{3r}{64\sqrt{n}}$, which is the first inequality.
\end{proof}

Finally, we give the key doubling condition for the excess cubes $\S_e$. This is most similar to that seen in the argument for planar frequency seen in Section \ref{sec:pff}, namely when we used Lemma \ref{lemma:doubling-2}. Just like for planar frequency, we need to move the center-point inward by an amount comparable to the radius $\ell(L)$, so that we can introduce the cut-off function $\phi$ from Section \ref{sec:pff} when controlling error terms by those involved in the definition of the frequency function.

\begin{lemma}[Excess Cube Doubling Condition]\label{lemma:cm-17}
	There exist constants $\bar{M}_0 = \bar{M}_0(n,k)$ and $\bar{C}_e = \bar{C}_e(\mathcal{V},M_0)$ such that if $M_0\geq \bar{M}_0$ and $C_e\geq \bar{C}_e$, and if $\eps_2 = \eps_2(\mathcal{V},M_0,N_0,C_e,C_h)$ is sufficiently small, then the following holds. If $L\in \S_e$ and $q\in P_0$ obeys $\dist(L,q)\leq 6\sqrt{n}\ell(L)$, then
	$$C\Etilt_V^2\ell(L)^{n+2(1-\frac{1}{12})}\leq \int_{\Phi(B_{\ell(L)/4}(q,P_0))}|DN|^2.$$
	Here, $C = C(\mathcal{V},M_0,N_0,C_e,C_h)\in (0,\infty)$.
\end{lemma}

\begin{proof}
	The proof is a modification of that seen in Lemma \ref{lemma:doubling-2}, where now we need to prove the doubling condition after ``subtracting the center manifold''. As the center manifold is close to the average $(u_L)_{\text{a}}$, and the average is close to being harmonic, one expects to argue in a similar fashion.
	
	Throughout, we fix a constant $\lambda = \lambda(n)\in (0,1)$ obeying\footnote{Here, the exponent $1/12 \equiv 1 - (1-\frac{1}{12})$ comes from the choice of exponent in (EX). This proof illustrates why the exponent in (EX) cannot be taken to be $1$.}
 $(1+\lambda)^{n+2}<2^{1/12}$.
		
	\textbf{Step 1: Linear setting.} We first start with a variant of Lemma \ref{lemma:doubling-1}. Fix $\delta>0$. We claim that there is $\beta = \beta(\mathcal{V},\delta)\in (0,\infty)$ with the following property: if $v\in \mathfrak{B}_{\mathcal{V}}$ obeys (for $r<1/2$)
	$$\int_{B_{(1+\lambda) r}(0)}|Dv-Dv_{\text{a}}(0)|^2 \geq 2^{1/12}\cdot2^{-n-2}\int_{B_{2r}(0)}|Dv|^2$$
	then, for all $B_s(q)\subset B_{2r}(0)$ with $s\geq\delta r$ we have
	$$\int_{B_s(q)}|Dv_{f}|^2\geq\beta^2\int_{B_{(1+\lambda)r}(0)}|Dv|^2.$$
	To see this we argue by contradiction. If it were false, after rescaling to assume $2r=1$ and $\int_{B_1(0)}|Dv|^2=1$ (which we can do in a similar manner to that seen in Lemma \ref{lemma:doubling-1}) we can find a sequence $v_j\in \mathfrak{B}_{\mathcal{V}}$ with $\int_{B_1(0)}|Dv_j|^2=1$ for all $k$ and
	$$\int_{B_{(1+\lambda)/2}(0)}|Dv_j-D(v_j)_{\text{a}}(0)|^2\geq 2^{1/12}\cdot 2^{-n-2}$$
	yet
	$$\int_{B_{s_j}(q_j)}|D(v_j)_f|^2<\frac{1}{j}\int_{B_{(1+\lambda)/2}(0)}|Dv_j|^2 \leq \frac{1}{j}$$
	for some $B_{s_j}(q_j)\subset B_{1}(0)$ with $s_j\geq\delta/2$. Passing to a subsequence so that $s_j\to s\geq\delta/2$, $q_j\to q$, with $B_{s/2}(q)\subset \overline{B}_{1-s/4}(0)$, and $v_j\to v\in\mathfrak{B}_{\mathcal{V}}$ strongly in $W^{1,2}_{\text{loc}}(B_1(0))$, we get
	$$ \int_{B_{s/2}(q)}|Dv_f|^2 = 0.$$
	Hence, as a result of the monotonicity from Theorem \ref{thm:blow-up}, we see that $|Dv_f|\equiv 0$ on $B_1(0)$; in particular $\int_{B_{(1+\lambda)/2}(0)}|D(v_j)_f|^2\to 0$.
	
	We then have, as $(v_j)_{\text{a}}$ is harmonic for each $k$, from the mean-value property of harmonic functions,
	$$Q\int_{B_{(1+\lambda)/2}(0)}|D(v_j)_{\text{a}}-D(v_j)_{\text{a}}(0)|^2 = \int_{B_{(1+\lambda)/2}(0)}|Dv_j - D(v_j)_{\text{a}}(0)|^2 - |D(v_j)_f|^2$$
	and thus
	$$Q\int_{B_{(1+\lambda)/2}(0)}|D(v_j)_{\text{a}} - D(v_j)_{\text{a}}(0)|^2 \geq 2^{1/12}\cdot 2^{-n-2}-\int_{B_{(1+\lambda)/2}(0)}|D(v_j)_f|^2.$$
	Taking $j\to\infty$, we therefore get (noting also that $1 = \int_{B_1(0)}|Dv_j|^2 = \int Q|D(v_j)_{\text{a}}|^2 + |D(v_j)_f|^2$)
	$$\int_{B_{(1+\lambda)/2}(0)}|Dv_{\text{a}}-Dv_{\text{a}}(0)|^2\geq Q^{-1}\cdot 2^{1/12}\cdot 2^{-n-2}\geq 2^{1/12}\cdot2^{-n-2}\int_{B_1(0)}|Dv_{\text{a}}|^2.$$
	But as $v_{\text{a}}$ is harmonic, it should satisfy the decay estimate
	$$\int_{B_{(1+\lambda)/2}(0)}|Dv_{\text{a}}-Dv_{\text{a}}(0)|^2 \leq (1+\lambda)^{n+2}\cdot 2^{-n-2}\int_{B_1(0)}|Dv_{\text{a}}|^2$$
	which by choice of $\lambda$ therefore gives the desired contradiction.
	
	\textbf{Step 2: Pass linear estimate to non-linear setting.} Now we want to use Step 1 to prove a form of Lemma \ref{lemma:doubling-2} for the present setting, which is the claim of the lemma.
	
	First pick $j$ with $L\in \S_e^j$. Then let $H\in \mathcal{R}^{j-1}$ and $J\in \mathcal{R}^{j-6}$ be ancestors of $L$. From Lemma \ref{lemma:cm-5} we know that there is a function $u_{HJ}:B_{16r_J}(p_J,P_H)\to \A_Q(P_H^\perp)$ representing $V$ in the cylinder $C_{16r_J}(p_J,P_H)$. We also know that $\mathbf{B}_L\subset\mathbf{B}_H\subset C_{16r_J}(p_J,P_H)$ (c.f. \eqref{lemma:cm-2}). By Lemma \ref{lemma:cm-3} we also have
	$$\Etilt_V(C_{32r_J}(p_J,P_H))\leq C\Etilt_V\ell(L)^{11/12}\ \ \ \text{and}\ \ \ \mathbf{h}_V(C_{32r_J}(p_J,P_H))\leq C\Etilt_V\ell(L)$$
	where $C = C(\mathcal{V},M_0,N_0,C_e,C_h)$. Moreover, as $L\in \mathcal{S}_e$, $\mathbf{B}_L\subset C_{16r_J}(p_J,P_H)$, and $r_L/r_J = 2^{-6}$, we know that
	$$C_e\Etilt_V \ell(L)^{11/12}\leq \Etilt_V(\mathbf{B}_L) \leq  c(n)\Etilt_V(C_{32r_J}(p_J,P_H)).$$
	Now write $\rho:= 32r_H - \tilde{C}\Etilt_V\ell(L)$, for suitably chosen $\tilde{C} = \tilde{C}(\mathcal{V},M_0,N_0,C_e,C_h)$. From the above bound on $\mathbf{h}_V(C_{32r_J}(p_J,P_H))$, since $p_H\in\spt\|V\|$ it evidently follows that (for suitable $\tilde{C}$)
	$$\spt\|V\|\cap C_{2\rho}(p_H,P_H)\subset\mathbf{B}_H.$$
	Thus, expressing the tilt-excess in terms of $u_{HJ}$, we get
	\begin{align*}
    (2\rho)^{-n}\int_{B_{2\rho}(\pi_{P_H}(p_H),P_H)}|Du_{HJ}|^2 & \leq (1+C\Etilt_V^2)\cdot\Etilt_V^2(C_{2\rho}(p_H,P_H))\\
    &\leq \Etilt_V^2(\mathbf{B}_H) + C\Etilt^2_V\cdot \Etilt_V^2(C_{2\rho}(p_H,P_H)).
    \end{align*}
	Now, since $|p_H-p_L|\leq 2\sqrt{n}\ell(H)$, if we set $\sigma:= 64r_L+2\sqrt{n}\ell(H)$, we have
	$$\sigma = 32r_H + \frac{2}{M_0}r_H \leq \left(\frac{1}{2}+\frac{\lambda}{4}\right)64r_H\leq \left(1+\frac{\lambda}{2}\right)\rho + 2\tilde{C}\Etilt_V\ell(L).$$
	This string of inequalities is possible if $1/M_0\leq 8\lambda$, i.e.~if $M_0 = M_0(n)$ is sufficiently large. Hence, if $\eps_2$ is sufficiently small, we see that $\sigma\leq (1+\lambda)\rho$. Hence, as $|\pi_{P_H}(p_L)-\pi_{P_H}(p_H)|\leq |p_L-p_H|\leq 2\sqrt{n}\ell(L)$, we see that
	$$\mathbf{B}_L\subset C_{64r_L}(p_L,P_H)\subset C_{(1+\lambda)\rho}(p_H,P_H).$$
	Next, let $A:= \frac{1}{\w_n[(1+\lambda)\rho]^n}\int_{B_{(1+\lambda)\rho}(\pi_{P_H}(p_H),P_H)}D(u_{HJ})_a$ be the average of $D(u_{HJ})_a$ on this ball, and let $\tau$ be the plane whose graph is that of the function $P_H\to P_H^\perp$ which sends $x\mapsto A\cdot x$. Expressing the tilt-excess in terms of $u_{HJ}$, we then get
	\begin{align*}
		(1+C\Etilt_V^2)[(1+\lambda)\rho]^{-n}\int_{B_{(1+\lambda)\rho}(\pi_{P_H}(p_H))}\G(Du_{HJ},Q\llbracket A\rrbracket)^2 & \geq \Etilt_V^2(C_{(1+\lambda)\rho}(p_H,P_H),\tau)\\
		& \geq \Etilt_V^2(\mathbf{B}_L,\tau)\\
		& \geq \Etilt_V^2(\mathbf{B}_L).
	\end{align*}
	Now, we know from the application of Theorem \ref{thm:eps-reg} in $C_{32r_J}(p_J,P_H)$ (enabled by Lemma \ref{lemma:cm-5}) that for any $\eta>0$, provided $\eps_2$ is sufficiently small (also now depending on $\eta$) there is a function $v\in \mathfrak{B}_{\mathcal{V}}$ with (up to rotation and translation) $v:B_{16r_J}(\pi_{P_H}(p_J))\to \A_Q(P_H^\perp)$ such that (writing $\mathcal{E}:= \Etilt_V(C_{2\rho}(p_H,P_H))$ for ease of notation)
	$$(2\rho)^{-n}\int_{B_{2\rho}(\pi_{P_H}(p_H), P_H)}\left|\mathcal{E}^{-1}|Du_{HJ}|- |Dv|\right|^2\leq\eta.$$
	This in particular implies the same smallness condition but for the averages:
	$$(2\rho)^{-n}\int_{B_{2\rho}(\pi_{P_H}(p_H),P_H)}|\mathcal{E}^{-1}D(u_{HJ})_{\text{a}}-Dv_{\text{a}}|^2\leq Q^{-1}\eta \leq \eta.$$
	Since $v_{\text{a}}$ is harmonic, we know from the mean-value property of harmonic functions that we have $Dv_{\text{a}}(\pi_{P_H}(p_H)) = \frac{1}{\w_n[(1+\lambda)\rho]^n}\int_{B_{(1+\lambda)\rho}(\pi_{P_H}(p_H),P_H)}Dv_{\text{a}}$, and hence combining this with the above inequalities we get\footnote{Here we are also using that $\inf_{a\in\R^k}\int_B\G(v,Q\llbracket a\rrbracket)^2 = \int_B\G(v,Q\llbracket\bar{v}\rrbracket)^2$, where $\bar{v} = \frac{1}{|B|}\int_Bv_{\text{a}}$.}
	$$(1+C\Etilt_V^2)\mathcal{E}^2\int_{B_{(1+\lambda)\rho}(\pi_{P_H}(p_H),P_H)}\G(Dv,Q\llbracket Dv_{\text{a}}(\pi_{P_H}(p_H))\rrbracket)^2\geq [(1+\lambda)\rho]^{n}\Etilt_V^2(\mathbf{B}_L) - C\rho^{n}\mathcal{E}^2\eta$$
	and
	$$\mathcal{E}^2\int_{B_{2\rho}(\pi_{P_H}(p_H),P_H)}|Dv|^2 \leq (2\rho)^n\Etilt_V^2(\mathbf{B}_H) + C\rho^n\mathcal{E}^2\eta + C\rho^n\Etilt^2_V\mathcal{E}^2.$$
	Now since $\Etilt_V(\mathbf{B}_L) \geq C_e\Etilt_V\ell(L)^{11/12}\geq 2^{-11/12}\Etilt_V(\mathbf{B}_H)$, we therefore have (as $\lambda>0$)
	\begin{align*}
        (1+C\Etilt_V^2)\int_{B_{(1+\lambda)\rho}(\pi_{P_H}(p_H),P_H)}\G(Dv,\,&Q\llbracket Dv_{\text{a}}(\pi_{P_H}(p_H))\rrbracket)^2\\
        &\geq 2^{2/12}\cdot 2^{-n-2}\int_{B_{2\rho}(\pi_{P_H}(p_H),P_H)}|Dv|^2 - C\rho^n\eta - C\rho^n\Etilt^2_V.
    \end{align*}
	Furthermore, notice that since $(2\rho)^{-n}\int_{B_{2\rho}(\pi_{P_H}(p_H),P_H)}\mathcal{E}^{-2}|Du_{HJ}|^2$ can be made arbitrarily close to $1$, and therefore we know that for suitably small $\eta$ that $(2\rho)^{-n}\int_{B_{2\rho}(\pi_{P_H}(p_H),P_H)}|Dv|^2\geq 1/2$, we see that if we choose $\eta$ and $\eps_2$ sufficiently small we get
	$$\int_{B_{(1+\lambda)\rho}(\pi_{P_H}(p_H),P_H)}\G(Dv,Q\llbracket Dv_{\text{a}}(\pi_{P_H}(p_H))\rrbracket)^2 \geq 2^{1/12}\cdot 2^{-n-2}\int_{B_{2\rho}(\pi_{P_H}(p_H),P_H)}|Dv|^2.$$
	Therefore, now we can apply the claim from Step 1 (with an appropriate choice of $\delta = \delta(n,M_0)$) to conclude that there is a constant $\beta = \beta(\mathcal{V},M_0)\in (0,1)$ such that
	$$\beta^2\int_{B_{(1+\lambda)\rho}(\pi_{P_H}(p_H),P_H)}|Dv|^2\leq \int_{B_{\ell(L)/8}(q,P_H)}|Dv_f|^2$$
	for any ball $B_{\ell(L)/8}(q,P_H)\subset B_{2\rho}(\pi_{P_H}(p_H),P_H)$. In particular, passing this estimate back to $u_{HJ}$ using the estimates above and absorbing error terms, we get
	\begin{equation*}
		C\ell(L)^n\Etilt_V^2(\mathbf{B}_L)\leq \int_{B_{\ell(L)/8}(q,P_H)}|D(u_{HJ})_f|^2
	\end{equation*}
	i.e.~using the lower bound $\Etilt_V(\mathbf{B}_L)\geq C_e\Etilt_V\ell(L)^{11/12}$, we have
	\begin{equation*}\tag{$\ddagger$}\label{E:cm-doubling}
		C\Etilt_V^2\ell(L)^{n+11/6}\leq\int_{B_{\ell(L)/8}(q,P_H)}|D(u_{HJ})_f|^2.
	\end{equation*}
	All that remains is to pass from this inequality to the final estimate by changing from $u_{HJ}$ to $N$. So consider any ball $B_{\ell(L)/4}(q,P_0)$ where $\dist(L,q)\leq 6\sqrt{n}\ell(L)$. For simplicity, write $\Omega:=\Phi(B_{\ell(L)/4}(q,P_0))$. Note that $\pi_{P_H}(\Omega)$ must contain a ball $B_{\ell(L)/8}(\tilde{q},P_H)$ of the form above due to the estimates on $\wp$ (from Theorem \ref{thm:cm}) and $\|\pi_{P_H}-\pi_{P_0}\|$ (from Lemma \ref{lemma:cm-3}); choose any such ball, and note that we have the estimate \eqref{E:cm-doubling} with $\tilde{q}$ in place of $q$.
	
	Now, noting that, for each $x$, $\inf_{a}\G(Du_{HJ}(x),Q\llbracket a\rrbracket) = \G(Du_{HJ}(x),Q\llbracket D(u_{HJ})_{\text{a}}(x)\rrbracket)$, we get that $|D(u_{HJ})_f|\leq \mathcal{G}(Du_{HJ},Q\llbracket D\wp^\prime\rrbracket)$ everywhere, where $\wp^\prime$ is the function defined on $B_{16r_J}(\pi_{P_H}(p_J),P_H)$ whose graph coincides with $\mathcal{M}$. Thus we have
	$$\int_{B_{\ell(L)/8}(\tilde{q},P_H)}|D(u_{HJ})_f|^2\leq \int_{B_{\ell(L)/8}(\tilde{q},P_H)}\G(Du_{HJ},Q\llbracket D\wp^\prime\rrbracket)^2.$$
	The right-hand side of this inequality is then controlled by
	$$\int_{B_{\ell(L)/8}(\tilde{q},P_H)}\G(Du_{HJ},Q\llbracket D\wp^\prime\rrbracket)^2 \leq C(\mathcal{V})\int_{C_{\ell(L)/8}(\tilde{q},P_H)}\|\pi_{T_XV}-\pi_{T_{x + \wp^\prime(x)}\mathcal{M}}\|^2\ \ext\|V\|(X)$$
	where here we have written $x = \pi_{P_H}(X)$ (note that $x+\wp^\prime(x)\in \mathcal{M}$ by definition). Now, because we have $|\mathbf{p}(X)-(x+\wp^\prime(x))|\leq C\Etilt_V^2\ell(L)^{1+11/12}$ (which follows from the graph structure as the left-hand side is controlled by $C\Etilt_V\ell(L)^{11/12}|X-\mathbf{p}(X)|$ from our estimates in Theorem \ref{thm:cm}), we get
    $$\|\pi_{T_{\mathbf{p}(X)}\mathcal{M}}-\pi_{T_{x+\wp^\prime(x)}\mathcal{M}}\| \leq C\|D^2\wp^\prime\|_{C^0}|\mathbf{p}(X)-(x+\wp^\prime(x))| \leq C\Etilt_V\cdot \Etilt_V^2\ell(L)^{1+11/12}$$
    and thus
	$$\int_{B_{\ell(L)/8}(\tilde{q},P_H)}\G(Du_{HJ},Q\llbracket\wp^\prime\rrbracket)^2 \leq C\int_{\mathbf{p}^{-1}(\Omega)}\|\pi_{T_XV}-\pi_{T_{\mathbf{p}(X)}\mathcal{M}}\|^2\ \ext\|V\|(X) + C\Etilt_V^6\ell(L)^{n+2+11/6}.$$
	But now expanding this latter tilt-excess term in terms of the normal map $N$, we have (see \cite[Proposition 3.4]{DLS13})
	$$\int_{\mathbf{p}^{-1}(\Omega)}\|\pi_{T_XV}-\pi_{T_{\mathbf{p}(X)}\mathcal{M}}\|^2\leq (1+C\Etilt_V^2\ell(L)^{11/6})\int_{\Omega}|DN|^2 + C\Etilt_V^4\ell(L)^{n+2}.$$
	Putting all of these estimates together, including with \eqref{E:cm-doubling}, we get
	$$C\Etilt_V^2\ell(L)^{n+11/6}\leq C(1+\Etilt_V^2\ell(L)^{11/6})\int_{\Omega}|DN|^2 + C\Etilt_V^4\ell(L)^{n+2}.$$
	From this we evidently see that, for $\eps_2$ sufficiently small, we have
	$$C\Etilt_V^2\ell(L)^{n+11/6}\leq \int_{\Omega}|DN|^2$$
	which concludes the proof.
\end{proof}

\subsection{Frequency relative to center manifold}\label{sec:cm-frequency}

We now combine all the estimates from the previous subsection to show that the normal map $N$ has an associated frequency function (namely, the frequency of $V$ relative to $\mathcal{M}$) which is approximately monotone. We now fix choices of $M_0,N_0,C_e,C_h$, depending only on $\mathcal{V}$, so that the results of the previous sections hold.

Let $\phi:[0,\infty]\to[0,1]$ be the same function from Section \ref{sec:pff}. For each $r\in (0,3]$ we then define
$$\mathbf{D}(r):= r^{2-n}\int_{\mathcal{M}}\phi\left(\frac{d(p)}{r}\right)|DN(p)|^2\ \ext p$$
and
$$\mathbf{H}(r):= -r^{1-n}\int_{\mathcal{M}}\phi^\prime\left(\frac{d(p)}{r}\right)d(p)^{-1}|N(p)|^2\ \ext p,$$
where $d(p)$ denotes the geodesic distance in $\mathcal{M}$ between $p$ and $0\in\mathcal{M}$. Here, $N$ is the normal map to $\mathcal{M}$ given by Theorem \ref{thm:cm}. Note that these are the corresponding expressions to those defined in Section \ref{sec:pff} (up to controlled factors) except relative to $\mathcal{M}$ rather than a plane. When $\mathbf{H}(r)>0$, we define the \emph{center manifold frequency function} by
$$\mathbf{I}(r):=\frac{\mathbf{D}(r)}{\mathbf{H}(r)}.$$
\textbf{Remark:} We will show later in Remark \ref{remark:cm-H-non-zero} that indeed $\mathbf{H}(r)>0$ for all $r>0$: this will follow in an identical fashion to that seen in Section \ref{sec:pff}, once we are able to control certain extra error terms which arise in the equivalent choice of test function in the first variation.

Just like with the planar frequency, where we wanted to rule out an infinite order of contact between $V$ and the plane, here we want to rule out that $V$ has an infinite order of contact with $\mathcal{M}$, for which we will establish an approximate monotonicity of the frequency function $\mathbf{I}(r)$. The argument will mirror that of the planar case, except with more error terms due to $\mathcal{M}$ being potentially curved.

Throughout, we will write $\del_{\hat{r}}$ for the derivative with respect to arc-length along geodesics in $\mathcal{M}$ starting at $0$. Let $\exp$ denote the exponential map $T_0\mathcal{M}\supset B_3\to \mathcal{M}$. We will also write $\mathcal{B}_r(0)$ for the geodesic ball in $\mathcal{M}$ of radius $r$ and centred at $0$.

By the area formula, we can write $\mathbf{H}(r)$ as
$$\mathbf{H}(r)= - \int_{T_0\mathcal{M}}\phi^\prime(|z|)\cdot|z|^{-1}\cdot|N(\exp(rz))|^2\cdot\mathbf{J}_{\exp}(rz)\ \ext z$$
where $\mathbf{J}_{\exp}$ is the Jacobian of $\exp$. Therefore, if we differentiate under the integral sign, we get
\begin{align*}
	\mathbf{H}^\prime(r) = -2\int_{T_0\mathcal{M}}&\phi^\prime(|z|)\cdot|z|^{-1}\sum_i\langle N_i(\exp(rz)),\del_{\hat{r}}N_i(\exp(rz))\rangle\cdot\mathbf{J}_{\exp}(rz)\ \ext z\\
	& -\int_{T_0\mathcal{M}}\phi^\prime(|z|)\cdot|z|^{-1}|N(\exp(rz))|^2\cdot\frac{\ext}{\ext r}\mathbf{J}_{\exp}(rz)\ \ext z.
\end{align*}
Now, as $\mathcal{M}$ is a $C^{3,1/2}$ manifold, $\exp$ is a $C^{2,1/2}$ map, and so certainly $\frac{\ext}{\ext r}\mathbf{J}_{\exp}(rz)$ is bounded by $\|\wp\|_{C^{3,1/2}}$, i.e.~by $C(\mathcal{V})\Etilt_V$. Hence, as $\ext \exp_0(y) = |y|$ for every $y\in B^n_1(0)$, we see that if we change coordinates back via the area formula,
\begin{equation}\label{E:H-prime-1}
	\mathbf{H}^\prime(r) = -2r^{-n}\int_{\mathcal{M}}\phi^\prime\left(\frac{d(p)}{r}\right)\sum^Q_{i=1}\langle N_i(p),\del_{\hat{r}}N_i(p)\rangle\ \ext p + O(1)\Etilt_V\mathbf{H}(r)
\end{equation}
where by $O(1)$ we mean a term which is bounded (in terms of $\mathcal{V}$). For ease of notation we will write
$$\mathbf{E}(r):= -r^{1-n}\int_{\mathcal{M}}\phi^\prime\left(\frac{d(p)}{r}\right)\sum^Q_{i=1}\langle N_i(p),\del_{\hat{r}}N_i(p)\rangle\ \ext p$$
and so \eqref{E:H-prime-1} is equivalent to
\begin{equation}\label{E:H-prime-2}
	|\mathbf{H}^\prime(r)-2r^{-1}\mathbf{E}(r)|\leq C\Etilt_V\mathbf{H}(r)
\end{equation}
where $C = C(\mathcal{V})$.

Recall that the normal projection map $\mathbf{p}:\mathbf{U}\to\mathcal{M}$ is a $C^{2,1/2}$ function. We now consider two types of variation, entirely analogous to those used in Section \ref{sec:pff} for planar frequency (namely, variations in the normal and tangential directions to $\mathcal{M}$) to derive the key variational identities needed to show the approximate monotonicity of $\mathbf{I}(r)$.

First, we consider the \emph{outer variation} (i.e.~\emph{normal variation}) $X_{\text{o}}(p):=\phi\left(\frac{d(\mathbf{p}(p))}{r}\right)\mathbf{p}^\perp(p)$ where $\mathbf{p}^\perp(p):= p-\mathbf{p}(p)$. Observe that $X_{\text{o}}$ is supported in $\mathbf{p}^{-1}(\mathcal{B}_r(0))$ (we should also cut $X_{\text{o}}$ off vertically, as the support of $V$ is in $\mathbf{U}$, to assume $X_{\text{o}}$ has compact support in $\R^{n+k}$ without changing the computations, but we shall omit this detail). To ease notation, for $p\in\mathcal{M}$ we will write $\phi_r(p):=\phi\left(\frac{d(p)}{r}\right)$ and $\phi^\prime_r(p):= \phi^\prime\left(\frac{d(p)}{r}\right)$. So, applying the first variation formula for $V$ and writing everything in terms of the normal map $N$ (which represents $V$ \emph{everywhere}) we get (see \cite[Theorem 4.2]{DLS13}):
\begin{equation}\label{E:cm-outer-variation}
\int_{\mathcal{M}}\phi_r|DN|^2 + \sum^Q_{i=1}N_i\otimes \nabla\phi_r:DN_i = \sum^3_{j=1}\text{Err}^{\text{o}}_j
\end{equation}
where
$$\text{Err}_1^{\text{o}} = Q\int_{\mathcal{M}}\phi_r\langle H_{\mathcal{M}},N_{\text{a}}\rangle,$$
$$|\text{Err}_2^{\text{o}}|\leq C_0\int_{\mathcal{M}}\phi_r|A|^2|N|^2,$$
$$|\text{Err}_3^{\text{o}}|\leq C_0\int_{\mathcal{M}}(|N||A|+|DN|^2)(\phi_r|DN|^2+|D\phi_r||DN||N|)$$
where $H_{\mathcal{M}}$ is the mean curvature vector of $\mathcal{M}$ and $A$ is the second fundamental form of $\mathcal{M}$ (technically, \cite[Theorem 4.2]{DLS13} requires $\phi$ to be $C^1$, but we can get around this by an approximation argument). Multiplying this by $r^{2-n}$ we get
$$\mathbf{D}(r) = -r^{1-n}\int_{\mathcal{M}}\phi_r^\prime\sum_i\langle N_i(p),\del_{\hat{r}}N_i(p)\rangle\ \ext p + r^{2-n}\sum^3_{j=1}\text{Err}_j^{\text{o}}$$
i.e.
\begin{equation}\label{E:D-E}
	\mathbf{D}(r) - \mathbf{E}(r) = r^{2-n}\sum^3_{j=1}\text{Err}_j^{\text{o}}.
\end{equation}
Second, consider the \emph{inner variation} (i.e.~\emph{tangential variation}) given by $X_{\text{in}}(p):= Y(\mathbf{p}(p))$, where for $p\in\mathcal{M}$,
$$Y(p):= \frac{d(p)}{r}\cdot\phi\left(\frac{d(p)}{r}\right)\cdot\frac{\del}{\del\hat{r}}.$$
If we take $X_{\text{in}}$ in the first variation formula for $V$, we get (see \cite[Theorem 4.3]{DLS13})
\begin{equation}\label{E:cm-inner-variation}
\int_{\mathcal{M}}\frac{1}{2}|DN|^2\div_{\mathcal{M}}(Y) - \sum^Q_{i=1}\langle DN_i:(DN_i\cdot D_{\mathcal{M}}Y)\rangle = \sum^3_{j=1}\text{Err}_j^{\text{in}}
\end{equation}
where
$$\text{Err}_1^{\text{in}} = Q\int_{\mathcal{M}}\langle H_{\mathcal{M}},N_{\text{a}}\rangle\div_{\mathcal{M}}(Y) + \langle D_YH_{\mathcal{M}},N_{\text{a}}\rangle,$$
$$|\text{Err}_2^{\text{in}}|\leq C_0\int_{\mathcal{M}}|A|^2(|DY||N|^2 + |Y||N||DN|),$$
$$|\text{Err}_3^{\text{in}}|\leq C_0\int_{\mathcal{M}}|Y||A||DN|^2(|N|+|DN|) + |DY|(|A||N|^2|DN| + |DN|^4).$$
We may compute
$$D_{\mathcal{M}}Y(p) = \phi^\prime\left(\frac{d(p)}{r}\right)\frac{d(p)}{r^2}\cdot\frac{\del}{\del \hat{r}}\otimes\frac{\del}{\del\hat{r}} + \phi\left(\frac{d(p)}{r}\right)\left(\frac{\text{Id}}{r} + O(1)\Etilt_V\right),$$
$$\div_{\mathcal{M}}Y(p) = \phi^\prime\left(\frac{d(p)}{r}\right)\cdot\frac{d(p)}{r^2} + \phi\left(\frac{d(p)}{r}\right)\left(\frac{n}{r} + O(1)\Etilt_V\right)$$
and so therefore we get, after multiplying by $2r^{2-n}$,
\begin{equation}\label{E:D-prime-1}
	\mathbf{D}^\prime(r) = -2r^{-n}\int_{\mathcal{M}}\phi^\prime_r(p)d(p)|\del_{\hat{r}}N(p)|^2\ \ext p + O(1)\Etilt_V\mathbf{D}(r) + 2r^{2-n}\sum^3_{j=1}\text{Err}_j^{\text{in}},
\end{equation}
which, if we simplify notation by writing
$$\mathbf{G}(r) := -r^{-n} \int_{\mathcal{M}}\phi_r^\prime(p)d(p)|\del_{\hat{r}}N(p)|^2\ \ext p,$$
is equivalent to
\begin{equation}\label{E:D-prime-2}
	\mathbf{D}^\prime(r) - 2\mathbf{G}(r) = O(1)\Etilt_V\mathbf{D}(r) + 2r^{2-n}\sum^3_{j=1}\text{Err}^{\text{in}}_j.
\end{equation}
The main work that remains is to bound the error terms $\text{Err}_j^{\text{o}}$ and $\text{Err}_j^{\text{in}}$. Note that ``error'' is only a reasonable name for these terms if they are actually smaller than the main terms, i.e.~those involving $\mathbf{D}(r)$ and $\mathbf{H}(r)$. As $\mathbf{D}(r)$ and $\mathbf{H}(r)$ are quadratic in $N$ or $DN$, we see that the worse error terms a priori are $\text{Err}_1^{\text{o}}$ and $\text{Err}_1^{\text{in}}$, as these are only linear in $N$ (notice that if $\mathcal{M}$ were minimal, as was the case for planar frequency when $\mathcal{M}$ was taken to be a plane, these terms would vanish). However, since they are in fact linear in $N_{\text{a}}$, we hope to control them by making $N_{\text{a}}$ small. \emph{This is exactly the reason for the center manifold construction}. We will bound the errors in terms of $\mathbf{D}(r)$ and $\mathbf{H}(r)$, which is why we needed to control both the tilt and the height in the center manifold construction, and moreover why we had to be delicate with ensuring that $N_{\text{a}}$ is always of a higher order than $\mathbf{D}(r)$ and $\mathbf{H}(r)$ (cf.~Theorem \ref{thm:cm-14}), which we did by building the center manifold over regions where decay stops.

When estimating the error terms, we will need to pass from the domain of integration, namely $\mathcal{B}_r(0)$, to the region in the plane $B:=\pi_{P_0}(\mathcal{B}_r(0))$ where we had our Whitney decomposition forming $\mathcal{M}$. Notice that as $\|\wp\|_{C^{3,1/2}}\leq C\eps_2$ (see Theorem \ref{thm:cm}) we know that $B$ is a $C^2$ convex set which moreover has the property that, for all $z\in \del B$ there is a ball $B_{r/2}(y,P_0)\subset B$ with $\overline{B}_{r/2}(y,P_0)\cap \del B = \{z\}$.

We therefore need to look at the Whitney regions that intersect $B$ and control the error terms across them. Analogously to the situation in Section \ref{sec:pff}, in order to reintroduce the function $\phi_r$ in our estimates, we need to be able to choose smaller balls which stay away from $\del B$ by a definite amount compared to the size of the Whitney region: for this we use our doubling-style conditions, such as Lemma \ref{lemma:cm-17} in the case of excess cubes $\S_e$. Morally, in Section \ref{sec:pff} we only had to do this for excess cubes $\S_e$, whereas here we now need to also look at $\S_h$ and $\S_n$. The argument is essentially analogous for $\S_h$ to that of $\S_e$ (easier, in fact, as there is always a definite amount of separation), and so $\S_h$ does not cause any significant complication. For $\S_n$, however, we do not have an a priori decay-change, and so we need to take the additional step at looking at its spawner, which has a decay-change, and use this with Lemma \ref{lemma:cm-17}. The complication is that, whilst the spawner is in $\S_e$, it might not lie inside $B$, and so we need to be careful choosing the ball for the doubling condition. For such cubes we will need a slight modification, as we shall see.

So, we define a family of cubes $\mathcal{T}\subset \S_e\cup\S_h$ corresponding to each Whitney region intersecting $B$ as follows:
\begin{itemize}
	\item [(i)] $\mathcal{T}$ contains all $L\in\S_e\cup\S_h$ which intersect $B$;
	\item [(ii)] if $L^\prime\in \S_n$ intersects $B$, then $L\in \mathcal{T}$, where $L\in\S_e$ is the spawner of $H$.
\end{itemize}
By Lemma \ref{lemma:cm-16} and Lemma \ref{lemma:cm-12} we know that $\ell(L)\leq \frac{3r}{64\sqrt{n}}$ and $\dist(L,B)\leq 3\sqrt{n}\ell(L)$ if $L\in\mathcal{T}$. For each $L\in\mathcal{T}$, recall that $x_L$ is the center of $L$. We then define an associated ball by
$$B^L:= B_{s(L)}(x_L,P_0)\qquad \text{where}\qquad s(L):=\sqrt{n}\ell(L) + \dist(x_L,\overline{B}).$$
Necessarily $B^L\cap B\neq\emptyset$ for each $L\in\mathcal{T}$ (and thus the ball $B^L$ always ``comes inside'' $B$ slightly whilst being centred at $x_L$).

\begin{remark}\label{remark:cover-1}
Note that for $L\in\mathcal{T}$ we have $s(L)\leq 5\sqrt{n}\ell(L)$. Indeed, if $L\cap B\neq\emptyset$, then $\dist(x_L,\overline{B})\leq \sqrt{n}\ell(L)$ and so $s(L)\leq 2\sqrt{n}\ell(L)$. Otherwise, $L$ must be the spawner of some $L^\prime\in\S_n$ which intersects $B$. By Lemma \ref{lemma:cm-16} we know $\dist(L,B)\leq 3\sqrt{n}\ell(L)$, and so $\dist(x_L,\overline{B})\leq 4\sqrt{n}\ell(L)$. Hence, $s(L)\leq 5\sqrt{n}\ell(L)$.
\end{remark}

We now define a countable family $\mathcal{F}\subset \mathcal{T}$ of cubes such that $\{B^L\}_{L\in\mathcal{F}}$ is a \emph{maximal} pairwise disjoint collection of balls. For this, first write $S:=\sup_{L\in\mathcal{T}}s(L)$ and select a maximal pairwise disjoint subcollection $\mathcal{F}_1\subset\mathcal{T}$ of cubes with $s(L)\geq S/2$; clearly $\mathcal{F}_1$ is finite. Then, inductively at stage $k\geq 2$, select a maximal subcollection $\mathcal{F}_k\subset\mathcal{T}$ such that $\{B^L\}_{L\in\mathcal{F}_k}$ are pairwise disjoint, these balls do not intersect any of the balls $B^{\tilde{L}}$ with $\tilde{L}\in \mathcal{F}_1\cup\cdots\cup\mathcal{F}_{k-1}$, and $s(L)\in [2^{-k}S,2^{1-k}S)$ for $L\in \mathcal{F}_k$. We then set $\mathcal{F} := \bigcup_k\mathcal{F}_k$. As $\mathcal{F}$ is a countable union of finite sets, we may fix an ordering $\mathcal{F} = \{L_i\}_{i\in \N}$ of $\mathcal{F}$ such that $s(L_i)$ is decreasing.

Of course, we may then partition the cubes of $\S$ which intersect $B$ into disjoint families $\{\S(L_i)\}_{L_i\in\mathcal{F}}$ depending on the associated ball which is intersected. Indeed, let $H\in\S$ have $B\cap H\neq\emptyset$. Then:
\begin{enumerate}
	\item [(a)] if $H\in\S_e\cup\S_h$, then $H\in\mathcal{T}$, and so there is $i\in \N$ with $B^H\cap B^{L_i}\neq\emptyset$. Choose $i$ minimal with this property and set $H\in\S(L_i)$.
	\item [(b)] if $H\in\S_n$, then its spawner $J$ obeys $J\in \S_e\cap\mathcal{T}$. Hence, there is $i\in \N$ with $B^J\cap B^{L_i}\neq\emptyset$. Choose $i$ minimal with this property and set $H\in \S(L_i)$.
\end{enumerate}

\begin{remark}\label{remark:cover-2}
Notice that for (a) we must have $s(L_i)\geq\frac{s(H)}{2}$ and in (b) we must have $s(L_i)\geq \frac{s(J)}{2}$. This follows simply because if $H\in\mathcal{T}$ had $B^H\cap B^{L_i}=\emptyset$ for all $L_i$ with $s(L_i)\geq\frac{s(H)}{2}$, then, by construction of $\mathcal{F}$, if we choose $k$ such that $s(H)\in [2^{-k}S,2^{1-k}S)$ then the collection $\mathcal{F}_k$ could be increased by including $H$ in it, contradicting the maximality of $\mathcal{F}_k$.
\end{remark}

\textbf{Remark:} We have $\S(L_i)\neq\emptyset$ for each $i$. Indeed, if $L_i\cap B\neq\emptyset$ then $L_i\in\S(L_i)$ (as all the $\{B^{L_j}\}_{j\in \N}$ are pairwise disjoint). Similarly, if $L_i\cap B=\emptyset$, then there is $H\in \S_n$ with $H\cap B\neq\emptyset$ for which $L_i$ is the spawner of $H$, and again $H\in\S(L_i)$.

\begin{remark}\label{remark:cover-4}
We claim that if $H\in \S(L_i)$ then
$$H\subset B_{65\sqrt{n}\ell(L_i)}(x_{L_i}).$$
To see this, first suppose $H\in \S_n$. Let $J\in\mathcal{T}$ be the spawner of $H$. We know from Lemma \ref{lemma:cm-16} that $\dist(H,J)\leq 3\sqrt{n}\ell(J)$, and so $H\subset B_{3\sqrt{n}\ell(J) + 2\sqrt{n}\ell(H) + \sqrt{n}\ell(J)}(x_J)\subset B_{5\sqrt{n}\ell(J)}(x_J)$. As $B^{L_i}\cap B^J\neq\emptyset$, we know $|x_J-x_{L_i}|\leq s(L_i) + s(J) \leq 3s(L_i)$, where we have used $s(J)\leq 2s(L_i)$ from Remark \ref{remark:cover-2}. We also know from Remark \ref{remark:cover-1} that $s(L_i)\leq 5\sqrt{n}\ell(L_i)$, and so $|x_J-x_{L_i}|\leq 15\sqrt{n}\ell(L_i)$. Hence
$$H\subset B_{5\sqrt{n}\ell(J)}(x_J)\subset B_{5\sqrt{n}\ell(J)+15\sqrt{n}\ell(L_i)}(x_{L_i}).$$
But as $\sqrt{n}\ell(J)\leq s(J)\leq 2s(L_i)\leq 10\sqrt{n}\ell(L_i)$, we therefore get
$$H\subset B_{65\sqrt{n}\ell(L_i)}(x_{L_i})$$
which is the claimed inclusion. In the other case, we have $H\in \S_e\cup \S_h$ and so $B^{L_i}\cap B^J\neq\emptyset$. As above, this gives $|x_H-x_{L_i}|\leq 3s(L_i)\leq 18\sqrt{n}\ell(L_i)$, and hence
$$H\subset B_{\sqrt{n}\ell(H)}(x_H)\subset B_{\sqrt{n}\ell(H)+18\sqrt{n}\ell(L_i)}(x_{L_i})\subset B_{24\sqrt{n}\ell(L_i)}(x_{L_i})$$
which thus proves the claim.
\end{remark}

\begin{remark}\label{remark:cover-5}
The inclusion in Remark \ref{remark:cover-4} immediately gives that if $H\in \S(L_i)$ then $\ell(H)\leq 65\ell(L_i)$.
\end{remark}

For each $L_i\in\mathcal{T}$, by definition of $B^{L_i}$ we can clearly choose a ball $B_{\ell(L_i)/4}(\tilde{x},P_0)\subset B^{L_i}\cap B$ which lies at least a distance $\ell(L_i)/4$ from $\del B$, where $\tilde{x}\in B^{L_i}\cap B$ is some point. (This was essentially the point of the definition of $s(L_i)$, to come into $B$ a definite amount: we can choose $\tilde{x}$ on the line segment connecting $x_{L_i}$ to $0$.) Choose one such ball for each cube $L_i\in\mathcal{F}$, and we will denote this choice of ball by $B(L_i)$. To summarise, we have the following properties of $\mathcal{F} = \{L_i\}_{i\in \N}$:
\begin{enumerate}
	\item [(I)] if $L_i\in\mathcal{F}$, then $L_i\in \S_e\cup\S_h$, and $L_i$ has an associated ball $B(L_i)$ of radius $\ell(L_i)/4$ which obeys $B(L_i)\subset B^{L_i}\cap B$, $\dist(B(L_i),\del B)\geq\ell(L_i)/4$, and whose center lies on the line segment connecting $x_{L_i}$ to $0$;
	\item [(II)] if $i\neq j$, then $B^{L_i}\cap B^{L_j}=\emptyset$, and so in particular $B(L_i)\cap B(L_j) = \emptyset$ and $L_i,L_j$ are distinct;
	\item [(III)] if $H\in\S$ intersects $B$, then $H\in\S(L_i)$ for some $i$, and $H\subset B_{65\sqrt{n}\ell(L_i)}(x_{L_i})$. Moreover, the families $\S(L_i)$ are disjoint, i.e.~$\S(L_i)\cap \S(L_j) = \emptyset$ for $i\neq j$.
\end{enumerate}
Now write
$$\mathcal{B}_{L_i}:= \Phi(B(L_i))\qquad \text{and}\qquad \mathcal{U}_{L_i} := \bigcup_{H\in\S(L_i)}\Phi(H)\cap \mathcal{B}_r(0).$$
These are the analogous balls from those in Section \ref{sec:pff} used to cover via the doubling condition. Observe that, as $\del B = \pi_{P_0}(\del\mathcal{B}_r(0))$, from simple properties of projections we have $\dist(\mathcal{B}_{L_i},\del\mathcal{B}_r(0))\geq \dist(B(L_i),\del B)$, and so from (I) above we have
$$\dist(\mathcal{B}_{L_i},\del\mathcal{B}_r(0))\geq\ell(L_i)/4.$$
Thus, we have
$$\inf_{\mathcal{B}_{L_i}}\phi_r\geq\frac{\ell(L_i)}{4r}.$$
Noting by (III) that $\mathcal{U}_{L_i}\subset\Phi(B_{65\sqrt{n}\ell(L_i)}(x_{L_i}))$, we see that
$$\sup_{\mathcal{U}_{L_i}}\phi_r - \inf_{\mathcal{U}_{L_i}}\phi_r \leq C(n)\text{Lip}(\phi_r)\ell(L_i)\leq \frac{C}{r}\ell(L_i)\leq C\inf_{\mathcal{B}_{L_i}}\phi_r.$$
But now one can check\footnote{Indeed, if $L_i\cap B\neq\emptyset$, then $L_i\in \S(L_i)$ and so the claim follows simply because the center point of $B(L_i)$ is closer to $0$ than that of $x_{L_i}$ (cf.~(I) above) and the radius of $B(L_i)$ is smaller than the side-length of $L_i$. If $L_i\cap B=\emptyset$, then $L_i$ must be the spawner of some cube $H\in\S_n$ with $H\cap B\neq\emptyset$, which moreover we know obeys $H\in \S(L_i)$. From Lemma \ref{lemma:cm-16} we know that $\dist(H,L_i)\leq 3\sqrt{n}\ell(L_i)$, and thus as $L_i\cap B=\emptyset$, there is a point $x$ in $H$ obeying $\dist(x,\del B)\leq 3\sqrt{n}\ell(L_i)$, and thus $\inf_{\mathcal{U}_{L_i}}\phi_r \leq \frac{7\sqrt{n}\ell(L_i)}{r}\leq 28\sqrt{n}\inf_{\mathcal{B}_{L_i}}\phi_r$.} that $\inf_{\mathcal{U}_{L_i}}\phi_r \leq C(n)\inf_{\mathcal{B}_{L_i}}\phi_r$, and so we therefore get
\begin{equation}\label{E:phi-sup-inf}
	\sup_{\mathcal{U}_{L_i}}\phi_r\leq C\inf_{\mathcal{B}_{L_i}}\phi_r.
\end{equation}
Finally, note that from Theorem \ref{thm:cm-14} we have the following estimates:
\begin{equation}\label{E:cm-N-bound}
\|N\|_{C^0(\mathcal{U}_{L_i})}\leq C\Etilt_V\ell(L_i);
\end{equation}
\begin{equation}\label{E:cm-DN-bound}
\|DN\|_{GC^{0,\alpha}(\mathcal{U}_{L_i})}\leq C\Etilt_V\ell(L_i)^{\frac{11}{12}};
\end{equation}
\begin{equation}\label{E:cm-average-bound}
\int_{\mathcal{U}_{L_i}}|N_{\text{a}}|\leq C\Etilt_V^3\ell(L_i)^{n+\frac{17}{6}}.
\end{equation}
Indeed, the first two follow directly from Theorem \ref{thm:cm-14}, the definition of $\mathcal{U}_{L_i}$ as a union over $H\in \S(L_i)$, and the fact that $\ell(H)\leq 65\ell(L_i)$ (see Remark \ref{remark:cover-5}) for $H\in \S(L_i)$. For the third, this also follows in the same manner, namely summing the corresponding estimates from Theorem \ref{thm:cm-14} over each $H\in\S(L_i)$, just additionally noting that $\sum_{H\in\S(L_i)}\ell(H)^n\leq C(n)\ell(L_i)^n$, because all $H\in \S(L_i)$ are disjoint (except for sets of measure zero) and are contained in the ball $B_{65\sqrt{n}\ell(L_i)}(x_{L_i})$ (so, one can compare the measures). This also then implies that $\sum_{H\in \S(L_i)}\ell(H)^{n+\eps}\leq 65^\eps C(n)\ell(L_i)^{n+\eps}$ for every $\eps>0$, simply using the fact $\ell(H)\leq 65\ell(L_i)$ for each $H\in \S(L_i)$ (see Remark \ref{remark:cover-5}).

We can now begin controlling the error terms. First, we note the following simple lemma. Here, we write
$$\boldsymbol{\Sigma}(r):= r^{-n}\int_{\mathcal{M}}\phi_r|N|^2.$$

\begin{lemma}\label{lemma:cm-18}
	There exists a constant $C = C(n)\in (0,\infty)$ such that
	$$\boldsymbol{\Sigma}(r)\leq C\mathbf{D}(r) + C\mathbf{H}(r) \qquad \text{and} \qquad r\boldsymbol{\Sigma}^\prime(r) \leq - n\boldsymbol{\Sigma}(r) + \mathbf{H}(r),$$
	$$r^{-n}\int_{\mathcal{B}_r(0)}|N|^2\leq\boldsymbol{\Sigma}(r)+\mathbf{H}(r),$$
	$$r^{2-n}\int_{\mathcal{B}_r(0)}|DN|^2\leq (n-1)\mathbf{D}(r)+r\mathbf{D}^\prime(r).$$
	Moreover, if there is a constant $\delta>0$ such that $\mathbf{I}(r)\geq\delta$, then we have $\boldsymbol{\Sigma}(r)$, $r\boldsymbol{\Sigma}^\prime(r)\leq C\mathbf{D}(r)$, and
	$$r^{-n}\int_{\mathcal{B}_r(0)}|N|^2\leq C\mathbf{D}(r)$$
	where now $C$ is allowed to depend also on $\delta$.
\end{lemma}

\textbf{Remark:} We will subsequently show that in fact we unconditionally have $\mathbf{I}(r)\geq\delta$ for some fixed constant $\delta$ (which could depend on $V$); see Corollary \ref{cor:cm-21}. With this, the above statements hold with no assumption on $\mathbf{I}(r)$, but allowing the constants to depend on this lower bound on $\mathbf{I}(r)$. However, one still needs the above lemma as an intermediate step in the proof of this frequency lower bound.

\begin{proof}
	Note that $\phi_r|N|^2$ is a Lipschitz function of compact support in $\mathcal{B}_r(0)$, and so using the Poincaré inequality followed by Cauchy--Schwarz we get
	\begin{align*}
		\boldsymbol{\Sigma}(r)\leq r^{-n}\cdot Cr\int_{\mathcal{M}}|D(\phi_r|N|^2)| & \leq Cr^{-n}\int_{\mathcal{M}}|\phi_r^\prime||N|^2 + r^{1-n}\int_{\mathcal{M}}\phi_r|N||DN|\\
		& \leq C\mathbf{H}(r) + C\boldsymbol{\Sigma}(r)^{1/2}\mathbf{D}(r)^{1/2}\\
		& \leq C\mathbf{H}(r) + \frac{1}{2}\boldsymbol{\Sigma}(r) + C\mathbf{D}(r),
	\end{align*}
	where in the last inequality we have used $2ab\leq a^2+b^2$ for suitable $a,b\in \R$. Rearranging clearly gives the first claimed inequality. We also have
	$$\boldsymbol{\Sigma}^\prime(r) = -nr^{-n-1}\int_{\mathcal{M}}\phi_r|N|^2 - r^{-n}\int_{\mathcal{M}}\frac{d(p)}{r^2}\phi_r^\prime|N|^2 \leq -nr^{-1}\boldsymbol{\Sigma}(r) + r^{-1}\mathbf{H}(r)$$
	giving the second inequality. For the third, as $\phi^\prime\equiv -2$ on $(1/2,1)$ and $\phi^\prime\equiv0$ on $(0,1/2)$, we easily have
	$$\int_{\mathcal{B}_r(0)\setminus\mathcal{B}_{r/2}(0)}|N|^2\leq \frac{1}{2}r^n\mathbf{H}(r)$$
	and as $\phi\equiv 1$ on $(0,1/2)$ we have
	$$\int_{\mathcal{B}_{r/2}(0)}|N|^2\leq r^n\boldsymbol{\Sigma}(r).$$
	Summing these inequalities gives the third inequality. For the fourth, note that
	\begin{align*}
		\mathbf{D}^\prime(r) & = (2-n)r^{1-n}\int_{\mathcal{M}}\phi_r|DN|^2 - r^{2-n}\int_{\mathcal{M}}\frac{d(p)}{r^2}\phi_r^\prime|DN|^2\\
		& \geq (2-n)r^{-1}\mathbf{D}(r) + r^{1-n}\int_{\mathcal{B}_{r}(0)\setminus\mathcal{B}_{r/2}(0)}|DN|^2
	\end{align*}
	and again, as $\phi\equiv 1$ on $(0,1/2)$,
	$$r^{2-n}\int_{\mathcal{B}_{r/2}(0)}|DN|^2\leq \mathbf{D}(r).$$
	Multiplying the first inequality here by $r$ and summing with the second, we get
	$$r^{2-n}\int_{\mathcal{B}_r(0)}|DN|^2\leq (n-1)\mathbf{D}(r) + r\mathbf{D}^\prime(r).$$
	Finally, if $\mathbf{I}(r)\geq\delta$, this gives $\mathbf{H}(r)\leq \delta^{-1}\mathbf{D}(r)$. Hence the first inequality of the lemma gives $\boldsymbol{\Sigma}(r)\leq C(1+\delta^{-1})\mathbf{D}(r)$, whilst the second gives (as $\boldsymbol{\Sigma}(r)\geq 0$) $r\boldsymbol{\Sigma}^\prime(r)\leq\mathbf{D}(r)$. For the final inequality, this now follows from the third inequality of the lemma.
\end{proof}

The last ingredient we need in order to bound the error terms correctly is to provide a way to control the different error terms. For this, we will control the scales $\ell(L_i)$ for the family of cubes $\mathcal{F} = \{L_i\}_{i\in \N}$ described previously. As we can relate all the relevant quantities to these scales, this will provide us with a convenient way to control the error terms in terms of the quantities $\mathbf{D}(r)$ (and $\mathbf{H}(r)$). We stress again, this works because we have built into our center manifold the failure of the decay of tilt and height.

\begin{lemma}\label{lemma:cm-19}
	Fix $\delta>0$ and suppose $\mathbf{I}(r)\geq\delta$. Then,
	$$\Etilt_V^2\sum_i\ell(L_i)^{n+2}\inf_{\mathcal{B}_{L_i}}\phi_r \leq Cr^{n-2}\mathbf{D}(r),$$
	$$\Etilt_V^2\sum_i\ell(L_i)^{n+2}\leq Cr^{n-2}\left((n-1)\mathbf{D}(r) + r\mathbf{D}^\prime(r)\right).$$
	Moreover,
	$$\Etilt_V^{\frac{2}{n+3}}\sup_i \ell(L_i) \leq C\left(r^{n-2}\mathbf{D}(r)\right)^{\frac{1}{n+3}}.$$
	Here, $C = C(\mathcal{V},\delta)\in (0,\infty)$.
\end{lemma}

\begin{proof}
	First note that if $L_i\in\S_h$ then necessarily $L_i\cap B\neq\emptyset$, and so $s(L_i)\leq 2\sqrt{n}\ell(L)$. Hence, $B(L_i)\subset B^{L_i}\subset B_{2\sqrt{n}\ell(L_i)}(x_{L_i},P_0)$. In particular, Lemma \ref{lemma:cm-15}(ii) gives that $|N_f|\geq\frac{1}{4}C_h\Etilt_V\ell(L_i)$ on $\mathcal{B}_{L_i}$. Hence, if $L_i\in\S_h$, we have (as we may assume $C_h\geq 1$)
	$$\int_{\mathcal{B}_{L_i}}\phi_r|N|^2\geq C\inf_{\mathcal{B}_{L_i}}\phi_r\cdot \Etilt_V^2\ell(L_i)^{n+2}.$$
	Moreover, if $L_i\in \S_e$, then since we know $s(L_i)\leq 6\sqrt{n}\ell(L_i)$, we know that the center of $B(L_i)$ is at a distance at most $6\sqrt{n}\ell(L_i)$ from $L_i$, and hence we may apply Lemma \ref{lemma:cm-17} with the ball $B(L_i)$. This therefore gives, for any $L_i\in \S_e$,
	$$\int_{\mathcal{B}_{L_i}}\phi_r|DN|^2\geq C\inf_{\mathcal{B}_{L_i}}\phi_r\cdot \Etilt_V^2\ell(L_i)^{n+\frac{11}{6}}.$$
	In particular, we have for every $i$,
	\begin{equation}\label{E:cm-scale-bound}
		C\Etilt_V^2\ell(L_i)^{n+2}\inf_{\mathcal{B}_{L_i}}\phi_r \leq \max\left\{\int_{\mathcal{B}_{L_i}}\phi_r|DN|^2,\, \int_{\mathcal{B}_{L_i}}\phi_r|N|^2\right\}.
	\end{equation}
	If we now sum the above estimate over $i$, using the fact that $\mathcal{B}_{L_i}$ are pairwise disjoint, we therefore get
    $$C\Etilt_V^2 \sum_i\ell(L_i)^{n+2} \inf_{\mathcal{B}_{L_i}}\phi_r \leq \int\phi_r|DN|^2 + \int\phi_r|N|^2$$
	and so if $\mathbf{I}(r)\geq\delta$, then from Lemma \ref{lemma:cm-18} we get
	$$\Etilt_V^2\sum_i\ell(L_i)^{n+2}\inf_{\mathcal{B}_{L_i}}\phi_r\leq Cr^{n-2}\mathbf{D}(r)$$
	which is the first claimed inequality. For the second, note that by the same argument as above (just without $\phi_r$) we get for each $L_i\in \S_e\cup\S_h$,
	$$C\Etilt_V^2\ell(L_i)^{n+2}\leq\max\left\{\int_{\mathcal{B}_{L_i}}|DN|^2,\, \int_{\mathcal{B}_{L_i}}|N|^2\right\}.$$
	Summing this over $i$ and using Lemma \ref{lemma:cm-18} (using again $\mathbf{I}(r)\geq\delta$) we get
	$$C\Etilt_V^2\sum_i\ell(L_i)^{n+2}\leq\int_{\mathcal{B}_r(0)}|DN|^2 + \int_{\mathcal{B}_r(0)}|N|^2 \leq Cr^{n-2}\left((n-1)\mathbf{D}(r) + r\mathbf{D}^\prime(r)\right)$$
	giving the second claimed inequality. For the final one, return to \eqref{E:cm-scale-bound} which gives, using Lemma \ref{lemma:cm-18} with $\mathbf{I}(r)\geq\delta$ for the right-hand side and using $\inf_{\mathcal{B}_{L_i}}\phi_r\geq \frac{\ell(L_i)}{4r}$ on the left-hand side:
	$$\Etilt_V^2\ell(L_i)^{n+3}\leq Cr\cdot r^{n-2}\mathbf{D}(r).$$
	Raising this inequality to the power of $\frac{1}{n+3}$ and taking the supremum over $i$ gives the result.
\end{proof}

We can now prove the desired error bounds. Throughout, we will assume that $\mathbf{I}(r)\geq\delta$ for some constant $\delta>0$.

Notice first that, since $\|D\wp\|_{C^{2,1/2}} \leq C\Etilt_V$, we have that $\|H_{\mathcal{M}}\|_\infty + \|DH_{\mathcal{M}}\|_\infty\leq C\Etilt_V$. Therefore, as $\bigcup_i \mathcal{U}_{L_i} = \mathbf{B}_r(0)$ and the $\mathcal{U}_{L_i}$ are disjoint, using the smallness in $L^1$ of $N_{\text{a}}$ from \eqref{E:cm-average-bound}, and the bound \eqref{E:phi-sup-inf}, as well as Lemma \ref{lemma:cm-19}, we get
\begin{align*}
	|\text{Err}_1^{\text{o}}| & \leq Q\int_{\mathcal{M}}\phi_r|H_{\mathcal{M}}||N_{\text{a}}|\\
	& \leq C\Etilt_V\sum_i\int_{\mathcal{U}_{L_i}}\phi_r|N_{\text{a}}|\\
	& \leq C\Etilt_V\sum_i \sup_{\mathcal{U}_{L_i}}\phi_r \cdot \Etilt_V^3\ell(L_i)^{n+2+\frac{5}{6}}\\
	& \leq C\Etilt_V^4\sum_i\inf_{\mathcal{B}_{L_i}}\phi_r \cdot \ell(L_i)^{n+2+\frac{5}{6}}\\
	& \leq C\Etilt_V^2\cdot r^{n-2}\mathbf{D}(r)\cdot\sup_i\ell(L_i)^{\frac{5}{6}}\\
	& \leq C\Etilt_V\cdot\left(r^{n-2}\mathbf{D}(r)\right)^{1+\frac{5}{6(n+3)}}.
\end{align*}
Analogously, we have
\begin{align*}
	|\text{Err}_1^{\text{in}}| & \leq C(n)Qr^{-1}\int_{\mathcal{M}}\left(|H_{\mathcal{M}}| + |D_YH_{\mathcal{M}}|\right)|N_{\text{a}}|\\
	& \leq C\Etilt_Vr^{-1}\sum_i \Etilt_V^3\ell(L_i)^{n+2+\frac{5}{6}}\\
	& \leq C\Etilt_V\cdot r^{-1}\cdot r^{n-2}\left((n-1)\mathbf{D}(r)+r\mathbf{D}^\prime(r)\right)\cdot\left(r^{n-2}\mathbf{D}(r)\right)^{\frac{5}{6(n+3)}}.
\end{align*}
These are the desired bounds on the first error terms; again, we stress that controlling these ``linear'' error terms was precisely the reason we had to construct a center manifold in the first place. The other error terms are in some sense ``higher order'' and are in theory easier to deal with, or the same order and can be incorporated into the frequency argument.

For the second error terms, note that $\|A\|_\infty\leq C\|D\wp\|_{C^2}\leq C\Etilt_V$. It follows immediately that
$$|\text{Err}_2^{\text{o}}|\leq C\Etilt_V^2\cdot r^n\boldsymbol{\Sigma}(r).$$
Moreover, as $|DX_{\text{in}}|\leq Cr^{-1}\one_{\mathcal{B}_r(0)\setminus\mathcal{B}_{r/2}(0)}$, we have (using Cauchy--Schwarz and Lemma \ref{lemma:cm-18})
$$|\text{Err}_2^{\text{in}}|\leq C\Etilt_V^2r^{-1}\int_{\mathcal{B}_r(0)}|N|^2 + C\Etilt_V^2\int\phi_r|N||DN|\leq C\Etilt_V^2r^{n-1}\mathbf{D}(r).$$
Finally we bound the errors of the third type. We have:
$$|\text{Err}_3^{\text{o}}|\leq C\underbrace{\int\phi_r\left(\Etilt_V|DN|^2|N| + |DN|^4\right)}_{=:\,I_1} + C\underbrace{r^{-1}\int_{\mathcal{B}_r(0)}|DN|^3|N|}_{=:\,I_2} + C\underbrace{r^{-1}\int_{\mathcal{B}_r(0)}\Etilt_V|N|^2|DN|}_{=:\,I_3}.$$
We estimate each term separately using \eqref{E:cm-N-bound}, \eqref{E:cm-DN-bound}, \eqref{E:phi-sup-inf}, and Lemma \ref{lemma:cm-18}, similarly to those for the errors of the first type:
$$I_1\leq C\Etilt_V^4\sum_i\sup_{\mathcal{U}_{L_i}}\phi_r\cdot\ell(L_i)^{n+2+\frac{5}{6}} \leq C\Etilt_V\left(r^{n-2}\mathbf{D}(r)\right)^{1+\frac{5}{6(n+3)}};$$
$$I_2\leq Cr^{-1}\sum_i\Etilt_V^4\ell(L_i)^{n+3+\frac{3}{4}} \leq C\Etilt_V^4\sum_i\ell(L_i)^{n+2+\frac{3}{4}}\inf_{\mathcal{B}_{L_i}}\phi_r \leq C\Etilt_V\left(r^{n-2}\mathbf{D}(r)\right)^{1+\frac{3}{4(n+3)}};$$
\begin{align*}
I_3\leq Cr^{-1}\sum_i\Etilt_V^2\ell(L_i)^{\frac{11}{12}}\int_{\mathcal{U}_{L_i}}|N|^2 & \leq C\Etilt_V r^{-1}\left(r^{n-2}\mathbf{D}(r)\right)^{\frac{11}{12(n+3)}}\int_{\mathcal{B}_r(0)}|N|^2\\
& \leq C\Etilt_V\left(r^{n-2}\mathbf{D}(r)\right)^{1+\frac{11}{12(n+3)}}.
\end{align*}
Thus combining these estimates, we see that
$$|\text{Err}_3^{\text{o}}|\leq C\left(r^{n-2}\mathbf{D}(r)\right)^{1+\frac{3}{4(n+3)}}.$$
Finally, for the inner variations we have
$$|\text{Err}_3^{\text{in}}|\leq C\Etilt_V\int\phi_r|N||DN|^2 + C\Etilt_V\int\phi_r|DN|^3 + C\Etilt_Vr^{-1}\int_{\mathcal{B}_r(0)}|N|^2|DN| + Cr^{-1}\int_{\mathcal{B}_r(0)}|DN|^4.$$
The first term here is exactly the first term in $I_1$ above, and the third term is exactly $I_3$ above; thus, they are both controlled by $C\left(r^{n-2}\mathbf{D}(r)\right)^{1+\frac{3}{4(n+3)}}$. We then have
$$\Etilt_V\int\phi_r|DN|^3\leq C\Etilt_V^4\sum_i\sup_{\mathcal{U}_{L_i}}\phi_r\cdot\ell(L_i)^{n+2+\frac{3}{4}}\leq C\Etilt_V\left(r^{n-2}\mathbf{D}(r)\right)^{1+\frac{3}{4(n+3)}}$$
bounding as above, and
\begin{align*}
r^{-1}\int_{\mathcal{B}_r(0)}|DN|^4 \leq C\Etilt_V^4r^{-1}\sum_i\ell(L_i)^{n+3+\frac{2}{3}} & \leq C\Etilt_V^4\sum_i\ell(L_i)^{n+2+\frac{2}{3}}\inf_{\mathcal{B}_{L_i}}\phi_r\\
& \leq C\Etilt_V\left(r^{n-2}\mathbf{D}(r)\right)^{1+\frac{2}{3(n+3)}}
\end{align*}
bounding as we did in $I_2$. So, we see that
$$|\text{Err}_3^{\text{in}}|\leq C\Etilt_V\left(r^{n-2}\mathbf{D}(r)\right)^{1+\frac{2}{3(n+3)}}.$$
To conclude, we have:
\begin{equation}\label{E:cm-err-0-bounds}
r^{2-n}\sum^3_{j=1}|\text{Err}_j^{\text{o}}|\leq C\Etilt_V\mathbf{D}(r)^{1+\frac{2}{3(n+3)}} + C\Etilt_V^2r^2\boldsymbol{\Sigma}(r),
\end{equation}
$$r^{2-n}\sum^3_{j=1}|\text{Err}_j^{\text{in}}|\leq C\Etilt_V\mathbf{D}(r)^{\frac{2}{3(n+3)}}\mathbf{D}^\prime(r) + C\Etilt_Vr^{-1}\mathbf{D}(r)^{1+\frac{2}{3(n+3)}} + C\Etilt_V^2 r\mathbf{D}(r).$$
(Note that we do not know the sign of $\mathbf{D}^\prime(r)$ meaning we have been imprecise here with the constants, however we omit this detail as it is a minor point.)

To summarise, we have now shown:
\begin{equation}\label{E:cm-H-prime-3}
	|\mathbf{H}^\prime(r)-2r^{-1}\mathbf{E}(r)|\leq C\Etilt_V\mathbf{H}(r);
\end{equation}
\begin{equation}\label{E:cm-D-E-3}
	|\mathbf{D}(r)-\mathbf{E}(r)|\leq C\Etilt_V\mathbf{D}(r)^{1+\frac{2}{3(n+3)}} + C\Etilt_V^2r^2\boldsymbol{\Sigma}(r);
\end{equation}
\begin{equation}\label{E:cm-D-prime-3}
	|\mathbf{D}^\prime(r)-2\mathbf{G}(r)|\leq C\Etilt_V\mathbf{D}(r) + C\Etilt_Vr^{-1}\mathbf{D}(r)^{1+\frac{2}{3(n+3)}} + C\Etilt_V\mathbf{D}(r)^{\frac{2}{3(n+3)}}\mathbf{D}^\prime(r).
\end{equation}
We stress that these bounds are currently under the assumption that $\mathbf{I}(r)\geq\delta$, with the constant $C = C(\mathcal{V},\delta)$. We will remove this assumption shortly with an appropriate choice of $\delta$.

\begin{remark}\label{remark:cm-H-non-zero}
	We can now show that $\mathbf{H}(r)>0$ is always true. Indeed, if $\mathbf{H}(r)=0$ for some $r$, then $|N|\equiv 0$ on $\mathcal{B}_r(0)\setminus\mathcal{B}_{r/2}(0)$, which in turn gives that $\mathbf{E}(r)=0$. In fact, if $\mathbf{H}(r)=0$, then \emph{all} conclusions of Lemma \ref{lemma:cm-18} go through without any assumption on a lower bound for $\mathbf{I}(r)$, and thus all the above computations hold. Thus, \eqref{E:cm-D-E-3} still holds (which we note comes from the analogous variation as that used in the usual `squash inequality', which gives the equivalent conclusion for the planar frequency) and hence we would have
$$\mathbf{D}(r)\leq C\Etilt_V\mathbf{D}(r)^{1+\frac{2}{3(n+3)}} + C\Etilt_V^2r^2\boldsymbol{\Sigma}(r).$$
But we also have from the first inequality in Lemma \ref{lemma:cm-18} (as $\mathbf{H}(r)=0$) that $\boldsymbol{\Sigma}(r)\leq C\mathbf{D}(r)$, and hence we get
$$\mathbf{D}(r)\leq C\Etilt_V\mathbf{D}(r)^{1+\frac{2}{3(n+3)}} + C\Etilt_V^2r^2\mathbf{D}(r).$$
Thus, if $\Etilt_V$ is sufficiently small, rearranging this would give that $\mathbf{D}(r)\leq 0$, i.e. $\mathbf{D}(r)=0$, and so $|DN|=0$ on $\mathcal{B}_r(0)$. But as $|N|\equiv 0$ on $\mathcal{B}_r(0)\setminus\mathcal{B}_{r/2}(0)$, this would mean that $|N|\equiv 0$ on $\mathcal{B}_{r}(0)$, i.e.~$V$ coincides with $Q\llbracket \mathcal{B}_r(0)\rrbracket$ here. But as $\mathcal{M}$ was a $C^{3,1/2}$ manifold, this would contradict $0\in\sing(V)$. Thus we have reached the desired contradiction, and so $\mathbf{H}(r)>0$ for all $r>0$. In particular, $\mathbf{I}(r)$ is well-defined for all $r>0$ (and again we stress this is true without assuming any lower bound on $\mathbf{I}(r)$).
\end{remark}

We can now prove the first fact about the center manifold frequency, which is a growth-bound. In particular, this gives that $\sup_{r}\mathbf{I}(r)<\infty$. We will use this in turn to establish a uniform lower bound on $\mathbf{I}(r)$, i.e.~$\inf_r\mathbf{I}(r)>0$, giving our choice of $\delta$ above, and then finally we will use this to show that $\mathbf{I}(r)$ is approximately monotone in $r$.

\begin{theorem}[Frequency Growth Bound]\label{thm:cm-20}
	If $\eps_2 = \eps_2(\mathcal{V})\in (0,1)$ is sufficiently small, then for all $0<a\leq b<1$,
	$$\mathbf{I}(a)\leq e^{C\Etilt_Vb^{\frac{4}{3(n+3)}}}(1+\mathbf{I}(b)).$$
	Here, $C = C(\mathcal{V})$.
\end{theorem}

\textbf{Remark:} Clearly Theorem \ref{thm:cm-20} gives $\sup_{r\in (0,1)}\mathbf{I}(r)\leq e^{C\Etilt_V}(1+\mathbf{I}(1))<\infty$.

\begin{proof}
	Set $\boldsymbol{\Omega}(r):=\log(\max\{\mathbf{I}(r),1\})$. To prove the theorem, it suffices to show that
	\begin{equation*}\tag{$\star$}\label{E:Omega}
		\boldsymbol{\Omega}(a)\leq C\Etilt_Vb^{\frac{4}{3(n+3)}}+\boldsymbol{\Omega}(b)
	\end{equation*}
	for some $C = C(\mathcal{V})$. Indeed, once we have this we can take exponentials of both sides to get
	$$\max\{\mathbf{I}(a),1\}\leq e^{C\Etilt_Vb^{\frac{4}{3(n+3)}}}\max\{\mathbf{I}(b),1\}\ \ \Longrightarrow\ \ \mathbf{I}(a)\leq e^{C\Etilt_Vb^{\frac{4}{3(n+3)}}}(1+\mathbf{I}(b))$$
	as desired.
	
	To prove \eqref{E:Omega}, first note that if $\boldsymbol{\Omega}(a)=0$ then there is nothing to prove, and so we may assume that $\boldsymbol{\Omega}(a)>0$. Then set $\tilde{b}:=\sup\{t\in (a,b]:\boldsymbol{\Omega}>0\text{ on }(a,t)\}$, i.e.~$\tilde{b}$ is the first zero of $\boldsymbol{\Omega}$ greater than $a$ and at most $b$; if this set is non-empty and $\tilde{b}<b$, then $\boldsymbol{\Omega}(\tilde{b}) = 0\leq\boldsymbol{\Omega}(b)$, and so if we can prove $\boldsymbol{\Omega}(a)\leq C\Etilt_V\tilde{b}^{\frac{4}{3(n+3)}}+\boldsymbol{\Omega}(\tilde{b})$, the result follows. In particular, it suffices to work on $(a,\tilde{b})$ in this case, where $\boldsymbol{\Omega}>0$. In the other cases, we know that $\boldsymbol{\Omega}>0$ on $(a,b)$. Thus, we may without loss of generality assume that $\boldsymbol{\Omega}>0$ on $(a,b)$, i.e. $\mathbf{I}(r)>1$, as this case will imply the full result.
	
	Fix $r\in (a,b)$. As $\mathbf{I}(r)>1$, all of our previous estimates where we assumed a uniform lower bound $\mathbf{I}(r)\geq\delta$ now hold with $\delta=1$; in particular, the constants only depend on $\mathcal{V}$. From \eqref{E:cm-D-E-3} and the fact that $\boldsymbol{\Sigma}(r)\leq C\mathbf{D}(r)$ from Lemma \ref{lemma:cm-18}, we know
	$$|\mathbf{D}(r)-\mathbf{E}(r)|\leq C\Etilt_V\left(\mathbf{D}(r)^{\frac{2}{3(n+3)}} + C\Etilt_V r^2\right)\mathbf{D}(r)$$
	and so if $\Etilt_V$ is sufficiently small depending only on $\mathcal{V}$, this gives
	$$|\mathbf{D}(r)-\mathbf{E}(r)|\leq\frac{1}{2}\mathbf{D}(r)$$
	and hence
	\begin{equation}\label{E:D-E-comparison}
		\frac{1}{2}\mathbf{D}(r)\leq\mathbf{E}(r)\leq 2\mathbf{D}(r).
	\end{equation}
	From this we conclude, as $\mathbf{I}(r)>1>0$ (and thus $\mathbf{D}(r)>0$) that $\mathbf{E}(r)>0$ for $r\in (a,b)$. Now, as $\boldsymbol{\Omega}(r) = \log\mathbf{I}(r) = \log\mathbf{D}(r) - \log\mathbf{H}(r)$ here, we can compute
	\begin{align*}
		\boldsymbol{\Omega}^\prime(r) & = \frac{\mathbf{D}^\prime(r)}{\mathbf{D}(r)} - \frac{\mathbf{H}^\prime(r)}{\mathbf{H}(r)}\\
		& = \frac{\mathbf{D}^\prime(r)}{\mathbf{E}(r)} - \frac{\mathbf{H}^\prime(r)}{\mathbf{H}(r)} + \frac{\mathbf{D}^\prime(r)(\mathbf{E}(r)-\mathbf{D}(r))}{\mathbf{E}(r)\mathbf{D}(r)}.
	\end{align*}
	From \eqref{E:cm-H-prime-3} we know
	$$\frac{\mathbf{H}^\prime(r)}{\mathbf{H}(r)}\leq \frac{2r^{-1}\mathbf{E}(r)}{\mathbf{H}(r)} + C\Etilt_V$$
	and from \eqref{E:cm-D-prime-3} we know
	\begin{align*}
		\frac{\mathbf{D}^\prime(r)}{\mathbf{E}(r)} & \geq\frac{2\mathbf{G}(r)}{\mathbf{E}(r)} - C\Etilt_V\frac{\mathbf{D}(r)}{\mathbf{E}(r)} - C\Etilt_Vr^{-1}\frac{\mathbf{D}(r)^{1+\frac{2}{3(n+3)}}}{\mathbf{E}(r)} - C\Etilt_V\frac{\mathbf{D}(r)^{\frac{2}{3(n+3)}}}{\mathbf{E}(r)}\mathbf{D}^\prime(r)\\
		& \geq \frac{2\mathbf{G}(r)}{\mathbf{E}(r)} - C\Etilt_V - C\Etilt_Vr^{-1}\mathbf{D}(r)^{\frac{2}{3(n+3)}} - C\Etilt_V\mathbf{D}(r)^{\frac{2}{3(n+3)}-1}\mathbf{D}^\prime(r)
	\end{align*}
	where in the second inequality here we have used \eqref{E:D-E-comparison}. Finally, since from \eqref{E:cm-D-E-3}
	$$|\mathbf{D}(r)-\mathbf{E}(r)|\leq C\Etilt_V\mathbf{D}(r)^{1+\frac{2}{3(n+3)}} + C\Etilt_V^2r^2\boldsymbol{\Sigma}(r)$$
	we have, again using \eqref{E:D-E-comparison} in the denominator,
	\begin{align*}
		\frac{\mathbf{D}^\prime(r)(\mathbf{E}(r)-\mathbf{D}(r))}{\mathbf{E}(r)\mathbf{D}(r)} & \geq -C\Etilt_V\frac{\mathbf{D}^\prime(r)(\mathbf{D}(r)^{1+\frac{2}{3(n+3)}}+r^2\boldsymbol{\Sigma}(r))}{\mathbf{D}(r)^2}\\
		& \geq -C\Etilt_V\mathbf{D}(r)^{\frac{2}{3(n+3)}-1}\mathbf{D}^\prime(r) - Cr^2\Etilt_V\frac{\mathbf{D}^\prime(r)\boldsymbol{\Sigma}(r)}{\mathbf{D}(r)^2}.
	\end{align*}
	Combining the above inequalities, we get
	\begin{align*}
		\boldsymbol{\Omega}^\prime(r) \geq 2\cdot\frac{\mathbf{G}(r)\mathbf{H}(r)-r^{-1}\mathbf{E}(r)^2}{\mathbf{E}(r)\mathbf{H}(r)} - C\Etilt_V - C&\Etilt_Vr^{-1}\mathbf{D}(r)^{\frac{2}{3(n+3)}}\\
		& - C\Etilt_V\mathbf{D}(r)^{\frac{2}{3(n+3)}-1}\mathbf{D}^\prime(r) - Cr^2\Etilt_V\frac{\mathbf{D}^\prime(r)\boldsymbol{\Sigma}(r)}{\mathbf{D}(r)^2}.
	\end{align*}
	By Cauchy--Schwarz, we have $\mathbf{G}(r)\mathbf{H}(r)\geq r^{-1}\mathbf{E}(r)^2$, and hence, noting that $\mathbf{D}(r)\leq C\Etilt_V^2r^2$ (as $0$ is a point of planar frequency $\geq 2$)
	$$\boldsymbol{\Omega}^\prime(r) \geq -C\Etilt_V-C\Etilt_Vr^{\frac{4}{3(n+3)}-1} - C\Etilt_V\mathbf{D}(r)^{\frac{2}{3(n+3)}-1}\mathbf{D}^\prime(r) - C\Etilt_V\frac{\mathbf{D}^\prime(r)\cdot r^2\boldsymbol{\Sigma}(r)}{\mathbf{D}(r)^2}.$$
	Integrating this over $(a,b)$ (and integrating by parts in the last term), we get
	\begin{align*}
		\boldsymbol{\Omega}(b)-\boldsymbol{\Omega}(a)\geq -C\Etilt_V(b-a)-C&\Etilt_V\left(b^{\frac{4}{3(n+3)}} - a^{\frac{4}{3(n+3)}}\right) - C\Etilt_V\left(\mathbf{D}(b)^{\frac{2}{3(n+3)}} - \mathbf{D}(a)^{\frac{2}{3(n+3)}}\right)\\
		& - C\Etilt_V\left[-\frac{b^2\boldsymbol{\Sigma}(b)}{\mathbf{D}(b)} + \frac{a^2\boldsymbol{\Sigma}(a)}{\mathbf{D}(a)} + \int^b_a\frac{r^2\boldsymbol{\Sigma}^\prime(r) + 2r\boldsymbol{\Sigma}(r)}{\mathbf{D}(r)}\ \ext r\right].
	\end{align*}
	But now, recalling from Lemma \ref{lemma:cm-18} that $\boldsymbol{\Sigma}(r)\leq C\mathbf{D}(r)$ and $r\boldsymbol{\Sigma}^\prime(r)\leq C\mathbf{D}(r)$, this gives
	$$\boldsymbol{\Omega}(b)-\boldsymbol{\Omega}(a)\geq -C\Etilt_Vb^{\frac{4}{3(n+3)}}$$
	(where again we have used $\mathbf{D}(r)\leq C\Etilt_V^2r^2$); here, $C = C(\mathcal{V})$. But this is exactly \eqref{E:Omega}, and hence we have completed the proof.
\end{proof}

Next, we use Theorem \ref{thm:cm-20} to establish a fixed lower bound on the center manifold frequency. In fact, we will be able to show that if $\eps_2 = \eps_2(\mathcal{V})$ is sufficiently small, then $\mathbf{I}(r)\geq 1/2$ for all $r$ (the choice of a lower bound is $1/2$ is arbitrary, and a lower bound of $1-\eta$ could be achieved provided $\eps_2$ is allowed to depend on $\eta\in (0,1)$ as well).  

\begin{corollary}[Uniform Frequency Lower Bound]\label{cor:cm-21}
	Provided $\eps_2 = \eps_2(\mathcal{V})$ is sufficiently small, we have $\mathbf{I}(r)\geq 1/2$ for all $r>0$.
\end{corollary}

\begin{proof}
    Firstly, note that $\mathbf{I}(r)>0$ for all $r>0$: indeed, we have already shown that $\mathbf{H}(r)>0$ always, and by a similar argument\footnote{To see that $\mathbf{D}(r)>0$ for all $r>0$, suppose for some $r>0$ we had $\mathbf{D}(r)=0$. Then we would have $|DN|=0$ on $\mathcal{B}_r(0)$. But then as $0\in\mathcal{M}\cap\spt\|V\|$ (from Lemma \ref{lemma:cm-12}) we know $|N(0)|=0$, and thus we would need to have $|N|\equiv 0$ on $\mathcal{B}_r(0)$, i.e.~$V$ agrees with $Q\llbracket \mathcal{B}_r(0)\rrbracket$ here. Hence, as $\mathcal{M}$ was $C^{3,1/2}$, we see that $0\in\reg(V)$, a contradiction.} we have that $\mathbf{D}(r)>0$. Moreover, $\mathbf{I}(r)$ is clearly continuous in $r>0$.

    Now suppose for contradiction that there was a radius $\tilde{r}>0$ such that $\mathbf{I}(\tilde{r})<1/2$. Then provided $\eps_2 = \eps_2(\mathcal{V})\in (0,1)$ is sufficiently small, Theorem \ref{thm:cm-20} gives that for all $0<r\leq \tilde{r}<1$,
    $$\mathbf{I}(r) \leq e^{C\Etilt_V \tilde{r}^{\frac{4}{3(n+3)}}}(1+\mathbf{I}(\tilde{r})) < e^{C\Etilt_V}\left(1+\frac{1}{2}\right) < \frac{7}{4}.$$
    We just need to contradict this to complete the proof. To this end, we return to our first variation estimates (the reader should note the similarity of the present argument with that in Corollary \ref{cor:decay-estimates} for planar frequency). From \eqref{E:H-prime-1} and \eqref{E:D-E} we see
	$$\mathbf{H}^\prime(r) - 2r^{-1}\mathbf{D}(r) = O(1)\Etilt_V\mathbf{H}(r) - 2r^{-1}\sum^3_{j=1}r^{2-n}\text{Err}_j^{\text{o}}.$$
	Dividing this by $\mathbf{H}(r)$ $(>0)$ we get
	$$\left|\frac{\mathbf{H}^\prime(r)}{\mathbf{H}(r)} - \frac{2\mathbf{I}(r)}{r}\right| \leq C\Etilt_V + 2r^{-1}\sum^3_{j=1}\frac{r^{2-n}|\text{Err}_j^{\text{o}}|}{\mathbf{H}(r)}.$$
	We now want to estimate the error terms $\text{Err}_j^{\text{o}}$ in exactly the same way as before, \emph{but without using any assumption regarding a lower bound on }$\mathbf{I}(r)$. For example, in Lemma \ref{lemma:cm-19}, the best bound we can get is
	$$\Etilt_V^2\sum_i\ell(L_i)^{n+2}\inf_{\mathcal{B}_{L_i}}\phi_r\leq Cr^{n-2}\mathbf{D}(r) + Cr^n\boldsymbol{\Sigma}(r).$$
	Thus, using the first inequality in Lemma \ref{lemma:cm-18}, we get
	$$\Etilt_V^2\sum_i\ell(L_i)^{n+2}\inf_{\mathcal{B}_{L_i}}\phi_r\leq Cr^{n-2}\mathbf{D}(r) + Cr^n\mathbf{H}(r).$$
	Similarly, we get
	$$\Etilt_V^{\frac{2}{n+3}}\sup_i\ell(L_i) \leq C\left(r^{n-2}\mathbf{D}(r)\right)^{\frac{1}{n+3}} + C\left(r^n\mathbf{H}(r)\right)^{\frac{1}{n+3}}.$$
	Here, $C = C(\mathcal{V})$. Using these bounds instead in our previous estimates on $\text{Err}_j^{\text{o}}$, we get
	$$\sum^3_{j=1}r^{2-n}|\text{Err}_j^{\text{o}}|\leq C\Etilt_Vr^2(\mathbf{D}(r) + \mathbf{H}(r)) + C\Etilt_V\left(\mathbf{D}(r)^{\frac{3}{4(n+3)}}+\mathbf{H}(r)^{\frac{3}{4(n+3)}}\right)(\mathbf{D}(r)+r^2\mathbf{H}(r))$$
	and thus we get
	$$\sum^3_{j=1}\frac{r^{2-n}|\text{Err}_j^{\text{o}}|}{\mathbf{H}(r)} \leq C\Etilt_Vr^2\mathbf{I}(r) + C\Etilt_Vr^2 + C\Etilt_V\left(\mathbf{D}(r)^{\frac{3}{4(n+3)}} + \mathbf{H}(r)^{\frac{3}{4(n+3)}}\right)(\mathbf{I}(r)+r^2)$$
	i.e.
	$$2r^{-1}\sum^3_{j=1}\frac{r^{2-n}|\text{Err}_j^{\text{o}}|}{\mathbf{H}(r)} \leq C\Etilt_V r\mathbf{I}(r) + C\Etilt_V r + C\Etilt_V\left(\mathbf{D}(r)^{\frac{3}{4(n+3)}} + \mathbf{H}(r)^{\frac{3}{4(n+3)}}\right)\left(\frac{\mathbf{I}(r)}{r} + r\right).$$
    Now fix $\eps_*>0$ to be chosen. Then, choose $r_* = r_*(\mathcal{V},\eps_*)\in (0,\tilde{r})$ sufficiently small so that for all $r\in (0,r_*)$ we have $C\mathbf{D}(r)^{\frac{3}{4(n+3)}}<\eps_*$, $C\mathbf{H}(r)^{\frac{3}{4(n+3)}}<\eps_*$, and $Cr_*<\eps_*$, where $C = C(\mathcal{V})$ is the constant from the above inequality (note that this is possible because $\mathbf{D}(r)\leq C\Etilt_V^2 r^2$ and $\mathbf{H}(r)\leq C\Etilt^2_Vr^4$, as $0$ is a point of planar frequency $\geq 2$, using also our estimates on $\wp$). The above bound therefore becomes
	$$2r^{-1}\sum^3_{j=1}\frac{r^{2-n}|\text{Err}_j^{\text{o}}|}{\mathbf{H}(r)} \leq 3\eps_*\cdot\frac{\mathbf{I}(r)}{r} + 3\eps_*$$
	and hence we see
	$$\left|\frac{\mathbf{H}^\prime(r)}{\mathbf{H}(r)} - \frac{2\mathbf{I}(r)}{r}\right|\leq C + 3\eps_*\cdot\frac{\mathbf{I}(r)}{r}.$$
	Thus, we get
	$$\frac{\mathbf{H}^\prime(r)}{\mathbf{H}(r)}\leq C + (2+3\eps_*)\cdot\frac{\mathbf{I}(r)}{r}.$$
	Now, recall that we are supposing (as a result of our contradiction assumption) that $\mathbf{I}(r)<7/4$ for all $r\in (0,r_*)$ (as $r_*<\tilde{r}$). Hence, we know
	$$\frac{\mathbf{H}^\prime(r)}{\mathbf{H}(r)} \leq C + \frac{\frac{7}{4}(2+3\eps_*)}{r} \qquad \text{for all }r\in (0,r_*].$$
    Integrating this over $(r,r_*)$, we get
	$$\log\left(\frac{\mathbf{H}(r_*)}{\mathbf{H}(r)}\right)\leq Cr_* + \frac{7}{4}(2+3\eps_*)\log\left(\frac{r_*}{r}\right)$$
	i.e.
	$$\mathbf{H}(r_*)\leq e^{Cr_*}\mathbf{H}(r)\cdot\left(\frac{r_*}{r}\right)^{\frac{7}{4}(2+3\eps_*)} \qquad \text{for all }r\in (0,r_*).$$
	In particular, we must have
	$$\frac{\mathbf{H}(r)}{r^{\frac{7}{4}(2+3\eps_*)}}\geq e^{-Cr_*}\mathbf{H}(r_*)r_*^{-\frac{7}{4}(2+3\eps_*)}>0 \qquad \text{for all }r\in (0,r_*)$$
	i.e. $\liminf_{r\downarrow 0}\frac{\mathbf{H}(r)}{r^{\frac{7}{4}(2+3\eps_*)}}>0$. However, we know that $\mathbf{H}(r)\leq C\Etilt_V^2r^{4}$ for all $r\in (0,1)$ (again, as the origin is a point of planar frequency $\geq 2$, and moreover $|\wp|_{C^3}\leq C\Etilt_V$ with $\wp(0) = 0$, $D\wp(0) = 0$, by Lemma \ref{lemma:cm-12}, and so $|\wp(x)|\leq C\Etilt_V|x|^2$). Thus, if we choose $\frac{7}{4}(2+3\eps_*)<4$, e.g. $\eps_* = 1/21$, we get the desired contradiction and so complete the proof.
\end{proof}

\begin{remark}
	Currently, all our arguments have been based at the origin, i.e.~defining the center manifold frequency at $0$. We can also base oursevles at different points of planar frequency $\geq 2$ and the above arguments will also work, which allows us to define the center manifold frequency function $\mathbf{I}_y(r)$ at different points of planar frequency $\geq 2$. Corollary \ref{cor:cm-21} applies at these other base points as well, provided $\eps_2 = \eps_2(\mathcal{V})\in (0,1)$ is sufficiently small, to provide a uniform lower bound.
\end{remark}

Now note that one can rerun the entire argument from Theorem \ref{thm:cm-20} using the uniform lower bound $1/2$ provided by Corollary \ref{cor:cm-21}. Indeed, in the proof of Theorem \ref{thm:cm-20} instead define $\boldsymbol{\Omega}(r):=\log(\max\{\mathbf{I}(r),1/2\})\equiv\log\mathbf{I}(r)$, and then the whole argument leading to \eqref{E:Omega} goes through, which in turn gives $\mathbf{I}(a)\leq e^{C\Etilt_Vb^{\frac{4}{3(n+3)}}}\mathbf{I}(b)$ for any $0<a\leq b<1$. Thus, we have now established:

\begin{theorem}[Approximate Monotonicity of Center Manifold Frequency]\label{thm:cm-22}
	If $\eps_2 = \eps_2(\mathcal{V})\in (0,1)$ is sufficiently small, then for any $0<a\leq b<1$,
	$$\mathbf{I}(a)\leq e^{C\Etilt_Vb^{\frac{4}{3(n+3)}}}\mathbf{I}(b).$$
	Here, $C = C(\mathcal{V})\in (0,\infty)$.
\end{theorem}

In particular, we now get that the limit $\mathbf{I}(0):=\lim_{r\downarrow 0}\mathbf{I}(r)$ must exist (and is $\geq 1/2>0$). In fact, the argument in Corollary \ref{cor:cm-21} now gives, for any $\eps_*>0$,
$$\liminf_{r\downarrow 0}\frac{\mathbf{H}(r)}{r^{2\mathbf{I}(0)+\eps_*}}>0.$$
But then since $\mathbf{H}(r)\leq C\Etilt_V^2r^4$, we get that we need $\mathbf{I}(0)\geq 2$. We therefore have:
\begin{corollary}\label{cor:cm-23}
	If $\eps_2 = \eps_2(\mathcal{V})\in (0,1)$ is sufficiently small, then $\mathbf{I}(0)\geq 2$.
\end{corollary}

\begin{remark}\label{remark:cm-planar-lower-bounds-cm}
    In fact, one has $\mathbf{I}(0)\geq\mathcal{N}(0)$, i.e.~the center manifold frequency is at least as large as the planar frequency at the point. See Remark \ref{remark:cm+planar-frequency} for the justification. We note that it is entirely possible that $\mathbf{I}(0) = \mathcal{N}(0)$, as is the case at any point where the planar frequency is not an integer.
\end{remark}

\begin{remark}\label{remark:cm-height-decay}
    A more precise re-run of the proof of Corollary \ref{cor:cm-21}, using the lower bound provided by Corollary \ref{cor:cm-21} as well as Lemma \ref{lemma:cm-18} and \eqref{E:cm-err-0-bounds} gives that, provided $\eps_2 = \eps_2(\mathcal{V})\in (0,1)$ is sufficiently small, for any $0<r<R<1$,
    $$\left(\frac{R}{r}\right)^{2(1-o(1))\mathbf{I}(r)} \leq \frac{\mathbf{H}(R)}{\mathbf{H}(r)} \leq \left(\frac{R}{r}\right)^{2(1+o(1))\mathbf{I}(R)},$$
    where $o(1)$ denotes a term which $\to 0$ as $R\to 0$. Moreover, for any $\eps_*>0$, for all $r\in (0,1)$ sufficiently small (depending on $\eps_*$) we also get:
    $$cr^{2\mathbf{I}(0)+2\eps_*}\leq \mathbf{H}(r) \leq Cr^{2\mathbf{I}(0)}$$
    where $c = c(\mathcal{V},\mathbf{I}(1))$ and $C = C(\mathcal{V},\mathbf{I}(1))$.
\end{remark}

\textbf{Note:} Theorem \ref{thm:cm-22} gives that for all $r = r(\mathcal{V})$ sufficiently small, we have (for instance) $\mathbf{I}(r)\geq 3/2$.

These are the key facts we need to know regarding the center manifold frequency function in order to conclude our argument. In particular, the approximate monotonicity of frequency provided by Theorem \ref{thm:cm-22}, rather than the growth control given by Theorem \ref{thm:cm-20}, will allow us to show that blow-ups relative to the center manifold are homogeneous of degree $\mathbf{I}(0)$.

We now prove that a certain reverse Poincaré inequality holds for the normal map $N$ at appropriate scales, thus controlling its full $W^{1,2}$ norm by its $L^2$ norm at those scales.

\begin{corollary}[Center Manifold Reverse Poincaré Inequality]\label{cor:cm-24}
	If $\eps_2 = \eps_2(\mathcal{V})\in (0,1)$ is sufficiently small, then for every $r\in (0,1]$ there is $s\in (\frac{3}{2}r,3r)$ such that
	$$\int_{\mathcal{B}_s(0)}|DN|^2\leq\frac{C}{s^2}\int_{\mathcal{B}_s(0)}|N|^2.$$
	Here, $C = C(\mathcal{V},\mathbf{I}(4))\in (0,\infty)$.
\end{corollary}

\begin{proof}
	From the coarea formula we have
	$$(3r)^{n-1}\mathbf{H}(3r) \equiv \int_{\mathcal{B}_{3r}(0)\setminus\mathcal{B}_{3r/2}(0)}2d(p)^{-1}|N|^2 = 2\int^{3r}_{3r/2}\frac{1}{t}\int_{\del\mathcal{B}_t(0)}|N|^2$$
	whereas, using Fubini's theorem, we have
	$$\int^{3r}_{3r/2}\int_{\mathcal{B}_t(0)}|DN|^2 = \int_{\mathcal{M}}|DN(x)|^2\int^{3r}_{3r/2}\one_{(|x|,\infty)}(t)\ \ext t\ext x = \frac{3r}{2}\cdot (3r)^{n-2}\mathbf{D}(3r) \equiv \frac{1}{2}(3r)^{n-1}\mathbf{D}(3r).$$
	Now, the frequency upper bound $\mathbf{I}(r)\leq C^*$ provided by Theorem \ref{thm:cm-20} gives $\mathbf{D}(3r)\leq C^*\mathbf{H}(3r)$, where $C^* = C^*(\mathcal{V},\mathbf{I}(4))$. Hence, from the above equalities we see
	$$\int^{3r}_{3r/2}\int_{\mathcal{B}_t(0)}|DN|^2\leq 4C^*\int^{3r}_{3r/2}\frac{1}{t}\int_{\del\mathcal{B}_t(0)}|N|^2.$$
	Thus, basic measure theory tells us that there must exist $s\in (\frac{3}{2}r,3r)$ with
	\begin{equation}\label{E:cm-rpi-1}
		\int_{\mathcal{B}_s(0)}|DN|^2\leq \frac{4C^*}{s}\int_{\del\mathcal{B}_s(0)}|N|^2.
	\end{equation}
	Now, fix any $\sigma\in (s/2,s)$ and any point $x\in\del\mathcal{B}_s(0)$. Consider the geodesic line $\gamma\subset\mathcal{M}$ passing through $0$ and $x$, and let $\tilde{\gamma}\subset\gamma$ be the arc having one endpoint on $\del\mathcal{B}_\sigma(0)$ (call this endpoint $\tilde{x}$) and the other endpoint $x$. By the fundamental theorem of calculus (applied to $|N|^2$) we have
	$$|N(x)|^2\leq |N(\tilde{x})|^2 + 2\int_{\tilde{\gamma}}|DN||N|.$$
	Integrating this over $x\in\del\mathcal{B}_s(0)$ we get
	$$\int_{\del\mathcal{B}_s(0)}|N|^2\leq C\int_{\del\mathcal{B}_\sigma(0)}|N|^2 + C\int_{\mathcal{B}_s(0)\setminus\mathcal{B}_{s/2}(0)}|DN||N|$$
	where we have used that $\sigma>s/2$ and so $\mathcal{B}_s(0)\setminus\mathcal{B}_\sigma(0)\subset \mathcal{B}_{s}(0)\setminus\mathcal{B}_{s/2}(0)$ in the last term; here, $C = C(\mathcal{V})$ only depends on the curvature of $\mathcal{M}$ and so is controlled by $\|D\wp\|_{C^2}\leq C(\mathcal{V})$. If we further integrate this over $\sigma\in (s/2,s)$ we get
	$$\frac{s}{2}\int_{\del\mathcal{B}_s(0)}|N|^2\leq C\int_{\mathcal{B}_s(0)\setminus\mathcal{B}_{s/2}(0)}|N|^2 + Cs\int_{\mathcal{B}_s(0)\setminus\mathcal{B}_{s/2}(0)}|DN||N|.$$
	Using $2ab\leq a^2+b^2$ for suitable $a,b\in\R$, we get
	$$\frac{s}{2}\int_{\del \mathcal{B}_s(0)}|N|^2\leq C\int_{\mathcal{B}_s(0)\setminus\mathcal{B}_{s/2}(0)}|N|^2 + \frac{s^2}{16C^*}\int_{\mathcal{B}_s(0)\setminus\mathcal{B}_{s/2}(0)}|DN|^2$$
	where $C^*$ is as in \eqref{E:cm-rpi-1}. Substituting this into \eqref{E:cm-rpi-1} and rearranging, we get
	$$\int_{\mathcal{B}_s(0)}|DN|^2\leq \frac{C}{s^2}\int_{\mathcal{B}_s(0)\setminus\mathcal{B}_{s/2}(0)}|N|^2$$
	which proves the result.
\end{proof}

\begin{remark}\label{remark:cm-general-test}
	We note that if one takes in place of $\phi_r$ a more general test function in the outer and inner variations relative to the center manifold (\eqref{E:cm-outer-variation}, \eqref{E:cm-inner-variation}) we can analogously control the error terms to show:
	$$\int \phi|DN|^2 + \sum^Q_{i=1}N_i  DN_i\cdot D\phi = o(1)\|\phi\|_{C^1}\|N\|^2_{W^{1,2}} \qquad \text{for }\phi\in C^\infty_c(B^n_1(0));$$
	$$\int \left(|DN|^2 - 2D_iND_jN\right)D_i\phi^j = o(1)\|\phi\|_{C^1}\|N\|^2_{W^{1,2}} \qquad \text{for }\phi\in C^\infty_c(B^n_1(0);\R^n);$$
	where $o(1)$ denotes a term which $\to 0$ as $\Etilt_V\to 0$.
\end{remark}

\subsection{Center manifold blow-up and conclusion}\label{sec:cm-blow-up}

We are now in a position to perform our blow-up/linearisation procedure relative to the center manifold $\mathcal{M}$ from Theorem \ref{thm:cm}. Again, we have fixed constants $M_0 = M_0(\mathcal{V})$, $N_0 = N_0(\mathcal{V},M_0)$, $C_e = C_e(\mathcal{V},M_0,N_0)$, $C_h = C_h(\mathcal{V},M_0,N_0,C_e)$, and $\eps_2 = \eps_2(\mathcal{V},M_0,N_0,C_e,C_h)$ such that all the previous results of the section hold. In particular, $\eps_2 = \eps_2(\mathcal{V})$ only depends on $\mathcal{V}$.

Recall that $V\in\mathcal{V}$ is such that $0\in\mathcal{B}_V^{\geq 2}$ is a density $Q$ branch point with planar frequency $\geq 2$ and whose (unique) tangent cone is $Q|P_0|$; this is of course without loss of generality, by translating and rotating $V$. Moreover, by rescaling $V$ we may also assume that at scale $6\sqrt{n}$ the conclusions of the previous sections hold. Let $\mathcal{M}$ be the corresponding center manifold given by Theorem \ref{thm:cm}, with $N$ the associated normal map representing $V$ over $\mathcal{M}$, given by Theorem \ref{thm:cm-14}.

Now, take any sequence $r_j\downarrow 0$, and let $s_j\in (\frac{3}{2}r_j,3r_j)$ be the radius from Corollary \ref{cor:cm-24} corresponding to $r_j$. Consider the rescaled sequences
$$V_j:= (\eta_{0,s_j})_\#V,\qquad \mathcal{M}_j:= \eta_{0,s_j}(\mathcal{M}),$$
and $N_j:\mathcal{M}_j\to \A_Q(\R^{n+k})$ given by $N_j(p):= s_j^{-1}N(s_jp)$. Clearly from the uniqueness of the tangent cone to $V$ at $0$ we have $V_j\weakly Q|P_0|$, and moreover from the bounds in Theorem \ref{thm:cm} (cf.~Lemma \ref{lemma:cm-12}) we know that $\mathcal{M}_j\to P_0$ in $C^{3,1/2}$. Now, we define a sequence of \emph{blow-up maps} $N^b_j:B^n_1(0)\to\A_Q(\R^{n+k})$ by
$$N^b_j(x):=\|N_j\|^{-1}_{L^2(\mathcal{B}_1(0))}\cdot N_j(\exp_j(x))$$
where $\exp_j: B^n_1(0)\to \mathcal{M}_j$ is the exponential map at $0\in \mathcal{M}_j$ (here we have identified the domain $B^n_1(0)$ with the unit ball in $T_0\mathcal{M}_j\cong \R^n$). We also know from the estimates in Theorem \ref{thm:cm} that $\exp_j$ converges in $C^{2,1/2}$ to the identity map $B^n_1(0)\to B^n_1(0)$. 

From the center manifold reverse Poincaré inequality (Corollary \ref{cor:cm-24}) we know that
$$\|N_j^b\|_{W^{1,2}(B_1)}\leq C$$
i.e.~$N_j^b$ is uniformly bounded in $W^{1,2}$. In particular, up to passing to a subsequence, we know that $N^b_j$ converge to a $W^{1,2}$ function $N^b:B^n_1(0)\to \A_Q(\R^k)$, where the convergence is strongly in $L^2_{\text{loc}}(B^n_1(0))$ and weakly in $W^{1,2}_{\text{loc}}(B^n_1(0))$\footnote{By this we mean that if $\boldsymbol{\xi}:\A_Q(\R^k)\to \R^L$ is Almgren's bi-Lipschitz embedding \cite{Alm00} (here, $L = L(n,k,Q)\in\Z_{\geq 1}$) then $\boldsymbol{\xi}\circ N^b_j\to \boldsymbol{\xi}\circ N^b$ weakly in $W^{1,2}_{\text{loc}}(B^n_1(0);\R^L)$, where $N^b$ is the strong $L^2_{\text{loc}}(B^n_1(0);\A_Q(\R^k))$-limit of $N^b_j$.}. Notice that the limit $N^b$ is valued a.e.~in $\A_Q(\R^k)$, i.e.~taking values in $P_0^\perp$ rather than just $\R^{n+k}$. This is readily seen from Theorem \ref{thm:cm} and Lemma \ref{lemma:cm-12}, since $T_0\mathcal{M} = P_0$ and thus $\mathcal{M}_j\to P_0$ in $C^3$ (as a graph) in $B^{n+k}_1(0)$, meaning that the normal bundle to $\mathcal{M}_j$ converges to uniformly to $\{0\}^n\times\R^k$ in $B^{n+k}_1(0)$.

\begin{remark}\label{remark:cm-N-lower-bound}
In particular, from the approximate monotonicity of the frequency function over $\mathcal{M}$ as shown in Theorem \ref{thm:cm-22}, we get $\|N^b\|_{L^2(B^n_\rho(0))}\geq C(\mathcal{V},V,\rho)>0$ for each $\rho\in (0,1)$ (this is just blowing-up the equivalent bounds one gets from having a monotone frequency function equivalent to those seen in Corollary \ref{cor:decay-estimates} for planar frequency).
\end{remark}

We now give two methods of how to conclude Theorem \ref{thm:main}, depending on whether one can upgrade the above convergence of $N^b_j$ to $N^b$ to strong convergence in $W^{1,2}_{\text{loc}}(B^n_1(0))$ or not. We note that strong convergence is known when $V$ is the varifold corresponding to an area minimising current, or an area minimising current mod $p$ (this follows from the arguments in \cite[Proof of Theorem 6.2]{DLS16b} and \cite[Proof of Theorem 28.2]{DLHMS20} respectively).

\begin{remark}
	We do not rule out the possibility that one can unconditionally prove strong convergence in $W^{1,2}_{\text{loc}}(B^n_1(0))$ in our setting. One might be able to do this using our $\eps$-regularity property and estimates coming from the first variation (adapted to $\mathcal{M}$) in an analogous manner to that seen in Appendix \ref{app:energy-convergence}.
\end{remark}

\subsubsection{When there is energy convergence to center manifold blow-up}\label{sec:when-energy-convergence-known}

Suppose we are in a situation where we can in fact show that $N^b_j\to N^b$ strongly in $W^{1,2}_{\text{loc}}(B_1^n(0))$. We then claim that the frequency function associated to $N^b$, namely
$$\mathcal{I}(\rho):=\frac{\mathcal{D}(\rho)}{\mathcal{H}(\rho)},$$
is monotone increasing in $\rho$, where
$$\mathcal{D}(\rho):=\rho^{2-n}\int\phi(r/\rho)|DN^b|^2 \qquad \text{and} \qquad \mathcal{H}(\rho):= \rho^{1-n}\int\phi^\prime(r/\rho)r^{-1}|N^b|^2;$$
here, $\phi$ is the Lipschitz cut-off function from Section \ref{sec:pff}. To see this, if we write $\mathbf{D}_j(\rho)$, $\mathbf{H}_j(\rho)$, $\mathbf{E}_j(\rho)$, $\boldsymbol{\Sigma}_j(\rho)$, and $\mathbf{G}_j(\rho)$ for the relevant frequency quantities of $N^b_j$ as in Section \ref{sec:cm-frequency}, we have (cf.~\eqref{E:cm-H-prime-3}, \eqref{E:cm-D-E-3}, \eqref{E:cm-D-prime-3})
\begin{equation}\label{E:H-prime-blow-up-1}
|\mathbf{H}^\prime_j(\rho) - 2\rho^{-1}\mathbf{E}_j(\rho)| \leq C\Etilt_{V_j}\mathbf{H}_j(\rho);
\end{equation}
\begin{equation}\label{E:D-alt-blow-up-1}
|\mathbf{D}_j(\rho)-\mathbf{E}_j(\rho)| \leq C\Etilt_{V_j}\mathbf{D}_j(\rho)^{1+\frac{2}{3(n+3)}} + C\Etilt_{V_j}^2\rho^2\boldsymbol{\Sigma}_j(\rho);
\end{equation}
\begin{equation}\label{E:D-prime-blow-up-1}
|\mathbf{D}_j^\prime(\rho)-2\mathbf{G}_j(\rho)|\leq C\Etilt_{V_j}\mathbf{D}_j(\rho) + C\Etilt_{V_j}\rho^{-1}\mathbf{D}_j(\rho)^{1+\frac{2}{3(n+3)}} + C\Etilt_{V_j}\mathbf{D}_j^{\frac{2}{3(n+3)}}\mathbf{D}_j^\prime(\rho)
\end{equation}
where here $C = C(\mathcal{V},V)$ (all the estimates on the center manifold scale appropriately to give bounds in terms of $\Etilt_{V_j}$). Thus, if we divide these quantities by $\|N_j\|_{L^2(\mathcal{B}_1(0))}^2$ and take $j\to\infty$, we get, using the strong $W^{1,2}_{\text{loc}}(B_1^n(0))$ convergence $N_j^b\to N^b$ and that $\Etilt_{V_j}\to 0$,
$$\mathcal{H}^\prime(\rho) = 2\rho^{-1}\mathcal{E}(\rho) \equiv -2\rho^{-n}\int\phi^\prime(r/\rho)N^b D_r N^b;$$
$$\mathcal{D}(\rho) = \mathcal{E}(\rho) \equiv -\rho^{1-n}\int\phi^\prime(r/\rho)N^bD_rN^b;$$
$$\mathcal{D}^\prime(\rho) = 2\mathcal{G}(\rho) \equiv -2\rho^{-n}\int r\phi^\prime(r/\rho)|D_rN^b|^2.$$
In particular, we see that $\mathcal{H}(\rho)>0$ for all $\rho>0$, else we would get $\phi^\prime(r/\rho)N^b\equiv Q\llbracket 0\rrbracket $, giving that $\mathcal{D}(\rho) = \mathcal{E}(\rho) = 0$, and so $N^b$ is constant on $B_\rho(0)$. But as $N^b\equiv Q\llbracket 0\rrbracket$ on $B_{\rho}(0)\setminus B_{\rho/2}(0)$, we therefore have that $N^b\equiv Q\llbracket 0\rrbracket$ on $B_\rho(0)$. But then this would contradict $\|N^b\|_{L^2(B_\rho(0))}\geq C(\mathcal{V},V,\rho)>0$ from Remark \ref{remark:cm-N-lower-bound} and the strong $L^2$ convergence.

From the above identities, we easily compute
\begin{equation*}
	\frac{\ext}{\ext\rho}\log\mathcal{I}(\rho) = \frac{\mathcal{D}^\prime(\rho)\mathcal{H}(\rho) - \mathcal{H}^\prime(\rho)\mathcal{D}(\rho)}{\mathcal{D}(\rho)\mathcal{H}(\rho)} =2\cdot\frac{\mathcal{G}(\rho)\mathcal{H}(\rho) - \rho^{-1}\mathcal{E}(\rho)^2}{\mathcal{D}(\rho)\mathcal{H}(\rho)}\geq 0
\end{equation*}
where in the last inequality we have used Cauchy--Schwarz. Thus, we get $\mathcal{I}(\sigma)\leq\mathcal{I}(\rho)$ for every $0<\sigma\leq\rho<1$.

Notice also that, even though the above computations are based at $0$, they also hold at any point $x$ which is a limit point in $B^n_1(0)$ of density $Q$ branch points (or flat singular points) $x_i$ of planar frequency $\geq 2$ in $V_{j_i}$ (as the corresponding computations over $\mathcal{M}$ also hold); this allows us to define the frequency function $\mathcal{I}_x(\rho)$ at any such point $x$. We then define $\mathcal{I}_x(0):=\lim_{\rho\downarrow 0}\mathcal{I}_x(\rho)$, which is well-defined at such $x$ by the above monotonicity. We can then also deduce from the above computation of monotonicity all the usual facts regarding $\mathcal{I}_x(\rho)$, such as:
\begin{enumerate}
	\item [(i)] if $\mathcal{I}_0(\rho)$ is constant on an interval $[a,b]$, then $N^b$ is homogeneous on $[a,b]$ of degree the constant value of $\mathcal{I}(\rho)$ on $[a,b]$;
	\item [(ii)] if $N^b$ is homogeneous of degree $\alpha$, then its \emph{spine} $S(N^b)$ (i.e. the set of points under which $N^b$ is translation invariant) is a subspace, which moreover is given by
	$$S(N^b) = \{x:\mathcal{I}_x(0) = \mathcal{I}_0(0)\}.$$
\end{enumerate}
Furthermore, because we are blowing-up a sequence of rescalings of a given $V$ about the fixed point $0$, it is easy to verify from the strong convergence in $W^{1,2}_{\text{loc}}(B^n_1(0))$ as well as the approximate almost monotonicity of $\mathbf{I}(\rho)$ (Theorem \ref{thm:cm-22}), in the same manner as that seen in Section \ref{sec:pff}, that
$$\mathcal{I}_0(\rho) = \mathbf{I}(0) \qquad \text{for every }\rho\in (0,1),$$
i.e.~the frequency value of $N^b$ at $0$ always agrees with the center manifold frequency of $V$ at $0$, and moreover $\mathcal{I}_0(\rho)$ is constant in $\rho$. Hence, $N^b$ must be a homogeneous function of degree $\mathbf{I}(0)\geq 2$ (from Corollary \ref{cor:cm-23}). In particular, its spine $S(N^b)$ is a subspace, and $\dim(S(N^b))\not\in\{n-1,n\}$ as $\mathcal{I}_0(0)\in [2,\infty)$, and so $\dim(S(N^b))\leq n-2$.

\begin{remark}\label{remark:cm-frequency-independence}
We have shown that, even if the blow-up depends on the subsequence we passed to, its degree of homogeneity does not. At present this all depends on the choice of center manifold: we can in fact show that both the degree of homogeneity, as well as the blow-up itself, are independent of the choice of center manifold. Indeed, let $\mathcal{M}$ be a choice of center manifold for $V$ (at $0$). Let $N$ denote the normal map of $V$ relative to $\mathcal{M}$ and $\mathbf{I}(0)$ the frequency of $V$ relative to $\mathcal{M}$ at $0$. Recalling Remark \ref{remark:cm-height-decay}, for every $\eps_*>0$ there exists $r_*>0$ such that for all $r\in (0,r_*)$,
$$cr^{2\mathbf{I}(0) + 2\eps_*} \leq \mathbf{H}(r) \leq Cr^{2\mathbf{I}(0)}$$
for some constants $c, C\in (0,\infty)$ independent of $\eps_*$ and $r$. Thus, after summing over suitable dyadic-like scales (cf.~the proof of Corollary \ref{cor:decay-estimates}):
\begin{equation}\label{E:cm-comparison-1}
cr^{n+2\mathbf{I}(0)+2\eps_*} \leq \int_{\mathcal{B}_r(0)}|N|^2 \leq Cr^{n+2\mathbf{I}(0)}
\end{equation}
for all $r\in (0,r_*)$. Now, suppose $\widetilde{\mathcal{M}}$ is a different choice of center manifold. Let $\widetilde{\mathbf{I}}(0)$ denote the frequency value of $V$ relative to $\widetilde{\mathcal{M}}$, and suppose for contradiction that $\mathbf{I}(0)<\widetilde{\mathbf{I}}(0)$. Since the Hausdorff distance of $\spt\|V\|$ to both $\mathcal{M}$ and $\widetilde{\mathcal{M}}$ is small, we can express $\widetilde{\mathcal{M}}$ as a normal graph of a single-valued function $\psi$ over $\mathcal{M}$. It follows from \eqref{E:cm-comparison-1} with $\widetilde{\mathcal{M}}$ in place of $\mathcal{M}$ as well as that $\G(N,Q\llbracket b\rrbracket)$ is minimised at $b = N_{a}$ that:
\begin{equation}\label{E:cm-comparison-2}
\int_{\mathcal{B}_r(0)}|N_f|^2 \leq \int_{\mathcal{B}_r(0)}\G(N,Q\llbracket\psi\rrbracket)^2 \leq Cr^{n+2\widetilde{\mathbf{I}}(0)}
\end{equation}
for all $r\in (0,r_*)$, where $N_f$ is the average-free part of the normal map of $N$, and $\mathcal{B}_r(0)$ denote geodesic balls in $\mathcal{M}$. Now let $r_j\to 0^+$ and $s_j\in (3r_j/2,3r_j)$ be the corresponding radii from Corollary \ref{cor:cm-24} for the center manifold $\mathcal{M}$. After passing to a subsequence, let $N^b$ be the blow-up of $(\eta_{0,s_j})_\#V$ relative to $\eta_{0,s_j}(\mathcal{M})$. Dividing both sides of \eqref{E:cm-comparison-2} with $r = s_j$ by $s_j^{-n-2}\|N\|^2_{L^2(\mathcal{B}_{s_j}(0))}$ and letting $j\to\infty$, we deduce that the average-free part of $N^b$ must vanish, contradicting the fact that $N^b$ is non-zero and average-free. Thus, we must have $\mathbf{I}(0)\geq\widetilde{\mathbf{I}}(0)$, and so by symmetry that $\mathbf{I}(0) = \widetilde{\mathbf{I}}(0)$. This shows that the frequency value is independent of the center manifold.

To see further that the blow-up itself is independent of the center manifold\footnote{Note that this is \emph{not} saying that the blow-up itself is unique, i.e.~independent of the choice of scales, but only that for a given choice of scales, the blow-up does not depend on the choice of center manifold.}, let $\widetilde{N}^b$ be the blow-up of $(\eta_{0,s_j})_\#V$ relative to $\eta_{0,s_j}(\widetilde{M})$. By similar reasoning to the above, $\widetilde{N}^b$ arises as the strong $L^2$ limit
$$\frac{N(\exp(s_j x))-\psi(\exp(s_j x))}{s_j^{-n/2}\|N-\psi\|_{L^2(\mathcal{B}_{s_j}(0))}}\to \widetilde{N}^b(x).$$
Taking the average-free part in this limit, and noting that both $N^b$ and $\widetilde{N}^b$ are non-zero and average-free, we see that $\gamma N^b = \widetilde{N}^b$, where $\gamma = \lim_{j\to\infty}\|N\|_{L^2(\mathcal{B}_{s_j}(0))}/\|N-\psi\|_{L^2(\mathcal{B}_{s_j}(0))}\in (0,\infty)$. But since $\|N^b\|_{L^2(B^n_1(0))} = \|\widetilde{N}^b\|_{L^2(B_1^n(0))}=1$ it follows that $\gamma=1$, and thus $\widetilde{N}^b = N^b$, as claimed.
\end{remark}

\begin{remark}\label{remark:cm+planar-frequency}
    The proof that the center manifold frequency value $\mathbf{I}(0)$ is independent of the choice of center manifold can also be used to verify Remark \ref{remark:cm-planar-lower-bounds-cm} that $\mathbf{I}(0)\geq\mathcal{N}(0)$. Indeed, given a center manifold $\mathcal{M}$, express the plane $P_0$ as a normal graph of a single-valued function $\psi$ over $\mathcal{M}$. Then the argument leading to \eqref{E:cm-comparison-2} now gives
    $$\int_{\mathcal{B}_r(0)}|N_f|^2 \leq \int_{\mathcal{B}_r(0)}\G(N,Q\llbracket\psi\rrbracket)^2 \leq Cr^{n+2\mathcal{N}(0)}$$
    using Corollary \ref{cor:decay-estimates}. Now if $\mathcal{N}(0)>\mathbf{I}(0)$ one performs the analogous blow-up as in Remark \ref{remark:cm-frequency-independence} to see that again $N^b$ must vanish, giving the same contradiction as before. Hence $\mathbf{I}(0)\geq\mathcal{N}(0)$. Notice that this argument cannot be used to show the other inequality, as all the corresponding argument would give is that if $\mathbf{I}(0)>\mathcal{N}(0)$ then the blow-up off the plane must have vanishing average-free part, which is not necessarily a contradiction as it is entirely possible that the blow-up off the plane would coincide with its average-part (with multiplicity $Q$). Notice that this can only happen when $\mathcal{N}(0)$ is an integer, and so if $\mathcal{N}(0)$ is not an integer then we do get a contradiction and thus need to have $\mathbf{I}(0) = \mathcal{N}(0)$. In fact, the argument in Remark \ref{remark:cm-frequency-independence} (showing that the blow-up itself is independent of the choice of center manifold) also shows that whenever the coarse blow-up has vanishing average, then the coarse blow-up off the plane \emph{agrees} with the blow-up off the center manifold. In particular, this happens at every point where the planar frequency is non-integer, showing that no extra information at the level of tangent maps is gained at such points from blowing-up relative to the center manifold compared to the tangent plane (cf.~Proposition \ref{prop:dimension-not-equal-2}).
\end{remark}

Finally, in an analogous manner to that seen in Corollary \ref{cor:usc-2} we have upper semi-continuity of frequency along the blow-up sequence, in the following sense: suppose $x\in \{0\}^k\times B^n_1(0)$ is such that there exist points $X_j\in \mathcal{B}_{V_j}^{\geq 2}$ with $X_j\to x$ (in $\R^{n+k}$), then
$$\mathcal{I}_x(0)\geq\limsup_{j\to\infty}\mathbf{I}_{X_j}(0).$$
Again, here we using Lemma \ref{lemma:cm-12} to know that $V_j$ and $\mathcal{M}_j$ touch at such $X_j$.

We are now ready to conclude the proof of Theorem \ref{thm:main} by showing that $\dim_{\H}(\mathcal{B}_V^{\geq 2}\cap\{\Theta_V=Q\})\leq n-2$. Notice first that, as $\mathcal{B}_{V}^{\geq 2}$ is a closed set, by a simple covering argument it suffices to show that at every point $X\in\mathcal{B}_V^{\geq 2}$ with $\Theta_V(X)=Q$, there is a radius $\rho_X>0$ such that $\dim_{\H}(\mathcal{B}_V^{\geq 2}\cap \{\Theta_V=Q\}\cap B_{\rho_X}(X))\leq n-2$.

Thus, fix a point $X\in\mathcal{B}_{V}^{\geq 2}$ with $\Theta_V(X) = Q$: by translating and rotating, we can without loss of generality assume that $X=0$ and the (unique) tangent cone to $V$ at $X$ is $Q|\{0\}\times\R^n|$. We may also rescale about $X=0$ to assume that, on $B^{n+k}_{6\sqrt{n}}(0)$, all the prior results in Section \ref{sec:cm} hold. Thus, from Theorem \ref{thm:cm} we may \emph{choose} a center manifold $\mathcal{M}$ which touches $V$ at all points in $\mathcal{B}_{V}^{\geq 2}\cap B^{n+k}_{1}(0)$; we fix such a choice $\mathcal{M}$ from now on. At every $Y\in \mathcal{B}_V^{\geq 2}\cap B^{n+k}_{1}(0)$, there is then an associated frequency value of $V$ relative to $\mathcal{M}$, which we denote by $\mathbf{I}^{\mathcal{M}}_V(Y)$.

The blow-up analysis of $V$ relative to $\mathcal{M}$ above then applies at \emph{every} such point $Y$. As such, we may stratify $\mathcal{B}_V^{\geq 2}\cap B^{n+k}_{1}(0)$ based on the maximal spine dimension of the blow-ups we get in this manner. So, set for $\ell\in \{0,1,\dotsc,n-2\}$,
$$\mathcal{S}^{\mathcal{M}}_\ell:=\{Y\in \mathcal{B}_V^{\geq 2}\cap B^{n+k}_1(0):\dim(S(N^b))\leq \ell\text{ for every blow-up $N^b$ of $V$ relative to $\mathcal{M}$ at $Y$}\}.$$
Here, we have included the superscript $\mathcal{M}$ to emphasise that this stratification could depend on the choice of $\mathcal{M}$. Clearly we have
$$\mathcal{S}_0^{\mathcal{M}}\subset \mathcal{S}^{\mathcal{M}}_1\subset\cdots\subset\mathcal{S}^{\mathcal{M}}_{n-2} = \mathcal{B}_V^{\geq 2}\cap B^{n+k}_1(0)$$
where the last equality follows from the discussion above.

It therefore suffices to prove that $\dim_\H(\mathcal{S}^{\mathcal{M}}_j)\leq j$ for each $j=0,1,\dotsc,n-2$. But because we have upper semi-continuity of the frequency along the blow-up sequence, this forces points of high frequency compared to the base point to accumulate around the spine of the blow-up, which has dimension at most $j$ if the base point is in $\S_j$. Thus, we may conclude $\dim_\H(\S^{\mathcal{M}}_j)\leq j$ in an analogous manner to that seen in Section \ref{sec:pff} using the weak linear approximation property of Almgren. This therefore completes the proof of Theorem \ref{thm:main} in the case where we have local energy convergence during the center manifold blow-up procedure.

\subsubsection{When energy convergence to the center manifold blow-up is not readily available}

In this situation, we instead use the theory of multi-valued gradient Young measures introduced by Hirsch--Spolaor \cite{HS24}. We summarise the theory, following \cite{HS24}, in Appendix \ref{app:ym}. Here, we shall use the results in Appendix \ref{app:ym} to prove Theorem \ref{thm:main}.

As in Section \ref{sec:when-energy-convergence-known}, we still have the identities \eqref{E:H-prime-blow-up-1}, \eqref{E:D-alt-blow-up-1}, and \eqref{E:D-prime-blow-up-1}. Thus, consider $\mathcal{E}_{j} := \mathcal{E}_{N_j^b}$, the elementary Young measure associated to $N_j^b$. Since $(N_j^b)_j$ is uniformly bounded in $W^{1,2}(B_1)$, the $(\mathcal{E}_j)_j$ are uniformly bounded in $\mathcal{Y}^Q$, and thus by Proposition \ref{prop:Y-compactness} we can find $\mathcal{F}\in \mathcal{Y}^Q$ and a subsequence such that $\mathcal{E}_{j}\weakly\mathcal{F}$ as $Q$-Young measures; in particular, $\mathcal{F}\in \grad\,\mathcal{Y}^Q$. Dividing again \eqref{E:H-prime-blow-up-1}, \eqref{E:D-alt-blow-up-1}, and \eqref{E:D-prime-blow-up-1} by $\|N_j\|_{L^2(\mathcal{B}_1(0))}^2$, Proposition \ref{prop:Y-compactness} combined with Remark \ref{remark:cm-general-test} tells us that $\mathcal{F}$ is stationary. Remark \ref{remark:cm-N-lower-bound} again gives that the associated map $f\in W^{1,2}(B_1(0);\A_Q(\R^k))$ to $\mathcal{F}$ given by Proposition \ref{prop:grad-Young} obeys $f\not\equiv 0$ on $B_1(0)$ (in fact, the proof of Proposition \ref{prop:grad-Young} gives that $f = N^b$).

Just as in Section \ref{sec:when-energy-convergence-known}, since the blow-up was constructed by a sequence of rescalings of a given $V$ about a fixed point, it is easy to verify that
$$\mathcal{I}(0,\rho) := \mathbf{I}(0) \qquad \text{for every }\rho\in (0,1),$$
and thus by Proposition \ref{prop:Y-constant-frequency}, we see that $\mathcal{F}$ must be homogeneous. Proposition \ref{prop:Y-spine} then guarantees that the points of maximal frequency in $\mathcal{F}$ belong to a subspace, which as $\mathcal{I}(0,\rho) = \mathbf{I}(0)\geq 2$, must be a subspace of dimension $\in \{0,1,\dotsc,n-2\}$ by Proposition \ref{prop:Y-spine-dimensions}. But now one can conclude in an analogous manner to that in Section \ref{sec:when-energy-convergence-known}, namely to stratify $\mathcal{B}_V^{\geq 2}\cap B^{n+k}_1(0)$ by the maximal spine dimension of the blow-up one achieves through this procedure, and using upper semi-continuity of the frequency along the blow-up sequence to show that nearby points of good frequency relative to the chosen base point must accumulate along the spine of the blow-up, from which standard measure theory again allows us to control the size. This therefore completes the proof of Theorem \ref{thm:main} in this case also, thereby also completing its proof.
\qed

\appendix

\section{Energy Convergence for Coarse Blow-Ups}\label{app:energy-convergence}

In this appendix we verify that Definition \ref{defn:eps-reg} guarantees energy convergence to coarse blow-ups, and thus strong convergence in $W^{1,2}_{\text{loc}}(B^n_1(0))$. In fact, we prove the following:
\begin{theorem}\label{thm:energy-convergence}
    Let $(V_j)_{j=1}^\infty$ be a sequence of stationary integral $n$-varifolds and $E_j\downarrow 0$ be such that for some Lipschitz $Q$-valued functions $u_j:B^n_1(0)\to \A_Q(\R^k)$ we have
    $$V_j = \mathbf{v}(u_j) \qquad \text{and} \qquad \sup|u_j| + \Lip(u_j)\leq E_j.$$
    Let $v:B^n_1(0)\to \A_Q(\R^k)$ be a Lipschitz $Q$-valued function such that $E_j^{-1}u_j\to v$ locally uniformly in $B^n_1(0)$. Then, for each $\sigma\in (0,1)$,
    $$\lim_{j\to\infty} \int_{B^n_\sigma(0)}\left|E_j^{-2}|Du_j|^2 - |Dv|^2\right| = 0.$$
    In particular, $E_j^{-1}u_j\to v$ strongly in $W^{1,2}_{\text{loc}}(B^n_{1}(0))$.
\end{theorem}

\textbf{Note:} We do not require (for instance) that $E_j = \hat{E}_{V_j}$.

\begin{proof}
    We shall proceed by induction on $Q$. By a standard covering argument, we may assume that $\sigma = 1/2$. In the base case $Q=1$, by standard elliptic estimates we have $\|u_j\|_{C^3(B^n_{3/4}(0))}\leq C(n,k)E_j$, and thus we get the stronger conclusion that $E_j^{-1}Du_j\to Dv$ uniformly in $B^n_{1/2}(0)$.

    Let us now suppose that the theorem holds true whenever we replace $Q$ with $Q^\prime\in \{1,2,\dotsc,Q-1\}$. Let $V_j$ and $u_j$ be as in the statement of the lemma. By the Rellich compactness theorem, $E_j^{-1}u_j^{\kappa,\alpha}\to v^{\kappa,\alpha}$ weakly in $W^{1,2}_{\text{loc}}(B^n_1(0))$, implying that
    $$\|Dv\|_{L^2(B^n_{1/2}(0))} \leq \liminf_{j\to\infty}E_j^{-1}\|Du_j\|_{L^2(B^n_{1/2}(0))}.$$
    Hence it suffices to prove
    $$\limsup_{j\to\infty}E_j^{-1}\|Du_j\|_{L^2(B^n_{1/2}(0))} \leq \|Dv\|_{L^2(B^n_{1/2}(0))}.$$
    Now, arguing as in \cite[$(\mathfrak{B}6)$]{BKMW25}, for any $\delta>0$ we get that for all $j$ sufficiently large,
    \begin{equation}\label{E:app-2}
        \sum^Q_{\alpha=1}\int_{B^n_{1/2}(0)\cap \{|v^\alpha-v_a|<\delta\}}|E_j^{-1}Du_j-Dv_a|^2 \leq C\delta
    \end{equation}
    where $C = C(n,k,Q)\in (0,\infty)$ is a constant.

    Now let $\xi\in \overline{B}^n_{1/2}(0)\cap \{|v^\alpha-v_a|\geq\delta\}$. Then
    $$v(\xi) = \sum^N_{\alpha=1}Q_\alpha \llbracket v^\alpha(\xi)\rrbracket$$
    where $N\geq 2$, $Q_1,\dotsc,Q_N\geq 1$ are integers such that $\sum^N_{\alpha=1}Q_\alpha=Q$, and $v^1(\xi),\dotsc,v^N(\xi)\in \R^k$ are distinct. Let
    $$s:=\min\{|v^\alpha(\xi)-v^\beta(\xi)|:1\leq\alpha<\beta\leq N\}.$$
    Recall that $\Lip\left(v|_{B^n_{1/2}(\xi)}\right) \leq 1$. Let $\rho:= \min\left\{\frac{1}{4},\frac{s}{6\sqrt{Q}}\right\}$. Then
    $$v(x) = \sum^N_{\alpha=1}\tilde{v}^\alpha(x)$$
    on $B^n_\rho(\xi)$ where $\tilde{v}^\alpha:B^n_\rho(\xi)\to \A_{Q_\alpha}(\R^k)$ are Lipschitz $Q_\alpha$-valued functions such that
    $$\G(\tilde{v}^\alpha(x),Q_\alpha\llbracket v^\alpha(\xi)\rrbracket)<s/6.$$
    Since $E^{-1}_{j}u_j\to v$ uniformly on $B^n_{7/8}(0)$, for all sufficiently large $j$ we can write
    $$u_j(x) = \sum^N_{\alpha=1}\tilde{u}^\alpha_j(x)$$
    on $B_\rho^n(\xi)$, where $\tilde{u}^\alpha_j:B^n_\rho(\xi)\to \A_{Q_\alpha}(\R^k)$ is a Lipschitz $Q_\alpha$-valued function such that
    $$\G(E_{j}^{-1}\tilde{u}^\alpha_j(x), Q_\alpha\llbracket v^\alpha(\xi)\rrbracket)<s/3.$$
    It follows that $\sup|\tilde{u}_j^\alpha| + \Lip(\tilde{u}_j^\alpha)\leq E_{j}$. By the Arzelà--Ascoli theorem, $E_j^{-1}\tilde{u}_j^\alpha\to \tilde{v}^\alpha$ uniformly on $B^n_\rho(\xi)$. Since $Q_\alpha<Q$, by the induction hypothesis we have
    $$\lim_{j\to\infty}\int_{B^n_\rho(\xi)} \left|E_j^{-2}|D\tilde{u}_j^\alpha|^2 - |D\tilde{v}^\alpha|^2\right| = 0$$
    for each $\alpha\in \{1,2,\dotsc,N\}$, and therefore
    $$\lim_{j\to\infty}\int_{B_\rho^n(\xi)} \left|E_j^{-2}|Du_j|^2-|Dv|^2\right| = 0.$$
    By a standard covering argument based on the compactness of $\overline{B}^n_{1/2}(0)\cap \{|v^\alpha-v_a|\geq\delta\}$, we therefore get
    \begin{equation}\label{E:app-3}
    \lim_{j\to\infty}\sum^Q_{\alpha=1}\int_{B^n_{1/2}(0)\cap \{|v^\alpha-v_a|\geq\delta\}}\left|E_j^{-2}|Du_j|^2 - |Dv|^2\right| = 0.
    \end{equation}
    Since $E_j^{-1}u^{\kappa,\alpha}_j\to v^{\kappa,\alpha}$ weakly in $W^{1,2}_{\text{loc}}(B^n_1(0))$ and \eqref{E:app-3} implies that $E_j^{-2}|Du_j^{\kappa,\alpha}|^2 \to |Dv^{\kappa,\alpha}|^2$ strongly in $L^1(B^n_{1/2}(0)\cap \{|v^\alpha-v_a|\geq\delta\})$, we thus get
    \begin{equation}\label{E:app-4}
    \lim_{j\to\infty}\sum^Q_{\alpha=1}\int_{B^n_{1/2}(0)\cap \{|v^\alpha-v_a|\geq\delta\}}|E_j^{-1}Du_j^{\kappa,\alpha}-Dv_a^\kappa|^2 = \sum^Q_{\alpha=1}\int_{B^n_{1/2}(0)\cap \{|v^\alpha-v_a|\geq\delta\}}|Dv^{\kappa,\alpha}-Dv_a^\kappa|^2
    \end{equation}
    for each $\kappa\in \{1,\dotsc,k\}$. Therefore, for all sufficiently large $j$ we get by \eqref{E:app-2} and \eqref{E:app-4}:
    \begin{align*}
        \sum^Q_{\alpha=1}\int_{B^n_{1/2}(0)}|E_j^{-1}Du_j^{\kappa,\alpha}-Dv_a^{\kappa}|^2 & \leq \sum^Q_{\alpha=1}\int_{B^n_{1/2}(0)\cap\{|v^\alpha-v_a|<\delta\}}|E_j^{-1}Du_j^{\kappa,\alpha}-Dv_a^\kappa|^2\\
        & \hspace{3em} + \sum^Q_{\alpha=1}\int_{B^n_{1/2}(0)\cap \{|v^\alpha-v_a|\geq\delta\}}|E_j^{-1}Du_j^{\kappa,\alpha}-Dv^\kappa_a|^2\\
        & \leq C\delta + \sum^Q_{\alpha=1}\int_{B^n_{1/2}(0)\cap\{|v^\alpha-v_a|\geq\delta\}}|Dv^{\kappa,\alpha}-Dv_a^\kappa|^2\\
        & \leq C\delta + \sum^Q_{\alpha=1}\int_{B^n_{1/2}(0)}|Dv^{\kappa,\alpha}-Dv_a^\kappa|^2
    \end{align*}
    for each $\kappa\in\{1,2,\dotsc,k\}$, where $C = C(n,k,Q)\in (0,\infty)$ is a constant. Hence, taking $\limsup_{j\to\infty}$ of both sides and letting $\delta\downarrow 0$, we get
    $$\limsup_{j\to\infty}\sum^Q_{\alpha=1}\int_{B^n_{1/2}(0)}|E_j^{-1}Du_j^{\kappa,\alpha} - Dv^\kappa_a|^2 \leq \sum^Q_{\alpha=1}\int_{B^n_{1/2}(0)}|Dv^{\kappa,\alpha}-Dv^\kappa_a|^2$$
    for each $\kappa\in \{1,2,\dotsc,k\}$. Since $E_j^{-1}u^{\kappa,\alpha}_j\to v^{\kappa,\alpha}$ weakly in $W^{1,2}(B^n_1(0))$, this then gives that (summing over $\kappa\in \{1,2,\dotsc,k\}$ also)
    $$\limsup_{j\to\infty} E_{j}^{-2}\sum^k_{\kappa=1}\sum^Q_{\alpha=1}\int_{B^n_{1/2}(0)}|Du_j^{\kappa,\alpha}|^2 \leq \sum^k_{\kappa=1}\sum^Q_{\alpha=1}\int_{B^n_{1/2}(0)}|Dv^{\kappa,\alpha}|^2.$$
    This gives us the desired upper semi-continuity, which completes the proof.
\end{proof}

\section{Multi-Valued Gradient Young Measures}\label{app:ym}

Here we provide an exposition of the multi-valued gradient Young measures introduced by Hirsch--Spolaor \cite{HS24}. We also provide further structural information regarding homogeneous multi-valued gradient Young measures. Much of this appendix is included within \cite{HS24} although there are some differences as, for instance, we avoid the need to discuss slicings of homogeneous gradient Young measures.

The classical theory of Young measures provides a way of associating a limit to a bounded sequence of functions $(f_j)_j\subset L^\infty(B^n_1(0);\R^k)$ where the limit of $\phi(f_j)$ is determined for every continuous function $\phi:\R^k\to \R$. For instance, if the $f_j$ oscillate a lot, the limiting Young measure represents the asymptotic statistical distribution of the values of the $f_j$ near each point.

For us, we are interested in limits of sequences $(f_j)_j$ of $W^{1,2}(B^n_1(0);\A_Q(\R^k))$ functions as (Young) measures which keep track of the action of \emph{certain} functions of the form $\phi(x,f_j,Df_j,Df_j\otimes Df_j)$, in particular those determining the inner and outer variations needed to prove the monotonicity of a frequency function (keeping track of every such continuous $\phi$ is not possible now as we only have $W^{1,2}$ control on the sequence). The aim is that whilst the usual weak $W^{1,2}$ limit may or may not have a monotone frequency function, the associated Young measure (which is related to the weak $W^{1,2}$ limit but contains more information) will have a notion of frequency which is monotone that we can make use of.

To this end, we introduce the vector space
$$\V:= \R^k\times \R^{k\times n}\times \R^{(k\times n)^2} \equiv \V_1\times \V_2\times \V_3.$$
Write $(y,p,M)\in \V$ for the corresponding coordinates on $\V$ adapted to $\V_1,\V_2,\V_3$. One should think of $y,p,M$ as playing the roles of $f(x),Df(x),Df(x)\otimes Df(x)$, respectively, for a function $f:B^n_1(0)\to \R^k$.

Let us begin by introducing a class of test functions which $W^{1,2}(B^n_1(0);\A_Q(\R^k))$ naturally acts on. Write $\widetilde{\CC}$ for the set of continuous functions $\phi:B^n_1(0)\times\V\to \R$ which obey:
\begin{itemize}
	\item $\spt(\phi)\subset K\times\V$ for some compact set $K\subset B^n_1(0)$;
	\item $\|\phi\|_*<\infty$, where
	$$\|\phi\|_* := \sup_{x,y,p,M}\frac{|\phi(x,y,p,M)|}{1+|y|^{2^*}+|p|^2+|M|}.$$
\end{itemize}
Here, for $n>2$ we take $2^*:= \frac{2n}{n-2}$ to be the Sobolev conjugate of $2$, and if $n=2$ we can take $2^*$ to be any (fixed) number we wish in $(2,\infty)$, e.g.~$2^* = 3$. The important properties that we will need $2^*$ to satisfy is that $2^*>2$ and we have the compact Sobolev embedding $W^{1,2}(B^n_1(0);\A_Q(\R^k))\hookrightarrow L^{2^*}(B^n_1(0))$ (which in our context follows from, e.g.~\cite[Proposition 2.11]{DLS11}).

The reader should note that $(\widetilde{\CC},\|\cdot\|_*)$ is a \emph{non-separable} normed vector space (due to the non-compactness of the support in the $\V$-variables). Nonetheless, a function $f\in W^{1,2}(B^n_1(0);\A_Q(\R^k))$ naturally gives rise to a linear map $\CE_f: \widetilde{\CC}\to \R$, called the \emph{elementary Young measure associated to $f$}, by
$$\CE_f(\phi):= \int_{B^n_1(0)}\sum^Q_{\alpha=1}\phi(x,f^\alpha(x),Df^\alpha(x),Df^\alpha(x)\otimes Df^\alpha(x))\, \ext x.$$
This is well-defined by definition of $\widetilde{\CC}$ and Sobolev embedding. In general, for a linear map $\CF:\widetilde{\CC}\to \R$ we define its corresponding operator norm by
$$\|\CF\|_* := \sup\{\CF(\phi):\phi\in\widetilde{\CC} \text{ such that }\|\phi\|_*\leq 1\}.$$
Notice then that for $f\in W^{1,2}(B^n_1(0);\A_Q(\R^k))$ we have (using the Sobolev embedding $W^{1,2}\hookrightarrow L^{2^*}$)
$$\|f\|^2_{W^{1,2}(B^n_1(0))}\leq \|\CE_f\|_* \leq Q\w_n + \|f\|^{2^*}_{W^{1,2}(B_1^n(0))}+2\|f\|^2_{W^{1,2}(B^n_1(0))}.$$
Thus, $(f_j)_j\subset W^{1,2}(B^n_1(0);\A_Q(\R^k))$ is a uniformly bounded sequence if and only if $(\CE_{f_j})_j$ is uniformly bounded. However, since $(\widetilde{\CC},\|\cdot\|_*)$ is non-separable, we cannot apply the sequential Banach--Alaoglu theorem in this space to always find a subsequential limit of a uniformly bounded sequence. In order to resolve this issue, we will restrict the functions in $\widetilde{\CC}$ we act on so that the vector space of functions we act on \emph{is} separable. Indeed, ultimately we are most interested in the limits coming from the inner and outer variations, which correspond to the choices
$$\phi_1(x,y,p,M) := p_i^{\kappa} y^{\kappa}\del_i\phi + M_{ii}^{\kappa\kappa}\phi \qquad \text{for }\phi\in C^\infty_c(B^n_1(0)),$$
$$\phi_2(x,y,p,M) := (2M^{\kappa\kappa}_{ij} - \delta_{ij}M^{\kappa\kappa}_{\ell\ell})\del_i\phi^j \qquad \text{for }\phi\in C^\infty_c(B^n_1(0);\R^n).$$
Here, $i,j,\ell\in \{1,\dotsc,n\}$ and $\kappa\in \{1,\dotsc,k\}$. Note that both $\phi_1,\phi_2\in \widetilde{\CC}$, but even more so they have the property that, as $(y,p,M)\to \infty$, the limiting behaviour of these functions (up to lower order terms determined by the scaling in the norm $\|\cdot\|_*$) is determined entirely by functions of $(x,M)$ in a uniform manner. More precisely, if we set
$$\phi_1^\infty(x,M):= \phi_1(x,0,0,M) \qquad \text{and} \qquad \phi_2^\infty(x,M) := \phi_2(x,0,0,M),$$
then we have
$$\lim_{(y,p,M)\to\infty}\sup_{x\in B^n_1(0)}\frac{|\phi_i(x,y,p,M)-\phi_i^\infty(x,M)|}{1+|y|^{2^*}+|p|^2+|M|} = 0 \qquad \text{for }i=1,2.$$
The reader should note that this is why we require the exponent of $2^*>2$ in the power of $|y|$ in the definition of $\|\cdot\|_*$, to ensure this limiting behaviour is true for the inner variation $\phi_1$.

This offers a subclass of functions in $\widetilde{\CC}$ which \emph{is} separable (as we have essentially compactified the behaviour in the $\V$-variables) and contains the functions we care about. We therefore define:
\begin{defn}
	Write $\CC(B^n_1(0)\times\V)$ for the set of functions $\phi\in \widetilde{\CC}$ for which there is a continuous function $\phi^\infty = \phi^\infty(x,M):B^n_1(0)\times \V_3\to \R$ which is $1$-homogeneous in the $M$-variable and obeys:
	$$\lim_{(y,p,M)\to\infty}\sup_{x\in B^n_1(0)}\frac{|\phi(x,y,p,M)-\phi^\infty(x,M)|}{1+|y|^{2^*}+|p|^2+|M|} = 0.$$
	We call $\phi^\infty$ the \emph{recession function} of $\phi$.
\end{defn}
Since $\CC(B^n_1(0)\times\V)\subset \widetilde{\CC}$, the elementary Young measures still act on $\CC(B^n_1(0)\times\V)$. The added benefit is that $\CC(B^n_1(0)\times\V)$ \emph{is} a separable normed vector space, and so we can apply separable Banach--Alaoglu to pass to a weak limit for a bounded sequence of elementary Young measures. We will prove this compactness property directly, however one could also see it by applying a suitable ``spherical compactification map", which is the (linear, isometric) map $T:\widetilde{\CC}\to C^0(B^n_1(0)\times B^{\V}_1(0))$ given by:
$$T(\phi)(x,y,p,M):= R\phi\left(x,\,\frac{y}{R^{1/2^*}},\,\frac{p}{R^{1/2}},\,\frac{M}{R}\right)$$
where $R:= 1-|y|^{2^*}-|p|^2-|M|$ and $B^{\V}_1(0) := \{(y,p,M)\in \V: |y|^{2^*}+|p|^2+|M|<1\}$.  The maps $\phi \in \CC(B^n_1(0)\times\V)$ are those ones such that $T(\phi)$ has $\spt(T(\phi))\subset K\times B^{\V}_1(0)$ for some compact $K\subset B^n_1(0)$, and have a continuous extension to $\overline{B^n_1(0)\times B^{\V}_1(0)}$ such that on $B^n_1(0)\times \del B^{\V}_1(0)$ the extension is a function of only $x$ and $M$; the pull-back under $T$ of the function on the ``boundary'' $B^n_1(0)\times \del B_1^{\V}(0)$ is then the recession function. In the language of spherical compactification, the Young measure is then a measure on $\overline{B^n_1(0)\times B^{\V_3}_1(0)}$: the part of the measure restricted to $B^n_1(0)\times \del B^{\V_3}_1(0)$ is then the part of the measure which measures the energy loss.

To make this precise, let us write $C^0_{\text{rec}}(\V)$ for the set of functions $\phi\in C^0(\V)$ for which there is a continuous $1$-homogeneous function $\phi^\infty = \phi^\infty(M):\V_3\to \R$ such that
$$\lim_{(y,p,M)\to \infty}\frac{|\phi(y,p,M)-\phi^\infty(M)|}{1+|y|^{2^*}+|p|^2+|M|} = 0.$$
For such a function we slightly abuse notation and write
$$\|\phi\|_* := \sup_{y,p,M}\frac{|\phi(y,p,M)|}{1+|y|^{2^*}+|p|^2+|M|}.$$
A linear map $\CF:C^0_{\text{rec}}(\V)\to \R$ then has its \emph{mass} defined through its operator norm, i.e.
$$\mathbf{M}_*(\CF) \equiv \|\CF\|_{*} := \sup\{\CF(\phi): \phi\in C^0_{\text{rec}}(\V)\text{ obeys }\|\phi\|_*\leq 1\}.$$
We stress this is different to the `usual' mass, which would be defined by
$$\mathbf{M}(\CF) := \sup\{\CF(\phi): \phi\in C^0(\V)\text{ obeys }\sup|\phi|\leq 1\}.$$
Notice also that if $\phi\in C^0(\V)$ has $\sup|\phi|<\infty$, then $\phi\in C^0_{\text{rec}}(\V)$ with $\phi^\infty\equiv 0$.
\begin{defn}\label{defn:Young}
	We call a linear map $\mathcal{F}:\CC(B^n_1(0)\times\V)\to \R$ a $Q$\emph{-Young measure} if there is a measurable function $F:B^n_1(0)\to C^0_{\textnormal{rec}}(\V)^*$ and $F^\infty\in C_c^0(B^n_1(0)\times \del B^{\V_3}_1(0))^*$ such that:
	\begin{enumerate}
		\item [(a)] for $\phi \in \CC(B^n_1(0)\times\V)$ we have
			$$\mathcal{F}(\phi) := \int_{B^n_1(0)}F_x(\phi_x)\, \ext x + F^\infty(\phi^\infty),$$
			where $F_x \equiv F(x)$ and $\phi_x\equiv \phi(x,\cdot)$;
		\item [(b)] $F_x(\id_{\V}) = \mathbf{M}(F_x) = Q$ for a.e.~$x\in B^n_1(0)$, where $\id_{\V}\equiv 1$ on $\V$;
		\item [(c)] $\int_{B^n_1(0)}\mathbf{M}_*(F_x)\, \ext x + \mathbf{M}(F^\infty)<\infty$, where
		$$\mathbf{M}(F^\infty):=\sup\{F^\infty(\phi):\phi\in C^0_c(B^n_1(0)\times\del B^{\V_3}_1(0))\text{ obeys }\sup|\phi|\leq 1\}.$$
	\end{enumerate}
	In this case we write $\CF:= (F\otimes\ext x,F^\infty)$. We write $\mathcal{Y}^Q$ for the set of all $Q$-Young measures.
\end{defn}
Here, for a vector space $W$ we have written $W^*$ for its dual space.

\textbf{Note:} 
\begin{enumerate}
	\item [(1)] In \cite{HS24}, $Q$-Young measures are referred to as \emph{$\A_Q$-generalised Young measures}.
	\item [(2)] Technically in the above definition $\phi^\infty\in C^0(B^n_1(0)\times \V_3)$ whilst $F^\infty$ acts on functions in $C^0_c(B^n_1(0)\times \del B^{\V_3}_1(0))$. As recession functions are always $1$-homogeneous in $M$, this requires identifying the set of such functions with $C^0_c(B^n_1(0)\times \del B^{\V_3}_1(0))$ via restriction.
	\item [(3)] One can readily check that the integral $\int_{B^n_1(0)}F_x(\phi_x)\, \ext x$ in Definition \ref{defn:Young} is well-defined. Indeed, since $|F_x(\phi_x)|\leq \mathbf{M}_*(F_x)\|\phi_x\|_*$, by condition (c) in Definition \ref{defn:Young} it suffices to show that $\|\phi_x\|_*$ is uniformly bounded in $x$ for $\phi\in \CC(B^n_1(0)\times\V)$, which follows as $\phi\in\widetilde{\CC}$.
	\item [(4)] The condition in Definition \ref{defn:Young}(b) in particular implies that $F_x$ must be a positive Radon measure for a.e.~$x\in B^n_1(0)$, i.e. if $\phi\geq 0$, then $F_x(\phi)\geq 0$.
\end{enumerate}
For $\CF\in \mathcal{Y}^Q$ we then define
$$\|\CF\|_{\mathcal{Y}^Q} := \int_{B^n_1(0)}\mathbf{M}_*(F_x)\, \ext x + \mathbf{M}(F^\infty).$$
\begin{defn}
	We say that $(\CF_j)_j\subset\mathcal{Y}^Q$ \emph{converges} to $\CF\in \mathcal{Y}^Q$, and write $\CF_j\weakly \CF$, if
	$$\lim_{j\to\infty}\CF_j(\phi) = \CF(\phi) \qquad \text{for every }\phi\in \CC(B^n_1(0)\times\V).$$
\end{defn}
We start by proving the weak compactness of $\mathcal{Y}^Q$.
\begin{prop}[{\cite[Proposition 2.2]{HS24}}]\label{prop:Y-compactness}
	Suppose $(\CF_j)_j\subset\mathcal{Y}^Q$ is a sequence with $\sup_j\|\CF_j\|_{\mathcal{Y}^Q}<\infty$. Then, there exists $\CF\in \mathcal{Y}^Q$ and a subsequence $(j^\prime)$ such that:
	\begin{enumerate}
		\item [(i)] $\|\CF\|_{\mathcal{Y}^Q}\leq\liminf_j\|\CF_j\|_{\mathcal{Y}^Q}$;
		\item [(ii)] $\CF_{j^\prime}\weakly\CF$.
	\end{enumerate}
\end{prop}
\begin{proof}
	Write $\CF_j = (F_j\otimes \ext x, F_j^\infty)$ for each $j$. Separable Banach--Alaoglu provides the existence of \emph{some} weak limit $\mathcal{F}$ satisfying the conclusions of the proposition, but we need to verify that $\CF\in \mathcal{Y}^Q$, i.e.~that $\CF = (F\otimes \ext x,F^\infty)$ for appropriate $F$ and $F^\infty$.
	
	Note first that $F_j\otimes \ext x$ defines a linear map $\CC(B^n_1(0)\times\V)\to \R$ by
	$$(F_j\otimes \ext x)(\phi):=\int_{B^n_1(0)}F_x(\phi_x)\, \ext x.$$
	Clearly $\mathbf{M}_*(F_j\otimes\ext x)\leq\|\mathcal{F}_j\|_{\mathcal{Y}^Q}$ for each $j$ (here $\mathbf{M}_*$ is defined analogously for linear functionals $\CC(B^n_1(0)\times\V)\to \R$) and thus $(F_j\otimes\ext x)_j$ is bounded, Similarly, $F_j^\infty$ gives a bounded sequence of linear maps $C^0_c(B^n_1(0)\times \del B^{\V_3}_1(0))\to \R$. We may therefore apply separable Banach--Alaoglu to both giving, up to passing to a subsequence,
	\begin{align*}
	F_j\otimes\ext x \weakly\hat{F} \qquad &\text{in }\CC(B^n_1(0)\times\V)^*,\\
	F^\infty_j\weakly \hat{F}^\infty \qquad &\text{in }C^0_c(B^n_1(0)\times\del B^{\V_3}_1(0))^*.
	\end{align*}
	Clearly we have $\mathcal{F}(\phi) = \hat{F}(\phi) + \hat{F}^\infty(\phi^\infty)$ for $\phi\in \CC(B^n_1(0)\times\V)$. We now claim that in fact we have
	$$\hat{F}(\phi) = \int_{B^n_1(0)}(F_x)(\phi_x)\, \ext x + \widetilde{F}^\infty(\phi^\infty) \qquad \text{for }\phi\in \CC(B^n_1(0)\times\V),$$
	for some measurable $F:B^n_1(0)\to C^0_{\text{rec}}(\V)^*$ with $\mathbf{M}(F_x) = Q$ for a.e.~$x\in B^n_1(0)$ and $\widetilde{F}^\infty\in C^0_c(B^n_1(0)\times\del B^{\V_3}_1(0))^*$; once we have shown this the proof is complete, since $\CF = (F\otimes \ext x, F^\infty)$ with $F^\infty := \hat{F}^\infty + \widetilde{F}^\infty$.
	
	To show this, we wish to understand more the action $F_j\otimes\ext x$. For this, we build up its limiting behaviour on progressively more general functions in $\CC(B^n_1(0)\times\V)$. To begin with, applying sequential Banach--Alaoglu on $ C^0_c(B^n_1(0)\times\V)\subset \CC(B^n_1(0)\times\V)$, followed by the Riesz representation theorem and the standard disintegration theorem \cite[Theorem 2.28]{AFP00}, we see that there is a measurable function $F:B^n_1(0)\to C^0_c(\V)^*$ such that for any $\phi\in C^0_c(B^n_1(0)\times\V)$,
	\begin{equation}\label{E:compactness-1}
	(F_j\otimes\ext x)(\phi) \to (F\otimes \ext x)(\phi)
	\end{equation}
	i.e.~for $\phi\in C^0_c(B^n_1(0)\times\V)$ we have $\hat{F}(\phi) = (F\otimes \ext x)(\phi)$ (note that for such $\phi$ the corresponding recession function vanishes).
	
	Next, let us consider a general $\phi\in \CC(B^n_1(0)\times\V)$ with $\phi^\infty\equiv 0$. This means that
	$$\lim_{(y,p,M)\to\infty}\sup_{x\in B^n_1(0)}\frac{|\phi(x,y,p,M)|}{1+|y|^{2^*}+|p|^2+|M|} = 0,$$
	which implies that for every $\delta>0$ there exists $R_\delta>0$ such that
	$$\sup_{x\in B^n_1(0)}|\phi(x,y,p,M)| \leq \delta(1+|y|^{2^*}+|p|^2+|M|) \qquad \text{whenever }|y|^{2^*}+|p|^2+|M|>R_\delta.$$
	Now, fix $\delta>0$ and let $\chi_\delta:\V\to \R$ be a smooth, decreasing, radial function which is identically $1$ on $B^{\V}_{R_{\delta}}(0)$ and vanishes outside $B^{\V}_{2R_\delta}(0)$. Then $\chi_\delta\phi\in C^0_c(B^n_1(0)\times\V)$, and so we know from \eqref{E:compactness-1},
	$$(F_j\otimes\ext x)(\chi_\delta\phi) \to (F\otimes \ext x)(\chi_\delta\phi) = \hat{F}(\chi_\delta\phi).$$
	On the other hand, $\|(1-\chi_\delta)\phi\|_* \leq \delta$, and so
	$$|(F_j\otimes \ext x)((1-\chi_\delta)\phi)| \leq \|\mathcal{F}_j\|_{\mathcal{Y}^Q}\cdot\delta$$
	and similarly
	$$|\hat{F}((1-\chi_\delta)\phi)| \leq\mathbf{M}_*(\hat{F})\cdot\delta \leq \sup_i\|\mathcal{F}_i\|_{\mathcal{Y}^Q}\cdot\delta$$
	(using the lower semi-continuity of $\|\cdot\|_{\mathcal{Y}^Q}$ along the convergent subsequence). Hence, for any fixed $\delta>0$, for all sufficiently large $j$ we have
	$$|(F_j\otimes \ext x)(\phi) - (F\otimes\ext x)(\chi_\delta \phi)|\leq \delta + \delta\sup_i\|\CF_i\|_{\mathcal{Y}^Q}.$$
	Taking $j\to\infty$ and then $\delta\downarrow 0$ (using dominated convergence) we therefore see that
	\begin{equation}\label{E:compactness-2}
	\lim_{j\to\infty}(F_j\otimes \ext x)(\phi) = (F\otimes\ext x)(\phi)
	\end{equation}
	for any $\phi\in \CC(B^n_1(0)\times\V)$ whose recession function vanishes; in particular, the action of $F\otimes \ext x$ on such functions is well-defined (and agrees with $\hat{F}$). This in particular implies that $\mathbf{M}(F_x) = Q$ for a.e.~$x\in B^n_1(0)$ (the upper bound of $Q$ comes form mass lower semi-continuity along the weakly convergent subsequence, and the lower bound comes from taking $\phi = \psi(x)\one_{\V}$ in the above convergence for any $\psi\in C^\infty_c(B^n_1(0))$).
	
	Finally, we are left with a general $\phi\in \CC(B^n_1(0)\times\V)$. First notice that $\phi - \phi^\infty\in\CC(B^n_1(0)\times\V)$ has vanishing recession function, and so by \eqref{E:compactness-2}
	$$(F_j\otimes\ext x)(\phi-\phi^\infty)\to (F\otimes \ext x)(\phi-\phi^\infty).$$
	We are therefore left with evaluating $\lim_{j\to\infty}(F_j\otimes\ext x)(\phi^\infty)$: if we can show that this is equal to $(F\otimes\ext x)(\phi^\infty) + \widetilde{F}^\infty(\phi^\infty)$ for a suitable linear map $\widetilde{F}^\infty:C^0_c(B^n_1(0)\times\del B^{\V_3}_1(0))\to \R$, we will be done.
	
	For this, for each $j$ define a linear map $\widetilde{F}_j$ which sends $\psi\in C^0_c(B^n_1(0)\times \overline{B}^{\V_3}_1(0))$ to
	$$\widetilde{F}_j(\psi):=\int_{B^n_1(0)}(F_j)_x(\widetilde{\psi}_x)\, \ext x$$
	where $\widetilde{\psi}(x,M):= \psi(x,\frac{M}{1+|M|})(1+|M|)$; this is well-defined because $\widetilde{\psi}\in \CC(B^n_1(0)\times\V)$. Notice that the recession function of $\widetilde{\psi}$ is $\widetilde{\psi}^\infty(x,M):=|M|\psi(x,M/|M|)$. Note that for such $\psi$,
	$$\widetilde{F}_j(\psi) \leq \|\CF_j\|_{\mathcal{Y}^Q}\cdot\sup|\psi|$$
	and thus $(\widetilde{F}_j)_j$ is a bounded sequence in $C^0_c(B^n_1(0)\times \overline{B}^{\V_3}_1(0))^*$. Applying sequential Banach--Alaoglu, we can therefore pass to a subsequence and find $\widetilde{F}\in C^0_c(B^n_1(0)\times \overline{B}^{\V_3}_1(0))^*$ such that
	$$\widetilde{F}_j(\psi)\to \widetilde{F}(\psi) \qquad \text{for all }\psi\in C^0_c(B^n_1(0)\times\overline{B}^{\V_3}_1(0)).$$
	Thus, if $\phi = \phi(x,M)\in \CC(B^n_1(0)\times\V)$ is homogeneous degree one in the $M$-variable, then
	$$(F_j\otimes\ext x)(\phi) = \int_{B^n_1(0)}(F_j)_x(\phi(x,M))\, \ext x = \int_{B^n_1(0)}(F_j)_x\left(\phi\left(x,\frac{M}{1+|M|}\right)(1+|M|)\right)\, \ext x \equiv \widetilde{F}_j(\phi_*)$$
	where $\phi_* := \phi|_{B^n_1(0)\times\overline{B}_1^{\V_3}(0)}$. So,
	$$\lim_{j\to\infty}(F_j\otimes \ext x)(\phi) = \widetilde{F}(\phi_*).$$
	Now fix $\eps>0$ and let $\eta_\eps:[0,1]\to [0,1]$ be a smooth, decreasing function, obeying $\eta_\eps|_{[0,1-\eps]} \equiv 1$ and $\eta_\eps\equiv 0$ on a neighbourhood of $1$. Write
	$$\phi_*(x,M) = \eta_\eps(|M|)\phi_*(x,M) + (1-\eta_\eps(|M|))\phi_*(x,M).$$
	Notice that $\eta_\eps(|M|)\phi_*\in C^0_c(B^n_1(0)\times\overline{B}^{\V_3}_1(0))$ vanishes on a neighbourhood of $\del B^{\V_3}_1(0)$. We therefore have:
	\begin{align*}
		\widetilde{F}(\eta_\eps(|M|)\phi_*) & = \lim_{j\to\infty}\widetilde{F}_j(\eta_\eps(|M|)\phi_*)\\
		& = \lim_{j\to\infty}\int_{B^n_1(0)}(F_j)_x\left(\eta_\eps\left(\frac{|M|}{1+|M|}\right)\phi_*\left(x,\frac{M}{1+|M|}\right)(1+|M|)\right)\, \ext x\\
		& = \lim_{j\to\infty}\int_{B^n_1(0)}(F_j)_x\left(\eta_\eps\left(\frac{|M|}{1+|M|}\right)\phi(x,M)\right)\, \ext x\\
		& = \lim_{j\to\infty}(F_j\otimes \ext x)\left(\eta_{\eps}\left(\frac{|M|}{1+|M|}\right)\phi(x,M)\right)\\
		& = (F\otimes \ext x)\left(\eta_{\eps}\left(\frac{|M|}{1+|M|}\right)\phi(x,M)\right),
	\end{align*}
	where in the last line we have used the previous case, since $\eta_\eps\left(\frac{|M|}{1+|M|}\right)\phi(x,M)\in \CC(B^n_1(0)\times\V)$ has vanishing recession function. In particular, taking $\eps\to 0$ and using the form of $F\otimes \ext x$ (noting that $\eta_\eps\left(\frac{|M|}{1+|M|}\right)\phi(x,M)$ converges pointwise to $\phi(x,M)$ as $\eps\to 0$ and so we may apply the dominated convergence theorem), we get
	$$\lim_{\eps\to 0}\widetilde{F}(\eta_\eps(|M|)\phi_*) = (F\otimes \ext x)(\phi(x,M)).$$
	In particular, this tells us that $F\otimes \ext x$ is now well-defined acting on $\CC(B^n_1(0)\times\V)$. The remaining part left to analyse is $\widetilde{F}((1-\eta_\eps(|M|))\phi_*(x,M))$. Note that pointwise, as $\eps\to 0$,
	$$(1-\eta_\eps(|M|))\phi_*(x,M) \to \phi_*(x,M)\one_{B^n_1(0)\times\del B^{\V_3}_1(0)}(x,M).$$
	Recall that as $\widetilde{F}$ is an element of the dual space of $C^0_c(B^n_1(0)\times\overline{B}^{\V_3}_1(0))$, the Riesz representation theorem identifies it with a Borel measure on $B^n_1(0)\times \overline{B}^{\V_3}_1(0)$. We may therefore consider the measure-restriction $\widetilde{F}^\infty := \widetilde{F}\res (B^n_1(0)\times\del B^{\V_3}_1(0))$. The dominated convergence theorem combined with the above pointwise convergence therefore gives
	$$\lim_{\eps\to 0}\widetilde{F}((1-\eta_\eps(|M|))\phi_*) = \widetilde{F}^\infty(\phi_*|_{B^n_1(0)\times\del B^{\V_3}_1(0)}) \equiv \widetilde{F}^\infty(\phi).$$
	Thus, combining everything we have shown, we have for any $\phi\in \CC(B^n_1(0)\times\V)$ which is homogeneous degree one in the $M$-variables, for any $\eps>0$,
	\begin{align*}
		\lim_{j\to\infty}(F_j\otimes \ext x)(\phi) & = \widetilde{F}(\phi_*) = \widetilde{F}(\eta_\eps(|M|)\phi_*) + \widetilde{F}((1-\eta_\eps(|M|))\phi_*)
	\end{align*}
	and then taking $\eps\to 0$ we get
	$$\lim_{j\to\infty}(F_j\otimes\ext x)(\phi) = (F\otimes \ext x)(\phi) + \widetilde{F}^\infty(\phi)$$
	which is exactly what we had left to show. This therefore completes the proof.
\end{proof}
We are most interested in the limits of elementary Young measures. In terms of the language of $\mathcal{Y}^Q$, these are:
\begin{defn}
	The \emph{elementary Young measure} $\mathcal{E}_f$ associated to $f\in W^{1,2}(B^n_1(0);\A_Q(\R^k))$ is the $Q$-Young measure $\mathcal{E}_f = (E_f\otimes\ext x,0)$, where
	$$(E_f)_x(\phi) = \sum^Q_{\alpha=1}\phi(f^\alpha(x),Df^\alpha(x),Df^\alpha(x)\otimes Df^\alpha(x)) \qquad \text{for }\phi\in C^0_{\text{rec}}(\V).$$
\end{defn}
\textbf{Note:} If $f\in W^{1,2}(B^n_1(0);\A_Q(\R^k))$, its elementary Young measure evidently satisfies $\mathbf{M}((E_f)_x) = Q$ and
$$\sum^Q_{\alpha=1}\left(|f^\alpha(x)|^{2} + |Df^\alpha(x)|^2\right) \leq \mathbf{M}_*((E_f)_x) \leq Q + \sum^Q_{\alpha=1}(|f^\alpha(x)|^{2^*} + 2|Df^\alpha(x)|^2)$$
for a.e.~$x\in B^n_1(0)$. Thus, $\mathcal{E}_f\in \mathcal{Y}^Q$ and, as noted before,
\begin{equation}\label{E:Young-bounds-function}
\|f\|^2_{W^{1,2}(B^n_1(0))}\leq \|\CE_f\|_{\mathcal{Y}^Q} \leq Q\w_n + \|f\|^{2^*}_{W^{1,2}(B_1^n(0))}+2\|f\|^2_{W^{1,2}(B^n_1(0))}.
\end{equation}
\begin{defn}
	We say $\mathcal{F}\in\mathcal{Y}^Q$ is a \emph{gradient $Q$-Young measure}, and write $\mathcal{F}\in \grad\,\mathcal{Y}^Q$, if there exists a sequence of elementary Young measures $\mathcal{E}_{f_j}$ associated to $f_j\in W^{1,2}(B^n_1(0);\A_Q(\R^k))$ such that $\mathcal{E}_{f_j}\weakly\mathcal{F}$.
\end{defn}
\textbf{Note:} In \cite{HS24} gradient $Q$-Young measures are called \emph{$\mathcal{A}_Q$-generalised gradient Young measures}.

The following proposition provides a characterisation of gradient $Q$-Young measures.

\begin{prop}[{\cite[Proposition 2.5]{HS24}}]\label{prop:grad-Young}
	Let $\mathcal{F} = (F\otimes \ext x, F^\infty)\in \grad\, \mathcal{Y}^Q$. Then, there exists $f\in W^{1,2}(B^n_1(0);\A_Q(\R^k))$ and a family of probability measures $(\nu_{x,y})_{(x,y)\in B^n_1(0)\times\V_1}$ on $\V_2\times\V_3$ such that:
	\begin{enumerate}
		\item [(a)] For a.e.~$x\in B^n_1(0)$ we have, for $\phi\in C^0_{\textnormal{rec}}(\V)$,
		\begin{equation}\label{E:grad-Young-1}
			F_x(\phi) = \int_{\V_2\times\V_3}\sum^Q_{\alpha=1}\phi(f^\alpha(x),p,M)\, \ext\nu_{x,f^\alpha(x)};
		\end{equation}
		\item [(b)] For a.e.~$x\in B^n_1(0)$ we have, for any $i,j\in \{1,\dotsc,n\}$, $\kappa,\kappa^\prime\in\{1,\dotsc,k\}$, and $\alpha \in \{1,\dotsc,Q\}$,
		\begin{equation}\label{E:grad-Young-2}
			\int_{\V_2\times\V_3}p^\kappa_i\, \ext\nu_{x,f^\alpha(x)} = \del_i f^{\kappa,\alpha};
		\end{equation}
		\begin{equation}\label{E:grad-Young-3}
			\int_{\V_2\times\V_3} M_{ij}^{\kappa\kappa^\prime}\, \ext\nu_{x,f^\alpha(x)} = \int_{\V_2\times\V_3}p_i^\kappa p_j^{\kappa^\prime}\, \ext\nu_{x,f^\alpha(x)} \geq \del_i f^{\kappa,\alpha}\del_j f^{\kappa^\prime,\alpha};
		\end{equation}
		\item [(c)] $\spt(F^\infty)\subset B^n_1(0)\times \{M\in \del B_{1}^{\V_3}(0): M\geq 0\}$, and furthermore $F^\infty$ is positive, i.e.~if $\phi\geq 0$, then $F^\infty(\phi)\geq 0$.
	\end{enumerate}
	Here, $f^{\kappa,\alpha} \equiv e_\kappa\cdot f^\alpha$, where $e_1,\dotsc,e_k$ is the standard basis for $\R^k$.
\end{prop}
\textbf{Note:} In Proposition \ref{prop:grad-Young}, the inequalities in \eqref{E:grad-Young-3} and (c) are understood in the sense of matrices, i.e.~for real $m\times m$ matrices $A,B$, we write $A\geq B$ if and only if $x^T Ax\geq x^TBx$ for all $x\in \R^m$. In particular, $A\geq B$ if and only if $A-B$ is non-negative definite.

\begin{proof}
	We first prove (a). Fix a generating sequence $f_j\in W^{1,2}(B^n_1(0);\A_Q(\R^k))$ with $\CE_{f_j}\weakly\mathcal{F}$. First note that, by setting $\phi(x,y,p,M):= |y|^2 + M^{\kappa\kappa}_{ii}\in \mathcal{C}(B^n_1(0)\times\mathbb{V})$, we have
    $$\lim_{j\to\infty}\|f_j\|_{W^{1,2}(B^n_1(0);\A_Q(\R^k))} \equiv \lim_{j\to\infty}\mathcal{E}_{f_j}(\phi) = \mathcal{F}(\phi)$$
    and thus the sequence $\|f_j\|_{W^{1,2}(B^n_1(0);\A_Q(\R^k))}$ is bounded\footnote{This can also be seen by combining \eqref{E:Young-bounds-function} and the uniform boundedness principle, which gives $\limsup_j\|\mathcal{E}_{f_j}\|_{\mathcal{Y}^Q}<\infty$.}. We may therefore apply the Sobolev embedding theorem (for multi-valued functions, see e.g.~\cite[Proposition 2.11]{DLS11}) to find $f\in W^{1,2}(B^n_1(0);\A_Q(\R^k))$ such that, up to passing to a subsequence, $f_j\to f$ strongly in $L^p$ for all $p<2^*$ and $\|Df\|_{L^2}\leq\liminf_{j}\|Df_j\|_{L^2}$.
	
	Now let $\phi = \phi(x,y)\in C^1_c(B^n_1(0)\times\V_1)$. We have
	\begin{align*}
		|\mathcal{E}_{f_j}(\phi) - \mathcal{E}_f(\phi)| & = \left|\int_{B^n_1(0)}\sum^Q_{\alpha=1}\phi(x,f^\alpha_j(x)) - \sum^Q_{\beta=1}\phi(x,f^\beta(x))\, \ext x\right|\\
		& \leq \int_{B^n_1(0)}\inf_{\pi}\sum^Q_{\alpha=1}\left|\phi(x,f_j^\alpha(x))-\phi(x,f^{\pi(\alpha)}(x))\right|\, \ext x\\
		& \leq \int_{B^n_1(0)}\|D_y\phi\|_\infty\cdot\mathcal{G}(f_j(x),f(x))\, \ext x\\
		& = \|D_y\phi\|_\infty\|\mathcal{G}(f_j,f)\|_{L^1}\\
		& \to 0 \qquad \text{as }j\to\infty
	\end{align*}
	where in the second line the infimum is over all permutations $\pi:\{1,\dotsc,Q\}\to \{1,\dotsc,Q\}$. Since we also know $\mathcal{E}_{f_j}(\phi)\to \mathcal{F}(\phi)$, we see that we must have $\mathcal{F}(\phi) = \mathcal{E}_f(\phi)$ for such $\phi\in C^1_c(B^n_1(0)\times\mathbb{V}_1)$; by an approximation argument we then have $\mathcal{F}(\phi) = \mathcal{E}_f(\phi)$ for any $\phi\in C^0_c(B^n_1(0)\times\mathbb{V}_1)$. Furthermore, since $\phi$ only depends on $x,y$ here we have $\mathcal{F}(\phi) = (F\otimes \ext x)(\phi)$, where $F$ is as in Definition \ref{defn:Young} for $\mathcal{F}$. In particular, this is saying that
	$$(\pi_0\otimes\pi_1)_\#\mathcal{F} = \sum^Q_{\alpha=1}\llbracket f^\alpha\rrbracket\otimes\ext x.$$
	In particular, the classical disintegration theorem \cite[Theorem 2.28]{AFP00} implies (a).
	
	Next we prove \eqref{E:grad-Young-2}. We will divide this into three claims.
	
	\textbf{Claim 1:} Suppose there is a sequence of measurable sets $E_j\subset E$, where $E\subset B^n_1(0)$ is measurable, with $\H^n(E\setminus E_j)\to 0$, as well as two uniformly bounded sequences $f_j,g_j\in W^{1,2}(B^n_1(0);\A_Q(\R^k))$ with $\mathcal{E}_{f_j}\weakly \mathcal{F}$ and $\mathcal{E}_{g_j}\weakly \widetilde{\mathcal{F}}$ such that $f_j = g_j$ on $E_j$. Then, if $\mathcal{F} = (F\otimes \ext x, F^\infty)$ and $\widetilde{\mathcal{F}} = (\widetilde{F}\otimes \ext x, \widetilde{F}^\infty)$, we have $F_x = \widetilde{F}_x$ for a.e.~$x\in E$.
	
	\begin{proof}[Proof of Claim 1]
		This is a simple check. Indeed, approximate differentiability of $Q$-valued functions \cite[Corollary 2.7]{DLS11} gives that
		$$T_{x_0}f_j = T_{x_0}g_j$$
		for a.e.~$x\in E_j$, where $T_{x_0}f_j$ and $T_{x_0}g_j$ denote the first order approximations of $f_j$ and $g_j$ (respectively) at $x_0$. In particular, for a.e.~$x\in E_j$ we have $(E_{f_j})_x = (E_{g_j})_x$. Now, suppose that $\phi\in C^0_c(B^n_1(0)\times\V)$; then
		$$A:= \sup_{x,y,p,M}\frac{|\phi(x,y,p,M)|}{(1+|y|^{2^*}+|p|^2+|M|)^{1/2}} <\infty.$$
		Thus, we have
		\begin{align*}
			\int_{E\setminus E_j}|(E_{f_j})_x(\phi_x)\, +\, & |(E_{g_j})_x(\phi_x)|\, \ext x\\
			& \leq A\sum^Q_{\alpha=1}\int_{E\setminus E_j}(1+|f^\alpha_j(x)|^{2^*} + 2|Df^\alpha_j(x)|^2 + |g_j^\alpha(x)|^{2^*} + 2|Dg_j^\alpha(x)|^2)^{1/2}\\
			& = C(n)A\left(\|f_j\|_{W^{1,2}} + \|g_j\|_{W^{1,2}}\right)^{1/2}\H^n(E\setminus E_j)^{1/2}\\
		\end{align*}
		which therefore $\to 0$ as $j\to\infty$. As $(E_{f_j})_x = (E_{g_j})_x$ on $E_j$, we therefore see that
		$$\int_E F_x(\phi_x)\, \ext x = \int_{E} \widetilde{F}_x(\phi_x)\, \ext x.$$
		As $\phi\in C^0_c(B^n_1(0)\times \V)$ was arbitrary, we then conclude that $F_x = \widetilde{F}_x$ for a.e.~$x\in E$.
	\end{proof}
	
	\textbf{Claim 2:} \eqref{E:grad-Young-2} holds a.e.~on $E_0:= \{f = Q\llbracket f_a\rrbracket\}$.\footnote{Recall that $f_a \equiv \frac{1}{Q}\sum^Q_{\alpha=1}f$ denotes the average of $f$, and so $a$ is not a varying index.}
	
	\begin{proof}[Proof of Claim 2]
		Note first that $(f_j)_a\weakly f_a$ in $W^{1,2}(B^n_1(0);\R^k)$. Therefore, for any $\phi\in C^0_c(B^n_1(0))$, if we set $\psi(x,p):= \phi(x)p^\kappa_i$ for some $i\in\{1,\dotsc,n\}$ and $\kappa\in\{1,\dotsc,k\}$, we have
		$$\mathcal{F}(\psi) = \lim_{j\to\infty}\mathcal{E}_{f_j}(\psi) = \lim_{j\to\infty}\int_{B^n_1(0)}\phi(x)\sum^Q_{\alpha=1} \del_i f^{\kappa,\alpha} = Q\lim_{j\to\infty}\int_{B^n_1(0)}\phi(x)\del_i (f^{\kappa}_j)_a = Q\int_{B^n_1(0)}\phi(x)\del_i f^\kappa_a.$$
		But since such $\psi$ have vanishing recession function, we have $\mathcal{F}(\psi) = (F\otimes\ext x)(\psi)$, and thus using \eqref{E:grad-Young-1} we have
		$$\int_{B^n_1(0)}\phi(x)\int_{\V_2\times\V_3}\sum^Q_{\alpha=1}p^\kappa_i\, \ext\nu_{x,f^\alpha(x)}\, \ext x = Q\int_{B^n_1(0)}\phi(x)\del_if^{\kappa}_a.$$
	Since $\phi\in C^0_c(B^n_1(0))$ was arbitrary, this therefore implies for a.e.~$x\in B^n_1(0)$,
	\begin{equation}\label{E:grad-Young-2-1}
	\sum^Q_{\alpha=1}\int_{\V_2\times\V_3}p^{\kappa}_i\, \ext\nu_{x,f^\alpha(x)} = Q\del_i f^{\kappa}_a.
	\end{equation}
    Hence since on $\{f = Q\llbracket f_a\rrbracket\}$ we have $f^\alpha(x) = f_a$ for $\alpha=1,2,\dotsc,Q$, we see that for a.e.~$x\in E_0$ the above gives
	$$\int_{\V_2\times\V_3}p^\kappa_i\, \ext\nu_{x,f_a(x)} = \del_i f^\kappa_a(x).$$
	This proves \eqref{E:grad-Young-2} for a.e.~$x\in E_0$, proving Claim 2.
	\end{proof}
	
	\textbf{Claim 3:} \eqref{E:grad-Young-2} holds.
	
	\begin{proof}
		We work by induction on $Q$. The case $Q=1$ is just a statement about weak convergence of $W^{1,2}$ functions, and so is true. We may therefore assume that \eqref{E:grad-Young-2} holds for all $Q^\prime<Q$. By Claim 2, we know that \eqref{E:grad-Young-2} holds for a.e.~$x\in E_0:= \{f = Q\llbracket f_a\rrbracket\}$.
		
		Now fix any $S\in \A_Q(\R^k)$ with $S\neq Q\llbracket s\rrbracket$ for any $s\in \R^k$. We may therefore find $\eps_S>0$ and (globally defined) Lipschitz retractions $\chi_i:\A_Q(\R^k)\to \A_{Q_i}(\R^k)$, $i=1,2$, such that whenever $\G(T,S)<2\eps_S$ we may write
		$$T = \chi_1(T) + \chi_2(T).$$
		Now consider $E := \{x:\G(f(x),S)<\eps_S\}$, and define the sequences
		$$f_j^{(i)} := \chi_i\circ f_j\in W^{1,2}(B^n_1(0);\A_{Q_i}(\R^k)) \qquad \text{for }i=1,2.$$
		Also write $E_j := \{x\in E:\G(f_j(x),S)<\eps_S\}$ for each $j$. Clearly we have $\H^n(E\setminus E_j)\to 0$ since $f_j\to f$ a.e.~$x\in B^n_1(0)$. Furthermore, for a.e.~$x\in E_j$,
		$$(E_{f_j^{(1)}})_x + (E_{f_j^{(2)}})_x = (E_{f_j})_x$$
		by the almost everywhere differentiability of multi-valued $W^{1,2}$ functions. We may also assume, passing to a subsequence if necessary, that for $i=1,2$ we have $\mathcal{E}_{f_j^{(i)}}\weakly\mathcal{F}^{(i)} \equiv (F^{(i)}\otimes\ext x,(F^{(i)})^\infty)$ as $j\to\infty$. By Claim 1, we see that for a.e.~$x\in E$,
		$$F\otimes \ext x = F^{(1)}\otimes \ext x + F^{(2)}\otimes \ext x.$$
		Hence, by applying our inductive hypothesis to $f_j^{(1)}$ and $f_j^{(2)}$, we deduce that \eqref{E:grad-Young-2} holds for a.e.~$x\in E$. Since $S\in \A_Q(\R^k)\setminus \{Q\llbracket s\rrbracket: s\in \R^k\}$ was arbitrary, by covering $B^n_1(0)\setminus E_0$ by sets of the above type, and combining with Claim 2, we conclude the proof of \eqref{E:grad-Young-2}.
	\end{proof}
	
	Next we prove \eqref{E:grad-Young-3}. Fix $\chi\in C^\infty_c(\V_3)$ with $\chi\equiv 1$ on $B_1^{\V_3}(0)$ and let $\psi\in C^\infty_c(B^n_1(0)\times\V_1)$ be arbitrary. Now consider, for $R>0$,
	$$\phi(x,y,M):=\psi(x,y)\chi(M/R)M_{ij}^{\kappa\kappa^\prime} \qquad \text{and} \qquad \widetilde{\phi}(x,y,p):=\psi(x,y)\chi\left(\frac{p\otimes p}{R}\right)p^\kappa_i p^{\kappa^\prime}_j.$$
	Then, since $\phi$ and $\widetilde{\phi}$ have vanishing recession functions, and since $\mathcal{E}_{f_j}(\phi) = \mathcal{E}_{f_j}(\widetilde{\phi})$, we have
	\begin{align*}
		\int\sum^Q_{\alpha=1}\psi(x,f^\alpha(x))\int\chi\left(\frac{M}{R}\right)M^{\kappa\kappa^\prime}_{ij}\, \ext\nu_{x,f^\alpha(x)}\, \ext x & = (F\otimes \ext x)(\phi)\\
		& = \lim_{j\to\infty}\mathcal{E}_{f_j}(\phi)\\
		& = \lim_{j\to\infty}\mathcal{E}_{f_j}(\widetilde{\phi})\\
		& = \int\sum^Q_{\alpha=1}\psi(x,f^\alpha(x))\int\chi\left(\frac{p\otimes p}{R}\right)p^\kappa_i p^{\kappa^\prime}_j\, \ext\nu_{x,f^\alpha(x)}\, \ext x.
	\end{align*}
	Taking $R\to\infty$ in this equality and using the dominated convergence theorem, followed by the arbitrariness of $\psi\in C^\infty_c(B^n_1(0)\times\V_1)$, we see that for a.e.~$x\in B^n_1(0)$ the first equality in \eqref{E:grad-Young-3} holds. The inequality part of \eqref{E:grad-Young-3} then follows from Jensen's inequality applied to the convex function $G(p):= p^\kappa_i p^{\kappa^\prime}_j \xi^i\xi^j \eta_\kappa\eta_{\kappa^\prime}$, for fixed $\xi\in \R^n$ and $\eta\in \R^k$ (here we are using summation convention). Indeed, Jensen's inequality gives
	\begin{align*}
		\xi^i\xi^j\eta_k\eta_{\kappa^\prime}\int p^{\kappa}_ip^{\kappa^\prime}_j\, \ext\nu_{x,f^\alpha(x)} = \int G(p)\, \ext\nu_{x,f^\alpha(x)} & \geq G\left(\int p\, \ext\nu_{x,f^\alpha(x)}\right)\\
		& = \xi^i\xi^j\eta_{\kappa}\eta_{\kappa^\prime}\left(\int p^{\kappa}_i\, \ext\nu_{x,f^\alpha(x)}\right)\left(\int p^{\kappa^\prime}_j\, \ext\nu_{x,f^\alpha(x)}\right)\\
		& = \xi^i\xi^j\eta_{\kappa}\eta_{\kappa^\prime}\del_i f^{\kappa,\alpha}\del_j f^{\kappa^\prime,\alpha}\\
		& \equiv (\eta\cdot\del_\xi f^\alpha(x))^2
	\end{align*}
	which is the claimed inequality; here, we have used \eqref{E:grad-Young-2} in the second-to-last equality.
	
	Finally we prove (c). Let $M_0\in \del B^{\V_3}_1(0)$ be such that $\dist(M_0,\{M\geq0\}) = 2\delta>0$. Suppose $\phi \in C^0_c(B^n_1(0)\times \del B_{1}^{\V_3}(0))$ has $\spt(\phi)\subset B^n_1(0)\times (B_{\delta}(M_0)\cap \del B^{\V_3}_1(0))$. Let $\phi^\infty(M):= |M|\phi(x,M/|M|)$ be its $1$-homogeneous extension in the $M$-variable. Then consider:
	$$\phi_R(x,M) := \rho(|M|/R)\phi^\infty(x,M),$$
	where $\rho:[0,\infty)\to \R$ is a non-decreasing function with $\rho(t)\equiv 0$ for $t<1/2$ and $\rho(t)\equiv 1$ for $t\geq 1$. Note that $\phi_R^\infty := \phi^\infty$ is the recession function of $\phi_R$. Thus, by dominated convergence, we have
	\begin{align*}
		F^\infty(\phi) \equiv F^\infty(\phi^\infty) = \lim_{R\to \infty}\mathcal{F}(\phi_R) = \lim_{R\to\infty}\lim_{j\to\infty}\mathcal{E}_{f_j}(\phi_R) = 0,
	\end{align*}
	where the last equality follows because $\mathcal{E}_{f_j}(\phi_R) = 0$ for any $R>0$ and $j\geq 1$, as $Df(x)\otimes Df(x)$ is always a non-negative definite matrix. This therefore proves the first part of (c). The second part of (c) follows immediately from the proof of Proposition \ref{prop:Y-compactness} and how the concentration part of the $Q$-Young measure arises. This therefore completes the proof of Proposition \ref{prop:grad-Young}.
\end{proof}

\subsection{Stationary Gradient $Q$-Young Measures}

Throughout this section, we fix $\mathcal{F} = (F\otimes \ext x, F^\infty)\in \grad\,\mathcal{Y}^Q$. We consider the action of $\mathcal{F}$ on two specific types of functions in $\CC(B^n_1(0)\times \V)$, namely:
\begin{itemize}
	\item the \emph{outer variation}
	$$\mathfrak{O}(\phi) := \mathcal{F}(p_i^\kappa y^\kappa \del_i \phi + M_{ii}^{\kappa\kappa}\phi) \qquad \text{for }\phi\in C^\infty_c(B^n_1(0));$$
	\item the \emph{inner variation}
	$$\mathfrak{I}(\phi) := \mathcal{F}((2M^{\kappa\kappa}_{ij} - \delta_{ij}M^{\kappa\kappa}_{\ell\ell})\del_i\phi^j) \qquad \text{for }\phi\in C^\infty_c(B^n_1(0);\V_1).$$ 
\end{itemize}
Here, our convention is that a repeated index is summed over.

We will assume throughout this section that:
\begin{itemize}
	\item $\mathcal{F}$ is \emph{stationary}, meaning that $\mathfrak{O}\equiv 0$ and $\mathfrak{I}\equiv 0$;
	\item $f\not\equiv 0$ on $B^n_1(0)$, where $f\in W^{1,2}(B^n_1(0);\A_Q(\R^k))$ is as in Proposition \ref{prop:grad-Young}.
\end{itemize}

Fix $\phi$ the Lipschitz cut-off function from Section \ref{sec:pff}. For $x_0\in B^n_1(0)$ and $r\in (0,1-|x_0|)$, we define
$$\mathcal{D}(x_0,r) := r^{2-n}\mathcal{F}\left(\phi\left(\frac{|x-x_0|}{r}\right)M_{ii}^{\kappa\kappa}\right),$$
$$\mathcal{H}(x_0,r) := r^{1-n}\mathcal{F}\left(-\phi^\prime\left(\frac{|x-x_0|}{r}\right)\frac{1}{|x-x_0|}|y|^2\right).$$
We stress that whilst $-\phi^\prime$ is not strictly a continuous function, one can still make sense of $\mathcal{H}(x_0,r)$ by approximation, and indeed using Proposition \ref{prop:grad-Young} we have
\begin{equation}\label{E:Y-H-alt}
\mathcal{H}(x_0,r) = -r^{1-n}\int \phi^\prime\left(\frac{|x-x_0|}{r}\right)\frac{1}{|x-x_0|}|f|^2.
\end{equation}
Then, provided $\mathcal{H}(x_0,r)>0$, we define the \emph{frequency function} by
$$\mathcal{I}(x_0,r) := \frac{\mathcal{D}(x_0,r)}{\mathcal{H}(x_0,r)}.$$

\begin{remark}\label{remark:Y-approximation}
    At various points we will need to evaluate $\mathcal{F}$ on certain functions arising from $\phi^\prime$ which are not continuous, as well as differentiating $\mathcal{D}(r)$. All of this can be justified by Proposition \ref{prop:grad-Young}, since we can regard $F_x$ as corresponding to the probability measures $\nu_{x,f^\alpha(x)}$ and we can regard $F^\infty$ as a positive Radon measure (and so $\|F^\infty\|(\{|x|=r\}) = 0$ for a.e.~$r\in (0,1)$). Thus, by standard approximation arguments and the Dominated Convergence Theorem, $\mathcal{D}(r)$ is absolutely continuous on $(0,1)$, and furthermore equations such as \eqref{E:Y-freq-1} and the formula for $\mathcal{D}^\prime(r)$ below both hold true for a.e.~$r\in (0,1)$.
\end{remark}

\begin{prop}[{\cite[Proposition 3.7(4)]{HS24}}]\label{prop:Y-frequency}
	Let $\mathcal{F}\in\grad\,\mathcal{Y}^Q$ be as above, and suppose $x_0\in B^n_1(0)$ is such that $f\not\equiv 0$ on $B_{r}(x_0)$ for all $r\in (0,1-|x_0|)$. Then, $r\mapsto \mathcal{I}(x_0,r)$ is monotone non-decreasing on $(0,1-|x_0|)$.
\end{prop}

\textbf{Remark:} In Remark \ref{remark:Y-defined}, we will see that one can remove the assumption that $f\not\equiv 0$ on $B_r(x_0)$ for all $r\in (0,1-|x_0|)$, as this will be implied by the assumption that $f\not\equiv 0$ on $B_1(0)$.

\begin{proof}
	We can without loss of generality assume that $x_0 = 0$, else consider $(\tilde{\eta}_{x_0,1-|x_0|})_\#\mathcal{F}$, where $\tilde{\eta}_{x_0,r}(x,y,p,M) := (\frac{x-x_0}{r},y,rp,r^2M)$. We will write $\mathcal{D}(r)\equiv \mathcal{D}(0,r)$ and so forth.
	
	For $\phi$ the cut-off function as in Section \ref{sec:pff}, we know by assumption that $\mathfrak{I}(\phi(|x|/r)x) = 0$ and $\mathfrak{O}(\phi(|x|/r)) = 0$ (cf.~Remark \ref{remark:Y-approximation}). These give the identities:
	\begin{equation}\label{E:Y-freq-1}
		\mathcal{F}\left((2-n)\phi(|x|/r)M_{ii}^{\kappa\kappa}-\phi^\prime(|x|/r)\frac{|x|}{r}M_{ii}^{\kappa\kappa}\right) + 2\mathcal{F}\left(\phi^\prime\left(\frac{|x|}{r}\right)\frac{|x|}{r}\frac{x^T}{|x|}M^{\kappa\kappa}\frac{x}{|x|}\right) = 0;
	\end{equation}
	\begin{equation}\label{E:Y-freq-2}
		\mathcal{F}(\phi(|x|/r)M_{ii}^{\kappa\kappa}) + \int\frac{\phi^\prime(|x|/r)}{|x|r}\sum^Q_{\alpha=1}f^\alpha(x)\cdot Df^\alpha(x) x\, \ext x = 0;
	\end{equation}
	where we have used \eqref{E:grad-Young-2} in the outer variation computation. We may then compute:
	\begin{align*}
		\mathcal{D}^\prime(r) & = r^{1-n}\mathcal{F}\left((2-n)\phi(|x|/r)M_{ii}^{\kappa\kappa}-\phi^\prime(|x|/r)\frac{|x|}{r}M_{ii}^{\kappa\kappa}\right)\\
		& = -2r^{1-n}\mathcal{F}\left(\phi^\prime(|x|/r)\frac{|x|}{r}\frac{x^T}{|x|}M^{\kappa\kappa}\frac{x}{|x|}\right),
	\end{align*}
	where the second equality comes from \eqref{E:Y-freq-1}. To compute the derivative of $\mathcal{H}(r)$, note that by \eqref{E:Y-H-alt} and a change of variables $w = x/r$ we have
	$$\mathcal{H}(r) = -\int\phi^\prime(w)\frac{1}{|w|}|f(rw)|^2\, \ext w$$
	and so, differentiating under the integral sign and changing variables back to $x$, we get
	$$\mathcal{H}^\prime(r) = -2r^{1-n}\int\frac{\phi^\prime(|x|/r)}{r|x|}\sum^Q_{\alpha=1}f^\alpha(x)\cdot Df^\alpha(x)x\, \ext x.$$
	Thus, \eqref{E:Y-freq-2} gives
	\begin{equation}\label{E:Y-freq-3}
	\mathcal{H}^\prime(r) = \frac{2}{r}\mathcal{D}(r).
	\end{equation}
	In particular, Proposition \ref{prop:grad-Young}(c) gives that both $\mathcal{D}^\prime(r)\geq 0$ and $\mathcal{H}^\prime(r)\geq 0$ (recall that $-\phi^\prime\geq 0$), and thus both $\mathcal{H}(r)$ and $\mathcal{D}(r)$ are non-decreasing.
	
	Let us note that the given condition gives that $\mathcal{H}(r)>0$ for all $r\in (0,1)$. Indeed, if $\mathcal{H}(r) = 0$ for some $r$, then $f\equiv 0$ on $B_r(0)\setminus\overline{B}_{r/2}(0)$. But then from \eqref{E:Y-freq-2} (combined with \eqref{E:grad-Young-3}), we would need $|Df|\equiv 0$ on $B_r(0)$, and thus that $f$ is constant on $B_r(0)$. As it is zero on $B_r(0)\setminus\overline{B}_{r/2}(0)$, it would therefore need $f\equiv 0$ on $B_r(0)$, contradicting our assumption.
	
	Thus, $\mathcal{H}(r)>0$ for all $r\in (0,1)$, and so $\mathcal{I}(r)$ is well-defined for all $r\in (0,1$). Using the above expressions we can then compute:
	\begin{align*}
		\mathcal{I}^\prime(r) & = \frac{\mathcal{D}^\prime(r)\mathcal{H}(r) - \mathcal{D}(r)\mathcal{H}^\prime(r)}{\mathcal{H}(r)^2}\\
		& = \frac{2r^{1-2n}}{\mathcal{H}(r)^2}\left[\mathcal{F}\left(-\phi^\prime(|x|/r)\frac{|x|}{r}\frac{x^T}{|x|}M^{\kappa\kappa}\frac{x}{|x|}\right)\left(\int-\phi^\prime(|x|/r)\frac{r}{|x|}|f|^2\right)\right.\\
		& \hspace{18em}\left.-\left(\int\phi^\prime(|x|/r)\sum^Q_{\alpha=1}f^\alpha(x)\cdot Df^\alpha(x)\frac{x}{|x|}\, \ext x\right)^2\right].
	\end{align*}
	But now using Cauchy--Schwartz followed by \eqref{E:grad-Young-1}, \eqref{E:grad-Young-3}, and (c) of Proposition \ref{prop:grad-Young} (as then the function inside $\mathcal{F}$ in the above is non-negative due to the restriction on the support of $\mathcal{F}$), we see that $\mathcal{I}^\prime(r)\geq 0$, completing the proof.
\end{proof}

In particular, Proposition \ref{prop:Y-frequency} tells us that about any point $x_0\in B^n_1(0)$ for which $f\not\equiv 0$ about $x_0$, we have that the \emph{frequency} of $f$ at $x_0$, namely
$$\mathcal{I}(x_0) := \lim_{r\to 0}\mathcal{I}(x_0,r)$$
is well-defined as a number in $[0,\infty)$.

We get several immediate corollaries from Proposition \ref{prop:Y-frequency}.

\begin{prop}[{\cite[Proposition 3.7(5)]{HS24}}]\label{prop:Y-bounds}
	Let $\mathcal{F}$ be as above. Then for a given $x_0\in B^n_1(0)$, we either have that $f\equiv 0$ on $B_{1-|x_0|}(x_0)$, or $\mathcal{H}(x_0,r)>0$ for all $r\in (0,1-|x_0|)$.
	
	In fact, for any $0<r\leq R<1-|x_0|$ we always have
	$$\left(\frac{R}{r}\right)^{2\mathcal{I}(x_0,r)}\mathcal{H}(x_0,r) \leq \mathcal{H}(x_0,R) \leq \left(\frac{R}{r}\right)^{2\mathcal{I}(x_0,R)}\mathcal{H}(x_0,r).$$
\end{prop}

\begin{remark}\label{remark:Y-defined}
An immediate consequence of Proposition \ref{prop:Y-bounds} is that either $f\equiv 0$, or $\mathcal{I}(x_0,r)$ is well-defined for all $x_0\in B^n_1(0)$ and $r\in (0,1-|x_0|)$. Indeed, if $\mathcal{H}(x_0,r) = 0$ for some $x_0\in B^n_1(0)$ and $r\in (0,1-|x_0|)$, then Proposition \ref{prop:Y-bounds} gives that $f\equiv 0$ on $B^n_{1-|x_0|}(x_0)$. But then for all $y\in B^n_{1-|x_0|}(x_0)$ and sufficiently small $r$ we have $\mathcal{H}(y,r) = 0$, and so again $f\equiv 0$ on $B^n_{1-|y|}(y)$ for all $y\in B^n_{1-|x_0|}(x_0)$. Repeating this eventually we get that we can take $y = 0$, and so $f\equiv 0$ on $B^n_1(0)$.
\end{remark}

In particular, since we assumed that $f\not\equiv 0$ on $B_1^n(0)$, we see that $\mathcal{H}(x_0,r)>0$ for all $x_0\in B^n_1(0)$ and $r\in (0,1-|x_0|)$, meaning that $\mathcal{I}(x_0,r)$ is well-defined always.

\begin{proof}
	Again by translating and rescaling we can without loss of generality assume that $x_0 = 0$.
	
	If $\mathcal{H}(r_0)>0$, then $\mathcal{H}(r)>0$ for all $r\in (r_0-\eps,r_0+\eps)$ for some $\eps>0$; let us therefore suppose $(a,b)$ is the largest interval containing $r_0$ on which $\mathcal{H}>0$ (we will show that in fact $(a,b) = (0,1)$). We can then compute for $r\in (a,b)$:
	$$\frac{\ext}{\ext r}\log \mathcal{H}(r) = \frac{\mathcal{H}^\prime(r)}{\mathcal{H}(r)} = \frac{2}{r}\mathcal{I}(r)$$
	where we have used \eqref{E:Y-freq-3}. Integrating we therefore get for $r,R\in (a,b)$, $r<R$,
	$$\frac{\mathcal{H}(R)}{\mathcal{H}(r)} = e^{\int^R_r\frac{2}{\rho}\mathcal{I}(\rho)\, \ext\rho}.$$
	Using the monotonicity of $\mathcal{I}$ from Proposition \ref{prop:Y-frequency} we therefore get
	$$\left(\frac{R}{r}\right)^{2\mathcal{I}(r)}\mathcal{H}(r) \leq\mathcal{H}(R) \leq \left(\frac{R}{r}\right)^{2\mathcal{I}(R)}\mathcal{H}(r).$$
	In particular, since $\mathcal{H}$ is increasing (from \eqref{E:Y-freq-3}), as $\mathcal{H}(r_0)>0$ we know $\mathcal{H}>0$ on $(a,1)$, i.e.~$b=1$. Thus, for $r\in (a,1)$ we have
	$$\mathcal{H}(r) \geq r^{2\mathcal{I}(1)}\mathcal{H}(1).$$
	This provides a constant lower bound on $\mathcal{H}(r)$, and thus we see that we must in fact have $\mathcal{H}(r)>0$ for all $r\in (0,1)$, i.e.~$a=0$. This completes the proof.
\end{proof}

\begin{prop}\label{prop:Y-usc}
	Let $\mathcal{F}$ be as above. Then, the function $x\mapsto \mathcal{I}(x)$ is upper semi-continuous.
\end{prop}

\begin{proof}
	Suppose we have $x_i\to x\in B^n_1(0)$ with $\mathcal{I}(x_i)$ defined for all $i$; in particular, $\mathcal{I}(x)$ must be well-defined, else $f\equiv 0$ on a neighbourhood of $x$ which implies that $f\equiv 0$ on a neighbourhood of $x_i$ for all $i$ sufficiently large, a contradiction. Now fix $r>0$. We clearly have $\mathcal{D}(x_i,r)\to \mathcal{D}(x,r)$ and $\mathcal{H}(x_i,r)\to \mathcal{H}(x,r)$, and so $\mathcal{I}(x_i,r)\to \mathcal{I}(x,r)$. Hence from the monotonicity of $\mathcal{I}$,
	$$\limsup_{i\to\infty}\mathcal{I}(x_i) \leq \limsup_{i\to\infty} \mathcal{I}(x_i,r) = \mathcal{I}(x,r).$$
	Taking $r\downarrow 0$ then gives the result.
\end{proof}

From the proof of Proposition \ref{prop:Y-frequency}, we get a characterisation of constant frequency via \emph{homogeneity}, i.e.:

\begin{defn}
	For $I\in [0,\infty)$, we say that $\mathcal{F}\in \grad\,\mathcal{Y}^Q$ is $I$\emph{-homogeneous} if:
	\begin{itemize}
		\item $\frac{\del}{\del |x|}\left(\frac{f(x)}{|x|^I}\right) = 0$ for a.e.~$x\in B^n_1(0)$;
		\item $\int x^TM^{\kappa\kappa}x\, \ext\nu_{x,f^\alpha(x)} = I^2|f^{\alpha}(x)|^2$ for a.e.~$x\in B^n_1(0)$;
		\item $F^\infty(x^TM^{\kappa\kappa}x) = 0$.
	\end{itemize}
	Here, $f$ is as in Proposition \ref{prop:grad-Young} for $\mathcal{F}$.
\end{defn}

\begin{prop}[{\cite[Corollary 3.8]{HS24}}]\label{prop:Y-constant-frequency}
	Let $\mathcal{F}$ be as above. Then, $\mathcal{I}(0,r)\equiv I$ for $r\in (0,1)$ if and only if $\mathcal{F}$ is $I$-homogeneous.
\end{prop}

\begin{proof}
	This follows immediately from the proof of Proposition \ref{prop:Y-frequency} via the equality case in the Cauchy--Schwarz inequality as well as the inequalities used from Proposition \ref{prop:grad-Young}. To see the homogeneity is determined by the (constant) frequency value, we note that the expression for $\mathcal{I}^\prime(r)$ can also be written as
    \begin{align*}
        \mathcal{I}^\prime(r) & = \frac{2r^{1-n}}{\mathcal{H}(r)}\left[\mathcal{F}\left(-\phi^\prime(|x|/r)\frac{|x|}{r}\frac{x^T}{|x|}M^{\kappa\kappa}\frac{x}{|x|}\right) + \int\phi^\prime(|x|/r)\frac{1}{|x|r}\sum^Q_{\alpha=1}|Df^\alpha(x)x|^2\, \ext x\right.\\
        &\hspace{15em}\left. -\int\phi^\prime(|x|/r)\frac{1}{|x|r}\sum^Q_{\alpha=1}|Df^\alpha(x)x-\mathcal{I}(r)f^\alpha(x)|^2\, \ext x\right].
    \end{align*}
\end{proof}

It is only at this point, in our final two propositions, that we deviate from \cite{HS24}. Nonetheless, we use arguments similar to those seen in \cite[Proposition 3.7 \& Lemma 3.11]{HS24}.

If $\mathcal{F}$ as above is $I$-homogeneous for some $I>0$, we know that the associated function $f\in W^{1,2}(B^n_1(0);\A_Q(\R^k))$ is $I$-homogeneous; in such a situation, we will abuse notation and also write $f\in W^{1,2}(\R^n;\A_Q(\R^k))$ for its $I$-homogeneous extension to $\R^n$. In particular, standard arguments imply that its \emph{spine}
$$S_f := \{z\in\R^n: f(x+z) = f(x) \text{ for a.e. }x\in \R^n\}$$
is a subspace of $\R^n$. In this situation, we claim:
\begin{prop}\label{prop:Y-spine}
	Suppose $\mathcal{F}$ as above is $I$-homogeneous for some $I>0$. Then:
	$$S_f\cap B^n_1(0) = \{x\in B^n_1(0): \mathcal{I}(x) = \mathcal{I}(0)\}.$$
	Furthermore, $\mathcal{F}$ is $I$-homogeneous centered at any $x\in S_f\cap B^n_1(0)$.
\end{prop}

\begin{proof}
	First notice that the inclusion $S_f\cap B^n_1(0)\subset\{x \in B^n_1(0):\mathcal{I}(x) = \mathcal{I}(0)\}$ is clear. Indeed, if $z\in S_f$, then
	\begin{align*}
	\mathcal{H}(z,r) & = -r^{1-n}\int\phi^\prime\left(\frac{|x-z|}{r}\right)\frac{1}{|x-z|}|f(x)|^2\, \ext x\\
	& = -r^{1-n}\int\phi\left(\frac{|x|}{r}\right)\frac{1}{|x|}|f(x+z)|^2\, \ext x\\
	& = -r^{1-n}\int\phi\left(\frac{|x|}{r}\right)\frac{1}{|x|}|f(x)|^2\, \ext x\\
	& = \mathcal{H}(0,r),
	\end{align*}
	and so by Proposition \ref{prop:Y-bounds}, gives for any $0<r<R<1-|z|$,
	\begin{align*}
		\mathcal{H}(z,R) \leq \left(\frac{R}{r}\right)^{2\mathcal{I}(z,R)}\mathcal{H}(z,r) & = \left(\frac{R}{r}\right)^{2\mathcal{I}(z,R)}\mathcal{H}(0,r)\\
		& \leq \left(\frac{R}{r}\right)^{2(\mathcal{I}(z,R)-\mathcal{I}(0,r))}\mathcal{H}(0,R) =  \left(\frac{R}{r}\right)^{2(\mathcal{I}(z,R)-\mathcal{I}(0,r))}\mathcal{H}(z,R) 
	\end{align*}
	i.e.
	$$\left(\frac{R}{r}\right)^{\mathcal{I}(z,R)-\mathcal{I}(0,r)} \geq 1,$$
	and hence $\mathcal{I}(0,r)\leq \mathcal{I}(z,R)$ for all $0<r<R<1-|z|$. Swapping the roles of $0$ and $z$ in the above argument also gives $\mathcal{I}(z,r) \leq \mathcal{I}(0,R)$ for all $0<r<R<1-|z|$, and hence taking $r\downarrow 0$ followed by $R\downarrow 0$ we see that we need $\mathcal{I}(0) = \mathcal{I}(z)$. This proves the first inclusion.
	
	Now we wish to show that if $z\in B^n_1(0)$ has $\mathcal{I}(z) = \mathcal{I}(0)$, then $z\in S_f$. For this, we will show that $\mathcal{I}(z,r)$ is constant with $r$ (where defined), giving that we can apply Proposition \ref{prop:Y-constant-frequency} to give that $f$ is homogeneous of degree $I \equiv \mathcal{I}(0)$ based at $z$\footnote{Notice that, as $f$ is defined on $\R^n$ from a homogeneous extension on $B^n_1(0)$, it suffices to show this homogeneity within $B^n_1(0)$ in order to get it on $\R^n$.}, from which a standard argument gives $z\in S_f$. Indeed, if we know that, for any $\lambda>0$, $f(\lambda x) = \lambda^I f(x)$ and $f(z+\lambda x) = \lambda^If(z+x)$ for a.e.~$x\in\R^n$, then
	$$f(x+z) = 2^I f\left(\frac{x+z}{2}\right) = 2^If\left(z + \frac{x-z}{2}\right) = f(z+ (x-z)) = f(x)$$
	for a.e.~$x\in \R^n$, and thus $z\in S_f$. So, it suffices to show that $\mathcal{I}(z,r)$ is constant in $r\in (0,1-|z|)$.

    To see this, fix $z\in S_f$ and $r\in (0,1-|z|)$. By the homogeneity of $f$ based at $0$ and \eqref{E:Y-freq-2} (or rather, the analogous formula centered at an arbitrary point rather than $0$) one can easily check that for all $t\in (0,1)$ we have $\mathcal{H}(tz,tr) = t^{2I}\mathcal{H}(z,r)$ and $\mathcal{D}(tz,tr) = t^{2I}\mathcal{D}(z,r)$, and hence $\mathcal{I}(tz,tr) = \mathcal{I}(z,r)$ for any such $z,r,t$. But then by the monotonicity formula
    $$\mathcal{I}(z,r) = \lim_{t\downarrow 0}\mathcal{I}(tz,tr) \leq \lim_{t\to 0}\mathcal{I}(tz,r).$$
    Now, for $r>0$ we have $\mathcal{D}(tz,r)\to \mathcal{D}(0,r)$ and $\mathcal{H}(tz,r)\to \mathcal{H}(0,r)$, and so $\lim_{t\to 0}\mathcal{I}(tz,r) = \mathcal{I}(0,r)$. Therefore,
    $$\mathcal{I}(z,r) \leq \lim_{t\to 0}\mathcal{I}(tz,r) = \mathcal{I}(0,r) = I,$$
    where the last equality is from the $I$-homogeneity. But then since $\mathcal{I}(z,r)$ is increasing and $\lim_{r\to 0}\mathcal{I}(z,r) = I$ by assumption this means that $\mathcal{I}(z,r)$ is constant for $r\in (0,1-|z|)$, which is what we needed to prove in order to complete the proof.
\end{proof}

Finally, we have:

\begin{prop}\label{prop:Y-spine-dimensions}
	Suppose $\mathcal{F}$ as above is $I$-homogeneous for some $I>0$ with $f\not\equiv 0$. Then $\dim S_f \in \{0,1,\dotsc,n-1\}$. Moreover, $\dim S_f = n-1$ only if $I = 1$.
\end{prop}

\begin{proof}
	Since $S_f$ is a subspace of $\R^n$ we know that $\dim S_f \in \{0,1,\dotsc,n\}$. However, if $\dim S_f = n$, then $f$ is constant almost everywhere, which being homogeneous of degree $I>0$ would mean that $f\equiv 0$, contradicting our assumption that $f\not\equiv 0$. Thus, $\dim S_f \in \{0,1,\dotsc,n-1\}$.
	
	Now suppose $\dim S_f = n-1$: we want to show that $I = 1$. Without loss of generality choose coordinates of $\R^n$ so that $S_f = \{0\}\times\R^{n-1}$, and thus $f(x) \equiv f(x_1)$. By Proposition \ref{prop:Y-spine}, we know that $\mathcal{I}(z,r)\equiv I$ for all $z\in S_f$. As $S_f$ is a subspace, this means that for any $t\in \R$ and $i=2,\dotsc,n$, if $e_i$ denotes the $i^{\text{th}}$ standard basis vector of $\R^n$, then for any $\psi\in C^0_c(B^n_1(0))$:
	$$\mathcal{F}(\psi(x)(x-te_i)^TM^{\kappa\kappa}(x-te_i)) = \int\psi(x)\sum^Q_{\alpha=1}|Df(x)(x-te_i)|^2.$$
	Indeed, this follows from (the proof of) Proposition \ref{prop:Y-frequency}, as the derivative of $r\mapsto \mathcal{I}(te_i,r)$ is zero (cf.~Proposition \ref{prop:Y-constant-frequency}). Differentiating the above expression twice in $t$ and using that $f$ is independent of $x_i$ then gives
    \begin{equation}\label{E:Y-spine-n-1-1}
        \mathcal{F}(\psi(x)M_{ii}^{\kappa\kappa}) = 0
    \end{equation}
	for such $i$ (we stress that here the index $\kappa$ is being summed over but the index $i$ is not). It follows that
    \begin{equation}\label{E:Y-spine-n-1-2}
         \mathcal{F}(\psi(x)M^{\kappa\kappa}_{ij}) = 0
    \end{equation}
    provided at least one of $i,j$ is in $\{2,\dotsc,n\}$. Indeed, to see this, let $(f_p)_p$ be the sequence in $W^{1,2}(B^n_1(0);\A_Q(\R^k))$ obeying $\mathcal{E}_{f_p}\to \mathcal{F}$. Then, for any such $\psi$, we have
    \begin{align*}
        \left|\mathcal{E}_{f_p}(\psi(x)M^{\kappa\kappa}_{ij})\right| & = \left|\int_{B^n_1(0)}\psi(x)\del_if^\kappa_p\del_jf^\kappa_p\, \ext x\right|\\
        & \leq \int_{B^n_1(0)}|\psi(x)|\|\del_if_p\|\cdot\|\del_jf_p\|\, \ext x\\
        & \leq \left(\int_{B^n_1(0)}|\psi(x)| \|\del_if_p\|^2\right)^{1/2} \left(\int_{B^n_1(0)}|\psi(x)| \|\del_jf_p\|^2\right)^{1/2}\\
        & = \sqrt{\mathcal{E}_{f_p}(|\psi(x) |M_{ii}^{\kappa\kappa})}\sqrt{\mathcal{E}_{f_p}(|\psi(x)|M_{jj}^{\kappa\kappa})}.
    \end{align*}
    Thus, taking $p\to\infty$ we get
    $$\left|\mathcal{F}(\psi(x)M_{ij}^{\kappa\kappa})\right| \leq \sqrt{\mathcal{F}(|\psi(x)|M_{ii}^{\kappa\kappa})} \sqrt{\mathcal{F}(|\psi(x)|M_{jj}^{\kappa\kappa})}.$$
    Hence, \eqref{E:Y-spine-n-1-2} follows from this and \eqref{E:Y-spine-n-1-1}.

    By \eqref{E:Y-spine-n-1-2}, the inner variation formula for $\mathcal{F}$ reduces to
    $$\mathcal{F}(2M_{11}^{\kappa\kappa}\del_1\phi^1 - M_{11}^{\kappa\kappa}\div(\phi)) = 0$$
    for all $\phi = (\phi^1,\dotsc,\phi^n)\in C^1_c(B^n_1(0);\R^n)$. Replacing $\phi^1$ by $-\phi^1$ in the above we get
    $$\mathcal{F}(M^{\kappa\kappa}_{11}\div(\phi)) = 0 \qquad \text{for all }\phi\in C^1_c(B^n_1(0);\R^n).$$
    Now define $\mathscr{G}:C^0_c(B^n_1(0))\to \R$ to be the continuous linear functional defined by $\mathscr{G}(\psi):= \mathcal{F}(\psi(x)M^{\kappa\kappa}_{11})$. By the Riesz representation theorem, we can regard $\mathscr{G}$ as a (signed) Radon measure. By the above, we know that $\mathscr{G}(\div(\phi))=0$ for all $\phi\in C^1_c(B^n_1(0);\R^n)$, i.e.~the distributional derivative of $\mathscr{G}$ is zero. By a standard argument (using mollification to approximate $\mathscr{G}$ by a smooth function with zero derivative) it follows that $\mathscr{G}$ is a constant multiple of the Lebesgue measure, i.e.~there exists a constant $m\geq 0$ such that
    $$\mathscr{G}(\psi) = \int_{B^n_1(0)}m\psi(x)\, \ext x \qquad \text{for all }\psi\in C^0_c(B^n_1(0)).$$
    So by definition,
    \begin{equation}\label{E:Y-spine-n-1-3}
    \mathcal{F}(\psi(x)M^{\kappa\kappa}_{11}) = \int_{B^n_1(0)}m\psi(x)\, \ext x \qquad \text{for all }\psi\in C^0_c(B^n_1(0)).
    \end{equation}
    But \eqref{E:Y-spine-n-1-3} implies, by the definition of $\mathcal{D}(r)$, that $\mathcal{D}(r) = Dr^2$ for some $D\geq 0$. Then recalling \eqref{E:Y-freq-3}, we know that $\mathcal{H}^\prime(r) = \frac{2}{r}\mathcal{D}(r) = 2Dr$ and so since $\lim_{r\to 0}\mathcal{H}(r) = 0$ (by Proposition \ref{prop:Y-bounds}) we therefore get that $\mathcal{H}(r) = Dr^2$ (in particular, $D>0$ as $\mathcal{H}(r)>0$). But then $\mathcal{H}(r) = \mathcal{D}(r)$ giving that $\mathcal{I}(r) = 1$ for all $r>0$, meaning that $I = \lim_{r\to 0}\mathcal{I}(r) = 1$, as claimed.
\end{proof}

\bibliographystyle{alpha} 
\bibliography{references}

\end{document}